\documentclass[aop]{imsart}

\RequirePackage{amsthm,amsmath,amsfonts,amssymb}
\RequirePackage[numbers]{natbib}
\RequirePackage[colorlinks,citecolor=blue,urlcolor=blue]{hyperref}
\RequirePackage{graphicx}
\usepackage{enumerate}

\startlocaldefs
\numberwithin{equation}{section}
\theoremstyle{plain}
 \newtheorem{theorem}{Theorem}[section]
 \newtheorem{lemma}[theorem]{Lemma}
 \newtheorem{prop}[theorem]{Proposition}
 \newtheorem{cor}[theorem]{Corollary}
 \newtheorem*{theorem*}{Theorem}
 \newtheorem{claim}[theorem]{Claim}
\theoremstyle{remark}
 \newtheorem{remark}[theorem]{Remark}

 \newtheorem{example}[theorem]{Example}

\newenvironment{enumeratea}{\begin{enumerate}[\upshape (a)]}{\end{enumerate}}

\newcommand\nc\newcommand
\nc\dmo\DeclareMathOperator
\nc\bs\boldsymbol

\nc{\al}{\alpha} \nc{\Al}{\Alpha}
\nc{\be}{\beta} \nc{\Be}{\Beta}
\nc{\gam}{\gamma} \nc{\Gam}{\Gamma}
\nc{\ep}{\epsilon} \nc{\Ep}{\Epsilon}
\nc{\vep}{\varepsilon}
\nc{\de}{\delta} \nc{\De}{\Delta}
\nc{\lm}{\lambda} \nc{\Lm}{\Lambda}
\nc{\ka}{\kappa} \nc{\Ka}{\Kappa}
\nc{\vpi}{\varpi} \nc{\vph}{\varphi}
\nc{\om}{\omega} \nc{\Om}{\Omega}
\nc{\si}{\sigma} \nc{\Si}{\Sigma}

\nc{\A}{\mathbb{A}}
\nc{\B}{\mathbb{B}}
\nc{\C}{\mathbb{C}}
\nc{\D}{\mathbb{D}}
\nc{\E}{\mathbb{E}}
\nc{\F}{\mathbb{F}}
\nc{\G}{\mathbb{G}}
\nc{\bI}{\mathbb{I}}
\nc{\J}{\mathbb{J}}
\nc{\K}{\mathbb{K}}

\nc{\M}{\mathbb{M}}
\nc{\N}{\mathbb{N}}
\nc{\bO}{\mathbb{O}}

	\renewcommand{\P}{\mathbb{P}}
\nc{\Q}{\mathbb{Q}}
\nc{\R}{\mathbb{R}}
\nc{\bS}{\mathbb{S}}
\nc{\T}{\mathbb{T}}
\nc{\U}{\mathbb{U}}
\nc{\V}{\mathbb{V}}
\nc{\W}{\mathbb{W}}
\nc{\X}{\mathbb{X}}
\nc{\Y}{\mathbb{Y}}
\nc{\Z}{\mathbb{Z}}

\nc{\cA}{\mathcal{A}}
\nc{\cB}{\mathcal{B}}
\nc{\cC}{\mathcal{C}}
\nc{\cD}{\mathcal{D}}
\nc{\cE}{\mathcal{E}}
\nc{\cF}{\mathcal{F}}
\nc{\cG}{\mathcal{G}}
\nc{\cH}{\mathcal{H}}
\nc{\cI}{\mathcal{I}}
\nc{\cJ}{\mathcal{J}}
\nc{\cK}{\mathcal{K}}
\nc{\cL}{\mathcal{L}}
\nc{\cM}{\mathcal{M}}
\nc{\cN}{\mathcal{N}}
\nc{\cO}{\mathcal{O}}
\nc{\cP}{\mathcal{P}}
\nc{\cQ}{\mathcal{Q}}
\nc{\cR}{\mathcal{R}}
\nc{\cS}{\mathcal{S}}
\nc{\cT}{\mathcal{T}}
\nc{\cU}{\mathcal{U}}
\nc{\cV}{\mathcal{V}}
\nc{\cW}{\mathcal{W}}
\nc{\cX}{\mathcal{X}}
\nc{\cY}{\mathcal{Y}}
\nc{\cZ}{\mathcal{Z}}

\nc{\eps}{\varepsilon}
\nc{\epp}{\epsilon}
\nc{\ii}{\mathrm{i}}
\renewcommand{\d}{\,d}
\nc{\ls}{\lesssim}
\nc{\gs}{\gtrsim}
\def \lf {\lfloor}
\def \rf {\rfloor}
\dmo{\I}{{I}}
\dmo{\II}{{II}}

\DeclareFontFamily{U}{mathx}{\hyphenchar\font45}
\DeclareFontShape{U}{mathx}{m}{n}{
      <5> <6> <7> <8> <9> <10>
      <10.95> <12> <14.4> <17.28> <20.74> <24.88>
      mathx10
      }{}
\DeclareSymbolFont{mathx}{U}{mathx}{m}{n}
\DeclareFontSubstitution{U}{mathx}{m}{n}
\DeclareMathAccent{\widecheck}{0}{mathx}{"71}
\nc{\ol}{\overline}
\renewcommand{\t}{\tilde}
\nc{\ul}{\underline}
\nc{\wt}{\widetilde}
\nc{\wh}{\widehat}
\nc{\wch}{\widecheck}

\dmo{\adj}{adj}
\dmo{\Hom}{Hom}
\dmo{\proj}{proj}
\dmo{\tr}{tr}
\dmo{\Tr}{Tr}
\nc{\Span}{\operatorname{span}}
\def \tran {\mathsf{T}}
\def \HS {\mathrm{HS}}

\dmo{\rank}{rank}

\dmo{\real}{Re}

\dmo{\argmin}{arg\,min}
\dmo{\argmax}{arg\,max}
\dmo{\area}{area}
\dmo{\diag}{diag}
\dmo{\df}{df}
\dmo{\diam}{diam}
\dmo{\esssup}{ess\,sup}
\dmo{\im}{im} 
\dmo{\loc}{loc}
\dmo{\logit}{logit}
\dmo{\osc}{osc}
\dmo{\rad}{rad}
\dmo{\sgn}{sgn}
\dmo{\supp}{supp}
\dmo{\vol}{vol}
\dmo{\br}{br}
\nc{\lp}{\left(}
\nc{\rp}{\right)}
\nc{\lb}{\left[}
\nc{\rb}{\right]}
\nc{\lset}{\left\{}
\nc{\rset}{\right\}}

\dmo{\Var}{Var}
\dmo{\Cov}{Cov}
\dmo{\1}{\mathbf{1}}
\dmo{\ind}{\mathbf{1}}
\nc{\eqd}{\stackrel{\text{\tiny $d$}}{=}}

\dmo{\Ber}{Ber}
\dmo{\Bin}{Bin}
\dmo{\Poi}{Poi}
\nc{\ER}{Erd\H{o}s--R\'enyi }

\nc{\asto}{\stackrel{a.s.}{\to}}
\nc{\dto}{\stackrel{d}{\to}}
\nc{\Lpto}[1]{\stackrel{L^{#1}}{\to}}
\nc{\pto}{\stackrel{p}{\to}}
\nc{\wkto}{\Ra}
\nc{\aslra}{\stackrel{a.s.}{\lra}}
\nc{\dlra}{\stackrel{d}{\lra}}
\nc{\Lplra}[1]{\stackrel{L^{#1}}{\lra}}
\nc{\plra}{\stackrel{p}{\lra}}
\nc{\wklra}{\stackrel{wk}{\lra}}

\dmo{\Sparse}{Sparse}
\dmo{\Comp}{Comp}
\dmo{\Incomp}{Incomp}

\nc{\sphere}{\mathbb{S}}
\nc{\sphereN}{{\sphere^{N-1}}}
\nc{\ball}{\mathbb{B}}
\nc\lam\lambda
\nc{\Sym}{\cH}
\nc{\Good}{\cG}
\nc{\Evalx}{\cE}
\nc{\psimax}{{\psi_\mu^{\sup}}}
\nc{\psiinfty}{{\psi_\mu^{\lim}}}
\nc{\psimin}{{\psi_\mu^{\inf}}}
\nc{\psiright}{{\psi_\mu^{+\infty}}}
\nc{\psileft}{{\psi_\mu^{-\infty}}}
\nc{\nn}{{n_0}}
\nc{\Uloc}{{\mathsf{U}}}
\nc{\Deloc}{{\mathsf{D}}}
\nc{\Bdeloc}{{\Deloc}}
\nc{\Aset}{{\mathsf{A}}}
\nc{\Bset}{{\mathsf{B}}}
\nc{\Lpot}{{\mathsf{V}}}
\nc{\approxd}{\stackrel{\text{\tiny $d$}}{\approx}}
\nc{{\tv}}{{\tilde v}}
\dmo{\error}{{Err.}}
\dmo{\DKL}{{H}}
\nc{\gsig}{{G_\sigma}}
\nc{\gsigx}{{\gsig(x)}}
\nc{\overlap}{{q}}
\dmo{\goe}{{GOE}}
\nc{\ham}{{h}}
\nc{\cushion}{{\tau}}
\dmo{\Lip}{{Lip}}
\nc{\vloc}{{z}}
\nc{\rate}{{\cI}}
\nc{\local}{{loc}}
\nc{\free}{{\varphi}}
\nc{\freeL}{{f_N}}		
\nc{\freeD}{{\free^{del}}}
\dmo{\LLa}{{\Lambda}}
\nc{\LLamu}{{\LLa_\mu}}
\nc{\psimu}{{\psi_\mu}}
\nc{\thetam}{{\theta_x^-}}
\nc{\thetap}{{\theta_x^+}}
\nc{\thetapm}{{\theta_x^\pm}}
\nc{\thetamy}{{\theta_y^-}}
\nc{\VP}{{\Phi}}
\nc{\tLL}{{\wt\Lambda}_\mu}
\nc{\ssq}{{\mathsf{sq}}}
\nc{\vcont}{{\mathrm{v}}}
\nc{\tfree}{\wt\free}
\nc{\hfree}{{\wh\free}}
\nc{\hcJ}{{\wh\cJ}}
\nc{\thresh}{{\xi}}
\nc{\heta}{{\eta_0}}
\nc{\event}{{\cE}}
\nc{\good}{{\Good}}
\nc{\sfA}{\cA}
\nc{\sfB}{\cB}
\nc{\sfD}{\cD}
\nc{\xcush}{{\kappa}}
\nc{\zcush}{{\rho}}
\nc{\tcush}{{\tau}}
\nc{\wcush}{{\zcush_0}}
\nc{\AS}{{S}}
\nc{\srad}{{r}}
\nc{\chw}{{\wch w}}
\nc{\tw}{{\wt w}}
\nc{\chz}{{\wch z}}
\nc{\tvloc}{{\wt\vloc}}
\nc{\cha}{{\wch\al}}
\nc{\tal}{{\wt\al}}

\usepackage{color}
\nc{\red}[1]{{\color{red} #1}}
\definecolor{purple}{rgb}{0.6,0,0.8}
\nc{\purple}[1]{{\color{purple} #1}}
\nc{\revision}[1]{\purple{#1}}

\endlocaldefs

\begin{document}

\begin{frontmatter}
\title{Full large deviation principles for the largest eigenvalue of sub-Gaussian Wigner matrices} 
\runtitle{Large deviations for the largest eigenvalue of Wigner matrices}

\begin{aug}
\author[A]{\fnms{Nicholas A.}~\snm{Cook}\ead[label=e1]{nicholas.cook@duke.edu}},
\author[B]{\fnms{Rapha\"el}~\snm{Ducatez}\ead[label=e2]{ducatez@math.univ-lyon1.fr}}
\and
\author[C]{\fnms{Alice}~\snm{Guionnet}\ead[label=e3]{alice.guionnet@ens-lyon.fr}}
\address[A]{Department of Mathematics, Duke University, 
120 Science Dr, Durham, NC 27710, USA\printead[presep={,\ }]{e1}}

\address[B]{Institut Camille Jordan, Universit\'e Claude
Bernard Lyon 1, UMR 5208\printead[presep={,\ }]{e2}}

\address[C]{CNRS, ENS de Lyon, 46 all\'ee d'Italie, 69007, Lyon, France\printead[presep={,\ }]{e3}}
\end{aug}

\begin{abstract}

We establish precise upper-tail asymptotics and large deviation principles for the rightmost eigenvalue $\lambda_1$ of Wigner matrices with sub-Gaussian entries. In contrast to the case of  heavier tails, where deviations of $\lambda_1$ are due to the appearance of a few large entries, and the sharp sub-Gaussian case that is governed by the collective deviation of entries in a delocalized rank-one pattern, we show that the general sub-Gaussian case is determined by a mixture of localized and delocalized effects.

Our key result is a finite-$N$ approximation for the upper tail of $\lambda_1$ by an optimization problem involving \emph{restricted annealed free energies} for a spherical spin glass model. This new type of argument  allows us to derive full large deviation principles when the log-Laplace transform of the entries' distribution $\mu$ has bounded second derivative, whereas previous results required much more restrictive assumptions, namely sharp sub-Gaussianity and symmetry, or only covered certain ranges of deviations.

We show that the sharp sub-Gaussian condition characterizes measures $\mu$ for which the rate function coincides with that of the Gaussian Orthogonal Ensemble (GOE). When $\mu$ is not sharp sub-Gaussian, at a certain distance from the bulk of the spectrum there is a transition from the GOE rate function to a non-universal rate function depending on $\mu$, and this transition coincides with the onset of a localization phenomenon for the associated eigenvector.

\end{abstract}



\end{frontmatter}
\setcounter{tocdepth}{1}
\tableofcontents


\section{Introduction}

\subsection{Background}
\label{sec:background}

Large random matrices appear in a wide variety of domains. They were first introduced in statistics in the work of Wishart \cite{wishart} to analyze large arrays of noisy data, and their relevance for principal component analysis and statistical learning persists to the present day.  In numerical analysis, Goldstine and Von Neumann considered random matrices to model the propagation of rounding errors in numerical algorithms \cite{GoVN}; more recently, they played an important role in Spielman and Teng's theory of smoothed analysis of algorithms \cite{ST02}. 
Wigner \cite{wigner} and Dyson \cite{dyson} conjectured that random matrix eigenvalue statistics model those of high energy levels in heavy nuclei. Even more surprisingly, Montgomery \cite{Montgo}  showed a connection with statistics of zeros of the Riemann zeta function, leading to far-reaching conjectures which nowadays provide important heuristics for the distribution of the primes, see e.g \cite{KeatingSnaith,ABB}. 
Random matrices have also played a central role in operator algebra theory  since  Voiculescu  proved that they are asymptotically free \cite{voicstflour,voi91}. 
They have been applied to model the stability of large dynamical systems such as food webs \cite{May72} and neural networks \cite{RaAb:neural},
and have recently played a central role in the study of the complexity of random energy landscapes \cite{ABA,BAMMN,BBMcK}.

The computation of  the joint law of the eigenvalues of the Gaussian ensembles  goes back to Weyl \cite{weyl} and Cartan \cite{cartan}, who showed that this distribution is characterized by a density proportional to a power of the Vandermonde determinant of the eigenvalues.
 As a consequence, the eigenvalues of random matrices furnish an example of strongly interacting particles system, in connection with  many other models such as Coulomb gases or random tilings. 
 
For all these reasons, the study of Large Random Matrices   has grown into a diverse and mature field during the last forty years, yielding answers to increasingly sophisticated questions. In this article we are concerned with large deviations for the largest eigenvalue of large random matrices. Such a question appeared in various contexts such as statistics \cite{BDMN11}, the complexity of random functions \cite{ABAC,ABA,BAMMN}, their relation with fluctuations \cite{Majum2,MS14}  or in statistical mechanics where similar questions were attacked in the more general context of spin glasses \cite{PR08,FD14,DZ15,LFD}.
 
We consider the real Wigner random matrix model: with $N$ large or going to infinity, let $(X_{ij})_{1\le i\le j\le N}$ be iid real random variables having distribution $\mu$ with mean zero and variance one, and let $H$ denote the real symmetric $N\times N$ matrix with entries
\begin{equation}	\label{def:H}
H_{ij} = \sqrt{\frac{2^{1_{i=j}}}{N}} X_{ij} \,.
\end{equation}
We assume $\mu$ has sub-Gaussian tails (see \eqref{subG-LLa} below).
If $\mu$ is the standard Gaussian measure then $H$ is matrix from the Gaussian Orthogonal Ensemble (GOE).
 We label the eigenvalues in non-increasing order $\lam_1(H)\ge\cdots\ge\lam_N(H)$ (we will usually drop the argument $H$ for brevity). 
 We restrict to the real case to keep the paper of a reasonable length, but note that complex Wigner matrices can be treated by the same arguments.
 
In order to provide some context for the large deviations behavior of the spectrum of $H$, which is the main concern of this article, we briefly recall what is known about the typical behavior of the spectrum 
for $H$ as above with sub-Gaussian entries.

\subsubsection*{Laws of Large Numbers}
The most basic problem is to determine the asymptotic locations of the eigenvalues to leading order. For the ``bulk'' of the eigenvalues this is addressed by Wigner's semicircle law \cite{wigner}:
writing $\sigma$ for the \emph{semicircle measure} with continuous compactly supported density $\sigma(dx) :=\frac1{2\pi}(4-x^2)_+^{1/2}dx$, and $\hat\mu_H:=\frac1N\sum_{i=1}^N \delta_{\lam_i(H)}$ for the empirical spectral distribution (ESD) of $H$, we have
\begin{equation}	\label{Wigner}
\int f d\hat\mu_H \to \int fd\sigma\qquad \text{ in probability}
\end{equation}
for any bounded continuous function $f:\R\to\R$.
In particular, with probability tending to 1, all but at most $o(N)$ of the eigenvalues are contained in the limiting support $[-2,2]$ (see Section \ref{sec:notation} for our conventions on asymptotic notation). 
By taking $f$ to approximate step functions $1_{(-\infty,x]}$ for fixed $x\in(-2,2)$ we deduce a law of large numbers for bulk eigenvalues: for $i=i_N\in [N]$,
\begin{equation}	\label{LLN.bulk}
i_N/N\to a\in(0,1) \quad\Longrightarrow\quad \lambda_{i_N}(H)\to x_a\quad\text{ in probability}
\end{equation}
where $x_a\in(-2,2)$ is the quantile such that $a=\sigma([x_a,2])$.
The law of large numbers behavior of ``edge'' eigenvalues was established much later by F\"uredi and Koml\'os \cite{furedi}, who showed
\begin{equation}	\label{FuKo}
\lambda_1(H)\to 2\qquad \text{ in probability}
\end{equation}
(and hence $\lam_N(H)\to -2$ in probability by considering $-H$ in place of $H$). We deduce that \eqref{LLN.bulk} holds for any sequence $i_N\in[N]$ such that $i_N/N$ tends to a limit (possibly 0 or 1).

\subsubsection*{Concentration}
Under further assumptions on $\mu$, general concentration of measure inequalities can be used to show that the left hand sides in \eqref{Wigner} (taking $f$ Lipschitz)  and \eqref{FuKo} concentrate around their limiting values with exponential rates of order $N^2$ and $N$, respectively \cite{GuZe}.
In a related direction, the celebrated \emph{local semicircle law} of Erd\" os--Schlein--Yau \cite{ESY09a} shows that $\int fd\hat\mu_H$ concentrates around $\int fd\sigma$ for $f$ supported on an interval of length  as small as $N^{-1+\eps}$, and related \emph{eigenvalue rigidity} results establish \eqref{LLN.bulk}--\eqref{FuKo} with near-optimal precision; we refer to the survey \cite{BeKn} for more on the local law and its consequences.

\subsubsection*{Fluctuations}
The next natural questions concern the size and law of the fluctuations of linear statistics $\int fd\hat\mu_H$ and the largest eigenvalue $\lam_1(H)$ around their limiting values of $\int fd\sigma$ and $2$, respectively.
The former have been shown to be Gaussian, with variance depending on the regularity of $f$; we refer to the recent work \cite{LaSo} for an overview of the large body of work on CLTs for linear statistics.
The fluctuations of individual eigenvalues, as well as eigenvalue gaps $\lam_i(H)-\lam_{i+1}(H)$, were originally understood for GOE matrices (where $\mu$ is the standard Gaussian measure) \cite{johansson, TW,ME}, and these were shown to be \emph{universal} in a series of remarkable breakthroughs
\cite{sosh,ESY, ESY11, EPRY, johanssonHS,TVun}.
In particular, the fluctuations of $\lam_1(H)$ are asymptotically described by the ($\beta=1$) Tracy--Widom distribution \cite{For93,TW,sosh}; notably, the scale $N^{-2/3}$ of fluctuations is smaller than the upper bound $O(N^{-1/2})$ implied by general concentration of measure estimates.

\subsubsection*{Large Deviations}
In contrast to the above results on typical behavior, the understanding of large deviations for the spectrum of Wigner matrices remains far less complete.
Here the problem is to estimate the probabilities of rare events that linear statistics or individual eigenvalues deviate significantly from their limiting values in \eqref{Wigner}--\eqref{FuKo}.

Recall that a function $\rate:\R\to[0,+\infty]$ is a \emph{good rate function} if it is lower semicontinuous, not identically $+\infty$, and its sub-level sets $\{\rate\le a\}$ are compact for all $a<\infty$.
A sequence of random variables $Y_N$ taking values in a Polish space $\cY$ satisfies a \emph{large deviation principle (LDP) with speed $r_N\to\infty$ and good rate function $\cI$} if 
\[
-\inf_{y\in E^\circ} \cI(y)\le \liminf_{N\to\infty}\frac1{r_N}\log\P(Y_N\in E) \le \limsup_{N\to\infty}\frac1{r_N}\log\P(Y_N\in E)\le -\inf_{y\in \overline{E}}\cI(y)
\]
for all Borel sets $E\subseteq \cY$.

LDPs were established for the spectrum of the Gaussian GOE matrix, for which the joint law of the eigenvalues has an explicit form, independent of the eigenvectors, displaying  a strong Coulomb gas interaction. This formula could be used to prove an LDP for the empirical measure in \cite{BAG97}, yielding LDPs with speed $N^2$ for linear statistics $\int fd\hat\mu_N$; an LDP for the largest eigenvalue {of GOE matrices was established in} \cite{BDG} (see also \cite{Majum} for further discussions of the Wishart case, and \cite{Majum2}). 
LDPs for the spectrum of deformed Gaussian matrices were obtained in \cite{GZ3,Maida:deformed,BGGM,Giulio}.

More recently, in a breakthrough paper,  C. Bordenave and P. Caputo \cite{BordCap}  tackled the case of matrices with tails heavier than Gaussians, that is  Wigner matrices with entries with stretched exponential tails, going to zero at infinity more slowly than a Gaussian tail. The driving idea to approach this question is to show that large deviations are in this case created by a few large entries -- an instance of what we call a \emph{localization phenomenon}. As a result, the empirical measure deviates towards the free convolution of the semicircle law and the limiting spectral measure of the matrix created by these few large entries. This idea could be also used to prove  the large deviations for the law of  the largest eigenvalue and spectral moments by F. Augeri \cite{Augeri:stretched,Augeri:moments}.  

Very recently in \cite{Augeri:esd}, Augeri has shown that for sparsified Wigner matrices with entries of the form $\frac1{\sqrt{N}}B_{ij}X_{ij}$ for iid bounded random variables $X_{ij}$ and independent  Bernoulli($p$) variables $B_{ij}$ with 
$p=o(1)$ and $p=\omega(N^{-1}\log N)$ (see Section \ref{sec:notation} for our conventions on asymptotic notation),
the ESD satisfies an LDP with rate function that is only finite on solutions of the Quadratic Vector Equations studied in \cite{AEK:QVE}. 
In general, we say that large deviations for a function of $d$ independent random variables exhibit a localization phenomenon when the driving mechanism for the deviation involves a deviation of $o(d)$ variables from their typical ranges.
Apart from the results of \cite{BordCap,Augeri:stretched} for Wigner matrices with stretched exponential tails, localization phenomena have been shown in recent years to govern large deviations for the extreme eigenvalues of adjacency matrices for sparse random graphs \cite{ChaVa,ChaDe,Augeri:ER,CoDe,BhGa,HMS,BaBa,BBG,Basak} and random networks \cite{GaNa,GHN,LeNa23,AuBa}.

In fact, neither the Gaussian case nor the case of tails heavier than Gaussians can cast light on the large deviations 
of the spectrum when the entries are compactly supported, or more generally sub-Gaussian. 
Recall that a random variable $X$ is sub-Gaussian if
\begin{equation}	\label{def:subG}
\E \exp( X^2/K^2)\le 2
\end{equation}
for some $K<\infty$, and the smallest such $K$ is called the \emph{sub-Gaussian constant} of $X$. 
When $X$ has centered distribution $\mu$, 
\eqref{def:subG} is equivalent up to modification of $K$ by a constant factor to
\begin{equation}	\label{subG-LLa}
\LLa_\mu(t) \le K^2 t^2 \qquad\forall t\in \R
\end{equation}
where we denote the (two-sided) log-Laplace transform of $\mu$ by
\begin{equation}	\label{def:LLa}
\LLa_\mu(t):=\log\int_\R e^{tx}d\mu(x)\,,\qquad t\in \R\,;
\end{equation}
see for instance \cite[Chapter 2]{Vershynin:book}.
For the standard Gaussian measure $d\gamma(x)= \frac1{2\pi} e^{-x^2/2}dx$ we have $\LLa_\gamma(t) = \frac12t^2$, and in particular \eqref{subG-LLa} holds with $K^2=\frac12$.

A breakthrough came in \cite{HuGu1}, establishing the following result for a wide class of sub-Gaussian Wigner matrices.

\begin{theorem}[Sharp sub-Gaussian case {\cite{HuGu1}}]
\label{thm:HuGu}
Assume the distribution $\mu$ is either compactly supported or satisfies a log-Sobolev inequality, and has log-Laplace transform satisfying the pointwise bound 
\begin{equation}	\label{def:sharp-subG}
\LLa_\mu(t) \le \LLa_\gamma(t)=\tfrac12t^2\qquad\forall t\in \R\,.
\end{equation}
Then $\lam_1(H)$ satisfies a large deviation principle with {speed $N$ and} good rate function given by
\begin{equation}	\label{def:igamma}
\rate^\gamma(x):= \begin{cases} \frac12\int_2^x\sqrt{y^2-4}dy & x\ge2\\ +\infty & x<2\,.\end{cases}
\end{equation}
In particular, for every fixed $x\in \R$ we have 
\begin{equation}	\label{tail.igamma}
\lim_{\delta\downarrow0}\limsup_{N\to\infty} \frac1N\log\P(|\lam_1(H)-x|\le \delta) 
=\lim_{\delta\downarrow0}\liminf_{N\to\infty} \frac1N\log\P(|\lam_1(H)-x|\le \delta) 
= -\rate^\gamma(x)\,.
\end{equation}
\end{theorem}

A measure $\mu$ satisfying \eqref{def:sharp-subG} is said to be \emph{sharp sub-Gaussian}.
Such distributions were recently studied in \cite{BCG:subgaussian}, where they were named ``strict sub-Gaussians''.
In addition to $\gamma$, this class of measures includes the important examples of the Rademacher distribution $\frac12(\delta_{+1}+\delta_{-1})$ and the uniform distribution on $[-\sqrt{3}, \sqrt{3}]$. 

\begin{remark}
The result of \cite{HuGu1} actually allows for the variables $X_{ij}$ to have varying sharp sub-Gaussian distributions $\mu_{ij}^N$, so long as the support/log-Sobolev conditions hold uniformly in $i,j$ and $N$.
In fact it is possible to remove the latter conditions using a truncation argument, leaving only the condition \eqref{def:sharp-subG}
{(see Appendix \ref{app:talagrand}; this was also recently noted in \cite{HuMc})}.
We further note that \cite{HuGu1} establishes an analogous result for complex Hermitian Wigner matrices, as well as sample covariance matrices. 
While the methods developed here could also be extended to such ensembles, we focus on the real Wigner case to keep the article of reasonable length.
\end{remark}

Perhaps the most remarkable feature of Theorem \ref{thm:HuGu} is that it establishes a \emph{universal} rate function \eqref{def:igamma} for $\lam_1(H)$ for a wide class of entry distributions $\mu$.
While universality is a pervading phenomenon in probability and random matrix theory, it is less common in large deviations theory, where results are usually sensitive to details of the tails of the input variables. (Compare for instance Cram\'er's theorem for the sample mean $\overline{X}_N=\frac1N(X_1+\cdots+X_N)$ of iid samples from $\mu$, where the rate function is given by the Legendre transform of $\LLa_\mu$, and is drastically different for the Gaussian and Rademacher cases.)

To establish Theorem \ref{thm:HuGu} a new strategy was introduced in \cite{HuGu1} based on tilting the law of $H$ by spherical integrals, given  for a symmetric $N\times N$ matrix $M$ and $\theta\ge0$ by
\begin{align}	\label{def:spherical}
I(M,\theta) & := \int_{\sphereN} e^{N\theta\langle u,Mu\rangle} dP(u)
\end{align}
where $P=P_N$ is the uniform measure on the unit sphere $\sphereN$ in $\mathbb R^{N}$. 
Spherical integrals are natural quantities to perform such a tilt as they play the role of the Laplace or Fourier transform \cite{GuMa05} in random matrix theory.  
The strategy of tilting by spherical integrals has since been applied to spectral large deviations problems in several works -- see for instance \cite{GMa20,McKenna:deformed,BHG,AGH,Giulio,HuMc,DGH24}.
We review this strategy in Sections \ref{sec:results.restricted} and \ref{sec:ideas} below.

The 
method of tilting by spherical integrals was applied to the case of symmetric sub-Gaussian $\mu$ not satisfying \eqref{def:sharp-subG} in \cite{AGH}.
Under rather general hypotheses it was shown that the probability that the largest eigenvalue is close to some value $x$ can be estimated when $x$ is large enough and the rate function is not the same as in the Gaussian case.
On the other hand, assuming the sub-Gaussian constant $K$ in \eqref{subG-LLa} is smaller than 1 (rather than $1/\sqrt{2}$ for the sharp-sub-Gaussian condition), it was shown that the large deviation rate function matches the GOE rate function $\rate^\gamma(x)$ in a neighborhood of 
$x=2$. 
For this class of distributions they hence obtained the LDP on $\R\setminus F_\mu$ for a compact set $F_\mu\subset (2, +\infty)$, with rate function $\rate^\mu(x)$ universal for small $x$ and non-universal for large $x$.
In fact, this transition from universality to non-universality can be detected by studying the limiting annealed spherical integral which  fails to be differentiable everywhere, see \cite[Proposition 7]{AGH}. The absence of differentiability of the limiting  log-Laplace transform of the variable  is a well known obstruction to derive a full LDP in Cram\'er-type proofs. 

The work \cite{AGH} hence left open the full LDP for $\lam_1$ on $\R$ and the nature of the transition to a non-universal limit somewhere in the intermediate range $F_\mu$. 
The authors of \cite{AGH}  noted there that the spherical integral method can only yield a convex rate function, and suggested that the true rate function $\rate^\mu$ may be non-convex in $F_\mu$, a prediction that we confirm in this article.  
 
In this work we greatly extend the spherical integrals method to permit the study of models with localization phenomena and non-convex rate functions. 
These innovations are especially useful for the study of models with structured distributions, and indeed the arguments developed here have already been applied in recent works on matrices with variance profiles \cite{DGH24} and deterministic shifts \cite{BoGu24}.

\subsection{Our contributions}
\label{sec:contributions}

In this article, we elucidate large deviations of  $\lam_1(H)$ for a wide class of sub-Gaussian distributions $\mu$, extending Theorem \ref{thm:HuGu} to obtain a full large deviation principle on all of $\R$ with a rate function $\rate^\mu$ that may be different from the GOE rate function $\rate^\gamma$. 
We show that under some mild technical conditions on $\mu$, the sharp sub-Gaussian assumption of Theorem \ref{thm:HuGu} in fact characterizes the universality class of distributions $\mu$ for which $\rate^\mu=\rate^\gamma$. 
Moreover, for sub-Gaussian $\mu$ that is not sharp sub-Gaussian, we still have $\rate^\mu(x)=\rate^\gamma(x)$ for $x$ in a neighborhood of $2$.
Our approach yields quantitative tail bounds, and moreover yields structural information on the associated eigenvector $v_1$ conditional on a large deviation event $\lam_1 \approx  x>2$. 
We find that
the transition to a non-universal rate function value $\rate^\mu(x)\ne \rate^\gamma(x)$ 
as $x$ exceeds a threshold value
is associated with the emergence of a 
localized component of $v_1$, i.e.\ a small set of large coordinates that carry a macroscopic fraction of the $\ell^2$-norm.

Roughly speaking, the reason for universality in the sharp sub-Gaussian case is that deviations of $\lam_1$ are due to a collective tilt of the matrix entries in the direction of a rank-one matrix $vv^\tran$ with $v$ a delocalized vector, so the rate function $\rate^\mu$ is ultimately determined by the behavior of $\LLa_\mu$ in a $o(1)$-neighborhood of 0, where the bound \eqref{def:sharp-subG} approaches an equality by Taylor expansion.
However, when the pointwise bound \eqref{def:sharp-subG} does not hold, competing \emph{localized} tilting strategies emerge, and the rate function $\rate^\mu$ depends on other details of $\mu$. 
One of our contributions is to determine
the rate function  by a \emph{mixture} of localized and delocalized strategies, which significantly complicates the analysis. This should be compared on the one hand with pure delocalization when $\mu$ has lighter tails (the sharp sub-Gaussian case) and pure localization when $\mu$ has heavier tails (the stretched exponential case or sparse random graphs/networks, as referenced previously). The general sub-Gaussian case is hence in some sense critical for the large deviations problem.

The following  highlights some of our results. 

\begin{theorem*}[Informal, see Theorems \ref{thm:smallx}, \ref{thm:fullLDP} and Corollary \ref{cor:ER}]\label{informal}
With $H$ as in \eqref{def:H}, assume $\mu$ is sub-Gaussian, centered and with unit variance. 
Under some further technical assumptions, there exists a good rate function $\rate^\mu$ on $\R$ that is infinite on $(-\infty,2)$ and continuous and non-decreasing on $[2,\infty)$ such that $\lam_1=\lam_1(H)$ satisfies a large deviation principle with speed $N$ and rate function $\rate^\mu$. In particular,
\eqref{tail.igamma} holds with $\rate^\mu$ in place of $\rate^\gamma$
for every fixed $x\in\R$.
Furthermore:
\begin{itemize}
\item $\rate^\mu\le\rate^\gamma$ pointwise on $\R$, i.e.\ large deviations are at least as likely as for the GOE case.
\item (Universality phase). There exists $x_\mu>2$ such that $\rate^\mu\equiv\rate^\gamma$ on $(-\infty,x_\mu]$.
Moreover, conditional on an event that $\lam_1$ lies in a small neighborhood of some fixed $x\in (2,x_\mu)$, with high probability the associated eigenvector $v_1$ is delocalized: 
for any fixed $\eps>0$ the $\ell^2$-mass of components larger than $N^{-1/2+\eps}$ is $o(1)$.

\item (Non-universality). If \eqref{def:sharp-subG} does not hold (i.e.\ $\mu$ is not sharp sub-Gaussian) then there exists $x_\mu'<\infty$ such that 
$\rate^\mu(x)<\rate^\gamma(x)$ for all $x>x_\mu'$.  
Moreover, conditional on an event that $\lam_1$ lies in a small neighborhood of 
some $x$ for which $\rate^\mu(x)<\rate^\gamma(x)$, with high probability the associated eigenvector $v_1$ has a localized component -- that is, $v_1$  has $\ell^2$-mass $\gs1$ on coordinates of size $\ge N^{-1/4-\eps}$. 
\end{itemize}

\end{theorem*}

See Figure \ref{fig:pgauss} for an example of the transition to non-universality for the rate function $\rate^\mu$ and the associated localization transition for the eigenvector $v_1$.

\begin{figure}
\includegraphics[width=6.7cm]{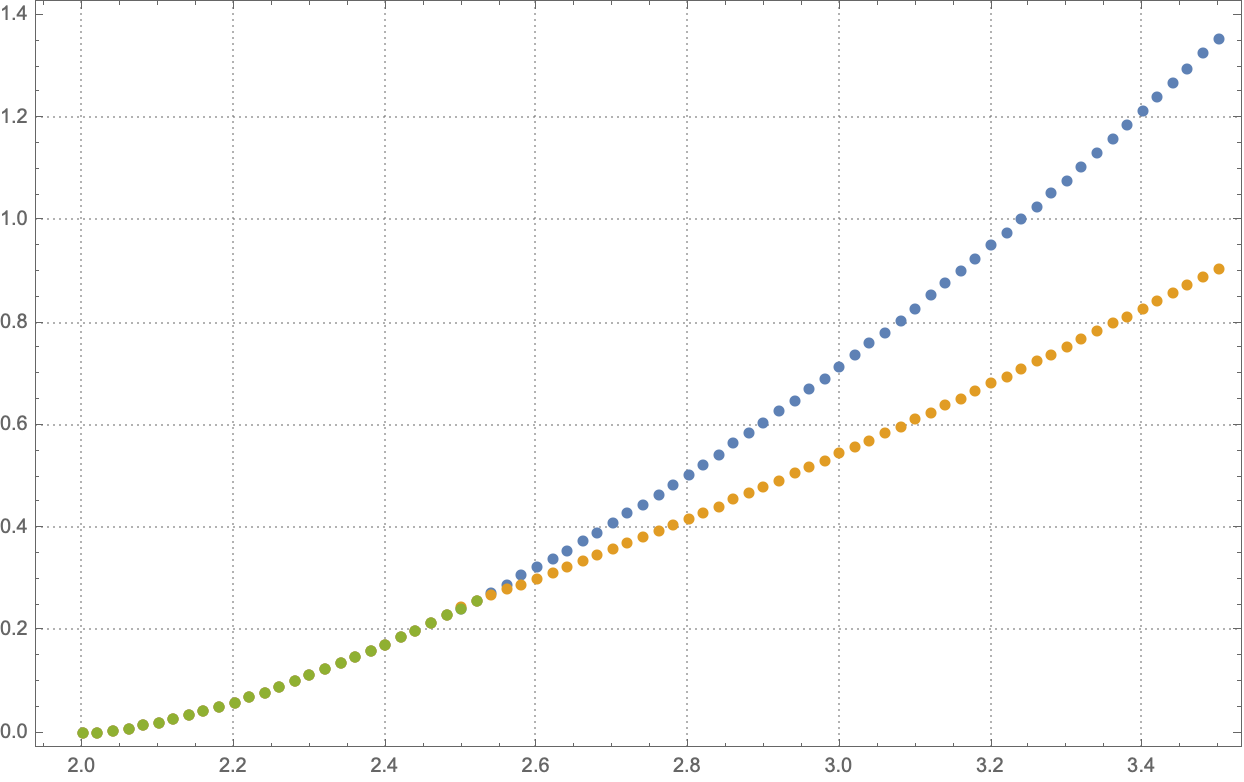}
\;
\includegraphics[width=6.7cm]{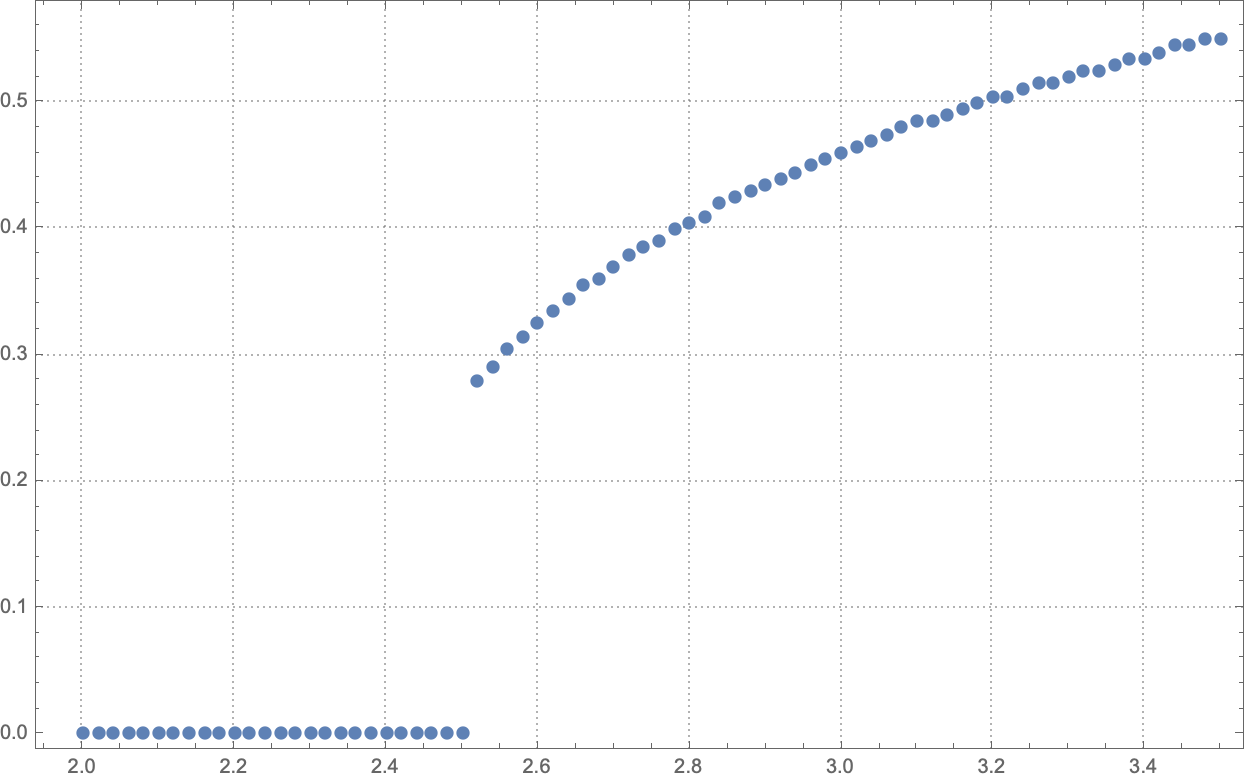}
\caption{\begin{small}
Large deviations for $\lam_1(H)$ for the case that $\mu$ is an equal mixture of the centered Gaussian of variance $2$ and the Dirac mass $\delta_0$ (the case $p=\frac12$ of Example \ref{ex:Gp}). Plots were made with Mathematica based on a variational formula for the rate function established in Theorem \ref{thm:increasing}; see \eqref{def:hfree}--\eqref{def:rate.increasing}. (The asymptotic restricted annealed free energy \eqref{def:hfree} in turn is computed using the non-variational formula noted in Remark \ref{rmk:tfree.alt}.) \\
{\bf Left:} Rate functions $\rate^\gamma(x)$ from \eqref{def:igamma} (green/blue) and $\rate^\mu(x)$ from \eqref{def:rate.increasing} (green/yellow), plotted for $x\in[2,3.5]$ (on a mesh of spacing $.02$). The rate functions match up to $x_\mu\approx 2.52$. 
Note that $\rate^\mu$ is not convex. \\
{\bf Right:} Optimizers $\al_x^*$ of $\al$ in \eqref{def:rate.increasing} are plotted for $x\in[2,3.5]$. From Theorem \ref{thm:increasing}(b), conditional on $\lam_1\approx x$ we have $\|v_1\|_\infty^2 \approx\al_x^*$ with high probability. The point $x_\mu$ above which $\rate^\mu(x)$ is non-universal coincides with a discontinuous jump in $\al_x^*$ from $0$ to $\approx 0.280$. 
\end{small}}
\label{fig:pgauss}
\end{figure}

We note that from standard estimates (see Lemma \ref{lem:tightness}) it follows that $\rate^\mu(x)$ grows at least quadratically as $x\to+\infty$.\\

Some further results we obtain:
\begin{itemize}
\item In Theorem \ref{thm:increasing} we provide a more detailed description of the localization transition for $v_1$ under the assumption that $\mu$ is symmetric and the ratio of log-Laplace transforms $\LLa_\mu(t)/\LLa_\gamma(t)$ is strictly increasing on $\R^+$. In this case the localized part of $v_1$ in the non-universality phase is supported on a single entry. (See Figure \ref{fig:pgauss}.)  
In the case where $\mu$ is compactly supported, conditional on a sufficiently large deviation of $\lambda_1$ the localized part of $v_1$ has roughly ${N}^{1/2}$ entries of order $N^{1/4}$, see  \cite[Proposition 15]{AGH}.  Moreover, our proof shows that if all the entries of $v_{1}$ are of order at most $N^{-1/4-\eps}$, we are in the Gaussian universality regime. Indeed, a key quantity is given by 
$\freeL(\theta,w)$ defined in \eqref{def:freeL}, where $w$ approximates the localized part of $v_1$. The Laplace transform $\LLa_\mu$ in \eqref{def:freeL} can be approximated by Taylor expansion to second order
(hence giving the same function as when $\mu$ is Gaussian)  only when all the entries  $\sqrt Nw_iw_j $ go to zero. 
This heuristically explain why the scale $N^{-1/4}$ is critical.

\item Corollary \ref{cor:ER} obtains the LDP for the largest eigenvalue of centered adjacency matrices for dense \ER graphs $G(N,p)$ at the scale $O(\sqrt{Np})$ of the bulk of the spectrum, complementing recent works covering larger deviations at scale $\gg \sqrt{Np}$, where the LDP is given by a simpler na\"ive mean-field approximation \cite{ChVa2,LuZh12,CoDe, BhGa}. 
For deviations at the scale of the bulk spectrum the na\"ive mean-field approximation is invalid and the LDP is more complicated, being related to the free energy of a disordered spin glass model.

\item Theorem \ref{thm:rateN} provides an asymptotic equivalent $\rate^\mu_N$ for the large deviations rate under quite general assumptions, as a consequence of quantitative estimates relating large deviation probabilities to a variational problem for \emph{restricted annealed free energies} of a spherical spin glass model -- see Theorem \ref{theomain} and Propositions \ref{prop:upper} and \ref{prop:lower0}.
\end{itemize}

We want to emphasize here that our technical assumptions are very mild, see \eqref{assu:usg}, \eqref{assu:LRlim} and 
\eqref{assu:maxR+}. In particular, we do not assume that $\mu$ is symmetric as in \cite{AGH,HuGu1}.
Further open problems are described in Section \ref{sec:open}.

Outside the hypotheses of Theorem \ref{thm:increasing}, the localized part of $v_1$ may be supported on a growing number of coordinates.
For instance, when $\mu$ is compactly supported localization happens on sets of size $\sqrt{N}$ and is more difficult to quantify. 
This is the reason why we proceed in two steps: we first get dimension-dependent estimates for localization scenarios, and then show that these estimates  converge as the dimension goes to infinity.
In fact, in the non-universal range where $\rate^\mu(x)<\rate^\gamma(x)$, the conditional structure of $v_1$ depends strongly on $\mu$ and our work is the first to describe it precisely.

In a notable recent work \cite{GHN}, a transition from a universal rate function for light-tailed entries to non-universal rate functions for heavy-tailed entries was also shown to occur for large deviations for the largest eigenvalue of diluted random matrices, where entries above the diagonal are independent and of the form $B_{ij}Y_{ij}$ for $B_{ij}\sim$Ber($\frac dn$) for a fixed $d>0$, and $Y_{ij}$ independent of $B_{ij}$ with stretched-exponential (Weibull) tails. In that setting, all large deviation mechanisms are of a localized nature, occurring on stars and cliques in the associated sparse graph.  
Both localization and delocalization phenomena were  shown to appear in the simpler setting of large deviations for quadratic forms in Gaussian random variables \cite{BGR};
 this is to  our knowledge the  only precursor to what we shall see happens for matrices with sub-Gaussian entries. However, localization for  quadratic forms in Gaussian random variables happens on a single site which is not always the case here.

\subsection{Innovations of the proof} 
Beyond
the generality and novelty of our results, we introduce several new ideas and techniques that should prove useful in other contexts. 
To highlight a few:
\begin{enumerate}

\item\label{new-restr} 
We greatly extend the spherical integral method from \cite{HuGu1} to permit the study of models with localization phenomena and non-convex rate functions (such as in Figure \ref{fig:pgauss}). 
The key to this is the use of \emph{restricted spherical integrals} where the integration in \eqref{def:spherical} is only taken over a part of the sphere  determined by the localized part of the leading eigenvector $v_1$.

As a general strategy, the use of restricted spherical integrals based on low-entropy data about top eigenvectors can be useful for the study of structured models. Indeed, following an earlier preprint version of this article the strategy has been applied to matrices with a variance profile in \cite{DGH24}, removing the assumptions on the variance profile required in \cite{Hu}. 

\item\label{new-joint} 
The restricted spherical integral method gives more than LDPs for $\lam_1$: we get joint large deviation estimates for $\lam_1$ and the localized part of $v_1$ (roughly defined as the coordinates of size $\gg N^{-1/4}$). The sharp LDP upper bounds are then deduced as contractions of the joint tail estimates by optimizing the joint rate function to identify the least unlikely structure of $v_1$. A stability analysis of this optimization then yields the typical structure of $v_1$ conditional on the large deviation event $\{\lam_1 \approx x\}$. 

\item\label{new-cont} 

Proving that the upper bound obtained from optimizing restricted spherical integrals is sharp requires substantial new arguments.
Whereas in the sharp sub-Gaussian case the optimal tilt could be located by a BBP phase transition computation, this is no longer possible when tilting by spherical integrals restricted to vectors $u$ with a localized component. 
To deal with this we develop a robust continuity argument, showing that the 
restricted annealed free energy (defined in Section \ref{sec:results.restricted}) localizes to a smaller portion of the sphere that varies continuously with the parameter $\theta$. Together with concentration and coupling arguments, we can show the mean of $\lam_1$ under the tilted law varies continuously with $\theta$, allowing us to locate the optimal tilting parameter $\theta_*$ using the intermediate value theorem.

\item 

We highlight a novel Markov chain argument to prove the rate function $\rate^\mu(x)$ for $x$ is monotone on 
$[2,\infty)$ (see Lemma \ref{lem:monotone}). The argument should also apply in more general situations, and has recently been applied in \cite{BoGu24} to shifted Wigner matrices.

The idea is to design a chain on the space of $N\times N$ symmetric matrices that has the law of $H$ as its stationary distribution, and for which $\lam_1$ can only change a small amount at each step. Thus, when the chain is initialized in the event $\{\lam_1 \approx   x\}$ it  must pass through $\{\lam_1 \approx   y\}$ on the way to the typical event $\{\lam_1 \approx   2\}$ for any $y\in (2,x)$, allowing us to compare the probabilities.

\item 

In general the optimization of the localized part of $v_1$ takes place over a space of growing dimension, and it is initially not clear how to establish an $N$-independent rate function for $\lam_1$. 
For this we develop a pigeonholing argument to separate scales in the entries of $v_1$ and transform to an optimization problem over an increasing sequence of compact subsets of the ball in $\ell^2(\N)$, allowing us to deduce existence of the limit from monotonicity. 
See the proof of Proposition \ref{prop:weakLDP}.

\end{enumerate}

We further discuss the proof ideas in Section \ref{sec:ideas}.

\subsection{Structure of the paper}

In Section \ref{sec:results} we state our model assumptions and main results.
Section \ref{sec:open} lists some open questions and directions for future work.
In Section \ref{sec:ideas} we outline the main ideas for the proof of our core results relating the upper tail for $\lam_1$ to a minimax problem for spherical integrals (Theorems \ref{thm:rateN} and \ref{theomain}). 
Section \ref{sec:notation} summarizes our notational conventions. 
In Section \ref{sec:high} we prove most of our main results after stating our main technical lemmas, which we prove in Sections \ref{sec:upper}--\ref{sec:rateprops}.
Corollary \ref{cor:nonuniv} on the transition from the GOE rate function to a nonuniversal rate function is proved in Section \ref{sec:nonuniv}, and Theorem \ref{thm:increasing} for the case that the tails of $\mu$ have a certain monotonicity property is proved in Section \ref{sec:increasing}.
Various technical lemmas of a more standard nature are proved in the appendices.

\section{Main results}
\label{sec:results}

\subsection{Model assumptions}
\label{sec:assu}

Recall the Wigner random matrix model $H$ introduced in \eqref{def:H}.
Throughout the article we make the following assumptions on the distribution $\mu$ of the rescaled entries $\sqrt{N/2^{1_{i=j}}}H_{ij}$. 
We assume throughout that the probability measure $\mu$ is centered with unit second moment (that is, $\int_\R xd\mu(x)=0$ and $\int_\R x^2d\mu(x) = 1$) 
and sub-Gaussian.
Recalling the log-Laplace transform $\LLa_\mu(t)$ from \eqref{def:LLa}, we additionally define
\begin{equation}
\label{def:psimu}
\psimu(t):= \frac1{t^2}\LLa_\mu(t) = \frac12\frac{\LLa_\mu(t)}{\LLa_\gamma(t)}\,.
\end{equation}
Since $\mu$ is standardized, 
$\LLa_\mu(t)=\frac12t^2+O(t^3)$ as $t\to0$, so we  take $\psimu(0):=\frac12$ to extend $\psimu$ continuously to 0.
 For the Gaussian measure $\gamma$ we have $\LLa_\gamma(t) = \frac12t^2$ and $\psi_\gamma(t)\equiv \frac12$.
The sub-Gaussian condition \eqref{subG-LLa} can be restated as
\begin{equation}	\label{assu:sg}
\psimax:= \sup_{t\in\R} \psimu(t) <\infty\,.	\tag{SG}
\end{equation}
The sharp sub-Gaussian condition \eqref{def:sharp-subG} is thus that $\psimax=\frac12$. 
Our large deviation results also depend on the tails of $\mu$ through the parameter
\begin{equation}
\label{def:psiinfty}
\psiinfty:= \limsup_{|t|\to\infty} \psimu(t)\,.
\end{equation}
Note that whenever $\mu$ is compactly supported we have $\LLa_\mu(t) \ls |t|$ and hence $\psiinfty=0$ in this case. 

We note that  $\psimax$, $\psiinfty$ correspond respectively to parameters $A/2$, $B/2$ from \cite{AGH}. 

For technical reasons, in our main results we impose the following strengthening of \eqref{assu:sg}, that $\mu$ has \emph{uniformly sub-Gaussian tilts:}
\begin{equation}	\label{assu:usg}
\sup_{t\in\R}\LLa_\mu''(t)<\infty\,. 	\tag{USG}
\end{equation}
Since $\LLa_\mu(0)=\LLa_\mu'(0) =0$, \eqref{assu:usg} implies \eqref{assu:sg}.
The former is equivalent to the assumption that the exponentially tilted measures  $d \mu^{t}(x) =e^{tx-\LLa_\mu(t)} d\mu(x)$, after recentering, are uniformly sub-Gaussian for $t\in\R$; see Lemma \ref{lem:subG-tilts}.

While \eqref{assu:usg} is stronger than \eqref{assu:sg}, it includes for instance all compactly supported measures, as well as finite mixtures of Gaussian measures (see for instance Example \ref{ex:Gp} below). 
Measures satisfying \eqref{assu:sg} but not \eqref{assu:usg} are somewhat pathological. Roughly speaking, such measures must have unbounded support, with an infinite sequence of ``gaps'' in the support -- a sequence intervals of unbounded length having measure much smaller than what is implied by the sub-Gaussian tail condition \eqref{assu:sg}. That is, for such $\mu$, there must be an infinite sequence of disjoint intervals $I_k=[a_k,b_k]$ with $a_k\to\infty$ and $b_k-a_k\to\infty$ such that $-\log\mu((b_k,\infty))=O(b_k^2)\ll -\log\mu(I_k)$ (or the analogous condition holds for the left tail of $\mu$). 
For example, one verifies that measures of the form
\begin{equation}
\label{ex:not-usg}
\mu = \frac1Z\sum_{k=1}^\infty e^{-a_k^2}(\delta_{a_k} + \delta_{-a_k})
\end{equation}
for a normalizing constant $Z$ and a sufficiently rapidly growing sequence $a_k\in \R^+$ (such as $\sqrt{2\log k}$) satisfy \eqref{assu:sg} but not \eqref{assu:usg}. 

\begin{remark}[Regularity properties] \label{rmk:assu.reg}
From sub-Gaussianity of $\mu$ \eqref{assu:sg} it follows that $\LLa_\mu$ and $\psimu$ are real-analytic functions on $\R$.
It will also be useful to note that under \eqref{assu:sg}, $\tLL:=\LLa_\mu\circ\,\ssq^{-1}$ is globally Lipschitz on $\R$, where we denote the signed square function $\ssq(x)=x^2(1_{x\ge0}-1_{x<0})$, with inverse $\ssq^{-1}(x):=\sgn(x)\sqrt{|x|}$ on $\R$.
Indeed, $\tLL$ is continuous at 0, while for $t\ne0$, $\tLL'(t)=\frac12 |t|^{-1/2}\LLa_\mu'(\ssq^{-1}(t))$, and it is routine to show that $\LLa_\mu'(s)=O(|s|)$ when $\mu$ is sub-Gaussian, 
so $|\tLL'(t)|=O(1)$.
\end{remark}

We note the following basic families of distributions satisfying \eqref{assu:usg} that are not always sharp sub-Gaussian, and hence are not covered by Theorem \ref{thm:HuGu}. For further examples we refer to \cite{AGH}.

\begin{figure}
\includegraphics[width=6cm]{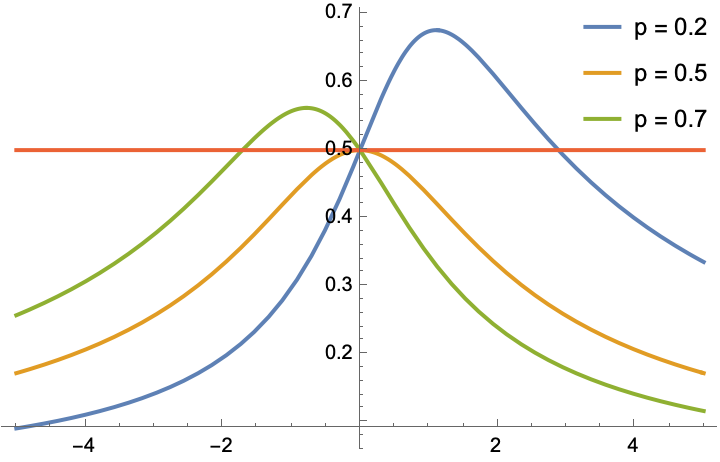}\\
\quad\\
\includegraphics[width=6.5cm]{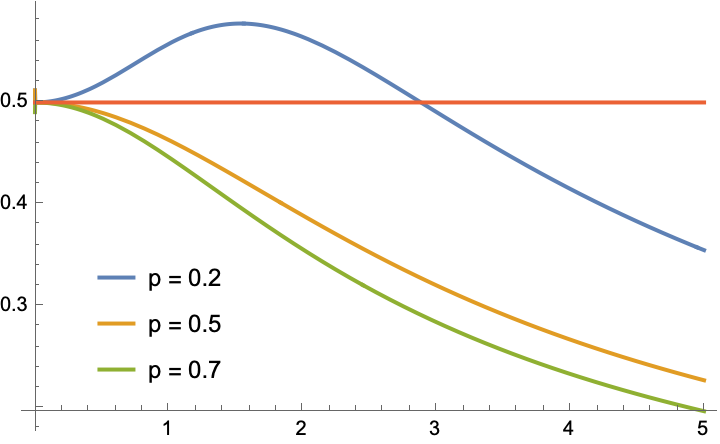}
\quad
\includegraphics[width=6cm]{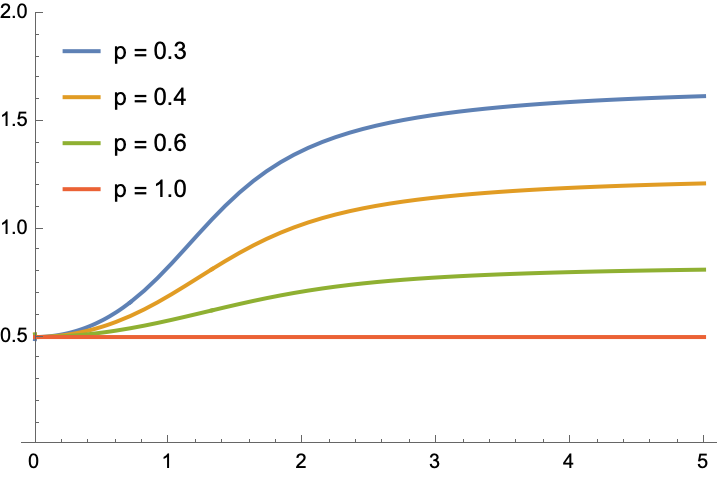}
\caption{
Plots of $\psimu(t)=\LLa_\mu(t)/t^2$ for the standardized Bernoulli($p$) measure (top), $p$-sparse Rademacher distribution (lower-left) and $p$-sparse Gaussian (lower-right). In the latter two cases $\psimu$ is symmetric. See Examples \ref{ex:Bp}, \ref{ex:Rp}, \ref{ex:Gp}. The line $\psi_\gamma(t)\equiv\frac12$ for the Gaussian measure is plotted in red for reference. The standardized Bernoulli($p$) measure is only sharp sub-Gaussian for $p=\frac12$ (the Rademacher case) and the $p$-sparse Gaussian is only sharp sub-Gaussian for $p=1$ (the Gaussian case). 
The $p$-sparse Rademacher is not sharp sub-Gaussian when $p<\frac13$.}
\label{fig:psimu}
\end{figure}

\begin{example}[Bernoulli]
\label{ex:Bp}
For $B_p$ a Bernoulli($p$) variable, the standarized variable $X=(B_p-p)/\sqrt{p(1-p)}$ has distribution
\[
\mu=(1-p)\delta_{-\sqrt{\frac{p}{1-p}}} + p \delta_{\sqrt{\frac{1-p}p}}
\]
with
\begin{equation}
\LLa_\mu(t) = \log \bigg[ (1-p) \exp\bigg( -t\sqrt{\frac{p}{1-p}}\bigg) + p \exp\bigg( t\sqrt{\frac{1-p}p}\bigg)\bigg]\,.
\end{equation}
Since $X$ is standardized, $\psimu(0)=\frac12$, but $\psi_\mu'(0)=0$ only if $p=\frac12$, i.e. when $X$ is Rademacher.
Hence, $\mu$ is only sharp sub-Gaussian when $p=\frac12$. Since $\mu$ is compactly supported we have $\psiinfty=0$. 
When $p<\frac12$ then $\psi_\mu'(0)>0$ and the maximum $\psimax>\frac12$ is achieved at a unique $t_p^*\in(0,\infty)$.
See Figure \ref{fig:psimu}.
\end{example}

\begin{example}[Sparse Rademacher]
\label{ex:Rp}
For fixed $p\in(0,1]$ suppose $\mu$ is the distribution of $\frac1{\sqrt{p}}B_p Y$ for $Y\in\{+1,-1\}$ uniform and $B_p$ is an independent Bernoulli($p$) variable. 
Then
\[
\LLa_\mu(t) = \log ( 1-p + p \cosh(t/\sqrt{p})).
\]
Since $\mu$ is symmetric, $\psimu$ is a symmetric function, and since $\mu$ is compactly supported we have $\psiinfty=0$. One verifies that $\psimu$ achieves its maximum on $\R^+$ at a unique point $t_*(p)\ge0$. Moreover, we have $t_*(p)=0$ when $p\ge \frac13$, in which case $\psimax=\psimu(0)=\frac12$ and $\mu$ is sharp sub-Gaussian, whereas for $p<\frac13$ we have $t_*(p)>0$ and $\psimax>\frac12$.
(See \cite[Example 3]{AGH}.) See Figure \ref{fig:psimu}(lower-left).
\end{example}

\begin{example}[Sparse Gaussian]
\label{ex:Gp}
For fixed $p\in(0,1]$ suppose $\mu$ is the distribution of $\frac1{\sqrt{p}}B_p G$, where $G$ is standard Gaussian and $B_p$ is an independent Bernoulli($p$) variable. 
Then
\[
\LLa_\mu(t) = \log ( 1-p + p e^{t^2/2p}).
\]
Since $\mu$ is symmetric, $\psimu$ is a symmetric function. 
One verifies that $\psimu$ is strictly increasing on $\R^+$, with $\psimax=\psiinfty=\frac1{2p}$, so $\mu$ is not sharp sub-Gaussian for any $p\in(0,1)$. See Figure \ref{fig:psimu}(lower-right).
\end{example}

\subsection{Minimax formula for the upper tail rate}
\label{sec:rate}

For $x\ge2$ and $\theta\ge0$ let
\begin{equation}	\label{def:Jsc}
J(x, \theta) := 
\begin{cases}
\theta^2 & \theta\le \thetam \\
\theta x- \frac12 \int \log(x-\lambda) d\sigma(\lambda) - \frac12\log(2\theta) -\frac12
& \theta\ge \thetam 
\end{cases}
\end{equation}
where we denote by
\begin{equation}	\label{def:thetapm}
\thetapm:= \frac14\Big(x\pm\sqrt{x^2-4}\Big)
\end{equation}
the solutions to the equation $x=2\theta + (2\theta)^{-1}$. 
We note that $J(x,\theta)$ is continuously differentiable in $\theta\in(0,\infty)$ for any fixed $x\ge2$ (see \cite[Section 4.1]{HuGu1};  using the relation $x=2\thetam +(2\thetam)^{-1}$ one sees that the function $v$ defined there is $C^1$ in $\theta$).
One verifies that $\thetam$ is related to the Stieltjes transform $G_\sigma(x)$ of the semicircular measure $\sigma$ at $x\ge2$ by
\begin{equation}	\label{def:Gsc}
G_\sigma(x) := \int\frac{d\sigma(\lambda)}{x-\lambda} = 2\thetam.
\end{equation}
For $x\ge2$ and $\theta\ge0$ we set
\begin{equation}	\label{def:overlapx}
\overlap_x(\theta) := 
\Big( 1-\frac{\thetam}{\theta}\Big)^{1/2}_+.
\end{equation}
The other root $\thetap$ in \eqref{def:thetapm} has significance for large deviations. Indeed, a computation shows (see \cite[Section 4.1]{HuGu1}) the GOE rate function 
\eqref{tail.igamma} can alternatively be expressed 
\begin{equation}	\label{igamma.var}
\rate^\gamma(x) = \sup_{\theta\ge0} \{J(x,\theta) - \theta^2\}
\end{equation}
where the supremum is attained at the unique value $\thetap$.

Our first main result gives an asymptotic minimax characterization of the upper tail for $\lam_1$, extending \eqref{igamma.var}.
The objective function bears some resemblance to the one in \eqref{igamma.var}, with $\theta^2$ replaced by a quantity $\free_{N,R}(\theta,w)$ that we now define, which gives the leading order asymptotic value of a certain \emph{restricted annealed free energy density} for a spherical spin glass model, whose definition we defer to Section \ref{sec:results.restricted}. 

Hereafter, $\B^N,\B$ denote the closed (Euclidean) unit balls in $\mathbb R^{N}$ and $\ell^2(\N)$, respectively. 
We identify $\R^N$ with the subspace $\R^{[N]}\subset\ell^2(\N)$ of sequences supported on $[N]=\{1,\dots,N\}$, and correspondingly view $\B^N$ as a subset of $\B$. We write $\cP(I)$ for the set of Borel probability measures supported on an interval $I$, and $\cP_\alpha(I)\subset\cP(I)$ for those measures with second moment $\int s^2d\nu(s)=\alpha$.
For $R,\alpha>0$ and $v\in \ell^2(\N)$ define
\begin{align}
\VP_R(v, \alpha) &:=
\sup_{\nu\in\cP_{\alpha}([-R,R])}
\bigg\{ \int \sum_{i\ge1} \LLa_\mu(2v_i s) 
 d\nu(s) - \DKL(\nu|\gamma)\bigg\} \label{def:VP}
\end{align}
where $\DKL(\nu|\gamma)$ is the relative entropy (see \eqref{def:DKL}). 
An alternative, non-variational expression for $\VP_R(v,\al)$ is given in Remark \ref{rmk:tfree.alt} below.
We also define $\VP_\infty(v,\al)$ as above but with the supremum   in \eqref{def:VP}
taken over $\cP_\al(\R)$.
From \eqref{assu:sg} and the non-negativity of $\DKL(\nu|\gamma)$ it follows that $\VP_R$ is finite on $\ell^2(\N)\times [0,R^2]$, specifically:
\[
\VP_R(v,\al)\le 4\psimax\alpha\|v\|_2^2\,.
\]
For $\theta\ge0,w\in\B$ and $R\ge1$ we set
\begin{align}
\label{def:freeN}
\free_{N,R}(\theta,w) 
&:=
\freeL(\theta,w)+
\theta^2(1-\|w\|_2^2)^2	
+ 
\VP_R(\theta w, 1-\|w\|_2^2) 
 -\frac12\|w\|_2^2
\end{align}
where 
\begin{equation}
\label{def:freeL}
\freeL(\theta,w):=
\frac1N \sum_{i\le j} \LLa_\mu(2^{\ep_{ij}} \theta\sqrt Nw_iw_j)	\,,\qquad
\ep_{ij} := \frac12(1+1_{i\ne j}). 
\end{equation}
Note that since $\LLa_\mu(0)=0$, the only nonzero summands in \eqref{def:VP} and \eqref{def:freeL} are for $i\in\supp(v)$ and $i,j\in\supp(w)$, respectively.
Note also that the quantities in \eqref{def:VP}, \eqref{def:freeN}, \eqref{def:freeL} are invariant under permutations of the coordinates of $v$ and $w$.

Define
\begin{equation}	\label{def:jayN0}
\cJ_N(x,\vloc):= \sup_{\theta\ge0} \Big\{ J(x,\theta)- \free_{N,N^{1/5}}\big(\theta, \overlap_x(\theta)z\big) \Big\}\,,\qquad x\ge2\,,\;\vloc\in\B\,.
\end{equation}
(Compare \eqref{igamma.var}.)
Note that $\free_{N,R}$ and $\cJ_N$ depend additionally on $\mu$, but we suppress this from the notation. 

\begin{theorem}\label{thm:rateN}  
Assume \eqref{assu:usg}.
There is a constant $c_\mu\in(0,1)$ depending only on $\mu$ such that
for any fixed $x\ge2$, 
\begin{equation}	\label{cormain.asymp}
\lim_{\delta\downarrow0}\limsup_{N\to\infty}
\bigg|\frac{1}{N}\log\mathbb{P}(|\lam_1-x|\leq\delta)+ \rate^\mu_N(x)\bigg|=0
\end{equation}
where, with $n=N^{3/4}$ and $\zcush_x=c_\mu/x^4$,
\begin{equation}	\label{def:rateN0}
\rate^\mu_N(x)
:=
\inf_{\vloc\in(1-\zcush_x)\B^n}
\cJ_N(x,\vloc)\,.
\end{equation}
\end{theorem}

\begin{remark}
\label{rmk:dropUSG}
\quad
\begin{enumerate}
\item The proof gives quantitative bounds -- see Propositions \ref{prop:upper-joint}, \ref{prop:no-comp} and Lemma \ref{lem:lower1}.
\item The assumption \eqref{assu:usg} is mainly needed in the proof of the lower bound for $\P(|\lam_1-x|\le \delta)$, in order to perform a tilting argument.
For the upper bound it is only used to rule out the event that the associated eigenvector $v_1$ is essentially supported on $o(N)$ coordinates (see Proposition \ref{prop:no-comp}), though we expect the assumption could be removed there. 
\item The choice of $n, \rho, R$ proposed in Theorem \ref{thm:rateN}  is sufficient for our arguments. 
The conclusion still holds if the infimum 
in \eqref{def:rateN0} is taken over $(1-\zcush)\B^{n}$ for any fixed $\zcush\in(0,\zcush_x)$. 
One can also replace $n$ with $N^{1-a}$ for any fixed $a\in (0,\frac12)$, in which case we can take any  $R\in [2N^{a/2},N^{1/4}/\log N]$ in place of $N^{1/5}$ in \eqref{def:jayN0}. 
\end{enumerate}
\end{remark}

\begin{remark}	\label{rmk:tfree.alt}
We have the following alternative, non-variational expression for $\VP_R(v,\al)$ from \eqref{def:VP}, which is useful for numerical evaluation of the rate function 
$\rate^\mu_N$ as we do in Figure \ref{fig:pgauss}. The case of finite $R$ is shown in Proposition \ref{prop:gibbs} and the case $R=\infty$ was established in \cite[Lemma 12]{AGH}. With
\begin{align*}
g_{v,R}(\zeta)
&:= \log\int_{-R}^R \exp \bigg(  - \zeta s^2 + \sum_{i\ge1} \LLa_\mu(2v_is) \bigg) ds
\,,\qquad
\end{align*}
we have
\begin{equation}
\VP_R(v,\al) = g_{v,R}(\zeta^\star) + \al\zeta^\star + \frac12(1-\al) - \frac12\log(2\pi e) 
\end{equation}
where $\zeta^\star=\zeta^\star_{v,\al,R}>0$ is the unique solution to the equation
$
g_{v,R}'(\zeta) +\al=0.
$
Hence, we can alternatively express \eqref{def:freeN} as
\begin{align*}
&\free_{N,R}(\theta,w) = \\
&\freeL(\theta,w) +
\theta^2(1-\|w\|_2^2)^2 +
g_{\theta w,R}\big(\zeta^\star_{\theta w,1-\|w\|_2^2,R}\big) + (1-\|w\|_2^2)\zeta^\star_{\theta w,1-\|w\|_2^2,R}
-\frac12\log(2\pi e).
\end{align*}
\end{remark}

One should think of $\cJ_N(x,\vloc)$ in \eqref{def:jayN0} as a joint large deviation rate function for $\lam_1$ and the restriction of the associated eigenvector $v_1$ to its large coordinates. That is, for a fixed parameter $\eta\in(0,\frac14)$ set
\begin{equation}
\label{def:v1eta}
v_1^{(\eta)}:=(v_i1_{|v_i|\ge N^{-1/2+\eta}})_{i=1}^N
\end{equation}
which has support of size at most $N^{1-2\eta}$
 (for Theorem \ref{thm:rateN} we take $\eta=\frac18$, but in the proofs we need to consider general $\eta$).
 Indeed, the proof roughly shows
\[
\frac1N\log\P(\lam_1\approx x, v_1^{(\eta)}\approx \vloc) \le - 
\cJ_{N}(x,\vloc)+o(1)\, .
\]
The infimum over $\vloc$ in \eqref{def:rateN0} then reflects a union bound over all possible choices for the localized part $v_1^{(\eta)}$ (where the cardinality of a net of approximations $\vloc$ is of negligible size $\exp(o(N))$). 
The idea to restrict to an event on which the large coordinates of $v_1$ are fixed is one of the key ideas of this work, which allows us to pin down the sharp large deviations rate in cases where the previous work \cite{AGH} hit a barrier. 




The 
expression \eqref{def:rateN0} bears some resemblance to rate functions appearing in Proposition 1 and Theorem 3 of \cite{AGH} giving one-sided bounds that are tight in certain cases.
The main, but crucial, difference from the rate functions in \cite{AGH} is that the supremum over $\vloc$ is taken after the infimum over $\theta$.
That is, we first establish joint large deviation estimates for $(\lam_1,v_1^{(\eta)})$, and then contract to get large deviation estimates for $\lam_1$.
In some cases 
this also allows to establish structural properties of $v_1$ conditional on a large deviation event for $\lam_1$, by understanding the structure of the optimizers $\vloc$ in \eqref{def:rateN0}.

Theorem \ref{thm:rateN} shows that asymptotically, the upper tail for $\lam_1$ is given by the $N$-dependent rate function 
$\rate^\mu_N$. Our next three results provide a genuine limiting rate function 
$\rate^\mu$ under further assumptions.

\subsection{Universal rate function close to the bulk}

We show in Lemma \ref{lem:Jprops} that $\free_{N,R}(\theta,0)=\theta^2+O(e^{-cR^2})$, and hence from \eqref{igamma.var},
\begin{equation}
\rate^\mu_N(x) \le \cJ_N(x,0) =\sup_{\theta\ge0} \big\{J(x,\theta) - \theta^2\big\} +O(e^{-cN^{1/4}}) =\rate^\gamma(x) +o(1).
\end{equation}
Thus, for the class of sub-Gaussian measures covered by Theorem \ref{thm:rateN}, large deviations of $\lam_1$ are at least as likely as in the GOE case.
The following result shows that $\rate^\mu_N(x)$ in fact always converges to the GOE rate function $\rate^\gamma(x)$ in a some neighborhood of $2$, because the infimum in \eqref{def:rateN0} is then 
taken at $\vloc=0$; this further entails that $v_1^{(\eta)}\approx 0$ on the event that $\lam_1 \approx   x$, i.e.\ $v_1$ is delocalized.

\begin{theorem} [GOE rate function in a neighborhood of the bulk]
\label{thm:smallx}
Assume \eqref{assu:usg}. 
\begin{enumerate}[(a)]
\item
There is a universal constant $c>0$ such that the following holds.
For any fixed $x>2$ and $x+N^{-c}\le y\le L$, 
\begin{equation}	\label{igamma-LB}
\frac1N\log \P( \lam_1\in [x,y)) \ge -\rate^\gamma(x) -N^{-c}
\end{equation}
for all $N$ sufficiently large depending on $x,L$ and $\mu$. 
Moreover, there exists $x_\mu>2$ depending only on $\mu$ such that if $2<x< x+N^{-c}\le y<x_\mu$ then
\begin{equation}	\label{igamma-UB}
\frac1N\log \P( \lam_1\in [x,y)) \le -\rate^\gamma(x) + N^{-c}
\end{equation}
for all $N$ sufficiently large depending on $x$ and $\mu$. In particular, \eqref{tail.igamma} holds for any fixed $x<x_\mu$.
\item For any fixed $\heta\in(0,\frac1{10})$, $\eta\in(\heta,\frac14-\heta)$, $\xcush>0$ and interval $I\subset[2+\kappa,x_\mu)$ of length at least $N^{-c}$, with $v_1^{(\eta)}$ as in \eqref{def:v1eta}, we have
\begin{equation}	\label{igamma-v1UB}
\P\big( \|v_1^{(\eta)}\|_2^2\ge\al \,\big|\, \lam_1\in I\big) \le \exp( - c'\al \sqrt{\xcush} N) \qquad\forall \al \ge N^{-c'\heta}
\end{equation}
for  all $N$ sufficiently large depending on $\eta,\xcush$ and $\mu$, and a constant $c'>0$ depending only on $\mu$. In particular, conditional on $\lam_1\in I$ we have  $\|v_1\|_\infty \le N^{-c'/16}$ with probability $1-o(1)$.
\end{enumerate}

\end{theorem}

The result of part (a) improves on \cite[Proposition 6]{AGH}, which established the GOE rate function $\rate^\gamma$ 
{for $x\in [2, (2\psimax-1)^{1/2}+ (2\psimax -1)^{-1/2}]$}
under the assumption $\psimax<1$.

\subsection{A full large deviation principle} 

For general sub-Gaussian $\mu$ and $x>x_\mu$ the GOE rate function $\rate^\gamma$ may underestimate the probability that $\lam_1 \approx   x$, due to the emergence of non-universal, localized large deviations mechanisms, coinciding with a non-vanishing mass $\|v_1^{(\eta)}\|_2^2$ in the large coordinates of the eigenvector $v_1$. 
Our next result provides a full large deviation principle on all of $\R$. 
For this it is necessary to assume that $\psimu(t)$ tends to a limit as $t\to\pm\infty$.
Our result assumes these limits are the same:
\begin{equation}
\label{assu:LRlim}
\psiinfty=\lim_{t\to+\infty}\psimu(t)= \lim_{t\to-\infty}\psimu(t)\,.
\end{equation}
We also assume that the supremum of $\psimu(t)$ is taken on $\R^+$ (possibly at $+\infty$):
\begin{equation}
\label{assu:maxR+}
\psimax = \sup_{t\ge0}\psimu(t)\,.
\end{equation}
\eqref{assu:LRlim} and \eqref{assu:maxR+} hold for instance when the $\mu$ is symmetric and $\lim_{t\to\infty}\psimu(t)$ exists, but also for some asymmetric measures such as the standardized Bernoulli measure $\mu_p=p\delta_{\sqrt{(1-p)/p}} + (1-p)\delta_{-\sqrt{p/(1-p)}}$ when $p\le \frac12$ (see Example \ref{ex:Bp}).
See Section \ref{sec:open} for further discussion of the assumptions \eqref{assu:LRlim}, \eqref{assu:maxR+}.

The rate function $\rate^\mu$ is obtained as a monotone limit of approximating rate functions defined as follows.
For  $R\ge1$ define
\begin{align}
&\tfree_{N,R}(\theta,\chw,\tal)
:=	  \label{def:tfree}\\
&\quad\theta^2\Big[ \beta^2 + 2\beta\tal + 2\psimax\tal^2+ 2\psiinfty(\|\chw\|_2^4+2\tal\|\chw\|_2^2) \Big]	\notag\\
&\qquad\qquad\qquad
+ \VP_R(\theta \chw, \beta) -\tfrac12(1-\beta)\,,\qquad\qquad\qquad	
&\theta\ge0,\chw\in\B,\tal\in[0,1]	\notag
\\
&\wt\cJ_{N,R}(x,\chz,\tal) 
:= \label{def:tjayN}\\
&\quad\sup_{\theta\ge0} \Big\{ J(x,\theta)- \tfree_{N,R}\big(\theta, \overlap_x(\theta)\chz,\overlap_x(\theta)^2\tal\big) \Big\}\,, 
\qquad
& \,x\ge2,\chz\in\B,\tal\in[0,1]	\notag
\end{align}
where in \eqref{def:tfree} we abbreviate $\beta:=1-\tal-\|\chw\|_2^2$, which is assumed to be nonnegative. 
For $\thresh>0$ denote
\begin{equation}	\label{def:Bthresh}
\B_{\ge\thresh}:= \{v\in\B: |v_j|\in \{0\}\cup[\thresh,1]\;\forall j\in\N\}
\end{equation}
and for $\zcush>0$ set
\begin{equation}
\wt\rate_{N,\thresh}(x,\zcush)
:= \inf_{\tal\in[0,1-\zcush]} \inf_{\substack{\chz\in\B_{\ge\xi} \\ \|\chz\|_2^2\le 1-\zcush-\tal}} 
\wt\cJ_{N,N^{1/5}}(x,\chz,\tal)  \,.
\label{def:trate}
\end{equation}

\begin{theorem}[Large deviation principle]\label{thm:fullLDP}
Assume \eqref{assu:usg}, \eqref{assu:LRlim} and \eqref{assu:maxR+}.  
Then with $c_\mu$ as in Theorem \ref{thm:rateN}, for any fixed $x\ge2$, $\eps\in(0,\frac1{10})$ and $\zcush\in(0,c_\mu x^{-4}]$, the limit 
\begin{equation}	\label{def:rate2}
\rate^\mu(x):=\lim_{N\to\infty} \wt\rate_{N,N^{-\eps}}(x,\zcush)
\end{equation}
exists and is independent of $\eps$ and $\zcush$, and defines a continuous non-decreasing function on $[2,\infty)$.
Moreover, $\lam_1$ satisfies a large deviation principle with speed $N$ and good rate function $\rate^\mu$ that is infinite on $(-\infty,2)$ and is otherwise given by $\eqref{def:rate2}$. 
\end{theorem}

Roughly speaking, the free energy function \eqref{def:tfree} arises as a reduction of $\free_{N,R}$ from \eqref{def:freeN} under a decomposition $w=\chw+\tw$, with $\chw\in \B_{\ge N^{-\eps}}$ containing the very large entries of $w$. The appropriate threshold is actually located via a pigeonholing argument which allows to show that cross terms simplify. It can then be shown for an appropriate choice of $R$ that  $\free_{N,R}$ depends on the moderately large entries $\tw_i\in(R,N^{-\eps})$ only through the norm $\tal=\|\tw\|_2^2$, leading to the expression \eqref{def:tfree}. The key point is that the resulting sequence \eqref{def:trate} is then monotone decreasing in $N$, yielding the existence of the limit.
In \cite[Proposition 5 and Section 6]{AGH}, it is shown under some additional technical hypotheses that when $\mu$ is symmetric and $\psiinfty<\psimax$, for $x$  large enough the optimum is taken at $\chw=0$, in which case a more explicit formula for $\rate^\mu$ can be given.

We have the following consequence of Theorem \ref{thm:fullLDP} for \ER graphs. Recall that the adjacency matrix $A$ for an \ER graph $G_{N,p}$ on $N$ vertices is symmetric with independent Bernoulli($p$) entries above the diagonal, and zeros on the diagonal. 

\begin{cor}
\label{cor:ER}
Fix $p\in(0,\frac12]$ and let $A$ be the adjacency matrix for an \ER graph $G_{N,p}$. Then $N^{-1/2}\lam_1(A - \E A)$ satisfies a large deviation principle with good rate function $\cI_p:\R\to[0,\infty]$ that is infinite on $(-\infty,2\sqrt{p(1-p)})$ and is a continuous nondecreasing function on $[2\sqrt{p(1-p)},\infty)$. Moreover, $\cI_p(x)=\cI^\gamma(x/\sqrt{p(1-p)})$ on $(-\infty, x_p]$ for some $x_p>2\sqrt{p(1-p)}$ depending only on $p$. 
\end{cor}

\begin{proof}
Let $H$ be as in \eqref{def:H} with $\mu=\mu_p$ the standardized Bernoulli measure. Recall from Example \ref{ex:Bp} that $\psi_{\mu_p}^\infty=0$ and $\psi_{\mu_p}$ attains its supremum at a finite point $t_p\ge0$. With $D=\diag(H_{ii})$ the diagonal of $H$, we have $\sqrt{p(1-p)N}(H-D)\eqd A-\E A$. Thus, $N^{-1/2}(A-\E A)$ and $\sqrt{p(1-p)}H$ differ by a diagonal matrix with entries almost-surely bounded by $N^{-1/2}$. From the Hoffman--Wielandt inequality it follows that $|\sqrt{p(1-p)}\lam_1(H)-N^{-1/2}\lam_1(A-\E A)| \le N^{-1/2}$ a.s. The claim then follows from Theorems \ref{thm:fullLDP} and \ref{thm:smallx}, with $\cI_p(x) = \cI^{\mu_p}(x/\sqrt{p(1-p)})$. 
\end{proof}

\begin{remark}
For the uncentered adjacency matrix we typically have $\lam_1(A) \approx   pN$ and $\lam_2(A) \approx   \lam_1(A-\E A)$, but it is not clear whether the latter holds in the large deviations regime.  A lower bound $\lam_2(A)\ge (1+o(1))\lam_1(A-\E A)$ can be deduced from the interlacing property of eigenvalues under rank-1 perturbations, and hence Theorem \ref{thm:smallx} shows that for any fixed $x\ge2$,
\[
\P\big(\lam_2(A) \ge x\sqrt{p(1-p)N}\,\big) \ge \exp\big( -\rate^\gamma(x)N +o(N)\big) \,.
\]
\end{remark}
\quad

The following shows that the sharp sub-Gaussian assumption gives a sharp characterization of the universality regime for large deviations of $\lam_1(H)$.

\begin{cor}[Non-universality away from the bulk]
\label{cor:nonuniv}
With hypotheses as in Theorem \ref{thm:fullLDP}:
\begin{enumerate}[(a)]
\item 
If $\Delta:=\psimax-\frac12>0$, then  $\rate^\mu(x)<\rate^\gamma(x)$ for all $x>2\sqrt{2}(\Delta^{1/2}+\Delta^{-1/2})$. 
In fact, 
\begin{equation}	\label{bd1:nonuniv}
\rate^\mu(x) \le \rate^\gamma(x) + \frac12- \frac{\Delta x^2}{16(1+\Delta)^2}\qquad 
\forall x\ge2 \,.
\end{equation}

\item 
For any $x>2$ there exists $a_\mu(x)>0$ depending only on $\mu$ and $x$ such that for any  fixed $\delta,\delta_0\in(0,\frac1{10})$ independent of $N$ and $x'\in (x+\delta,+\infty]$ possibly depending on $N$, if  
\begin{equation}
\label{rates-neq}
\rate^\gamma(x)\ge \rate^\mu(x)+\delta_0
\end{equation}
then
\begin{equation}
\label{bd2:nonuniv}
\P\Big( \big\|v_1^{(\frac14-\delta)}\big\|_2 \ge a_\mu(x)\delta_0 \,\Big|\, \lam_1\in [x,x')\Big)
\ge1- e^{-\delta_0N/2}
\end{equation}
for all $N$ sufficiently large depending on $\mu,x, \delta$ and $\delta_0$.

\end{enumerate}

\end{cor}

\begin{proof}
See Section \ref{sec:nonuniv}.
\end{proof}

\begin{remark}
From \eqref{bd1:nonuniv} we see that the condition \eqref{rates-neq} holds whenever $x>2\sqrt{2}(\Delta^{1/2}+\Delta^{-1/2})$ for some $\delta_0>0$ depending on $\mu$ and $x$, 
but we stress that \eqref{bd2:nonuniv} says more generally that $v_1$ has a non-vanishing localized component with high probability conditional on $\lam_1\in [x,x')$ for any $x$ where $\rate^\mu(x)\ne \rate^\gamma(x)$.
\end{remark}

\begin{remark}
The corollary gives an upper bound $x_\mu'\le2\sqrt{2} (\Delta^{1/2}+\Delta^{-1/2})$ for the threshold $x_\mu'$ above which $\rate^\mu(x)<\rate^\gamma(x)$.
While this bound on the regime of non-universal deviations is not sharp in general (and we have not optimized it in the proof), it is interesting to note that it agrees up to a constant factor with the upper bound $(2\Delta)^{1/2}+(2\Delta)^{-1/2}$ for the threshold $x_\mu$ below which $\rate^\mu(x)=\rate^\gamma(x)$ under the assumption $\Delta<\frac12$, as was shown in  \cite[Proposition 6]{AGH}.
\end{remark}

\subsection{The case of $\psimu(t)$ increasing}

When $\psimax=\psiinfty$ the expression \eqref{def:tfree} simplifies, and it is not hard to see that the infimum in \eqref{def:trate} will then be taken at $\tal=0$. In terms of the unreduced form of the rate function $\rate_N^\mu$ from \eqref{def:rateN0} this means that the infimum in $\vloc$ is taken at a vector with nonzero entries of size $\ge N^{-\eps}$. Thus, conditional on $\lam_1 \approx   x$, the eigenvector $v_1$ has a small number of entries of size $\ge N^{-\eps}$, and all remaining entries of size $\le N^{-1/2+\eta}$.

Assuming further that $\mu$ is symmetric and that $\psimu$ is increasing on $\R^+$, we can show that the optimizer $\chz$ is supported on a \emph{single coordinate}, giving rise to the following result, where the infimum over a high dimensional ball in \eqref{def:rateN0} is replaced by an infimum over an interval. 
This assumption includes the case of sparse Gaussian variables (see Example \ref{ex:Gp}). 

For $\theta\ge0,\al\in[0,1]$ and $x\ge2$  let
\begin{align}
&\hfree(\theta,\al) 	\label{def:hfree}\\
&:= \theta^2\big[ (1-\al)^2 + 2\psiinfty\al^2 \big] 
+\VP_\infty(\theta\al^{1/2}e_1, 1-\al) - \frac\al2		\notag\\
&\;= \theta^2\big[ (1-\al)^2 + 2\psiinfty\al^2 \big] 
+ \sup_{\nu\in\cP_{1-\al}(\R)} \bigg\{ \int \LLa_\mu(2\theta\al^{1/2}s)d\nu(s) - \DKL(\nu|\gamma)\bigg\} 	
- \frac\al2,	\notag\\
&\hcJ(x,\al) 
:= \sup_{\theta\ge0} \big\{ J(x,\theta) - \hfree\big(\theta,\overlap_x(\theta)^2\al\big)\big\}	\label{def:hcJ}
\end{align}
where we write $e_1=(1,0,0,\dots)\in\ell^2(\N)$ in \eqref{def:hfree}.

\begin{theorem}\label{thm:increasing} 
Assume \eqref{assu:usg}, that $\mu$ is symmetric, and that $\psimu$ is nondecreasing on $\R^+$.
(In particular $\psimax=\psiinfty$.)
\begin{enumerate}[(a)]
\item
$\lam_1$ satisfies a large deviation principle with speed $N$ and good rate function
$\rate^\mu$ which is infinite on $(-\infty,2)$ and is otherwise given by 
\begin{align}
\rate^\mu(x)
&=\inf_{0\le \al\le 1-\zcush_x}\hcJ(x,\al)
\label{def:rate.increasing}
\end{align}
with $\zcush_x=c_\mu/x^4$ as in Theorem \ref{thm:rateN}.
Moreover, for $x>2$ the infimum is achieved on a closed nonempty set $A^*_x\subset[0,1-\zcush_x]$.

\item
Assume further that $\psimu$ is \emph{strictly} increasing  on $\R^+$.
For any $x>2$ and $\eta,\eps\in(0,\frac1{10})$ there exist $\delta_0,\delta_1>0$ depending only on $x,\eps$ such that for any $\delta\in(0,\delta_0)$, with $\al_x^*:=\inf A_x^*$ we have
\begin{equation}	\label{bd:increasing.b}
\P\Big( \sqrt{\al_x^*} - \eps\le  \|v_1^{(\eta)}\|_2 \le \|v_1^{(\eta)}\|_\infty+\eps \,\Big| \, |\lam_1-x|\le \delta\Big) 
\ge 1- e^{-\delta_1N}
\end{equation}
for all $N$ sufficiently large depending on $x,\eta,\eps,\delta$ and $\mu$.
Thus, conditional on $|\lam_1-x|<\delta$ we have that with probability $1-o(1)$, $v_1$ is within distance $\eps$ of a vector with one entry of magnitude at least $\sqrt{\al_x^*}$ and all other entries bounded by $N^{-\frac12+\eta}$. 
\end{enumerate}
\end{theorem}

\begin{remark} \quad
\begin{enumerate}[($i$)]
\item For (a), Theorem \ref{thm:fullLDP} already established the LDP; what is new here is the alternative, non-asymptotic expression for $\rate^\mu$. The proof of Theorem \ref{thm:increasing} does not go through Theorem \ref{thm:fullLDP}, instead proceeding from Theorem \ref{thm:rateN}, but it follows that the expressions \eqref{def:rate.increasing} and \eqref{def:rate2} are equal, and (from Theorem \ref{thm:smallx}) equal to the GOE rate function $\rate^\gamma$ on $(-\infty,x_\mu]$.

\item 
We note that $\al_x^*>0$ for all sufficiently large $x$, so that the lower bound on $\|v_1^{(\eta)}\|_2$ in \eqref{bd:increasing.b} is nontrivial. 
Indeed, if this does not hold, i.e. $\alpha^*_x=0$, then the rate function \eqref{def:rate.increasing} reduces to the GOE rate function $\rate^\gamma(x)$ (see \eqref{igamma.var}), whereas from 
Corollary \ref{cor:nonuniv} (or alternatively by \cite[Theorem 1]{AGH} under the assumptions of Theorem \ref{thm:increasing}) we have
$\rate^\mu(x)<\rate^\gamma(x)$ for all sufficiently large $x$.
\end{enumerate}
\end{remark}

The reduced annealed free energy \eqref{def:hfree} also appeared in \cite[Proposition 8]{AGH}.
However, unlike in  \cite{AGH} we obtain 
a full large deviation principle, valid for all $x\in\R$ -- the crucial difference is that for our rate function \eqref{def:hcJ}--\eqref{def:rate.increasing} the infimum in $\al$ is taken \emph{after} the supremum in $\theta$.
This derives from the key idea of the present work, to obtain large deviation principles for $\lam_1$ as contractions of joint large deviation estimates for $\lam_1$ and $\|v_1^{(\eta)}\|_2$. 

See Figure \ref{fig:pgauss} for plots of $\rate^\mu(x)$ and $\al_x^*$ for the case that $\mu$ a sparsified Gaussian measure as in Example \ref{ex:Gp}.

\subsection{The restricted annealed free energy}
\label{sec:results.restricted}

In  the definition 
\eqref{def:jayN0} of the joint rate function 
$\cJ_{N}$ for $(\lam_1,v_1^{(\eta)})$,
under the supremum we have the difference between the quenched and annealed free energies for a spherical spin glass model at inverse temperature $\theta$ -- that is, a Gibbs measure on $\sphereN$ with random density $\propto\exp(\theta N\langle u,Hu\rangle)$ with respect to the uniform 
measure 
$P$.
These quantities arise from the approach of \emph{tilting by spherical integrals} introduced in \cite{HuGu1}.
Recalling the spherical integral $I(M,\theta)$ defined in \eqref{def:spherical}, conditional on the large deviation event $\{\lam_1\approx x\}$ the quenched free energy $\log I(H,\theta)$ is approximately $J(x,\theta)N$ (see \eqref{def:Jsc}) with probability $1-\exp(-\omega(N))$ (for our conventions on asymptotic notation see Section \ref{sec:notation}). 
Using this fact, the first part of the proof of Theorem \ref{thm:HuGu} in \cite{HuGu1} 
 relates the asymptotic upper tail of $\lam_1(H)$ to a variational problem involving the difference of the annealed and quenched free energy densities: for fixed $x\ge2$, when $\mu$ is sharp sub-Gaussian we have
\begin{equation}	\label{HuGu:var}
\frac1N\log\P(|\lam_1(H)-x|\le \delta) = \inf_{\theta\ge0} \{ F_N(\theta)-J(x,\theta)\} + o(1)
\end{equation}
where the error $o(1)$ tends to zero after sending $N\to\infty$ and then $\delta\downarrow 0$, and 
the \emph{annealed free energy density} is given by
\begin{equation}	\label{def:FN}
F_N(\theta) := \frac1N\log\E I(H,\theta) = \frac1N\log \int_{\sphereN} \E e^{N\theta\langle u,Hu\rangle} dP(u)\,.
\end{equation}
The second step is to show
\begin{equation}	\label{HuGu:FE}
F_N(\theta)\to \theta^2
\end{equation}
as $N\to\infty$ (with error bounds uniform in $\theta$). From \eqref{igamma.var} we see that \eqref{HuGu:var} and \eqref{HuGu:FE} combine to give \eqref{tail.igamma}.

In the general sub-Gaussian case both of the asymptotics \eqref{HuGu:var} and \eqref{HuGu:FE} can fail, and tilting by a spherical integral does not  correctly capture the large deviation rate. 
A key idea of this work is to notice that the  integral defining the free energy  $F_{N}(\theta)$ concentrates near vectors $u$ with overlap $\langle u,v_{1}\rangle\approx \pm\overlap_x(\theta)$ with the leading eigenvector $v_{1}$; and hence, on the joint large deviation event $\{\lam_1\approx x, v_1^{(\eta)}\approx \vloc\}$, the  free energy  concentrates on a section $\Uloc_{\overlap_x(\theta)\vloc}$ of the sphere where the restriction of $u$ to the support of $v_1^{(\eta)}$ is approximately $\overlap_x(\theta)\vloc$. 
We are thus led to compute a \emph{restricted annealed free energy}:
for nonempty measurable $\Uloc\subseteq\sphereN$ we define \begin{equation}	
\label{def:FN.U}
F_N(\theta;\Uloc):= \frac1N\log\E\int_{\Uloc}  e^{N\theta\langle u, Hu\rangle} dP(u)\,,\qquad \theta\ge0
\end{equation}
so $F_N(\theta)=F_N(\theta;\sphereN)$.  Proposition \ref{prop:FE} shows that $F_N(\theta;\Uloc_{\overlap_x(\theta)\vloc})\approx\free_{N,R}(\theta,\overlap_x(\theta)\vloc)$, leading to the expression \eqref{def:jayN0} for the joint rate function 
$\cJ_{N}$ (taking $R=N^{1/5}$). \\

We turn to the formal statements of our extensions of \eqref{HuGu:var} and \eqref{HuGu:FE} for the general sub-Gaussian case.
For a given (generally sparse) $w\in \R^N$ we use the shorthand notation
\begin{equation}	\label{w-wc}
u_{w} := u|_{\supp (w)}=(u_{i}1_{w_{i}\neq 0})_{1\le i\le N}\,,\qquad u_{w^{c}} := u|_{[N]\setminus\supp (w)}=(u_{i}1_{w_{i}= 0})_{1\le i\le N}\,.
\end{equation}
and for parameters $\srad>0, R\ge1$ (slowly decaying and growing, respectively)
we let
\begin{equation}	\label{def:UwrR}
\Uloc_w^N=\Uloc_w^N(\srad,R):=\bigg\{ u\in\sphereN \,:\, \|u_w-w\|_2\le \srad\,,\;\|u_{w^c}\|_\infty\le \frac R{\sqrt{N}}\bigg\}
\end{equation}
denote the set of unit vectors that are well approximated by $w$ on its
support and delocalized on all other coordinates. 

As a byproduct of the proof of Theorem \ref{thm:rateN}, we obtain the following extension of \eqref{HuGu:var}.

\begin{theorem}\label{theomain}  
Assume \eqref{assu:usg}. 
Fix $\eta\in(0,\frac14)$ and let $n_0:=\lf N^{1-2\eta}\rf$, $N^{-\eta/3}\le \srad=o(1)$ and $\log N\le R\le N^{1/4}/\log N$. 
Then for any fixed $x>2$ 
there exist $\zcush_x>0,T_x\ge10$ depending only on $x$ and $\mu$ such that 
for all fixed $T\ge T_x$, 
\begin{align}	\label{main-approx}
&\lim_{\delta\downarrow0}\limsup_{N\to\infty} \bigg|\frac{1}{N}\log\mathbb{P}(|\lambda_{1}-x|\leq\delta)	\\
&\qquad\qquad\qquad-
\sup_{\vloc \in(1-\zcush_x)\B^{n_0}  }\inf_{\theta\in[\thetam+T^{-1},T]}
\Big\{\, 
{F_N(\theta; \Uloc^N_{\overlap_x(\theta)\vloc})}
-J(x,\theta)
\,\Big\}
\bigg| =0\,.	\notag
\end{align}
\end{theorem}


The next result shows that  
 the restricted annealed free energy $F_N(\theta;\Uloc^N_w(\srad,R))$ is asymptotically given by the quantity $\free_{N,R}(\theta,w)$ from \eqref{def:freeN}, generalizing \eqref{HuGu:FE} as well as a result from \cite{AGH} to the general sub-Gaussian case.
For this result we only need the standing sub-Gaussian assumption \eqref{assu:sg}.

\begin{prop}[Restricted annealed free energy]
\label{prop:FE} 
Let $\eta\in(0,\frac14)$,
$\zcush\in (0,\frac12)$,  $w \in(1-\zcush)\B^N$ with $\|w\|_0\le N^{1-2\eta}$, and $T\ge1$.
For any $\theta\in[0,T]$, $R\in[\log N, N^{1/4}]$ and $\srad\in[N^{-4},\frac\zcush{10}]$, 
\begin{equation}	\label{FE.asymp}
F_N(\theta; \Uloc^N_w(\srad,R))=
\free_{N,R}(\theta,w) + O_{T,\zcush}( \srad+  R^2N^{-1/2}+ N^{-2\eta}\log N)
\end{equation}
for all $N$ sufficiently large depending on $T$ and $\zcush$.
\end{prop}

The proof {of Proposition \ref{prop:FE} builds on ideas developed in \cite{AGH} to analyze the full annealed free energy density $F_N(\theta)$ and} is given in Section \ref{sec:annealed}.

We briefly indicate how the various terms in the expression \eqref{def:freeN} for $\free_{N,R}(\theta,w)$ arise from the restricted annealed free energy $F_N(\theta,\Uloc_w)$.
From Fubini's theorem,
\begin{align*}
F_N(\theta;\Uloc)
=\frac1N \log \int_\Uloc \exp\bigg( \sum_{i\le j} \LLa_\mu( 2^{\ep_{ij}}\theta\sqrt{N} u_iu_j)\bigg) dP(u)
= \frac1N \log \int_\Uloc e^{N\freeL(\theta,u)}dP(u)
\end{align*}
recalling $\freeL(\theta,\cdot)$ from \eqref{def:freeL}.
On the other hand, we have
\begin{equation}
\free_{N,R}(\theta,w) =
\freeL(\theta,w) +
\freeD(\theta,\|w\|_2^2) +
\free^{cross}_R(\theta,w)	\label{def:free-split}
\end{equation}
where 
the ``localized'' contribution is $\freeL(\theta,w)$, 
and 
the ``delocalized'' and ``cross'' contributions are given by the dimension-free formulas
\begin{align}
\freeD(\theta,\alpha)&:= \theta^2(1-\alpha)^2	
\,,	\label{def:freeD}\\ 
\free^{cross}_R(\theta,w)&:=
\VP_R(\theta w, 1-\|w\|_2^2)
 -\frac12\|w\|_2^2		\label{def:freeX}
\end{align}
(recalling $\VP_R$ from \eqref{def:VP}).
One notes the formula \eqref{def:free-split} is considerably more complicated than the limit $\theta^2$ for $F_N(\theta)$ in the sharp sub-Gaussian case -- the new ``localized'' and ``cross'' contributions $\free_N^{loc}(\theta,w)$ and $\free^{cross}_R(\theta,w)$ arise from the large coordinates $u_w$.

The three contributions $F_N(\theta;\Uloc_w)\approx \freeL(\theta,w) +\freeD(\theta,\|w\|_2^2) + \free^{cross}_R(\theta,w)$ arise from the contributions to $\freeL(\theta,u)$ of indices $(i,j)$ in $\supp(w)\times\supp(w)$, $\supp(w)^c\times \supp(w)^c$, and $\supp(w)\times\supp(w)^c$, respectively. 
For the delocalized contribution of small coordinates $(i,j)\in\supp(w)^c\times \supp(w)^c$, we can Taylor expand $\LLa_\mu(t) \approx   \frac12t^2$ since $|u_iu_j|=o(N^{-1/2})$ there, resulting in the simple expression for $\freeD$. 
The contribution of large coordinates $(i,j)\in \supp(w)^2$ gives rise to $\freeL$ by approximating $u_w\approx w$. 
For the remaining cross contribution $(i,j)\in \supp(w)\times \supp(w)^c$, the integral $dP(u)$ over delocalized coordinates $u_j\in \supp(w)^c$ concentrates on vectors with empirical measure approximately given by the optimizing measure $\nu$ in $\VP_R(\theta w,1-\|w\|_2^2)$.

Note that that upon setting $w=0$ in \eqref{def:free-split} we reduce to the 
unrestricted free energy from \eqref{HuGu:FE}:
\[
\free_{N,R}(\theta,0) = \freeD(\theta,0) 
- \inf_{\nu\in\cP_1([-R,R])}\DKL(\nu|\gamma)= \theta^2  +O(e^{-cR^2})
\]
(for the error bound see Lemma \ref{lem:freeprops}).
More generally, if $\|w\|_\infty=O(N^{-1/4})$ and $\omega(1)\le R=o(N^{1/4})$ then the arguments of $\LLa_\mu$ in \eqref{def:freeL} and \eqref{def:freeX} are of size $o(1)$, and from Taylor expanding $\LLa_\mu(t) \approx   \frac12t^2$ we get
\[
\freeL(\theta,w)  \approx   \theta^2\|w\|_2^4\,,\qquad
\free^{cross}_R(\theta,w)  \approx   2\theta^2\|w\|_2^2(1-\|w\|_2^2)
\]
and hence $\free_{N,R}(\theta,w) \approx   \theta^2$ in this case as well.
(Here we used that the infimum of $\DKL(\nu|\gamma)$ over $\cP_{1-\|w\|_2^2}([-R,R])$ is $ \approx  \frac12\|w\|_2^2$, attained by a truncated centered Gaussian.)
On the other hand,
for $w$ of norm 1 we reduce to the localized contribution
\[
\free_{N,R}(\theta,w) = \freeL(\theta,w)\,,\qquad w\in \sphereN.
\]

\section{Outlook and open questions}
\label{sec:open}

Theorem \ref{thm:rateN} gives a complete characterization of the large deviation rate for $\lam_1(H)$ in terms of the restricted annealed free energy minimax problem $\inf_\vloc\sup_\theta\cJ_N(x,\vloc)$.
Under further assumptions we have established a genuine LDP {on the full line,} and proved the onset of a localization phenomenon for $v_1$ coinciding with the transition to a non-universal rate function. 
However, analyzing the minimax problem to extract {an explicit} limiting rate function and the conditional structure of $v_1$ at a level of detail comparable to Theorem \ref{thm:increasing} remains a challenging problem in general.

\subsubsection*{Single transition to non-universality?}

For the rate function $\rate^\mu:\R\to[0,+\infty]$ provided by Theorem \ref{thm:fullLDP}, we have from Theorem \ref{thm:smallx} that $\{x:\rate^\mu(x)=\rate^\gamma(x)\}$ contains an open neighborhood of $(-\infty, 2]$, while Corollary \ref{cor:nonuniv} shows this set is bounded from above. Is this set connected? 

\subsubsection*{Relaxing distribution assumptions}
For the large deviation principle of Theorem \ref{thm:fullLDP}, the condition \eqref{assu:usg} and  the existence of the limits in 
\eqref{assu:LRlim}
are natural assumptions, and we conjecture that without them the LDP may not hold in general. In particular,  it should then be possible for the random matrix to alternate between different localization phenomena infinitely often as $N\to\infty$. 
(We conjecture  that \eqref{assu:usg} can be dropped for the upper bound on $\log\P(|\lam_1-x|\le \delta)$ in Theorem \ref{thm:rateN} -- note that it is not needed for the joint upper bound of Proposition \ref{prop:upper-joint}.)

It would be interesting to drop the assumption that the limits are equal in \eqref{assu:LRlim}, or to drop the assumption
\eqref{assu:maxR+}.
If the supremum of $\psimu(t)$ were attained on $\R^-$, or if its left limit were larger than the right limit, then we would expect ``bipartite'' localization strategies to emerge. For example, with $\psimu(t)$ monotone decreasing on $\R^-$ and increasing on $\R^+$ as in Theorem \ref{thm:increasing}, but with  $\psimax =\lim_{t\to-\infty}\psimu(t)>\lim_{t\to+\infty}\psimu(t)$, it would no longer be optimal for $v_1$ to localize to a single coordinate, coinciding with $H$ having a single large diagonal entry. Instead we might expect $H$ to have a large off-diagonal entry $X_{ij}$, with $v_1$ localized to the two sites $i,j$. 
Dropping either of the assumptions \eqref{assu:LRlim}, \eqref{assu:maxR+} would lead to more terms in \eqref{def:tfree} and significantly complicate the analysis of the variational problem for the rate function, and tilting constructions for matching large deviation lower bounds, so we leave these questions for future work.

Removing \eqref{assu:maxR+} would let us drop the constraint $p\le \frac12$ in Corollary \ref{cor:ER}, or equivalently, to establish the LDP for $\lam_N(A-\E A)$ when $p\le \frac12$. The connection between the smallest eigenvalue of adjacency matrices and bipartite structure in the graph is well known \cite{Chung:book}.

It would also be interesting to allow the diagonal entries to have different variance from the GOE scaling \eqref{def:H}, or to allow different distributions on and off the diagonal. Such a setup was considered in the context of sparsified Wigner matrices in the recent work \cite{AuBa}.

\subsubsection*{Structure of $v_1$ in the compact case}

When $\mu$ is compactly supported we have $\psiinfty=0$ and the expressions \eqref{def:tfree}, \eqref{def:tjayN} for the rate function in Theorem \ref{thm:fullLDP} simplify, though the resulting expressions are more complicated than in Theorem \ref{thm:increasing} where $\psimax=\psiinfty$.
In particular, we need to consider localized vectors $\chw$ of 
unbounded support in variational problem $\VP_R(\theta\chw,\beta)$ in \eqref{def:tfree}.
We conjecture that the optimum in \eqref{def:trate} is attained with $\chz=0$, which would allow us to deduce that the localized part of $v_1$ is spread over $\asymp \sqrt{N}$ coordinates of size $\asymp N^{-1/4}$.  This scenario is shown to be optimal for large enough deviation in the proof of \cite[Proposition 5]{AGH}.
In the setting of \ER graphs this should coincide with the appearance of a clique on $\asymp \sqrt{N}$ vertices.

\section{Proof ideas}
\label{sec:ideas}

In Sections \ref{sec:ideas-classical}--\ref{sec:ideas-new} we give an informal sketch of the main ideas behind the proof of our core result Theorem \ref{thm:rateN}, and in particular of Theorem \ref{theomain} relating the upper tail for $\lam_1$ to the quenched and annealed restricted free energies for spherical integrals. The sketch includes important ideas from the preceding works \cite{HuGu1,AGH}. 
We conclude in Section \ref{sec:notation} with a summary of notational conventions that will be used throughout the article. 

In Sections \ref{sec:ideas-classical}--\ref{sec:ideas-new}, when discussing estimates for events of the form $\P(|Y-x|\le \delta)$ for small $\delta>0$ and a random variable $Y$ depending on $N$, to lighten notation we informally
write 
$\error$
for quantities that tend to 0 after taking $N\to\infty$ and then $\delta\downarrow0$:
\begin{equation}	\label{def:Err}
\lim_{\delta\downarrow0}\limsup_{N\to\infty} |\error| = 0.
\end{equation}
We will ignore issues of uniformity of the errors with respect to auxiliary parameters $\theta,\eta$, etc. 
In the proofs we often need to allow parameters to depend on $N,\delta$, significantly complicating the arguments, and we prefer to omit such technicalities here.

\subsection{The classical tilting argument}
\label{sec:ideas-classical}

A basic method for estimating the probability of a large deviation for a scalar random variable is to consider \emph{tilted distributions}. 
To motivate the approach for large deviations of $\lambda_1(H)$, we first sketch the key steps of the proof of the classical Cram\'er LDP for
 the sample mean $\overline{X}_N = \frac1N(X_1+\cdots + X_N)$ of iid centered random variables with distribution $\mu$. (For the full proof see for instance \cite[Chapter 2]{DZ}.) 
Informally, it says 
\begin{equation}	\label{Cramer}
\frac1N\log\P( |\overline{X}_N -x|\le \delta) = - \LLa_{\mu}^*(x) + 
\error
\end{equation}
(recall the notation from \eqref{def:Err})
for any fixed $x\in \R$, where $\LLa_{\mu}^*$ is the Legendre--Fenchel transform of the log-Laplace transform $\LLa_{\mu
}$ that is, $\LLa_{\mu}^*(x):=\sup_{\theta\in \R}\{ \theta x- \LLa_{\mu}(\theta)\}$. 
 For simplicity we consider the case that $\LLa_{\mu}(\theta)$ is finite for all $\theta\in\R$.

 In order to estimate $\P( |\overline{X}_N - x|\le \delta)$ one considers the one-parameter family of measures
\[
\P^{(\theta)}(\,\cdot\,) = \frac{ \E e^{ \theta N\overline{X}_N}\ind( \,\cdot\,) }{ \E e^{ \theta N\overline{X}_N}}\,,\quad \theta \in \R\,.
\]
In terms of these measures we can re-express
\begin{align}
\P( |\overline{X}_N -x|\le \delta) 
& =\E  \frac{ e^{\theta N \overline{X}_N}}{e^{\theta N\overline{X}_N}}\cdot \ind( |\overline{X}_N-x|\le \delta) \notag
\\
& = e^{-N(\theta x+
\error 
)} \E e^{\theta N\overline{X}_N}\ind( |\overline{X}_N-x|\le \delta) \notag\\
& =\P^{(\theta)}( |\overline{X}_N - x|\le \delta)\cdot \exp\big\{-N(\theta x-\LLa_\mu(\theta)+
\error 
)\big\} 	\label{intro-tilt}
\end{align}
where in the second line we used the restriction to the large deviation event to approximate the factor of $\exp(\theta N\overline{X}_N)$ from the denominator of the preceding line, and in the third line we used the fact that $\LLa_\mu(\theta) =  \log \E e^{\theta X_1}= \frac1N \log \E e^{\theta N \overline{X}_N}$ (since the variables are iid).
To obtain the upper bound in \eqref{Cramer} one can simply bound $\P^{(\theta)}(|\overline{X}_N-x|\le \delta)$ by one  in \eqref{intro-tilt} and optimize in $\theta$.

To show the matching lower bound in \eqref{Cramer} requires a closer examination of the measures $\P^{(\theta)}$. Indeed, from \eqref{intro-tilt} we see it suffices to show that the event $|\overline{X}_N-x|\le \delta$ is likely under $\P^{(\theta_x)}$, where $\theta_x$ is the optimizer in the definition of $\LLa_\mu^*(x)$. To see this, we note that under $\P^{(\theta)}$ we have that $N\overline{X}_N$ is a sum of iid variables with mean $\LLa_\mu'(\theta)$. Hence, we have $\E^{(\theta)}\overline{X}_N=\LLa_\mu'(\theta)$, and it is straightforward to show that $\overline{X}_N$ concentrates around this value under $\P^{(\theta)}$. Thus,
if $x$ lies in the range of $\LLa_\mu'$ (which is smooth and strictly increasing on $\R$) then letting $\theta$ be the unique solution of $x=\LLa_\mu'(\theta)$, we have $\P^{(\theta)}(|\overline{X}_N-x|\le \delta)\ge \exp(o(1)N)$. Noting that this choice of $\theta$ is precisely $\theta_x$, so the right hand side of \eqref{intro-tilt} is $\exp( -N(\LLa_\mu^*(x) +
\error ))$, we thus obtain the matching lower bound in \eqref{Cramer}.
When $x$ is not in the range of $\LLa_\mu'$ one verifies that both sides in \eqref{Cramer} diverge to $-\infty$.

The intuition is that with the measures $\P^{(\theta)}$ we are re-weighting the distribution of $\overline{X}_N$ so that the large deviation event becomes likely. 

\subsection{Tilting by spherical integrals}
\label{sec:ideas.GH}

We point out that Cram\'er's argument sketched above hinges on the fact that the moment generating function $\E \exp( \theta N\overline{X}_N)$ is straightforward to compute, owing to the independence of the summands $X_i$. 
Indeed,  a na\"ive attempt to apply this argument to the largest eigenvalue $\lambda_1=\lambda_1(H)$ of a Wigner matrix immediately runs into the problem that there is no easy way to compute the moment generating function $\E \exp( \theta N\lambda_1)$. 

A way to extend the tilting approach to obtain an LDP for $\lambda_1$ was found in \cite{HuGu1}. Rather than na\"ively tilt the distribution of $H$ by $\exp( \theta N\lambda_1)$, the key is to tilt by the \emph{spherical integral}
$I(H,\theta)$ defined in \eqref{def:spherical}.
One may view $u$ as a random vector with distribution $P$, independent of the Wigner matrix $H$, but we choose to keep this integration separate from the probability space. 
Two features of the spherical integral 
\eqref{def:spherical} make it well suited for the Cram\'er tilting strategy:
\begin{enumerate}
\item From results in \cite{HuGu1} it asymptotically depends in a smooth and monotone way on $\lambda_1$. Specifically, on the event that the bulk of the spectrum of $H$ is well approximated by the semicircle law (an event which fails with negligible probability of size 
{$\exp(-\omega(N))$})
(see Section \ref{sec:notation} for our conventions on asymptotic notation)
 we have 
\begin{equation}	\label{intro-Jasymp}
\frac1N\log I(H,\theta) = J(\lambda_1,\theta) +o(1)
\end{equation}
where $J(x,\theta)$ is defined in \eqref{def:Jsc}. (See Lemma \ref{lem:approxsc} for a precise statement.)
  One may hence expect to learn about large deviations of $\lambda_1$ from reweighting the distribution of $H$ by $I(H,\theta)$.
\item Unlike the exponential moment $\E \exp(\theta N\lam_1)$, the \emph{annealed spherical integral} $\E I(H,\theta)$ 
is tractable to compute, as the quadratic form $\langle u,Hu\rangle$ separates into a sum of independent random variables. 
\end{enumerate}

We now sketch the proof of Theorem \ref{thm:HuGu}.
We show
\begin{align}
\label{intro-GH.tail1}
\frac1N\log\P(|\lam_1-x|\le \delta ) 
&= \inf_{\theta\ge0} \big\{ F_N(\theta) - J(x,\theta) \big\} +
\error \\
&= -\rate^\gamma(x)+
\error 	\label{intro-GH.tail2}
\end{align}
where we recall the annealed free energy density $F_N(\theta)= \frac1N\log\E I(H,\theta)$. 
Recalling also $\freeL$ from \eqref{def:freeL}, note that
\begin{equation}	\label{freeL0}
\freeL(\theta,u) 
:= \frac1N \sum_{i\le j} \LLa_\mu(2^{\ep_{ij}} \theta\sqrt Nu_iu_j)
= \frac1N\log \E e^{\theta N\langle u, Hu\rangle}
\end{equation}
and from Fubini's theorem,
\begin{equation}	\label{FN-Fubini}
F_N(\theta) 
=  \int_{\sphereN} \E e^{\theta N \langle u,Hu\rangle}dP(u)
= \int_{\sphereN} e^{N\freeL(\theta,u)}dP(u)\,.
\end{equation}
We introduce a
family $\P^{(\theta,u)}$ of tilted measures on the background probability space
with Radon--Nikodym derivatives
\begin{equation}	\label{def:tiltP0}
\frac{d\P^{(\theta,u)}}{d\P} := e^{\theta N \langle u, Hu\rangle - N\freeL(\theta,u)},
\qquad \theta\ge0\,,\; u\in\sphereN \,.
\end{equation}
We further define a family $Q^{(\theta)}$ of tilted measures on $\sphereN$ with densities
\begin{equation}	\label{def:tiltQ0}
\frac{dQ^{(\theta)}}{dP}(u) := e^{N\freeL(\theta,u) - NF_N(\theta)}\,,\qquad \theta\ge0\,.
\end{equation}
We express the large deviation probability $\P(|\lam_1-x|\le \delta)$ in terms of these tilted measures. 
Let $\cG$ be the event that \eqref{intro-Jasymp} holds. Thus $\P( \cG) = 1-\exp( -\omega(N))$, so it suffices to estimate $\P( \cE_x)$, where we set
\begin{equation}	\label{ideas:Ex}
\cE_x:= \{ |\lam_1-x|\le \delta\}\cap  \cG.
\end{equation}
Fix an arbitrary $\theta\ge0$.
From \eqref{intro-Jasymp} and continuity of $x\mapsto J(x,\theta)$, we have
\begin{align} \label{intro-GH1.2}
\P( \cE_x  ) 
&=   \E \frac{ I(H,\theta)}{I(H,\theta)} \ind(\cE_x  ) 	
= e^{- N(J(x,\theta)+
\error 
)} \E I(H,\theta) \ind(\cE_x  )		
\end{align}
In terms of $\P^{(\theta,u)}, Q^{(\theta)}$ we can rewrite
\begin{align}
 \E I(H,\theta) \ind(\cE_x )	
&= 
\int_{\sphereN} \E e^{\theta N\langle u,Hu\rangle} \ind( \cE_x  )dP(u)	\label{intro-GH1.3}\\
&= \int_{\sphereN} \P^{(\theta,u)}( \cE_x  ) \,
e^{N\freeL(\theta,u)}
dP(u)	\notag\\
&=
e^{NF_N(\theta)} \int_{\sphereN} \P^{(\theta,u)}(\cE_x ) dQ^{(\theta)}(u)\,.	\notag
\end{align}
Combining with \eqref{intro-GH1.2}, we have
\begin{equation}
\label{intro-GH1}
\P( \cE_x  ) 
= e^{N(F_N(\theta)-J(x,\theta)+\error)} 
\int_{\sphereN} \P^{(\theta,u)}(\cE_x ) dQ^{(\theta)}(u) \,.
\end{equation}
The reader may compare with the lines leading to \eqref{intro-tilt} -- the major difference here is the additional integration over the high-dimensional sphere. 

By trivially bounding $ \P^{(\theta,u)}( \cE_x  )\le1$ 
and $Q^{(\theta)}(\sphereN)=1$
in \eqref{intro-GH1} we get
\begin{equation}
\frac1N\log \P(\cE_x  ) \le F_N(\theta) - J(x,\theta) + 
\error 
\end{equation}
showing \eqref{intro-GH.tail1} holds as an upper bound.

To prove the matching lower bound, from \eqref{intro-GH1} we see it is enough to show that there exists $\theta\ge0$ and a set $\Aset\subset\sphereN$ 
such that
\begin{equation}
\label{GH-LB-goalQ}
Q^{(\theta)}(\Aset) \ge e^{o(N)}
\end{equation}
and
\begin{equation}
\label{GH-LB-goalP}
\P^{(\theta,u)}(\cE_x ) \ge e^{o(N)} \quad \forall u\in \Aset\,.
\end{equation}
To get \eqref{GH-LB-goalP} it will suffice that $u$ be delocalized: specifically, that $\|u\|_\infty=o(N^{-1/4})$. 
Thus, denoting the set of \emph{$R$-delocalized vectors} 
\[
\Deloc_R:=\{v\in\B^N:\|v\|_\infty\le RN^{-1/2}\}
\]
we take $\Aset=\Deloc_{N^\eta}$ for $\eta<\frac14$.
To see why this is sufficient for \eqref{GH-LB-goalP}, we note that the
tilted means of the entries are
\begin{equation}	\label{intro.tilted-means}
\E^{(\theta, u)} H_{ij} =
\sqrt{\frac{2^{1_{i=j}}}N} \LLa_\mu'(2^{\ep_{ij}} \theta\sqrt{N} u_iu_j)\,.
\end{equation}
Since
$\LLa_\mu(t) \approx   \frac12t^2$ for small $t$,
if $\|u\|_\infty=o(N^{-1/4})$ then
$\E^{(\theta,u)} H\approx 2\theta uu^\tran$, a rank-one matrix,
and in fact one can show that under $\P^{(\theta,u)}$ we have an approximation in law
\begin{equation}	\label{approxd}
H\approxd 2\theta uu^\tran + \wt H
\end{equation}
for a Wigner matrix $\wt H$.
The largest eigenvalue of $H$ can be approximated using the classic BBP computation for the largest eigenvalue of a Wigner matrix under a rank-one perturbation \cite{BBP}, which gives
\begin{equation}	\label{intro.BBP}
\lambda_1 \approx   2\theta + \frac1{2\theta}
\end{equation}
with probability $1-o(1)$ under $\P^{(\theta,u)}$, for $\theta\ge\frac12$  (See \cite[Lemma 5.2]{HuGu1}).
Then noting that the right hand side is equal to $x$ for $\theta=\thetap$, we obtain \eqref{GH-LB-goalP}.

For \eqref{GH-LB-goalQ}, again from Taylor expansion we find that 
\begin{equation}	\label{freeL-deloc}
\freeL(\theta,u)= \theta^2+o(1)
\end{equation}
uniformly for $u\in \Aset$.
Together with \eqref{FN.sharp} this implies the density $dQ^{(\theta)}/dP$ is uniformly bounded below by $e^{o(N)}$ on $\Aset$. 
Since a random unit vector drawn from the uniform measure $P$ satisfies $\|u\|_\infty=N^{-1/2+o(1)}$ with probability $1-o(1)\ge \frac12$, we conclude $Q^{(\theta)}(\Aset) \ge e^{o(N)}P(\Aset) \ge e^{o(N)}$, giving \eqref{GH-LB-goalQ} to complete the proof of \eqref{intro-GH.tail1}.

Turning to prove \eqref{intro-GH.tail2}, from \eqref{freeL0} and the sharp sub-Gaussian hypothesis,
\begin{equation}
f_N(\theta, u) \le \frac1{2N} \sum_{i\le j}  2^{1+1_{i\ne j}} \theta^2 Nu_i^2u_j^2 = \theta^2\sum_{i,j} u_i^2u_j^2 = \theta^2
\end{equation}
for all $u\in \sphereN$, and hence $F_N(\theta)\le \theta^2$.
On the other hand, from \eqref{freeL-deloc},
\begin{equation}
e^{NF_N(\theta)} \ge \int_{\Aset}e^{Nf_N(\theta,u)}dP(u) = e^{\theta^2N+o(N)} P(A) = e^{\theta^2N + o(N)} 
\end{equation}
so
\begin{equation}	\label{FN.sharp}
F_N(\theta) 
= \theta^2 + o(1).
\end{equation}
Inserting this limiting value into \eqref{intro-GH.tail1} and optimizing over $\theta$, we have 
\begin{equation}	\label{intro-GH.UB1}
\frac1N\log \P(\cE_x  ) = \inf_{\theta\ge0} \{ \theta^2 - J(x,\theta)\} + 
\error 
\end{equation}
A computation shows that the infimum is achieved at $\thetap$ (recall \eqref{def:thetapm}), and moreover that the main term on the right hand side above can be expressed
\begin{equation}	\label{intro-GH.UB2}
\thetap^2  - J(x,\thetap  ) = - \frac12\int_2^x\sqrt{y^2-4}dy =- \rate^\gamma(x) 
\end{equation}
and \eqref{intro-GH.tail2} follows.

\subsection{New ideas to capture localization phenomena}
\label{sec:ideas-new}


From \eqref{intro-GH1} we can understand that for the sharp sub-Gaussian Wigner matrices, the main mechanism underlying a deviation of $\lambda_1$ to the neighborhood of some $x>2$ is for the entries of $H$ to collectively deviate in the direction of a rank-one matrix $uu^\tran$ which is ``delocalized'' in the sense that all entries of $uu^\tran$ are of size $o(N^{-1/2})$. 
In the general sub-Gaussian case we need to account for additional ``localized'' strategies, such as the existence of a single large entry of $H$ of size order one (whereas the typical size is of order $N^{-1/2}$). In fact, the key point is that deviations of $\lambda_1$ can occur due to a \emph{mixture} of localized and delocalized perturbations. 

To capture this, we keep track of the large entries of the eigenvector $v_1$ associated with $\lambda_1$.
For fixed $\eta\in(0,\frac14)$, recall the notation $v_1^{(\eta)}$ from \eqref{def:v1eta} for the restriction of $v_1$ to its entries of size at least $N^{-1/2+\eta}$.



\subsubsection{Upper bound}

For fixed $x>2$ and a vector $z\in \B^N$ supported on at most $N^{1-2\eta}$ coordinates, we denote the event
\begin{equation}	\label{ideas:Exz}
\cE_{x,z} := \big\{ 
\lam_1\approx x\,,\; v_1^{(\eta)}\approx z
\big\}\,.
\end{equation}
Our approach to the upper bound is to prove a sharp joint large deviation upper bound for the pair $(\lam_1,v^{(\eta)})$ of the form
\begin{equation}	\label{ideas:joint-upper}
\frac1N\log\P( \cE_{x,z}) \le - \cJ_N(x,z) + \error
\end{equation}
for fixed $x>2$ and sparse vector $z$ in the ball, with $\cJ_N$ as in \eqref{def:jayN0}.
See Proposition \ref{prop:upper-joint} for a precise statement. 
A large deviation upper bound for $\lam_1$ is obtained by minimizing $\cJ_N(x,z)$ over $z$, leading to the $N$-dependent rate function in Theorem \ref{thm:rateN}.
Minimizing over $z$ amounts to selecting the least unlikely localized part of $v_1$; if the minimum is attained at $z=0$ then we are reduced to the GOE rate function, leading to Theorem \ref{thm:smallx}. 
The joint upper bound \eqref{ideas:joint-upper} also allows us to easily deduce the statements on the conditional structure of $v_1^{(\eta)}$ in Theorems \ref{thm:smallx}, \ref{thm:increasing} and Corollary \ref{cor:nonuniv}.

Furthermore, under the hypotheses of Theorem \ref{thm:fullLDP}, we can use a pigeonholing argument to locate a gap in the sizes of the entries of $z$, allowing us to reduce $\cJ_N$ to the modified joint rate function $\wt\cJ_{N,R}$ of \eqref{def:tjayN}, from which we can get a genuine limiting rate function.

Turning to describe the proof of \eqref{ideas:joint-upper}, we follow the pattern of the argument from \cite{HuGu1}, but taking advantage of the restriction on $v_1^{(\eta)}$ to localize the spherical integral. 
Indeed, we can show that outside a negligible event (including, among others, the relatively rare event that $H$ has more than one eigenvalue near $x$), the spherical integral $I(H,\theta)$ concentrates on the portion of the sphere where $|\langle u, v_1\rangle|\approx \overlap_x(\theta)$, where the overlap function $\overlap_x(\cdot)$ was defined in \eqref{def:overlapx}. (For the precise statement see Lemma \ref{lem:restrict}.)
Together with the restriction to $\cE_{x,z}$ we can show 
\begin{equation}	\label{ideas:restrict}
I(H,\theta) \approx \int_{\Uloc_{\overlap_x(\theta)z}} e^{\theta N\langle u, Hu\rangle} dP(u)
\end{equation}
with $\Uloc_{\overlap_x(\theta)z}$ as in \eqref{def:UwrR}. 
Let $\cE_{x,z}'$ denote the intersection of $\cE_{x,z}$ with the high probability event that \eqref{ideas:restrict} holds, along with the event $\cG$ that \eqref{intro-Jasymp} holds.
Then arguing similarly to \eqref{intro-GH1.2}, we have
\begin{align*}
\P( \cE_{x,z}') &= \E \frac{I(H,\theta)}{I(H,\theta)}\ind(\cE_{x,z}')\\
&= e^{-N(J(x,\theta) + \error)} \E I(H,\theta)\ind(\cE_{x,z}')\\
&=e^{-N(J(x,\theta) + \error)} \E \int_{\Uloc_{\overlap z}}e^{\theta N\langle u, Hu\rangle}\ind(\cE_{x,z}')dP(u)\\
&\le e^{-N(J(x,\theta) + \error)} \E \int_{\Uloc_{\overlap z}}e^{\theta N\langle u, Hu\rangle}dP(u)\\
&= e^{N(F_N(\theta; \Uloc_{\overlap z}) - J(x,\theta) + \error)} 
\end{align*}
where we abbreviate $\overlap=\overlap_x(\theta)$, and we recall the restricted annealed free energy $F_N(\theta;\Uloc)$ from \eqref{def:FN.U}. 
Applying Proposition \ref{prop:FE} to replace $F_N(\theta; \Uloc_{\overlap z})$ with $\free_{N,N^{1/5}}(\theta, \overlap z)$ and then optimizing $\theta$, we obtain the desired upper bound \eqref{ideas:joint-upper}.

A technical point we have skipped is that in order to use \eqref{ideas:joint-upper} with a covering argument to establish bounds of the form
\[
\P( \lam_1\in I, v_1^{(\eta)} \in A) 
\]
for larger sets $I, A$, we need $\cJ_N$ to be continuous on $\R\times \B$ (in a suitable quantitative sense). It turns out the continuity can fail if the second argument is near the boundary of $\B$, which corresponds to the event $v_1^{(\eta)}\approx v_1$, i.e. $v_1$ is \emph{fully localized}.
We hence need a separate argument showing that this event is negligible even in the large deviations regime. We do this is Section \ref{sec:no-comp}.

\subsubsection{Lower bound}

Recall the tilted measures $\P^{(\theta,u)},Q^{(\theta)}$ defined in \eqref{def:tiltP0}, \eqref{def:tiltQ0}.
We would like to follow the argument of \cite{HuGu1} summarized in \eqref{intro-GH1} and \eqref{GH-LB-goalQ}--\eqref{GH-LB-goalP}. 
There, we could select a tilting parameter $\theta=\theta_x$ by an explicit BBP computation, thanks to the fact that most vectors $u\in\sphereN$ under $Q^{(\theta)}$ were delocalized. 
Here, however, it is crucial to restrict to vectors in the sphere with a localized part $v_1^{(\eta)}\approx z$, so that both \eqref{GH-LB-goalQ} and \eqref{GH-LB-goalP} fail to hold in our setting. 

With $\cE_x$ as in \eqref{ideas:Ex}, for any fixed $\theta\ge0$ we have from \eqref{intro-GH1.2} and the first two lines of \eqref{intro-GH1.3}
\[
\P(\cE_x)
=e^{-N(J(x,\theta) + \error) } \int_{\sphereN}  \P^{(\theta,u)}(\cE_x)e^{N\freeL(\theta,u)}dP(u).
\]
Now for any sparse $z\in \B^N$ we can restrict the spherical integral to lower bound
\begin{align}
\P(\cE_x)
&\ge e^{-N(J(x,\theta) + \error) } \int_{\Uloc_{\overlap z}}  \P^{(\theta,u)}(\cE_x)e^{N\freeL(\theta,u)}dP(u)	\notag\\
&= e^{N(F_N(\theta; \Uloc_{\overlap z}) - J(x,\theta) + \error) } \int_{\sphereN}  \P^{(\theta,u)}(\cE_x) dQ^{(\theta)}(u| \Uloc_{\overlap z})	\label{ideas:LB1}
\end{align}
where in the second line we are integrating with respect to the tilted measure $Q^{(\theta)}$ \emph{conditioned to $\Uloc_{\overlap z}$}, i.e. 
\begin{equation}	\label{def:tiltQ0-cond}
Q^{(\theta)}(\Aset| \Uloc_{\overlap z}) := \frac{Q^{(\theta)}(\Aset\cap \Uloc_{\overlap z})}{Q^{(\theta)}(\Uloc_{\overlap z})}
\end{equation}
for Borel sets $\Aset\subset\sphereN$. 

Now to conclude the lower bound for Theorem \ref{thm:rateN}, we need to show that for any sparse $z$ there exists $\hat\theta=\hat\theta_{x,z}\ge0$ such that 
\begin{equation}	\label{ideas:LBgoal1}
 \int_{\sphereN}  \P^{(\hat\theta,u)}(\cE_x) dQ^{(\hat\theta)}(u| \Uloc_{w(\hat\theta)})
 \ge e^{o(N)}
\end{equation}
where we now emphasize a key challenge that $w(\theta)=\overlap_x(\theta)z$ and hence the set $\Uloc_{w(\theta)} $ depend on $\theta$. 

Similarly, for the lower bound of Theorem \ref{thm:fullLDP}, we need to show that for any fixed $\chz\in \B_{\ge\thresh}$ and $\tal\ge0$ such that $\overlap^2(\|\chz\|^2+\tal)\le 1$, there exists $\hat\theta=\hat\theta_{x,\chz,\tal}\ge0$ such that \eqref{ideas:LBgoal1} holds with $w(\theta)=\overlap_x(\theta)(\chz + \t z)$, where $\t z=\t z(\theta)$ is a vector of squared $\ell^2$-norm $\tal$ taking a certain constant value of order $N^{-1/4}$ on its support.
The value is chosen so that the localized contribution $\freeL(\theta,\t z)$ to $\free_{N,R}(\theta, z)$ is $2\theta^2\psimax\tal^2$, so that we can match the expression \eqref{def:tfree} obtained in the proof of the upper bound. 

As in the argument of \cite{HuGu1} we  split the task of proving \eqref{ideas:LBgoal1} into two steps of the form \eqref{GH-LB-goalQ}, \eqref{GH-LB-goalP}. 
However, the sets $\Aset\subset\sphereN$ must now depend on $\theta$ and $z$. 

We show that there is a family of sets $\Aset_\theta\subset\sphereN$ indexed by $\theta\ge\frac12$ and a continuous curve $[\frac12,\infty)\ni \theta\mapsto x(\theta)$ such that for all $\theta\ge\frac12$,
\begin{equation}
\label{LB-goalQ}
Q^{(\theta)}(\Aset_\theta | \Uloc_{w(\theta)}) \ge e^{o(N)}
\end{equation}
and
\begin{equation}
\label{LB-goalP}
\P^{(\theta,u)}(\cE_{x(\theta)}) \ge \frac12 \qquad \forall u\in \Aset_\theta
\end{equation}
Moreover, $x(\theta)$ will satisfy
\begin{equation}	\label{xcurve}
x(\tfrac12)\approx 2\qquad \text{ and } \qquad x(\theta)\to +\infty \quad\text{ as } \theta\to+\infty.
\end{equation}
By the intermediate value theorem it will then follow that $x(\hat \theta)=x$ for some $\hat\theta\in(\frac12,\infty)$, completing the proof.

For \eqref{LB-goalQ}, recall that on $\Uloc_{w(\theta)}$ we have $u|_{\supp(w)}\approx w$. 
We show that for $u\in\sphereN$ drawn from the conditional tilted measure $Q^{(\theta)}(\,\cdot\,|\Uloc_{w(\theta)})$, up to permutation of the coordinates the restriction of $u$ to $[N]\setminus \supp(w)$ concentrates in a small neighborhood of a delocalized vector $\t v(\theta)$ determined by $\theta$ and $w(\theta)$. Moreover, $\t v$ varies continuously in $\ell^2$ as a function of $\theta$ (in a suitable quantitative sense). We hence obtain \eqref{LB-goalQ} with $\Aset_\theta$  a small neighborhood of the set of all vectors obtained from 
\begin{equation}	\label{optimalu}
u(\theta):=w(\theta) + \t v(\theta)
\end{equation}
obtained by permuting the coordinates of $\tv$. See Proposition \ref{prop:tiltQ} for the precise statement. 

Turning to  \eqref{LB-goalP}, we take
\begin{equation}
x(\theta):= \E^{(\theta, u(\theta))}\lam_1(H).
\end{equation}
It is shown in Proposition \ref{prop:tiltP} that this choice of $x(\cdot)$ is continuous and satisfies \eqref{xcurve}. We also show that $\lam_1(H)$ concentrates around its expectation under $\P^{(\theta,v)}$ for any $v\in\sphereN$, and moreover that $\E^{(\theta,v)}\lam_1(H)$ depends continuously on $v$ (in a suitable quantitative sense). From this it follows that for any fixed $u\in\Aset_\theta$, with high probability under $\P^{(\theta,u)}$ we have 
\[
\lam_1(H)\approx \E^{(\theta,u)}\lam_1(H) \approx \E^{(\theta, u(\theta))}\lam_1(H) = x(\theta)
\]
giving \eqref{LB-goalP}. \\

We remark on some of the technical challenges for making the above sketch rigorous. 

The necessary properties of the tilted measures $\P^{(\theta,u)}$ are established in Section \ref{sec:tiltP}. We wish to highlight a coupling argument for the continuity properties of these measures under variation of $\theta$ and $u$; see Lemma \ref{lem:NiceConnection0}.

The continuity and localization properties for the conditional tilted measures $Q^{(\theta)}(\,\cdot\,|\Uloc_{w(\theta)})$ require some substantial work. They are established in Section \ref{sec:annealed}, where we also prove Proposition \ref{prop:FE} on the restricted annealed free energies (note that $e^{NF_N(\theta;\Uloc_w)}$ is the partition function for $Q^{(\theta)}(\,\cdot\,|\Uloc_{w(\theta)})$).  

Establishing continuity for the tilted measures requires an understanding of the constrained Gibbs variational problem of \eqref{def:VP}, which is the subject of Section \ref{sec:gibbs}.
Results of this section are also needed for various regularity properties, established in Section \ref{sec:rateprops}, for the asymptotic free energies $\free_{N,R}(\theta,w)$ (see Lemma \ref{lem:freeprops}) and joint rate functions $\cJ_N(x,z)$ (Lemma \ref{lem:Jprops}) that are repeatedly invoked in the proofs. (As for the upper bound argument, delicate issues of continuity of the rate functions were glossed over in the above sketch.) 

\section{Notational conventions}
\label{sec:notation}

For $n\in\N$ we write $[n]$ for the discrete interval $\{1,\dots, n\}$.
For a statement $Q$ we write $1_Q$ for the associated Boolean variable. For a set $S$ we sometimes abusively write $1_S(\cdot)$ for the function with output $1_S(x)=1_{x\in S}$. 
We write 
$\ep_{ij}:=\frac12(1+1_{i\ne j})$ (which often enters into formulas for exponential moments due to the different variances of entries on and off the diagonal -- see \eqref{def:H}).
The Lipschitz constant of a function $f:\R\to\R$ is denoted $\|f\|_{\Lip}:=\sup_{x\ne y} |f(x)-f(y)|/|x-y|$.
Recall from Remark \ref{rmk:assu.reg} the notation $\tLL:=\LLa_\mu\circ\,\ssq^{-1}$, where we denote the signed square function 
\begin{equation}	\label{def:ssq}
\ssq(x)=x^2(1_{x\ge0}-1_{x<0})\,,\qquad x\in\R
\end{equation}
with inverse $\ssq^{-1}(x) =\sgn(x)\sqrt{|x|}$. We noted in Remark \ref{rmk:assu.reg} that $\|\tLL\|_{\Lip}<\infty$.

\subsection{Asymptotic notation and parameters}
$C,C_0,c$ etc. denote positive constants that may change from line to line. Constants may depend on $\mu$; dependence on other parameters will be noted explicitly. (For our assumptions on the measure $\mu$ see Section \ref{sec:assu}.)
For a positive real number $a$ we write $O(a)$ to denote an unspecified real number $b$ satisfying $|b|\le Ca$ for a constant $C>0$. 
We further write $b\ls a$ and $a\gs b$ to mean $b=O(a)$, and $b\asymp a$ to mean $b\ls a\ls b$.  
Dependence of implicit constants on parameters $p$ (not related to the entry distribution $\mu$) is indicated with subscripts, e.g. $O_p(a)$, $b\ls_pa$, etc.

For a positive real $a$ possibly depending on $N$ we write $o(a)$ for a real $b$ depending on $N$ satisfying $b/a\to0$ as $N\to\infty$, while $\omega(a)$ denotes a positive real number $b$ such that $b/a\to \infty$ as $N\to \infty$. We indicate dependence of the rate of convergence on parameters with subscripts.

\subsection{Measures}
$\cP(S)$ stands for the set of probability measures on a set $S$. 
For an interval $I\subset \R$ and $\al>0$ we write
\begin{equation}	\label{def:cPbeta}
\cP_\al(I):= \bigg\{ \nu\in \cP(I): \int x^2 d\nu(x) = \al\,\bigg\}
\end{equation}
for the set of probability measures supported on $I$ and with second moment equal to $\al$.
The standard Gaussian measure on $\R$ is denoted by $\gamma$ and the semicircle measure (with density $\frac1{2\pi}(4-x^2)_+^{1/2}$) by $\sigma$.
The relative entropy (or Kullback--Leibler divergence) of a probability measure $\nu_1$ with respect to a reference measure $\nu_0$ (not necessarily a probability measure) is denoted
\begin{equation}
\label{def:DKL}
\DKL(\nu_1|\nu_0):=\int \log \frac{d\nu_1}{d\nu_0} d\nu_1
\end{equation}
when $\nu_1\ll\nu_0$, and is otherwise equal to $+\infty$.
In particular, when $d\nu(x) = f(x)dx$ is continuous with respect to the Lebesgue measure $dx$ on $\R$ we have that $-\DKL(\nu|dx) = -\int f(x)\log f(x)dx$ is the differential Shannon entropy of $\nu$, and
\begin{equation}	\label{DKL-gamma}
\DKL(\nu|\gamma) = \DKL(\nu|dx) + \frac12\int x^2d\nu(x) + \frac12\log (2\pi).
\end{equation}
The $L^2$-Wasserstein distance on $\cP(\R)$ is denoted
\begin{equation}	\label{def:W2.meas}
\cW_2(\nu_1,\nu_2)=\inf \{ \|X-Y\|_{L^2(\P)}\}
\end{equation}
where the infimum is taken over all couplings $(X,Y)$  such that the marginal distributions of $X,Y$ are $\nu_1$ and $\nu_2$, respectively. 
For compactness we often write 
\[
\nu(f):= \int fd\nu.
\] 

\subsection{Vectors and matrices}

{We view $\R^N$ as a subspace of the Hilbert space $\ell^2(\N)$ of square-summable sequences, consisting of those sequences supported on $[N]$. For a finite set $J\subset\N$ we write $\R^J=\{ v\in \ell^2(\N): \supp(v)\subseteq J\}$ (thus $\R^N=\R^{[N]}$). The closed unit $\ell^2$-balls in $\ell^2(\N), \R^N, \R^J$ are denoted $\B, \B^N, \B^J$, respectively. $\sphereN$ is the boundary of $\B^N$.}
For $v\in \ell^2(\N)$ of finite support we write
\begin{equation}
\|v\|_0 := |\supp(v)|\,.
\end{equation}
The $\ell^p$ norms are denoted $\|\cdot\|_p$ {(there should be no risk of confusion with the $L^p$-norms for Lebesgue spaces)}. 
For matrices, $\|\cdot\|$ denotes the $\ell^2\to\ell^2$ operator norm, and the matrix Hilbert--Schmidt norm (or Frobenius norm) is $\|M\|_\HS = (\sum_{i,j}M_{ij}^2 )^{1/2}$. 

We write $\Sym_N$ for the set of $N\times N$ real symmetric matrices.
The eigenvalues of an element $M\in \Sym_N$ are labeled $\lambda_1(M)\ge\cdots\ge \lambda_N(M)$, and we write $v_1(M),\dots, v_N(M)$ for an associated orthonormal basis of eigenvectors. 
We will tend to write $\lam_i,v_i$ for the random $\lam_i(H),v_i(H)$ when there can be no confusion.
We write $\hat{\mu}_v=\frac1N\sum_{i=1}^N \delta_{v_i}$ for the empirical distribution of the coordinates of  $v\in \R^N$, and
 $\hat{\mu}_M:= \frac1N\sum_{i=1}^N \delta_{\lambda_i(M)}$ for the empirical spectral distribution of  $M\in \Sym_N$.
 
We denote the $\ell^2$-Wasserstein distance on $\R^N$ 
\begin{equation}	\label{def:W2.vec}
d_2(u,v) :=  \min_\varrho \Big( \sum_{i=1}^N |u_i - v_{\varrho(i)}|^2\Big)^{1/2}
\end{equation}
where the minimum is taken over all permutations $\varrho:[N]\to[N]$. (This is a pseudometric, but defines a metric on equivalence classes of vectors that are equal up to permutation of the coordinates.)
We remark that the $\ell^2$-Wasserstein distance is related to the $L^2$-Wasserstein distance (see \eqref{def:W2.meas}) of empirical measures scaled by $\sqrt{N}$:
\begin{equation}	\label{d2W2}
d_2(u,v) = \cW_2(\hat{\mu}_{\sqrt N u},\hat{\mu}_{\sqrt Nv}) .
\end{equation}
We write $\Bset_2(u,\eps)=\Bset_2^N(u,\eps):=\{v\in\R^N: d_2(u,v)<\eps\}$.
For $R\ge1$ we denote the set of \emph{$R$-delocalized} vectors
\begin{equation}	\label{def:deloc}
\Deloc_R:=\Deloc_R^N:= \{ v\in \ball^N: \|v\|_\infty \le RN^{-1/2}\}.
\end{equation}
Recalling the sets $\Uloc_w(\srad,R)$ from \eqref{def:UwrR} (we will often drop the superscript $N$), we note that with $w$ the zero vector, its support is empty, and from \eqref{w-wc} we see that $u_w=0$ and $u_{w^c}=u$. Hence
\begin{equation}	\label{cU0}
\Uloc_0(\srad,R) = \Deloc_R\qquad\forall\srad>0.
\end{equation}
For given $\eta\in(0,\frac14)$ and $v\in\B$, we denote
\begin{equation}	\label{def:eta}
v^{(\eta)}:=(v_i1_{|v_i|>N^{-1/2+\eta}})_i\,,\qquad
n_0= n_0(\eta):= \lf N^{1-2\eta}\rf\,.
\end{equation}

\subsection{Probability space and tilted laws}
\label{sec:notation.tilts}

We fix a background probability space $(\Omega,\cF, \P)$ which supports the random matrix $H$ and all other random variables, and write $\E$ for expectation under $\P$.
We assume that under $\P$, $H$ has the law defined in \eqref{def:H}.
While we could regard integration over the sphere as expectation with respect to a random unit vector $u$ independent of all other variables, we prefer to keep this separate writing $P=P_N$ for the uniform measure on $\sphere^{N-1}$.
Thus, we often work on the extended probability space $(\Omega\times \sphere^{N-1}, \P\otimes P_N)$.
{(This follows the spin glass literature in keeping separate notation for integration over states $u\in\sphereN$ and the disorder $H$.)}

As shown in Section \ref{sec:ideas}, the proofs involve various tilted measures on the sphere and the background probability space -- see for instance \eqref{def:tiltP0}, \eqref{def:tiltQ0}, \eqref{def:tiltQ0-cond}. 
For the reader's aid we summarize the notation here.
First, with $\mu$ the sub-Gaussian distribution of the entries $X_{ij}$ in \eqref{def:H}, we denote the exponentially tilted measures 
\begin{equation}	\label{def:tilt-mu}
\mu^t(A):=  \frac{\int_A e^{t x} d\mu(x)}{\int_\R e^{t x}d\mu(x)}
= \int_A e^{t x - \LLa_\mu(t)} d\mu(x)\,,\qquad t\in \R\,.
\end{equation}

We define families of tilted measures on the background probability space:
\begin{equation}	\label{def:tiltP}
\P^{(\theta,u)}(\,\cdot\,) := \frac{\E e^{\theta N\langle u,Hu\rangle}\ind(\,\cdot\,)}{\E e^{\theta N\langle u, Hu\rangle}}\,,\qquad \theta\ge0\,,\; u\in\sphereN\,.
\end{equation}

For $\theta\ge0\,,M\in\cH_N$  we define measures given for Borel sets $A\subseteq \sphereN$ by
\begin{equation}	\label{def:tiltQM}
Q^{(\theta,M)}( A) := \frac{ \int_A e^{\theta N\langle u, Mu\rangle} dP(u)}{I(M,\theta)}
\end{equation}
recalling the spherical integral $I(M,\theta)$ defined in \eqref{def:spherical}.
We further define a one-parameter family of measures 
\begin{equation}	\label{def:tiltQ}
Q^{(\theta)}(A):= \frac{\int_A \E e^{\theta N\langle u, Hu\rangle} dP(u)}{\int_{\sphereN} \E e^{\theta N\langle u, Hu\rangle} dP(u)}
= \frac{ \E Q^{(\theta, H)}(A) I(H,\theta)}{\E I(H,\theta)} \,,\qquad \theta\ge0\,.
\end{equation}
More generally, for Borel sets $A,B\subseteq\sphereN$ we define 
\begin{equation}	\label{def:tiltQAB}
Q^{(\theta)}(A|B) := \frac{Q^{(\theta)}(A\cap B)}{Q^{(\theta)}(B)}.
\end{equation}
In terms of the annealed and restricted annealed free energy densities from \eqref{def:FN}, \eqref{def:FN.U}, 
\begin{align}	\label{QAB-FAB}
Q^{(\theta)}(A|B) &= e^{N (F_N(\theta;A\cap B) - F_N(\theta; B))}.
\end{align}
In particular $\frac1N\log Q^{(\theta)}(A) = F_N(\theta;A) - F_N(\theta)$. 

For the proof of a general lower bound for large deviation probabilities in Section \ref{sec:lower} we will ultimately condition on a 1-parameter family of sets $\Uloc^{(\theta)}=\Uloc_{w(\theta)}(\srad,\wt R(\theta))$, with $\Uloc_w(\srad,R)\subset\sphereN$ as in \eqref{def:UwrR}, for well-chosen continuous curves $\theta\mapsto w(\theta)$, $\theta\mapsto \wt R(\theta)$ (see \eqref{def:tQtheta}).
We will thus obtain a 1-parameter family of measures denoted
\begin{equation}	\label{def:tQtheta0}
\wt Q^{(\theta)}(A):= Q^{(\theta)}(A| \Uloc^{(\theta)}).
\end{equation}

\section{Proofs of the main results}
\label{sec:high}

In this section we gather our main lemmas that will be proved in subsequent sections and use them to prove Theorems \ref{thm:rateN}, \ref{thm:smallx},  and \ref{thm:fullLDP}.
The proof of Theorem \ref{thm:increasing} involves more tools developed in later sections and  is deferred to Section \ref{sec:increasing}. 
Recall that the assumption \eqref{assu:sg} is in force throughout the article. Further assumptions such as \eqref{assu:usg} will be stated explicitly where they are needed. 

The following is standard. (The bounds are not the sharpest available but suffice for our purposes.)

\begin{lemma}[Exponential tightness]
\label{lem:tightness}
There are constants $C,c>0$ depending only on $\mu$ such that 
\begin{equation}
\label{tight-Rside}
\P(\lam_1\ge K) \le \P( \|H\| \ge K) \le 2\exp( - cK^2N)\qquad \forall\,K\ge C
\end{equation}
and
\begin{equation}
\label{tight-Lside}
\P( \lam_1\le 2-\eps) \le 2\exp( - cN^{3/2})\qquad \forall \,\eps\ge N^{-1/10}\,.
\end{equation}
\end{lemma}

\begin{proof}
See Appendix \ref{app:conc.linear}.
\end{proof}

The following shows that the large deviation rate for the event that $\lam_1 \approx   x$ is asymptotically monotone in $x$. 

\begin{lemma}[Monotonicity]
\label{lem:monotone}
Assume \eqref{assu:usg} holds, and that $N$ is sufficiently large depending on $\mu$. Let $\delta\in[N^{-1/3}, 1]$. For any $2\le y\le x\le N$, 
\begin{equation}
\frac1N\log\P( |\lam_1-x|\le \delta, \|H\|\le N) \le \frac1N\log\P(|\lam_1-y|\le \delta) + O(N^{-1}\log N).
\end{equation}
\end{lemma}

\begin{proof}
See Section \ref{sec:monotone}.
\end{proof}

Recall the notation \eqref{def:eta}.
Generalizing \eqref{def:jayN0}, for $R\ge1$ we define
\begin{equation}
\label{def:jayN}
\cJ_{N,R}(x,\vloc):=\sup_{\theta\ge0}\big\{ J(x,\theta) - \free_{N,R}(\theta,\overlap_x(\theta)\vloc)\big\}\,,\qquad x\ge2,\; \vloc\in\B
\end{equation}
so $\cJ_N(x,\vloc)=\cJ_{N,N^{1/5}}(x,\vloc)$. 

The following provides a joint large deviation upper bound for $\lam_1$ and $v_1^{(\eta)}$, and   is the main step toward the proofs of the large deviation upper bounds in our theorems. 

\begin{prop}[Joint eigenvalue-eigenvector upper bound]
\label{prop:upper-joint}
Let $\heta,\xcush,\zcush\in(0,\frac1{10})$. 
 For any $\eta\in(\heta,\frac14-\heta)$, $I=[x',x'')\subset[2+\xcush,\xcush^{-1}]$ and measurable set $ A\subseteq(1-\zcush)\B^N$,
\begin{equation}
\label{jointUB1}
\frac1N\log \P\big(\lam_1\in I, v_1^{(\eta)}\in  A\big) \le -\inf_{y\in I, \vloc\in  A\cap \B^{n_0} }\cJ_{N,2N^\eta}(y,\vloc) + N^{-c\heta} 
\end{equation}
for a universal constant $c>0$ and all $N$ sufficiently large depending on $\heta,\xcush,\zcush$ and $\mu$.
\end{prop}

\begin{proof}
See Section \ref{sec:upper}.
\end{proof}

Proposition \ref{prop:upper-joint} does not address the event that $v_1^{(\eta)}$ is near the boundary of $\B^N$, i.e.\ that  $v_1\approx v_1^{(\eta)}$, so that $v_1$ is almost completely localized to $N^{1-2\eta}$ coordinates. This event is shown to be negligible by the following,
allowing us to assume the localized portion of $v_1$ is a distance $\gs_x1$ from $\sphereN$.
Recall $\|w\|_0=|\supp(w)|$. For $0<s\le N$ and $\eps>0$ we denote the set of almost-sparse (``compressible'') unit vectors
\begin{equation}
\label{def:AS}
\Comp_N(s,\eps):= \big\{ v\in\sphereN: \|v-w\|_2\le \eps \, \text{ for some }w\in\B^N \text{ with } \|w\|_0\le s \,\big\}\,.
\end{equation}

\begin{prop}[Ruling out complete localization]
\label{prop:no-comp}
Assume \eqref{assu:usg}.
There are constants $c_0,c_1\in(0,1)$ depending only on $\mu$ such that the following holds. 
For any $s\le N/\log N$, $L\ge3$, $\eps\in (0,c_1/L^2)$  and interval $I\subset[2,L]$ of length at least $N^{-1/4}$, 
\begin{equation}	\label{bd:no-comp0}
\P\big( v_1\in \Comp_N(s, \eps)\,\big| \, \lam_1\in I\big) 
\le  e^{-c_0N}
\end{equation}
for all $N$ sufficiently large depending on $L$ and $\mu$. 
In particular, for any $\eta\in(0,\frac14)$,
\begin{equation*}	
\P\big(\|v_1^{(\eta)}\|_2^2\ge1-\eps^2\,\big| \, \lam_1\in I\big) 
\le  e^{-c_0N}
\end{equation*}
for all $N$ sufficiently large depending on $L,\eta$ and $\mu$. 
\end{prop}

\begin{proof}
See Section \ref{sec:localized}.
\end{proof}

\begin{lemma}[Large deviation lower bound]
\label{lem:lower1}
Assume \eqref{assu:usg}.
For any $\heta,\xcush,\zcush\in(0,\frac1{10})$,
$\eta\in(\heta,\frac14-\heta)$, $\log N\le R\le N^{1/4}/\log N$ and $x\in[2+\xcush,\xcush^{-1}]$,
\begin{equation}
\frac1N\log\P\big( |\lam_1-x|\le N^{-1/20}\big) \ge - \inf_{\vloc\in (1-\zcush)\B^{n_0}} \cJ_{N,R}
(x,\vloc) - O_{\xcush,\zcush}(N^{-c\heta} + R^2 N^{-1/2})
\end{equation}
for a universal constant $c>0$ and all $N$ sufficiently large depending on $\xcush,\zcush,\heta$ and $\mu$.
\end{lemma}

\begin{proof}
See Section \ref{sec:lower}.
\end{proof}

Finally, we gather some basic properties of the eigenvalue-eigenvector rate function $\cJ_{N,R}$ that will be used repeatedly.
In addition to $\cJ_{N,R}(x,z)$ from \eqref{def:jayN}, for $E\subset \R$ let
\begin{equation}	\label{def:JNE}
\cJ^{E}_{N,R}(x,z):= \sup_{\theta\in E} \big\{ J(x,\theta) - \free_{N,R}(\theta, \overlap_x(\theta)z)\big\}\,
\end{equation}
so $\cJ_{N,R}(x,\vloc) = \cJ^{\R_+}_{N,R}(x,\vloc)$.

\begin{lemma}[Properties of $\cJ_{N,R}$]
\label{lem:Jprops}
\quad
\begin{enumeratea}

\item\label{Jprop.smallz} (
Zero $z$).
 For all $x\ge2$, $R\ge1$, 
\begin{equation}	\label{Jprop.smallz-UB}
\rate^\gamma(x)\le
\cJ_{N,R}(x,0) \le 
\rate^\gamma(x) + O(e^{-cR^2}).
\end{equation}

\item\label{Jprop.smallx} (Small $x$). There exists $x_\mu>2$ depending only on $\mu$ and  universal constants $C_0,c_0>0$ such that for any $x\in[2,x_\mu), \vloc\in\B$ and $R\ge C_0$, 
\begin{equation}
\cJ_{N,R}(x,\vloc) \ge \rate^\gamma(x) + c_0\sqrt{x-2}\|\vloc\|_2^2 \,.
\end{equation}

\item\label{Jprop.T} (Bounded optimizer).
For $2\le x\le K,R\ge10, \zcush>0$ and $\vloc\in(1-\zcush)\B$ the supremum in \eqref{def:jayN} is attained in $[0,T]$ for some $T=O_{K,\zcush }(1)$, i.e.
$\cJ_{N,R}(x,\vloc) = \cJ^{[0,T]}_{N,R}(x,\vloc)$. 

If we further assume $x\ge 2+\xcush$ for $\xcush>0$ and $R\ge C\sqrt{\log(K/\xcush)}$ for a sufficiently large constant $C>0$, then $\cJ_{N,R}(x,\vloc) = \cJ^{[\thetam+\tcush,T]}_{N,R}(x,\vloc)$ for some $\tcush\gs_{K,\xcush,\zcush}1$.


\item\label{Jprop.cont} (Continuity).
Let $R\ge1,T\ge2$. 
\begin{itemize}
\item[(i)] For fixed $2\le x\le L$, $z\mapsto \cJ^{[0,T]}_{N,R}(x,z)$ is $O_{L,T}(1)$-Lipschitz on $\B$.
\item[(ii)] For fixed $z\in\ball, \kappa\in(0,\frac12)$ and $L\ge3$, $x\mapsto \cJ^{[0,T]}_{N,R}(x,z)$ is $O(T^3L^2\kappa^{-1/2})$-Lipschitz on $[2+\kappa,L]$.
\end{itemize}

\end{enumeratea}
\end{lemma}

\begin{proof}
See Section \ref{sec:rateprops}.
\end{proof}


\subsection{Proof of Theorem \ref{thm:rateN}}
\label{sec:rateN}
Here we establish Theorem \ref{thm:rateN} giving the asymptotics of the probability that $\lambda_{1}$ is close to $x\ge 2$ in terms of the rate  $\rate^{\mu}_{N}(x)$ defined in \eqref{def:rateN0}.
For the case that $x=2$, from \eqref{FuKo} it suffices to show $\rate^\mu_N(2)\to0$. 
For all $x\ge2$, since $\overlap_x(0)=0$ we can lower bound $\cJ_{N,R}(x,\vloc) \ge J(x,0) - \free_{N,R}(0,0) = 0$, so 
\begin{equation}	\label{rateN-LB0}
\rate^\mu_N(x) \ge 0 \qquad\forall x\ge2. 
\end{equation}
On the other hand, from Lemma \ref{lem:Jprops}(\ref{Jprop.smallz}),
\begin{equation}	\label{rateN-2UB}
\rate^\mu_N(2) \le \rate^\gamma(2) + O(e^{-cN^{2/5}}) = O(e^{-cN^{2/5}})
\end{equation}
so $\rate^\mu_N(2)\to0$ as desired. 

The following establishes the upper bound from Theorem \ref{thm:rateN} for the case $x>2$ with explicit error rates, while the lower bound is immediate from Lemma \ref{lem:lower1}.

\begin{prop}
\label{AquantUB}
Assume \eqref{assu:usg}.
Let $\heta,\xcush\in(0,\frac1{10})$, $L\ge3$ and $\eta\in(\heta,\frac14-\heta)$.
For any interval $I=[x-\delta,x+\delta]\subset[2+\xcush,L]$, 
\begin{equation}	\label{bd:AquantUB}
\frac1N\log\P(\lam_1\in I) \le - \inf_{\vloc\in(1-\zcush_0)\B^{n_0}} \cJ_{N,2N^\eta}(x,\vloc) + O_L(\xcush^{-1/2}\delta) + N^{-c\heta}
\end{equation}
for all $N$ sufficiently large depending on $\xcush,L,\heta$ and $\mu$,
where $\zcush_0=\zcush_0(L)=c/L^4$.
\end{prop}

\begin{proof}
We will assume without comment that $N$ is sufficiently large depending on fixed parameters. 
Since the right hand side in \eqref{bd:AquantUB} is unchanged under increasing $\delta\in(0,N^{-2c\heta})$,
we may assume without loss of generality that $\delta\ge N^{-2c\heta}$, and in particular that $\delta\ge N^{-1/4}$. 
From Proposition \ref{prop:no-comp} and taking $c>0$ sufficiently small, we have
\begin{equation}
\P(\lam_1\in I) 
\le \P(\lam_1\in I, \|v_1^{(\eta)}\|_2^2\le 1-\eps_0^2) + e^{-c_0N} \P(\lam_1\in I)
\end{equation}
and hence
\begin{equation}
\P(\lam_1\in I) 
\le 2\P(\lam_1\in I, \|v_1^{(\eta)}\|_2^2\le 1-\eps_0^2).
\end{equation}
Applying Proposition \ref{prop:upper-joint} with $\zcush=\eps_0^2$ and $ A=(1-\zcush)\B^N$, we have
\begin{align*}
\frac1N\log \P(\lam_1\in I) 
&\le \frac1N\log \P(\lam_1\in I, \|v_1^{(\eta)}\|_2^2\le 1-\eps_0^2) + \frac{\log 2}{N}\\
&\le -\inf_{y\in I, \vloc\in (1-\eps_0^2)\B^{n_0}} \cJ_{N,2N^\eta}(y,\vloc) + N^{-c\heta}
\end{align*}
for $c\eta_0<1$.
For any $y\in I$ and $\vloc\in\B$, from Lemma \ref{lem:Jprops}(\ref{Jprop.T}) we have $\cJ_{N,2N^\eta}(y,\vloc) = \cJ_N^{[0,T]}(y,\vloc;2N^\eta)$ for some $T=O_L(1)$, and from 
Lemma \ref{lem:Jprops}(\ref{Jprop.cont}) and another application of Lemma \ref{lem:Jprops}(\ref{Jprop.T}),
\[
\cJ_N^{[0,T]}(y,\vloc;2N^\eta) = \cJ_N^{[0,T]}(x,\vloc;2N^\eta) + O_L(\kappa^{-1/2}\delta)
= \cJ_{N,2N^\eta}(x,\vloc) + O_L(\kappa^{-1/2}\delta).
\]
The claim follows. 
\end{proof}

\subsection{Proof of Theorem \ref{thm:smallx}}
\label{sec:smallx}

The case that $x<2$ follows from \eqref{tight-Lside}, and the case $x=2$ follows from \eqref{FuKo}. 
The claims for the case $x>2$ are then a consequence of the following. 

\begin{prop}
\label{prop:Bquant}
Assume \eqref{assu:usg}.
Let $\heta,\xcush\in(0,\frac1{10})$, $L\ge3$,  $I=[x',x'')\subset[2+\kappa,L]$.
\begin{itemize}
\item If $x''\ge x'+2N^{-1/20}$, then
\begin{equation}	\label{Bquant.LB}
\frac1N\log\P(\lam_1\in I) \ge -\rate^\gamma(x') -N^{-c}
\end{equation}
for all $N$ sufficiently large depending on $\xcush,L,\mu$.

\item If $x''<x_\mu$, then for any $\al\ge0$ and $\eta\in(\heta,\frac14-\heta)$,
\begin{equation}	\label{Bquant.UB}
\frac1N\log\P(\lam_1\in I, \|v_1^{(\eta)}\|_2^2\ge \al) \le - \rate^\gamma(x') - c_0\sqrt{\xcush}\al + N^{-c\heta}
\end{equation}
for all $N$ sufficiently large depending on $\xcush,\heta$ and $\mu$, and a constant $c_0$ depending on $\mu$.
\end{itemize}
\end{prop}

To deduce Theorem \ref{thm:smallx}, 
the bounds \eqref{igamma-LB} and \eqref{igamma-UB} follow from \eqref{Bquant.LB} and \eqref{Bquant.UB}, respectively (taking $\al=0$ in the latter). 
For \eqref{igamma-v1UB}, combining \eqref{Bquant.LB} and \eqref{Bquant.UB}, we get that  for $\al>0$ and $2+\xcush\le x'+2N^{-1/20}\le x''<x_\mu$, 
\begin{equation}
\P(\|v_1^{(\eta)}\|_2^2\ge\al| \lam_1\in I) \le \exp( - c_0\sqrt{\kappa}\al N) \qquad\forall \al\ge N^{-c'\heta}
\end{equation}
for all $N$ sufficiently large depending on $\xcush,\eta$ and $\mu$.
We thus obtain the conditional tail bound of Theorem \ref{thm:smallx}(b). 

It only remains to prove Proposition \ref{prop:Bquant}.
From Lemma \ref{lem:lower1} with $\zcush=\frac1{20},\eta=\frac18$ and $R=N^{1/8}$, say, we have
\begin{align*}
\frac1N\log\P(\lam_1\in I)
&\ge \frac1N\log\P( |\lam_1-x'-N^{-1/20}|\le N^{-1/20})\\
&\ge - \cJ_{N,R}(x+N^{-1/20},0) - O(N^{-c}).
\end{align*}
From Lemma \ref{lem:Jprops}(\ref{Jprop.smallz}) we can further bound 
\[
\cJ_{N,R}(x+ N^{-1/20},0) \le \rate^\gamma(x+N^{-1/20}) + O(N^{-c}) \le \rate^\gamma(x) + O(N^{-c})
\]
since $\rate^\gamma$ is locally Lipschitz. This yields the first point. 

For the second point, we split
\begin{align}
\P( \lam_1\in I, \|v_1^{(\eta)}\|_2^2\ge \al)
&= \P( \lam_1\in I, \|v_1^{(\eta)}\|_2^2\ge 1-\eps_0^2) + \P( \lam_1\in I, \|v_1^{(\eta)}\|_2^2\in[\al,1-\eps_0^2))	\notag\\
&\le e^{-c_0N} \P(\lam_1\in I) + \P( \lam_1\in I, \|v_1^{(\eta)}\|_2^2\in[\al,1-\eps_0^2))		\label{B-decomp}
\end{align}
where we applied Proposition \ref{prop:no-comp}
(if $\al\ge1-\eps_0^2$ then the latter term is zero).
In particular, with $\al=1-\eps_0^2$ we get
\begin{align*}
\P( \lam_1\in I, \|v_1^{(\eta)}\|_2^2\ge 1-\eps_0^2)
&\le e^{-c_0N} \P(\lam_1\in I) \\
&= e^{-c_0N} \big[ \P( \lam_1\in I, \|v_1^{(\eta)}\|_2^2\ge 1-\eps_0^2) +  \P( \lam_1\in I, \|v_1^{(\eta)}\|_2^2< 1-\eps_0^2)\big]
\end{align*}
and rearranging yields
\begin{equation}	\label{large-al}
\P( \lam_1\in I, \|v_1^{(\eta)}\|_2^2\ge 1-\eps_0^2) 
\le2e^{-c_0N} \P( \lam_1\in I, \|v_1^{(\eta)}\|_2^2< 1-\eps_0^2).
\end{equation}
On the other hand, with $\al=0$ in \eqref{B-decomp} we similarly obtain
\begin{equation}	\label{B-al0}
\P(\lam_1\in I) \le 2\P(\lam_1\in I, \|v_1^{(\eta)}\|_2^2< 1-\eps_0^2)
\end{equation}

Now for any $\al\in[0,1-\eps_0^2)$, applying Proposition \ref{prop:upper-joint} with the annulus $ A=\{v\in\B^N: \al\le\|v\|_2^2<1-\eps_0^2\}$, we have
\begin{align*}
\frac1N\log\P( \lam_1\in I, \|v_1^{(\eta)}\|_2^2\in[\al,1-\eps_0^2))
&\le -\inf_{y\in I, \vloc\in  A\cap\B^{n_0}} \cJ_{N,2N^\eta}(y,\vloc) + N^{-c\heta}.
\end{align*}
By the assumption $x''<x_\mu$ and Lemma \ref{lem:Jprops}(\ref{Jprop.smallx}), for any $y\in I$ and $\vloc\in  A$ we have
\[
\cJ_{N,2N^\eta}(y,\vloc)\ge \rate^\gamma(y) + c_0\sqrt{y-2}\|\vloc\|_2^2 
\ge  \rate^\gamma(x') + c_0\sqrt{\xcush}\al.
\]
Combining with the previous display,
\begin{equation}	\label{B-mainal}
\frac1N\log\P( \lam_1\in I, \|v_1^{(\eta)}\|_2^2\in[\al,1-\eps_0^2))
\le - \rate^\gamma(x') - c_0\sqrt{\xcush}\al + N^{-c\heta}.
\end{equation}
Inserting this bound for the case $\al=0$ on the right hand of \eqref{large-al} yields
\[
\frac1N\log\P( \lam_1\in I ,\|v_1^{(\eta)}\|_2^2\ge 1-\eps_0^2) \le -\rate^\gamma(x') - c_0 + N^{-c\heta}
\]
giving the claim for the case $\al\in[1-\eps_0^2,1]$. 

For the case $\al<1-\eps_0^2$, inserting \eqref{B-mainal} with $\al=0$ on the right hand side of \eqref{B-al0} gives
\[
\frac1N\log\P( \lam_1\in I ) \le -\rate^\gamma(x')  + N^{-c\heta}
\]
and combining this with \eqref{B-decomp} and \eqref{B-mainal} yields
\begin{align*}
\P( \lam_1\in I ,\|v_1^{(\eta)}\|_2^2\ge\al) &\le e^{N^{1-c\heta} - \rate^\gamma(x')N-c_0N}
+ e^{N^{1-c\heta} - \rate^\gamma(x')N - c_0\sqrt{\xcush}\al N}\\
&\le \exp( - N( \rate^\gamma(x') + c_0 \sqrt{\xcush}\al) + O(N^{1-c\heta}))
\end{align*}
which completes the proof. 
\qed

\subsection{Proof of Theorem \ref{thm:fullLDP}}
\label{sec:fullLDP}

The main task is to prove the weak large deviation principle on $(2,+\infty)$, which is a consequence of the following:

\begin{prop}
\label{prop:weakLDP}
With assumptions as in Theorem \ref{thm:fullLDP},
let $\xcush,\eps\in(0,\frac1{10})$, $L\ge3$ and $\zcush\in(0,cL^{-4})$ for a sufficiently small constant $c>0$ depending only on $\mu$.
For any closed interval $I=[x-\delta,x+\delta]\subset[2+\xcush,L]$  for some $\delta>0$,
\begin{equation}	\label{lim:weakLDP}
\limsup_{N\to\infty}\bigg|\frac1N\log\P(\lam_1\in I) +
 \wt\cI_{N,N^{-2\eps}}(x,\zcush) \bigg|
\ls_L\xcush^{-1/2}\delta\,.
\end{equation}
where we recall that $ \wt\cI_{N,N^{-2\eps}}$ is defined in  \eqref{def:trate}.
\end{prop}

We also need the following lemma providing analogues for $\wt\cJ_{N,R}$ of some of the properties for $\cJ_{N,R}$ stated in Lemma \ref{lem:Jprops}.
(Note that \eqref{Jprop.smallz-UB} carries over immediately with $\wt\cJ_{N,R}(x,0,0)$ in place of $\cJ_{N,R}(x,0)$ since these quantities are equal.)

\begin{lemma}[Properties of $\wt\cJ_{N,R}$]
\label{lem:tJprops}
\quad
\begin{enumeratea}

\item\label{tJprop.smallx} (Small $x$). 
With $x_\mu,C_0,c_0$ as in Lemma \ref{lem:Jprops}(\ref{Jprop.smallx}), 
for any $x\in[2,x_\mu), \chz\in\B$, $\tal\in[0,1-\|\chz\|_2^2]$ and $R\ge C_0$, 
\begin{equation}
\wt\cJ_{N,R}(x,\chz,\tal) \ge \rate^\gamma(x) + c_0\sqrt{x-2}(\|\chz\|_2^2+\tal) \,.
\end{equation}

\item\label{tJprop.T} (Bounded optimizer).
For $x\in[2+\xcush,K]$, $R\ge C\sqrt{\log(K/\xcush)}$, $\zcush>0$, $\tal\in[0,1]$ and $\chz\in \B$ such that $\|\chz\|_2^2+\tal \le (1-\zcush)^2$, we have
$\wt\cJ_{N,R}(x,\vloc) = \wt\cJ_{N,R}^{[\thetam+\tcush,T]}(x,\vloc)$ for some $T\ls_{K,\zcush}1$ and $\tcush\gs_{K,\xcush,\zcush}1$. 

\item\label{tJprop.cont} (Continuity in $x$).
Let $R\ge1,T\ge2$. 
For fixed $\xcush\in(0,\frac12),L\ge3$, $\tal\in[0,1]$ and $\chz\in\B$ such that $\tal+\|\chz\|_2^2\le 1$,
$x\mapsto\wt\cJ^{[0,T]}_{N,R}(x,\chz,\tal)$ is $O(T^3L^2\kappa^{-1/2})$-Lipschitz on $[2+\kappa,L]$.
\end{enumeratea}
\end{lemma}

\begin{proof}
This is proved alongside Lemma \ref{lem:Jprops} in Section \ref{sec:rateprops}.
\end{proof}

We postpone the proof of Proposition \ref{prop:weakLDP} and complete the proof of Theorem \ref{thm:fullLDP}.
\begin{proof}[Proof of Theorem \ref{thm:fullLDP}]
Since $\lam_1$ is exponentially tight by Lemma \ref{lem:tightness}, it suffices (see \cite[Lemma 1.2.18]{DZ}) to show that the rate function $\rate^\mu$ is well defined (i.e.\ the limit in \eqref{def:trate} exists) and lower-semicontinuous, and that the weak large deviation principle holds, that is
\begin{equation}
\label{goal:weakLDP}
\lim_{\delta\downarrow0} \limsup_{N\to\infty} \frac1N\log\P(|\lam_1-x|\le \delta)
= \lim_{\delta\downarrow0} \liminf_{N\to\infty} \frac1N\log\P(|\lam_1-x|\le \delta) = -\rate^\mu(x)
\end{equation}
for every fixed $x\in\R$. 

For existence of the limit, note that for fixed $x\in(2,\infty),\eps\in(0,\frac1{10}),\zcush\in(0,c_\mu x^{-4})$ the sequence $\wt\rate_{N,N^{-2\eps}}(x,\zcush)$ is monotone decreasing in $N$. 
Indeed, $\VP_R$ (defined in \eqref{def:VP}) and hence $\tfree_{N,R}$ (defined in \eqref{def:tfree}) are increasing in $R$, so $\wt\cJ_{N,R}(x,\chz,\tal)$ is decreasing in $R$.
Moreover, for fixed $\tal$ the inner infimum in \eqref{def:trate} is taken over an increasing sequence of sets $\B_{\ge N^{-2\eps}}$. Furthermore, $ \wt\rate_{N,\xi}$ is bounded 
since $0\le \wt\rate_{N,\xi}(x,\zcush)\le \rate^\gamma(x)+O(e^{-cN^{2/5}})$ for all $x\ge2$ and $\xi,\zcush>0$ (for the first inequality we can bound $\wt\cJ_{N,R}(x,\vloc,\tal) \ge J(x,0) - \tfree_{N,R}(0,0,0)=0$ for any $\vloc,\tal$ by taking $\theta=0$ in \eqref{def:tjayN}, while the second bound follows from $\wt\rate_{N,\xi}(x,\zcush) \le \wt\cJ_{N,R}(x,0,0) =\cJ_{N,R}(x,0)\le \rate^\gamma(x)+O(e^{-cR^2})$ from \eqref{lem:Jprops}(\ref{Jprop.smallz})) the limit $\rate^\mu(x) =\lim_{N\to\infty}\wt\rate_{N,N^{-2\eps}}(x,\zcush)$ exists and is finite for every $x>2$. 
Since $\frac1N\log\P(\lam_1\in I)$ is independent of $\eps,\zcush$, the independence of the limit on these parameters follows from  \eqref{lim:weakLDP}.
At $x=2$, we note that for any $\thresh,\zcush>0$,
\begin{align*}
\wt\rate_{N,\thresh}(2,\zcush) \le \wt\cJ_{N,N^{1/5}}(2,0,0) = \cJ_N(2,0) \le \rate^\gamma(2) + e^{-cN^{1/4}} = e^{-cN^{1/4}}\to 0
\end{align*}
so $\rate^\mu(2)=0$. 

For the lower-semicontinuity,
since $\rate^\mu(2)=0\le\rate^\mu(x)$ for all $x\in\R$, it suffices to show $\rate^\mu$ is continuous on $(2,\infty)$.\footnote{In fact from Theorem \ref{thm:smallx} and \eqref{goal:weakLDP} we can deduce \emph{a posteriori} that $\rate^\mu=\rate^\gamma$ in a neighborhood of 2 and hence $\rate^\mu$ is continuous on all of $[2,\infty)$, but we do not need this here.}
From
Lemma \ref{lem:tJprops}(\ref{tJprop.T},\ref{tJprop.cont}) we have that for any fixed $L\ge10$, $x\mapsto \wt\cJ_{N,N^{1/5}}(x,\chz,\tal)$ is $O_L(1)$-Lipchitz on $[2+L^{-1},L]$ for all $\tal\le 1-cL^{-4}$ and $\chz\in (1-\tal^{1/2})\B$. Hence, $\wt\rate_{N,N^{-2\eps}}$ is $O_L(1)$-Lipschitz on $[2+L^{-1},L]$ for every $N$ and $L\ge10$.
It follows that $\rate^\mu$  is continuous (in fact locally Lipchitz) on $(2,\infty)$. 

For \eqref{goal:weakLDP}, the case $x<2$ follows from \eqref{tight-Lside}, and the case $x=2$ follows from \eqref{FuKo} and the fact just shown that $\rate^\mu(2)=0$. 
For fixed $x>2$, \eqref{goal:weakLDP} follows from Proposition \ref{prop:weakLDP} and the fact that $\wt\rate_{N,N^{-2\eps}}(x,\zcush)\to\rate^\mu(x)$. 

The fact that $\rate^\mu$ is non-decreasing on $[2,\infty)$ follows from Lemma \ref{lem:monotone} and \eqref{goal:weakLDP}, along with \eqref{tight-Rside} and the fact that $\rate^\mu$ is finite on $[2,\infty)$.
This completes the proof of of Theorem \ref{thm:fullLDP}.
\end{proof}

In the remainder of this subsection we establish Proposition \ref{prop:weakLDP}. We first state two lemmas.

For the upper bound, as in the proofs of Theorems \ref{thm:rateN} and \ref{thm:smallx} we will apply Proposition \ref{prop:upper-joint}, but only after applying a pigeonholing argument to locate a gap in sizes of the large coordinates of $v_1$; after fixing the gap, the rate function $\cJ_{N,R}$ 
reduces to $\wt\cJ_{N,R}$, as we show in Lemma \ref{lem:free-tfree} below. 

For the lower bound we need the following modification of Lemma \ref{lem:lower1}, where, rather than taking the supremum in $\theta$ followed by the infimum over localized vectors $\vloc$, we allow the localized vector to vary with $\theta$. This will allow us to select large values of $\psimu(t)$ for the contribution of the restricted free energy $\free_{N,R}(\theta,w)$ using coordinates of $w$ of size $\asymp N^{-1/4}$. 

To that end,
for $\theta,t>0, \al\in[0,1]$  let
\begin{equation}	\label{def:w*}
w^*=w^*(\theta,\al,t) = \sqrt{\frac{t}{2\theta}}N^{-1/4} \1_{[N-n_1+1,N]}\in\R^N
\,,\quad
n_1=n_1(\theta,\al,t):= \lf 2\theta\al t^{-1}N^{1/2}\rf \,.
\end{equation}
Thus, $w^*$ is constant on its support of size $n_1=2\theta\al t^{-1}N^{1/2} +O(1)$ with squared norm
\begin{equation}	\label{w*norm}
\al\ge \|w^*\|_2^2 = \al + O(t\theta^{-1}N^{-1/2})\,.
\end{equation}
(We take the support on the right end of the interval $[N]$ only for later notational convenience.) 

\begin{lemma}
\label{lem:lower2}
Assume \eqref{assu:usg}.
Let $\xcush,\zcush,\eps\in(0,\frac1{10})$ and $t>0$. 
There exist $T_0(\xcush,\zcush)\ge10$, $\tcush_0(\xcush,\zcush)\in(0,\frac1{10})$ such that for 
any $x\in[2+\xcush,\xcush^{-1}]$ and $\chz\in\B_{\ge N^{-\eps}}$ with $\supp(\chz)\subset[1,N/2]$, and $\tal\ge0$ such that $\tal+\|\chz\|_2^2\le 1-\zcush$, 
\begin{equation}
\frac1N\log\P( |\lam_1-x|\le N^{-1/20}) \\
\ge -  \sup_{\theta\in[\thetam+\tcush_0,T_0]} \Big\{ J(x,\theta) - \free_{N,N^{1/5}}\big(\theta,w(\theta)\big) \Big\} + O_{\xcush,\zcush,t}(N^{-c} )
\end{equation}
for all $N$ sufficiently large depending on $\xcush, \zcush,t$ and $\mu$,
where
\begin{equation}	\label{def:w.lower2}
w(\theta)=w(\theta; x,\chz,\tal,t):=
\overlap_x(\theta)\chz + w^*\big(\theta,\overlap_x(\theta)^2\tal,t\big)\,.
\end{equation}

\end{lemma}

\begin{proof} See Section \ref{sec:lower}. \end{proof}

The value $\sqrt{t/2\theta}N^{-1/4}$ for the nonzero entries of $w^*$ in \eqref{def:w*} is chosen to select the value $\psimu(t)$ in the localized contribution $\freeL$ for $\free_{N,N^{1/5}}(\theta,w(\theta))$; see \eqref{wt-select1}--\eqref{wt-select2} below. The support size $n_1$ is chosen to ensure $\|w^*\|_2^2\approx \alpha$.

\begin{lemma}
\label{lem:free-tfree}
Let $R\ge1$ and let $\chw,\tw\in\ell^2(\N)$ have disjoint supports, with $\chw+\tw\in\B$, and
\begin{equation}	\label{assu:tfree1}
\|\tw\|_\infty\le \frac1{MR},\qquad |\chw_i\tw_j|\ge MN^{-1/2}\quad\forall i\in\supp(\chw),\,j\in\supp(\tw)
\end{equation}
for some (large) $M\ge1$.
Set $\cha:=\|\chw\|_2^2$, $\tal:=\|\tw\|_2^2$, $\beta:=1-\cha-\tal$. 
For any $0<\tcush\le\theta\le T$,
\begin{align}
\free_{N,R}(\theta,\chw+\tw) 
&= \freeL(\theta,\tw) + \theta^2\big[ \beta + 2\psiinfty ( \cha^2 + 2\cha\tal )   + 2\tal\beta \big] \label{free-chwtw}\\
&\qquad+\VP_R(\theta\chw,1-\beta) -\tfrac12(1-\beta)	
+ O(T^2\delta_{\tcush M} + T^3M^{-1})	\notag
\end{align}
where
\begin{equation}
\delta_B:= \sup_{|t|\ge B} |\psimu(t)-\psiinfty|.
\end{equation}
Furthermore,
\begin{equation}	\label{free-tfree1}
\free_{N,R}(\theta,\chw+\tw) \le \tfree_{N,R}(\theta,\chw,\|\tw\|_2^2) + 
O(T^2\delta_{\tcush M} + T^3M^{-1} )
\end{equation}
and for any fixed $t>0$ and $0\le \al\le 1-\|\chw\|_2^2$, if $\chw$ and $w^*(\theta,\al,t)$ have disjoint supports and 
\begin{equation}	\label{assu:tfree2}
R\le \frac1M (\tcush/t)^{1/2} N^{1/4}\,,
\qquad
|\chw_i| \ge M (T/t)^{1/2} N^{-1/4}
\quad\forall i\in \supp(\chw),
\end{equation}
then
\begin{equation}	\label{free-tfree2}
\free_{N,R}\big(\theta,\chw+ w^*(\theta,\al,t) \big)
=\tfree^{(t)}_{N,R}(\theta,\chw,\al) + 
O\bigg(T^2\delta_{\tcush M} + \frac{T^3}M  + \frac{t^2}N+ \frac{Tt}{\sqrt{N}}\bigg)
\end{equation}
where
\begin{align}
\tfree^{(t)}_{N,R}(\theta,\chw,\al)
&:=\theta^2\Big[ \beta^2 + 2\beta\al + 2\psiinfty(\cha^2+2\al\cha) + 2\psimu(t)\al^2\Big] 
	\label{def:tfreep}\\
&\qquad+ \VP_R(\theta \chw, \beta) -\tfrac12(1-\beta)		
\notag
\end{align}
for $\beta=1-\cha-\al\ge0$.
\end{lemma}

Note that if the supremum of $\psimu(t)$ is attained at some $t_\mu^*\in\R$, then
\begin{equation}	\label{wt-select1}
\tfree_{N,R}^{(t_\mu^*)}(\theta,\chw,\tal)= \tfree_{N,R}(\theta,\chw,\tal)\,.
\end{equation}
Moreover, under \eqref{assu:maxR+}, for any $\delta>0$ there exists $t(\delta)>0$ depending only on $\mu$ and $\delta$ such that
\begin{equation}	\label{wt-select2}
\tfree_{N,R}(\theta,\chw,\tal)\ge  \tfree^{(t(\delta))}_{N,R}(\theta,\chw,\tal) \ge \tfree_{N,R}(\theta,\chw,\tal) - \delta\theta^2\tal^2\,.
\end{equation}

\begin{proof}
Since $\chw,\tw$ have disjoint supports, 
\[
\freeL(\theta,\chw+\tw) = \freeL(\theta,\chw) + \freeL(\theta,\tw) + \wt f_N  (\theta\chw,\tw)
\]
where
\begin{align*}
\wt f_N   (\theta\chw,\tw)
&:= \frac1N\sum_{i,j=1}^N \LLa_\mu(2\theta\sqrt{N}\chw_i\tw_j)\\
&= 4\theta^2 \sum_{i,j=1}^N \chw_i^2\tw_j^2 \psimu(2\theta\sqrt{N}\chw_i\tw_j)\\
&=(4\psiinfty+ O(\delta_{\tcush M}))\theta^2\cha\tal \,.
\end{align*}
We similarly find
\[
\freeL(\theta,\chw) = (2\psiinfty+ O(\delta_{\tcush M}))\theta^2 \cha^2. 
\]
For any $\nu\in\cP_\beta([-R,R])$ we have
\begin{align*}
\bigg| 2\theta^2\tal\beta - \int\sum_i \LLa_\mu(2\theta\tw_is)d\nu(s)\bigg|
&= \bigg|  \int\sum_i 4\theta^2\tw_i^2s^2\big( \psimu(2\theta\tw_is) -\tfrac12\big) d\nu(s)\bigg|\\
&\le \int\sum_i 4\theta^2\tw_i^2s^2\big| \psimu(2\theta\tw_is) -\tfrac12\big| d\nu(s)\\
&\le O(\theta/M) \int\sum_i 4\theta^2\tw_i^2s^2d\nu(s)\\
&=O(\theta^3\tal\beta/M) = O(T^3/M).
\end{align*}
where in the third line we used the hypothesis $\|\tw\|_\infty\le (MR)^{-1}$, together with the fact that $\psimu$ is locally Lipchitz and takes value $\frac12$ at 0.
Hence,
\begin{align*}
\VP_R(\theta(\chw+\tw), 1-\beta)
&= \sup_{\nu\in\cP_\beta([-R,R])} \bigg\{ \int\sum_i \LLa_\mu(2\theta \chw_is)d\nu(s)  +\int \sum_i \LLa_\mu(2\theta \tw_is)d\nu(s) - \DKL(\nu|\gamma) \bigg\}\\
&= \VP_R(\theta\chw,1-\beta) + 2\theta^2\tal\beta + O(T^3/M). 
\end{align*}
Combining the above estimates, we obtain \eqref{free-chwtw}.
Then \eqref{free-tfree1} follows by upper bounding
\begin{align*}
\freeL(\theta,\tw) &= 
\sum_{i\le j} 2^{1+1_{i\ne j}} \theta^2 \tw_i^2\tw_j^2 \psimu(2^{\ep_{ij}}\theta\sqrt{N}\tw_i\tw_j)
\le \theta^2\psimax \sum_{i\le j} 2^{1+1_{i\ne j}} \tw_i^2\tw_j^2 
= 2\theta^2\psimax \tal^2\,.
\end{align*}
For the case that $\tw=w^*(\theta,\al,t)$, the assumptions \eqref{assu:tfree2} imply that \eqref{assu:tfree1} holds (up to modification of $M$ by a constant factor), so \eqref{free-chwtw} holds in this case.
We can estimate
\begin{align*}
\freeL(\theta,\tw) &= 
\frac{t^2}{4N}\sum_{1\le i\le j\le n_1} 2^{1+1_{i\ne j}}  \psimu(2^{\ep_{ij}-1} t)\\
&=\frac{t^2}{4N} \bigg( 4{n_1\choose 2} \psimu(t) + 2n_1\psimu(t/\sqrt{2})\bigg)\\
&= 2\theta^2\psimu(t)\al^2 + O\bigg( \frac{t^2}{N} ( 1+ 2\theta\al t^{-1}N^{1/2})\bigg)\\
&= 2\theta^2\psimu(t)\al^2 + O\bigg( \frac{t^2}{N} + \frac{Tt}{\sqrt{N}}\bigg)\,.
\end{align*}
This together with \eqref{free-chwtw} yields \eqref{free-tfree2}.
\end{proof}

\begin{proof}[Proof of Proposition \ref{prop:weakLDP}]
Fix $L,\xcush,\eps,\zcush$ and interval $I=[x-\delta,x+\delta]$ as in the statement of the proposition.
We assume without comment that $N$ is sufficiently large depending on $L,\xcush,\eps$ and $\mu$. 

We begin with the lower bound.
Let $T_0\ge 10$, $\tcush_0\in(0,\frac1{10})$ depending only on $L,\xcush$ and $\zcush$ be as in Lemma \ref{lem:lower2} (taking $\min\{\xcush, 1/L\}$ in place of $\xcush$ there).
Under the assumption \eqref{assu:maxR+}, we can choose $t\in(0,\infty)$ depending on $\mu,T_0(L,\xcush,\zcush)$ and $\delta$ such that
\begin{equation}	\label{tmd}
\psimu(t) \ge \psimax - \delta T_0^{-2}.
\end{equation}
Fixing arbitrary $\chz\in\B_{\ge N^{-2\eps}}$ and $\tal\ge0$ with $\|\chz\|_2^2+\tal\le 1-\zcush$, from Lemma \ref{lem:lower2} we have
\begin{align}
\frac1N\log\P(\lam_1\in I) &\ge \label{LDP-LB1}
 -  \sup_{\theta\in[\thetam+\tcush_0,T_0]} \big\{ J(x,\theta) - \free_{N,N^{1/5}}(\theta,
 w(\theta) 
 ) \big\} + O_{L,\xcush,\zcush,\delta}(N^{-c} )	
\end{align}
with $w(\theta)$ as in \eqref{def:w.lower2}.
Fixing an arbitrary $\theta\in[\thetam+\tcush_0,T_0]$, the estimates in \eqref{assu:tfree2} hold for $\chw:=\overlap_x(\theta)\chz$,  $\tw:=w^*(\theta,\overlap_x(\theta)^2\tal,t)$ with $R=N^{1/5}$ and $M= \min\{1,t^{-1/2}\}N^{1/50}$. 
From Lemma \ref{lem:free-tfree} 
we thus conclude the right hand side in \eqref{LDP-LB1} is bounded below by
\[
-  \sup_{\theta\in[\thetam+\tcush_0,T_0]} \bigg\{ J(x,\theta) - \tfree^{(t)}_{N,N^{1/5}}\Big(\theta,\overlap_x(\theta)\chz ,\overlap_x(\theta)^2\tal\big)\Big) \bigg\} + o_{L,\xcush,\zcush,\delta}(1)
\]
where we have used that the error $\delta_{\tcush_0 M}$ in \eqref{free-tfree2} is $o_{L,\xcush,\zcush,\delta}(1)$ by  the assumption \eqref{assu:LRlim}. 
Finally, from \eqref{tmd} and \eqref{def:tfreep} we can replace $ \tfree^{(t)}_{N,N^{1/5}}$ with $ \tfree_{N,N^{1/5}}$ on the right hand side up to an additive error of size $O(T_0^2\delta T_0^{-2})=O(\delta)$.
We thus obtain the lower bound in \eqref{lim:weakLDP}.

We turn to prove the upper bound. Set
\begin{equation}
M:= \exp(\sqrt{\log N}).
\end{equation}
For $v\in\B$ let
\begin{equation}
\ell_\eps(v):= \min\bigg\{ \ell\in\Z\cap(\tfrac12\eps\sqrt{\log N}, \infty): 
\sum_j v_j^2 1_{|v_j|\in (M^{-2\ell-1}, M^{-2\ell+1}]}<\frac{10}{\eps\sqrt{\log N}}\bigg\}\,.
\end{equation}
From the pigeonhole principle, the sets
\begin{equation}
A_\ell:= \{ v\in \B: \ell_\eps(v)=\ell\}\,,\qquad \ell\in \Z\cap(\tfrac12\eps\sqrt{\log N}, \eps\sqrt{\log N})
\end{equation}
partition $\B$. 
We henceforth assume without comment that $\ell$ ranges over $\Z\cap(\frac12\eps\sqrt{\log N}, \eps\sqrt{\log N})$.
Denote
\begin{equation}
\eta_\ell:= \frac{2\ell}{\sqrt{\log N}}\in (\eps,2\eps).
\end{equation}
Thus,
\begin{equation}
M^{-2\ell} = N^{-\eta_\ell}\in (N^{-2\eps},N^{-\eps}).
\end{equation}
From Proposition \ref{prop:no-comp}, with $\zcush_0:=c/L^4$ we have
\begin{align*}
\P(\lam_1\in I) 
&\le 2\P( \lam_1\in I, \|v_1^{(\eps)}\|_2 \le 1-\zcush_0)\\
&=2\sum_\ell \P( \lam_1\in I, \|v_1^{(\eps)}\|_2 \le 1-\zcush_0, v_1\in A_\ell)\,.
\end{align*}
Since 
\[
N^{-1/2+\eps}<N^{-1/2+\eta_\ell}<N^{-1/2+2\eps}<N^{-2\eps} < N^{-\eta_\ell}
\]
for all $\ell$, 
we have 
\[
\|v^{(\eta_\ell)}\|_2\le \|v^{(\eps)}\|_2 \quad\text{ and }\quad
\ell_\eps(v)= \ell_\eps(v^{(\eta_\ell)})
\]
for all $\ell$. 
Thus, setting $\xi_\ell:=N^{-1/2+\eta_\ell}$ and
\begin{equation}	\label{def:Aell'}
A_\ell':= A_\ell \cap\B_{\ge \xi_\ell} \cap (1-\zcush_0)\B^N
\end{equation}
we have
\begin{equation}	\label{tfree.sum}
\P(\lam_1\in I) \le 2\sum_\ell \P(\lam_1\in I, v_1^{(\eta_\ell)}\in A_\ell')\,.
\end{equation}
Let $T=T(L)\ge1$ to be taken sufficiently large depending on $L$. 
Applying Proposition \ref{prop:upper-joint} with $\heta=\eps$ and $\zcush=\zcush_0$, followed by Lemma \ref{lem:Jprops}(\ref{Jprop.T},\ref{Jprop.cont}), assuming $T$ is sufficiently large we have that for each $\ell$,
\begin{align}
\frac1N\log\P(\lam_1\in I,  v_1^{(\eta_\ell)}\in A_\ell') 
&\le - \inf_{y\in I, \vloc\in A_\ell'} \cJ_{N,2N^{\eta_\ell}}(y,\vloc) + N^{-c\eps}	\notag\\
&= - \inf_{y\in I, \vloc\in A_\ell'} \cJ^{[0,T]}_{N,2N^{\eta_\ell}}(y,\vloc) + N^{-c\eps}	\notag\\
&\le - \inf_{\vloc\in A_\ell'} \cJ^{[0,T]}_{N,2N^{\eta_\ell}}(x,\vloc) +O_L(\xcush^{-1/2}\delta)+ N^{-c\eps}\,.
\label{tfree.ell1}
\end{align}
For arbitrary $\vloc\in A_\ell'$, set
\begin{equation}	\label{def:chzt}
\chz:= (\vloc_i 1_{|\vloc_i|>M^{-2\ell+1}})_{i=1}^N,\qquad
\tvloc:= (\vloc_i1_{|\vloc_i|\le M^{-2\ell-1}})_{i=1}^N.
\end{equation}
Letting $\tcush=\tcush(L,\xcush)>0$ be sufficiently small depending on $L,\xcush$, 
from Lemma \ref{lem:Jprops}(\ref{Jprop.T},\ref{Jprop.cont}) we have
\begin{equation}	\label{tfree-up.1}
\cJ^{[0,T]}_{N,2N^{\eta_\ell}}(x,\vloc) = \cJ_{N,2N^{\eta_\ell}}^{[\thetam+\tcush,T]}(x,\vloc)
=\cJ_{N,2N^{\eta_\ell}}^{[\thetam+\tcush,T]}(x,\chz+\tvloc)+ O_L\big( \eps^{-1/2}(\log N)^{-1/4}\big)\,.
\end{equation}
For any $\theta\in [\thetam+\tcush,T]$, with $\chw=\overlap_x(\theta)\chz$, $\tw=\overlap_x(\theta)\tvloc$, we have
\begin{equation}
\|\tw\|_\infty \le \|\tvloc\|_\infty\le M^{-2\ell-1} = M^{-1}N^{-\eta_\ell}
\end{equation}
and from the definitions \eqref{def:Aell'} and \eqref{def:chzt} of $A_\ell'$ and $\chz$,
\begin{equation}
|\chw_i\tw_j| = \overlap_x(\theta)^2|\chz_i\tvloc_j| \gs_\tcush |\chz_i\tvloc_j| \ge N^{-1/2+\eta_\ell} M^{-2\ell+1} \ge \frac{M}{\sqrt{N}} \quad\forall i\in\supp(\chz)\,,\;j\in\supp(\tvloc).
\end{equation}
We can hence apply Lemma \ref{lem:free-tfree}, along with the fact that $\tfree_{N,R}$ is monotone in $R$, to bound
\begin{align*}
\free_{N,2N^{\eta_\ell}}(\theta,\chw+\tw) 
&\le \tfree_{N,2N^{\eta_\ell}}(\theta,\chw,\|\tw\|_2^2) 
+ O(T^2\delta_{\tcush M} + T^3 M^{-1})\\
&= \tfree_{N,2N^{\eta_\ell}}(\theta,\overlap_x(\theta)\chz,\overlap_x(\theta)^2\|\tvloc\|_2^2) + O(T^2\delta_{\tcush M} + T^3 M^{-1})\\
&\le  \tfree_{N,2N^{2\eps}}(\theta,\overlap_x(\theta)\chz,\overlap_x(\theta)^2\|\tvloc\|_2^2) + \ep_N
\end{align*}
for a sequence $\ep_N=o(1)$ converging to zero at a rate depending only on $L,\xcush$ and $\mu$ (in particular the error is uniform in $\theta\in[\thetam+\tcush,T]$, $\ell$, $\chz$ and $\tvloc$).
Hence,
\begin{align*}
\cJ_{N,2N^{\eta_\ell}}^{[\thetam+\tcush,T]}(x,\chz+\tvloc) 
&\ge \sup_{\theta\in[\thetam+\tcush,T]} \big\{ J(x,\theta) - \tfree_{N,2N^{2\eps}}(\theta,\overlap_x(\theta)\chz,\overlap_x(\theta)^2\|\tvloc\|_2^2) \big\} - \ep_N\\
&= \wt\cJ_{N,2N^{2\eps}}(x,\chz,\|\tvloc\|_2^2) - \ep_N
\end{align*}
where in the last line we used Lemma \ref{lem:tJprops}(\ref{tJprop.T}), taking $T(L),\tcush(L,\xcush)$ larger and smaller, respectively, if necessary.
Combining with \eqref{tfree.sum}, \eqref{tfree.ell1}, \eqref{tfree-up.1} we have
\begin{align*}
&\frac1N\log \P(\lam_1\in I) \\
&\le - \min_\ell\inf_{\vloc\in A_\ell'} \wt\cJ_{N,2N^{2\eps}}(x,\chz,\|\tvloc\|_2^2) 
+ O_L(\xcush^{-1/2}\delta)+N^{-c\eps} + \frac{\log\log N}N + \frac{1}{\eps^{1/2}(\log N)^{1/4}}+ \ep_N\\
&\le - \min_\ell \inf_{\substack{\chz\in \B_{\ge MN^{-\eta_\ell}},\\ \tal\in [0,1-\zcush_0-\|\chz\|_2^2]}} \wt\cJ_{N,2N^{2\eps}}(x,\chz,\tal)
+O_L(\xcush^{-1/2}\delta)+\ep'_N\\
&\le - \inf_{\substack{\chz\in \B_{\ge N^{-2\eps}},\\ \tal\in [0,1-\zcush_0-\|\chz\|_2^2]}} \wt\cJ_{N,2N^{2\eps}}(x,\chz,\tal)
+O_L(\xcush^{-1/2}\delta)+\ep'_N\\
&\le - \wt\rate_{N,N^{-2\eps}}(x) +O_L(\xcush^{-1/2}\delta)+\ep'_N
\end{align*}
for some $\ep'_N=o(1)$ converging to zero at a rate depending only on $L,\xcush,\eps$ and $\mu$.
This completes the proof of Proposition \ref{prop:weakLDP}. 
\end{proof}

\section{Joint eigenvalue-eigenvector large deviation upper bound}
\label{sec:upper}

In this section we prove Proposition \ref{prop:upper-joint} giving a joint large deviations upper bound for $\lam_1$ and the localized part $v_1^{(\eta)}$ of the associated eigenvector. For the proofs we use of estimates on the quenched free energy for spherical integrals (Lemmas \ref{lem:approxsc} and \ref{lem:restrict}), which we prove in Section \ref{sec:quenched}. 

For $\eta\in(0,\frac14)$ recall the notations $v^{(\eta)}$, $n_0(\eta)$ from \eqref{def:eta}.
For $x\ge2$, $\vloc\in \B^N$ and $\delta>0$ denote 
\begin{align}
\event_x(\delta)&:= \{ M\in\Sym_N : |\lambda_1(M)-x|<\delta\}	,\label{def:Ex}\\
\event'_{x,\vloc}(\delta,\eta)
&:=\event_x(\delta)\cap\{M\in \Sym_N: \|v_1(M)^{(\eta)}-\vloc\|_\infty\le N^{-10}\}\,.	\label{def:Exw}
\end{align}

First we define some ``good'' properties of $H$ that hold outside events of probability that can be made negligible compared to the probability of the event that $H\in\event'_{x,z}(\delta, \eta)$. For a matrix $M\in\Sym_N$ with eigenvalues $\lambda_{N}\le\lambda_{N-1}\le \cdots\le\lambda_{1}$ and 
for $0\le k\le N-1$, we denote
\begin{equation}
G_M^{(k)}(y) := \frac1N\sum_{i=k+1}^N \frac1{y-\lambda_i}\,,\qquad G_M(y):= G_M^{(0)}(y)
\end{equation}
and set
\begin{equation}	\label{thetaMk}
\theta_M^{(k)} := \frac12G_M^{(k)}(\lambda_1)\,, \qquad \overlap_M^{(k)}(\theta) := \bigg(1-\frac{\theta_M^{(k)}}{\theta}\bigg)_+^{1/2}. 
\end{equation}
(Compare \eqref{def:Gsc}--\eqref{def:overlapx}.)
We also denote 
the log-potentials
\begin{equation}
\Lpot_M(y):= \int \log (y-\lambda)d\hat{\mu}_M(\lambda)\,,\qquad \Lpot_\sigma(y):= \int \log(y-\lambda)d\sigma(\lambda)
\end{equation}
(recalling from Section \ref{sec:notation} the notation $\hat{\mu}_M$ for the empirical spectral measure of $M$).
For $K,\xcush>0$ and $\eta\in(0,\frac14)$ let
\begin{align}
\Good(K,\xcush,\eta) &:= \Good_0(K)\cap\Good_1(\xcush,\eta)\cap\Good_2(K,\xcush,\eta)\,,	\label{def:good}\\
\Good_0(K)&:=\{M\in\Sym_N: \|M\|\le K\}\,,	\notag\\
\Good_1(\xcush,\eta) &:= \{ M\in \Sym_N: \lambda_{\lf N^{1/2+\eta}\rf}(M)\le 2+ \xcush\}\,,	\notag\\
\Good_2(K,\xcush,\eta)&:= \bigcap_{2+2\xcush\le y\le K} \bigg\{ M\in \Sym_N: \bigg|\int_{-\infty}^{2+\xcush} \frac{d\hat{\mu}_M(\lambda)}{y-\lambda} - G_\sigma(y)\bigg| \le N^{-1/2+2\eta}\,,	\notag\\
&\qquad\qquad\qquad \qquad
\bigg|\int_{-\infty}^{2+\xcush} \log(y-\lambda)d\hat{\mu}_M(\lambda) - \Lpot_\sigma(y) \bigg| \le N^{-1/2+2\eta} \bigg\}\,.		\notag
\end{align}

\begin{lemma}
\label{lem:good}
For any $A,\xcush>0$, $\eta\in(0,\frac14)$ and $K=K_\mu(A)$ sufficiently large depending on $A$ and $\psimax$, 
\[
\frac1N\log \P( H\notin \Good(K,\xcush,\eta)) \le -A
\]
for all $N$ sufficiently large depending on 
$A,\xcush, \eta$ and $\mu$.
\end{lemma}

The proof of Lemma \ref{lem:good} involves standard concentration and truncation arguments and is deferred to Appendix \ref{app:good}.

Proposition \ref{prop:upper-joint} is a consequence of the following upper bound in terms of the restricted annealed free energies, together with Proposition \ref{prop:FE} and Lemma \ref{lem:good}. 

\begin{prop}
\label{prop:upper} 
Let $\heta,\xcush,\zcush\in(0,\frac1{10})$.
For any $\eta\in(\heta,\frac14-\heta)$, $x\in[2+\xcush, \xcush^{-1}]$
and  $\vloc\in (1-\zcush)\B^N$ with $\|\vloc\|_0\le n_0(\eta)$,
\begin{align*}
&\frac1N\log\P\big(H\in \event'_{x,\vloc}(N^{-2\heta},\eta) \cap \Good(K,\kappa,\eta)\big) \\
&\qquad\qquad\le 
\inf_{\theta\in [\thetam+\tcush,N^{c\heta}]}\Big\{ F_N\big(\theta;\,\Uloc_{\overlap_x(\theta)\vloc}(\srad,R)\big) - J(x,\theta)  \Big\}
+ N^{-c\heta}
\end{align*}
for all $N$ sufficiently large depending on $\heta,\xcush,\zcush$ and $\tcush$, 
where $\srad=N^{-\heta/2}, R=2N^\eta$ and $c>0$ is an absolute constant. 
\end{prop}

\begin{proof}[Proof of Proposition \ref{prop:upper-joint}]

We will assume without comment that $N$ is sufficiently large depending on $\heta,\xcush,\zcush$ and $\mu$.
We first prove a localized form of the claim. We claim that  for any $x\in[2+\xcush,\xcush^{-1}]$ and $\vloc\in(1-\zcush)\B^N$ with $\|\vloc\|_0\le n_0$, 
\begin{equation}	\label{bd:upper}
\frac1N\log\P\Big( |\lam_1-x|\le N^{-2\heta},\, \|v_1^{(\eta)} - \vloc\|_\infty \le N^{-10}\Big)
 \le -\cJ_{N,R}(x,\vloc) + N^{-c\heta}\,.
\end{equation}
From the assumption $\|z\|_2\le 1-\zcush$ and Lemma \ref{lem:Jprops}(\ref{Jprop.T}) we can take $T$ and $\tcush=1/T$ with $T=O_{\xcush,\zcush}(1)$ such that $\cJ_{N,R}(x,\vloc)= \cJ^{[\thetam+\tcush,T]}_{N,R}(x,\vloc)$. 
Then note that from Lemma \ref{lem:Jprops}(\ref{Jprop.smallz},\ref{Jprop.cont}) we have 
\[
 \cJ^{[\thetam+\tcush,T]}_{N,R}(x,\vloc) = \cJ^{[\thetam+\tcush,T]}_{N,R}(x,0)+ O_{\xcush,\zcush}(1) = \rate^\gamma(x) + O_{\xcush,\zcush}(1) = O_{\xcush,\zcush}(1).
\]
Hence, the right hand side of \eqref{bd:upper} is bounded below by $-C_0$ for some finite constant $C_0=C_0(\xcush,\zcush)>0$. 
Applying Lemma \ref{lem:good} with $A=2C_0$, say, (note that $\Good(K,\xcush,\eta)$ is monotone in $\eta$) it suffices to show 
\begin{equation}
\label{upjoint-goal1}
\frac1N\log\P\big(H\in \event'_{x,\vloc}(N^{-2\heta},\eta) \cap \Good(K,\xcush,\eta)\big) \le - \cJ^{[\thetam+\tcush,T]}_{N,R}(x,\vloc) + N^{-c\heta}
\end{equation}
where $K=K_\mu(2C_0)$ as in Lemma \ref{lem:good} is sufficiently large depending only on $\xcush,\zcush$ and $\psimax$.
Now applying Proposition \ref{prop:upper} (taking $N_0$ large enough that $N^{c\heta}\ge T$) followed by Proposition \ref{prop:FE}, we have that the left hand side in \eqref{upjoint-goal1} is at most 
\begin{align*}
\inf_{\theta\in [\thetam+\tcush,T]}\Big\{ F_N(\theta;\Uloc_{\overlap_x(\theta)\vloc}(\srad,R)) - J(x,\theta)  \Big\}
+ N^{-c\heta}
&\le - \cJ^{[\thetam+\tcush,T]}_{N,R}(x,\vloc) + N^{-c\heta}
\end{align*}
which yields \eqref{upjoint-goal1} and hence \eqref{bd:upper}. 

Let $ A'= A\cap\B_{\ge N^{-1/2+\eta}}$. Note that all elements of $A'$ have support of size at most $N^{1-2\eta}$. 
Hence, we can take $\Sigma\subset A'$ an $N^{-10}$-net for $ A'$ under the $\ell^\infty$ norm of size $\exp( O( N^{1-2\eta}\log N))$. 
Let $\Lambda\subset I$ be an $N^{-2\heta}$-net for $I$ of size $O(N^{2\heta})$. As $\|v_1^{(\eta)}\|_0\le N^{1-2\eta}$ a.s. (as $v_1$ is a unit vector), applying the union bound followed by \eqref{bd:upper}, 
we have
\begin{align*}
\P( \lam_1\in I, v_1^{(\eta)}\in A) 
&\le \sum_{y\in I, \vloc\in\Sigma} \P( |\lam_1-y|\le N^{-2\heta}, \|v_1^{(\eta)}-\vloc\|_\infty\le N^{-10})\\
&\le O(N^{2\eta}) \exp(O(N^{1-2\eta}\log N)) \exp( N^{1-c\heta} - N \min_{y\in\Lambda,\vloc\in\Sigma}\cJ_{N,R}(y,\vloc))\\
&\le \exp( N^{1-c'\heta} - N\inf_{y\in I, \vloc\in  A\cap\B^{n_0}}\cJ_{N,R}(y,\vloc))
\end{align*}
where in the last line we used that $\cJ_{N,R}$ is invariant under permutations of the coordinates of $\vloc$.
Taking logs and dividing through by $N$ yields the claim.
\end{proof}

For the proof of Proposition \ref{prop:upper} we have the following two lemmas concerning spherical integrals for deterministic matrices $M\in\Sym_N$ having the ``good'' properties enforced by the events $\Good_i$; the proofs are given in Section \ref{sec:quenched}.
The first provides quantitative  asymptotics for the spherical integral $I(M,\theta)$ and $\theta_M^{(k)}$ of \eqref{thetaMk} for $M\in \Good(K,\xcush,\eta)$ having top eigenvalue near $x$.
This result is also used in the proof of the lower bound in Theorem \ref{theomain}.

\begin{lemma}
\label{lem:approxsc}
Let $\tcush ,\xcush\in(0,1)$, 
{$\eta\in(0,\frac14)$,} $2+3\xcush\le x<K$ and $\theta{\in[\thetam  + \tcush ,N^{10}]}$. 
There exists $\delta_0=\delta_0(\xcush,\tcush )>0$  
such that the following holds for any $\delta\in(0,\delta_0]$.
For
any $M\in\Evalx_x(\delta)\cap \Good(K,\xcush,\eta)$, 
\begin{equation}	\label{approx.Jsc}
 \frac1N \log I(M,\theta) = J(x, \theta)  + {O_{K,\xcush,\tcush }( \theta (\delta +  N^{-1/4} {\,+N^{-1/2+2\eta}}))}  
\end{equation}
for all $N\ge  {(4/\tcush )^{1/(1-4\eta)}}$. 
Moreover, for $k\le \lf N^{1/2+\eta}\rf$ and any $\delta>0$ we have
\begin{equation}	\label{approx.thetasc}
\theta_M^{(k)} = \thetam  + O_\xcush(\delta)+ O(N^{-1/2+2\eta})
\end{equation}
for all $N$.
\end{lemma}

On the other hand,  the next lemma shows that the main contribution to the spherical integral $I(M,\theta)$ (at least up to sub-exponential corrections) comes from unit vectors $u$ at a certain angle to the leading eigenvector $v_1(M)$, assuming the spectrum is not too concentrated near $\lambda_1(M)$. 
For technical reasons we show we can make some further restrictions on $u$.
Specifically, 
for given $v\in\sphereN, \overlap\in[0,1], L_0>0$ and subspace $W\subset \R^N$, let
\begin{align}
\Uloc'_v(\overlap,L_0,W) &:= \bigg\{u\in \sphereN:  |\langle v,u\rangle^2-\overlap^2|\le 
L_0N^{-1/2}
,\notag\\
&\qquad\qquad\qquad \|(I-vv^\tran)u\|_\infty
\le 
L_0\sqrt{\tfrac{\log N}N}\,  , \;
 \| \Pi_W(I-vv^\tran)u\|_2\le 
 L_0\sqrt{\tfrac{\dim W}N}\,\bigg\}
\label{def:U1}
\end{align}
where $\Pi_W$ denotes the orthogonal projection to $W$. 

\begin{lemma}
\label{lem:restrict}
For any $\delta_0,\tcush>0$, $1\le k\le N-1$, $M\in\Sym_N$ such that
\begin{equation}	\label{kgap}
\delta_0\le \lambda_1(M)-\lambda_{k+1}(M) \le N^{10},
\end{equation}
any 
$\theta\in[\theta_M^{(k)}+\tcush, N^{10}]$,
and any subspace $W$ of $\R^N$, 
{if $L_0$ is a sufficiently large constant depending on $\delta_0,\tcush$, then,}
    with $\overlap_M^{(k)}(\theta)$ as in \eqref{thetaMk},
\[
\frac1N\log I(M,\theta) = \frac1N \log \int_{\Uloc'_{v_1(M)}(\overlap_M^{(k)}(\theta), L_0,W) } e^{\theta N\langle u, Mu\rangle} dP(u) + 
{O\bigg(\frac{k}N\log N\bigg)}
\]
for all $N$ sufficiently large depending on $\delta_0,\tcush$.
\end{lemma}

We now conclude the proof of Proposition \ref{prop:upper} assuming Lemmas \ref{lem:approxsc} and \ref{lem:restrict}.

\begin{proof}[Proof of Proposition \ref{prop:upper}]
Let $\eta,\zcush,\tcush ,\xcush $, $x$ and $z$ be as in the statement of the proposition.
Throughout we abbreviate $v_1=v_1(H)$, $\lambda_1=\lambda_1(H)$. 
We will also assume $N$ is sufficiently large depending on $\xcush ,\zcush,\tcush $ and $\heta$ without comment.
For brevity we write $\Uloc_w:=\Uloc_w(N^{-\heta/2}, 2N^\eta)$,  $\event_x:=\event_x(\delta)$, $\event'_{x,z}:=\event'_{x,z}(\delta,\eta)$ and $\Good:=\Good(K,\xcush,\eta)$ throughout. 

Fixing an arbitrary $\theta\in[\thetam +\tcush ,N^{c\heta}]$, our aim is to show
\begin{equation}	\label{upper.goal1}
\frac1N\log\P\big(H\in \event'_{x,\vloc}\cap \Good\big) \le 
F_N(\theta;\Uloc_{\overlap_x(\theta)\vloc}) -J(x,\theta) 
+ N^{-c\heta}\,.
\end{equation}
We first apply Lemma \ref{lem:approxsc} (using the restriction to $\Evalx_x\cap\Good$) 
to obtain
\begin{align*}
&\frac1N\log\P(H\in \event'_{x,\vloc} \cap \Good) 
=\frac1N\log\E [\ind(H\in \event'_{x,\vloc} \cap \Good)\frac{I(H,\theta)}{I(H,\theta)}]\\
&\qquad\qquad=  \frac1N\log\E[ \ind(H\in \event'_{x,\vloc} \cap \Good)I(H,\theta)] -J(x,\theta) +O_{\xcush ,\tcush }({N^{-c\heta}})
\end{align*}
for $c$ sufficiently small. We can remove the implicit constant depending on $\xcush ,\tcush $ by further shrinking $c$ and assuming $N$ is sufficiently large.
Up to further modification of $c$, it thus suffices to show
\begin{equation}	\label{upper.goal3}
\frac1N\log\E[ \ind(H\in \event'_{x,\vloc} \cap \Good)I(H,\theta)] \le F_N(\theta;\Uloc_{\overlap_x(\theta)\vloc}) + N^{-c\heta}.
\end{equation}

To that end, 
taking $W:=\R^{\supp(\vloc)}$,
from \eqref{approx.thetasc} and our assumption on $\theta$, we can apply Lemma \ref{lem:restrict} with $k=\lf N^{1/2+\eta}\rf$ and 
$\delta_0$ sufficiently small depending on $\xcush $
to bound the left hand side of \eqref{upper.goal3} by
\begin{equation}	\label{upper.bound}
\frac1N\log \E \ind(H\in \event'_{x,\vloc} \cap \Good)\int_{\Uloc'_{v_1}(\overlap_H^{(k)}(\theta),L_0,W) } e^{\theta N\langle u, Hu\rangle}dP(u) + O(N^{-1/2+\eta}\log N)
\end{equation}
for some $L_0=O_{\delta_0,\tcush }(1)=O_{\xcush,\tcush}(1)$. 
We claim that for $H\in \Sym^{(\vloc)}$ and $\overlap,\overlap_0
>0$
with $|\overlap^2-\overlap_0^2|\le  
{N^{-3\heta/2}}$, 
\begin{equation}	\label{upper.goal4}
\Uloc'_{v_1(H)}(\overlap_0,L_0,W) \subseteq \Uloc_{\overlap\vloc}\cup \Uloc_{-\overlap\vloc}
\end{equation}
for all $N$ sufficiently large depending on $\overlap,\eta$ {and $L_0$}.
Indeed, fixing any element $u$ of the left hand side, we abbreviate $u^\perp = (I-v_1v_1^\tran)u$, so that $u=\langle v_1,u\rangle v_1 + u^\perp$.
Noting that $(u^\perp)_\vloc = \Pi_W(I-v_{1}v_{1}^\tran)u$, we have
\begin{align*}
\|u_{\vloc}\pm \overlap\vloc\|_2
&= \|  \Pi_W(I-v_{1}v_{1}^\tran)u+ \langle v_1,u\rangle (( v_1)_{\vloc}-\vloc)+  (\langle v_1,u\rangle \pm \overlap) \vloc\|_2\\
&\le \| \Pi_W(I-v_{1}v_{1}^\tran)u \|_2 + |\langle v_1,u\rangle| \| (v_1)_{\vloc}-\vloc\|_2 +  |\langle v_1,u\rangle \pm \overlap|\|\vloc\|_2\\
&\le 
{L_0N^{-\eta}}+ N^{-100} + |\langle v_1,u\rangle \pm \overlap|.
\end{align*}
Now since
\begin{align*}
 |\langle v_1,u\rangle + \overlap|\cdot |\langle v_1,u\rangle - \overlap| 
 &=  |\langle v_1,u\rangle^2 -\overlap^2|\\
& \le |\langle v_1,u\rangle^2 -\overlap_0^2| 
+ 
{N^{-3\heta/2}  }\\
&{\le L_0N^{-1/2}+N^{-3\heta/2}\le 2N^{-3\heta/2} }
\end{align*}
so that
\[
\min_{\epsilon=\pm1}\{  |\langle v_1,u\rangle +\epsilon \overlap|\}\le 
{\sqrt{2}N^{-3\heta/4}}
\]
we get 
\[
\min_{\epsilon=\pm 1}\| u_{\vloc}+\epsilon \overlap\vloc\|_2  \le {N^{-\heta/2}}\,. 
\]
Moreover, 
since $H\in \Sym^{(\vloc)}$ we have 
$
\|(v_1)_{\vloc^c}\|_\infty \le N^{-1/2+\eta}
$,
so
\begin{align*}
\| u_{\vloc^{c}}\|_\infty
&\leq | \langle v_1,u \rangle | \| (v_1)_{\vloc^{c}}\|_\infty +\| (u^\perp)_{\vloc^{c}}\|_\infty 
\le 
N^{-1/2+\eta}
+
L_0\sqrt{\frac{\log N}N}
\le 2 N^{-1/2+\eta} \,
\end{align*}
and \eqref{upper.goal4} follows.
From \eqref{approx.thetasc} we have
\[
\overlap_H^{(k)}(\theta)^2 - (\overlap_x(\theta))^2 \le O_{\xcush }(N^{-2\heta} + N^{-1/2+2\eta} ) \le N^{-3\heta/2}
\]
so we can apply \eqref{upper.goal4} with $\overlap_0=\overlap_H^{(k)}(\theta)$, $\overlap=\overlap_x(\theta)$. Noting also that the integral in \eqref{upper.bound} is invariant under replacing $u$ with $-u$, we get that \eqref{upper.bound} is bounded above by
\begin{align*}
\frac1N\log \E \ind(H\in\Evalx_x\cap\Sym^{(\vloc)}\cap\Good)\int_{\Uloc_{\overlap_x(\theta)\vloc}} e^{\theta N\langle u, Hu\rangle}dP(u) + \frac{\log 2}{N} + O_{\xcush }(N^{-1/2+\eta}\log N)\,.
\end{align*} 
The final two terms are bounded by $N^{-1/8}$.
Dropping the indicator $\ind(H\in\Evalx_x\cap\Sym^{(\vloc)}\cap\Good)$, the claim follows. 
\end{proof}

\section{Quenched asymptotics for restricted spherical integrals}
\label{sec:quenched}

In this section we prove Lemmas \ref{lem:approxsc} and \ref{lem:restrict}. 
Throughout this section we drop the dependence on $M$ in $v_i=v_i(M)$, $\lambda_i=\lambda_i(M)$ for brevity.

\subsection{Proof of Lemma \ref{lem:approxsc}} 

The following is a consequence of \cite[Lemma 2.3]{Ma07}.

\begin{lemma}
\label{lem:maida}
Suppose $\|M\|\le K$. Let $\theta>0$, and let $r=r(M,\theta)$ denote the unique solution in $[\lambda_N-\frac1{2\theta}, \lambda_1-\frac1{2\theta}]^c$ of 
\begin{equation}	\label{def:r}
G_M( r+ \frac1{2\theta}) = 2\theta.
\end{equation}
We have
\begin{equation}	\label{Maida.error}
\frac1N\log I(M,\theta) = \theta r - \frac1{2N} \sum_{i=1}^N \log( 1+ 2\theta( r - \lambda_i)) + {O_K((1+\theta)N^{-1/4})} \,.
\end{equation}
\end{lemma}
{(In \cite{Ma07} the dependence of the error in \eqref{Maida.error} on $\theta$ is not specified, but it is readily seen from the proof to be of the above form.)}
In the following we denote 
\[
f_y:(-\infty,y)\to \R^+\,,\quad f_y(\lambda) = \frac1{y-\lambda}
\]
so that for $y>\max\{2,\lambda_1\}$ we have
\[
G_M(y) = \hat{\mu}_M(f_y)\,,\quad G_\sigma(y) = \sigma(f_y) = \frac12( y - \sqrt{y^2-4}).
\]

\begin{lemma}
\label{lem:ylam1}
Let $x,\theta,K,\eta,\tcush $ and $\xcush$ be as in Lemma \ref{lem:approxsc}.
There exists $\delta_0=\delta_0(\xcush,\tcush )>0$ such that {the following holds} 
for any $M\in \Evalx_x(\delta_0)\cap\Good_1(\xcush,\eta)\cap\Good_2(K,\xcush,\eta)$.
Letting $y>\lambda_1$ be the unique solution to 
\[
G_M(y) = 2\theta\,,
\]
we have 
\begin{equation}\label{ylam1-LB}
y-\lam_1 \ge N^{-12} \qquad\forall N\ge 2
\end{equation}
and
\begin{equation}\label{ylam1-UB}
y-\lam_1 \le \tcush ^{-1} N^{-1/2+\eta} \qquad \forall N\ge (4/\tcush )^{1/(1-4\eta)}.
\end{equation}
\end{lemma}

\begin{proof}
We have
\begin{equation}	\label{ylam1.1}
2\theta= \frac1N\sum_{j:\lambda_j\le 2+\xcush} \frac1{y-\lambda_j} + \frac1N \sum_{j: \lambda_j>2+\xcush} \frac1{y-\lambda_j} =: (I) + (II).
\end{equation}
Rearranging and using that (I) is non-negative, we have  since we assumed $\theta\le N^{10}$, 
\[
\frac1{N(y-\lam_1)} \le (II) = 2\theta - (I) \le 2\theta \le 2N^{10}
\]
giving \eqref{ylam1-LB}. 
In the other direction, from \eqref{ylam1.1} we have
\[
\frac{N^{-1/2+\eta}}{y-\lam_1} \ge (II) = 2\theta-(I) \ge 2\theta-G_\sigma(y)-N^{-1/2+2\eta}
\]
where the first bound follows from the restriction to $\Good_1(\xcush,\eta)$ and the second bound from  the restriction to $\Good_2(K,\xcush,\eta)$.
Now since $y>\lambda_1 = x+ 
{O(\delta_0)}$
and $G_\sigma$ is monotone decreasing and 
{$O_\xcush(1)$-Lipschitz on $[2+\xcush,\infty)$,}
we have 
\[
G_\sigma(y) \le G_\sigma(x) + 
{O_\xcush(\delta_0)}
= 2\thetam  + 
{O_\xcush(\delta_0)}\,.
\]
Combining the last two displays with our assumption that $\theta\ge \thetam+\tcush $ gives
\[
\frac{N^{-1/2+\eta}}{y-\lam_1} \ge 2\tcush   -O_\xcush(\delta_0)- N^{-1/2+2\eta}.
\]
Taking $\delta_0$ sufficiently small depending on $\xcush$ and $N\ge (4/\tcush )^{1/(1-4\eta)}$, we can bound the right hand side below by $\tcush $, giving \eqref{ylam1-UB}. 
\end{proof}

We now conclude the proof of Lemma \ref{lem:approxsc}.
Continuing to denote by $y:=r(M,\theta)+\frac1{2\theta}>\lambda_1$ the unique solution of $G_M(y)=2\theta$, from Lemma \ref{lem:maida} we have
\begin{equation}	\label{Jsc.1}
\frac1N\log I(M,\theta) 
=  \theta y - \frac12 \int  \log( y - \lambda) d\hat{\mu}_M(\lambda) - \frac12\log(2\theta) -\frac12 + O_{K}((1+\theta)N^{-1/4})\,.
\end{equation}
Now splitting
\[
 \int  \log( y - \lambda) d\hat{\mu}_M(\lambda) 
 = \int_{-\infty}^{2+\xcush} \log(y-\lambda)d\hat{\mu}_M(\lambda) + 
  \int_{2+\xcush}^\infty \log(y-\lambda)d\hat{\mu}_M(\lambda) 
  =: (I) + (II)\,,
\]
from the restriction to $\Good_2(K,\xcush,\eta)$ we have
\[
(I) = \int\log(y-\lambda)d\sigma(\lambda) + O(N^{-1/2+2\eta})
\]
while the restriction to $\Good_1(\xcush,\eta)$ and 
the lower bound \eqref{ylam1-LB} from Lemma \ref{lem:ylam1} imply
\[
(II) \le N^{-1/2+\eta} |\log (y-\lambda_1)| \ls N^{-1/2+\eta}\log N\,.
\]
Substituting these estimates into \eqref{Jsc.1}, we have
\[
\frac1N\log I(M,\theta) 
=  \theta y - \frac12 \int  \log( y - \lambda) d\sigma(\lambda) - \frac12\log(2\theta e)  + O_{K}((1+\theta)N^{-1/4}{\,+N^{-1/2+2\eta}})\,.
\]
Finally, we can apply the upper bound \eqref{ylam1-UB} from Lemma \ref{lem:ylam1} and $|x-\lambda_1|\le 
{\delta}$
to replace $y$ with $x$ above, incurring an additive error of $O_{\tcush ,\xcush}({ \theta (N^{-1/2+\eta}+\delta) } )$. 
This completes the proof of \eqref{approx.Jsc}.
For \eqref{approx.thetasc} we 
have by the restriction to $\Good_1(\xcush,\eta)\cap \Good_2(K,\xcush,\eta)$ that
\begin{align*}
2\theta_M^{(k)}
&= \frac1N\sum_{j\ge k+1} \frac1{\lambda_1-\lambda_j} \\
&= \int_{-\infty}^{2+\xcush} \frac{d\hat{\mu}_M(\lambda)}{\lambda_1-\lambda} - \frac1N\sum_{j\le k} \frac1{\lambda_1-\lambda_j}1_{\lambda_j\le 2+\xcush}\\
&= G_\sigma(\lambda_1)+ O(N^{-1/2+2\eta}) + O(\xcush^{-1}N^{-1/2+\eta})\\
&=G_\sigma(x) + O_\xcush(\delta) + O(N^{-1/2+2\eta})
\end{align*}
as claimed.\qed

\subsection{Proof of Lemma \ref{lem:restrict}}

Let $M$ and $\theta$ be as in the statement of the lemma. 
We write $V=(v_1,\dots, v_N)$ and $\Lambda=\diag(\lambda_1,\dots,\lambda_N)$, so that $M=V\Lambda V^\tran$. 
For $S\subset [N]$ we will write $V_S$ for the $N\times |S|$ matrix with columns $(v_j)_{j\in S}$. In what follows, for $S\subset[N]$ we abbreviate $P_S:= V_SV_S^\tran$ for the spectral projections.
For a vector $\omega=(\omega_j)_{j=2}^N\in \R^{N-1}$, we write $\omega=(\omega',\omega'')\in \R^{[2,k]}\times \R^{[k+1,N]}$.     
{For a given $N\times (N-k)$ matrix $A$ {and $L\ge1$}, we  set}
\begin{align*}
E_1&:= \{ \omega\in \R^{N-1}: \|\omega'\|_2 < N^{-5}, |\|\omega''\|_2^2 - (1-(\overlap_M^{(k)})^2)| < 
{LN^{-1/2}}\}\\
E_2&:= \{ \omega\in \R^{N-1}: \|V_{[k+1,N]}\omega''\|_\infty \le 
{LN^{-1/2}\log^{1/2}N} \}\\
E_{3}&:=\Big\{  \omega\in \R^{N-1}: \| A\omega''\|_2 \le 
{L\|A\|_\HS N^{-1/2}}\,\Big\} \,.
\end{align*}

Recall the probability measures $Q^{(\theta,M)}$ on $\sphereN$ defined in \eqref{def:tiltQM}.
The main step of the proof of Lemma \ref{lem:restrict} is to establish the following: 

\begin{lemma}
\label{lem:sphere}
With the above definitions, and hypotheses as in Lemma \ref{lem:restrict}, 
assume $L$ is a sufficiently large constant depending on $\delta_0,\tcush $. Then
\begin{equation}	\label{bound.sphere}
Q^{(\theta,M)} ( \{ u\in\sphere^{N-1}: V_{[2,N]}^\tran u \in E_1\cap E_2\cap E_{3}\}) 
\ge N^{-O(k)}
\end{equation}
for all $N$ sufficiently large depending on $\delta_0$ and $\tcush $.
\end{lemma}

\begin{proof}[Proof of Lemma \ref{lem:restrict} granted Lemma \ref{lem:sphere}]
In terms of the measure $Q^{(\theta,M)} $ our aim is to show
\begin{equation}
\label{sphere.goal1}
Q^{(\theta,M)} ( \Uloc'_{v_1}( \overlap_M^{(k)}(\theta) , L_0,W)) = N^{-O(k)}.
\end{equation}
Taking $A:=\Pi_WV_{[k+1,N]}$ (which is independent of $u$ and thus can be considered as given), we have
\[
\|A\|_\HS \le \|\Pi_W\|_\HS \|V_{[k+1,N]}\| = \dim(W)^{1/2}.
\]
From Lemma \ref{lem:sphere}, to establish \eqref{sphere.goal1} it thus suffices to show  
\begin{equation}
\label{sphere.goal2}
 \Uloc'_{v_1}( \overlap_M^{(k)}(\theta) ,{2L}, W)
 \supset
 \{ u\in\sphere^{N-1}: V_{[2,N]}^\tran u \in E_1\cap E_2\cap E_{3}\}\,.
\end{equation} 
If $W=\{0\}$ then $E_3=\R^{N-1}$, so we may assume $\dim W\ge1$. 
 Fixing an arbitrary element $u$ of the right hand side, we set $\omega= V^{\tran}_{[2,N]}u$, so that $\omega'= V^{\tran}_{[2,k]}u$, $\omega''=V^{\tran}_{[k+1,N]}u$. 
Now to verify the first condition in $\Uloc'_{v_1}(\overlap_M^{(k)}(\theta), W)$, we have
\begin{align*}
|\langle v_1,u\rangle^2 - \overlap_M^{(k)}(\theta)^2| 
&= |1- \|P_{[2,k]}u\|_2^2 - \|P_{[k+1,N]}u\|_2^2 - \overlap_M^{(k)}(\theta)^2| \\
&\le \|P_{[2,k]}u\|_2^2 + | \|P_{[k+1,N]}u\|_2^2  - (1-  \overlap_M^{(k)}(\theta)^2)|\\
&\le N^{-10} + 
{LN^{-1/2} 
\le 2LN^{-1/2}}
\end{align*}
where we finally used that $\omega\in E_{1}$. We also write
\[
u^\perp = P_{[2,k]} u + P_{[k+1,N]}u 
\]
so that, since $\omega=V_{[2,N]}^\tran u\in E_{2}$ so that $\|P_{[k+1,N]}u \|_\infty =\|V_{[k+1,N]}\omega''\|_{\infty}\le 
{ L\sqrt{\frac{\log N}N}}$, we have  since $u\in E_{1}$
\[
\|u^\perp\|_\infty  \le \|P_{[2,k]} u\|_2  + \|P_{[k+1,N]}u \|_\infty \le N^{-5}+ 
{ L\sqrt{\frac{\log N}N} 
\le 2L\sqrt{\frac{\log N}N}}\,.
\]
Finally, because $u\in E_1\cap E_{3}$, we find
\begin{align*}
\|\Pi_W u^\perp\|_2 
&\leq \| P_{[2,k]}u\|_2 + \|\Pi_WP_{[k+1,N]}u\|_2 \\
&= \| P_{[2,k]}u\|_2 + \|A V_{[k+1,N]}^\tran u\|_2 \leq N^{-5}+
{L\sqrt{\frac{\dim W}N} 
\le 2L\sqrt{\frac{\dim W}N}}
\end{align*}
and then the right hand side in \eqref{sphere.goal2} satisfies all the  conditions of $\Uloc'_{v_1}(\overlap_M^{(k)}(\theta),{2L},W)$.
\end{proof}

For the proof of Lemma \ref{lem:sphere} we use the following Gaussian approximate representation for the tilted spherical measures $Q^{(\theta,M)} $.

\begin{lemma}
\label{lem:gaussian}
Let $g=(g_i)_{i=2}^N$ be independent centered Gaussians with  
\[
\sigma_i^2 =\E g_i^2:= \frac1{2\max(1, \theta N(\lambda_1-\lambda_i))}.
\]
Assuming  $M=\Lambda$ and so $V=I_{N}$ is the identity matrix, for any Borel set $E\subseteq\R^{N-1}$ we have
\[
Q^{(\theta,\Lambda)}((\R\times E)\cap\sphere^{N-1}) \asymp \frac{ \E [\ind( g\in E\cap \ball^{N-1}) (1-\|g\|_2^2)^{-1/2}]}{ \E[ \ind( g\in  \ball^{N-1}) (1-\|g\|_2^2)^{-1/2}]}\,.
\]
\end{lemma}

\begin{proof}
Under the uniform measure $dP(u)$ on $\sphere^{N-1}$, the marginal density of $\tilde u:=(u_2,\dots, u_N)\in \ball^{N-1}$ relative to Lebesgue measure on $\R^{N-1}$ is proportional to $(1-\|\tilde u\|_2^2)^{-1/2}1_{\|\tilde u\|_2\le 1}$. 
Thus,  noticing that for $u\in \sphere^{N-1}$, $\langle u, \Lambda u\rangle=\lambda_{1}-\sum_{j=2}^{N}(\lambda_{1}-\lambda_{j})u_{j}^{2}$,
we obtain the identity
\begin{equation}	\label{QMtheta.identity}
Q^{(\theta,M)} ((\R\times E)\cap\sphere^{N-1}) 
=\frac{ \int_{-1}^1\cdots\int_{-1}^1 1_{E\cap\ball^{N-1}}(\tilde u) e^{ -\theta N\sum_{j=2}^N (\lambda_1-\lambda_j) u_j^2} \frac{du_2\cdots du_N}{\sqrt{1-\|\tilde u\|_2^2}}}{\int_{-1}^1\cdots\int_{-1}^1 1_{\ball^{N-1}}(\tilde u) e^{-\theta N\sum_{j=2}^N (\lambda_1-\lambda_j) u_j^2}\frac{du_2\cdots du_N}{\sqrt{1-\|\tilde u\|_2^2}}}\,.
\end{equation}
It only remains to note
\begin{align}
\sum_{j=2}^Nu_j^2 \big[\max( 1, \theta N(\lambda_1-\lambda_j)) - \theta N (\lambda_1-\lambda_j)\big] 
&=\sum_{j=2}^N u_j^2\big[1- \theta N(\lambda_1-\lambda_j) \big]1_{\theta N(\lambda_1-\lambda_j)\in [0,1)}\nonumber\\
&\le \sum_{j=2}^N u_j^2\le 1 \label{h.errorGaussian}
\end{align}
and hence
\[
\exp( -\theta N\sum_{j=2}^N (\lambda_1-\lambda_j) u_j^2) 
\asymp \exp( -\sum_{j=2}^N \frac{u_j^2}{2\sigma_j^2}) 
\]
uniformly for $\tilde u\in \ball^{N-1}$, so we can substitute the latter expression in the numerator and denominator of \eqref{QMtheta.identity} to obtain the claim.
\end{proof}

\begin{proof}[Proof of Lemma \ref{lem:sphere}]
We abbreviate $\overlap:=\overlap_M^{(k)}$ throughout the proof. 
By replacing the variable $u$ with 
$V u$
we see that the left hand side of \eqref{bound.sphere} is equal to 
\begin{equation}
Q^{(\theta,\Lambda)} ( \{ u\in \sphere^{N-1}: (\omega',\omega'')\in E_1\cap E_2\cap E_3\})
\end{equation}
where we denote 
by $\omega', \omega''$ the restrictions of $u$ to indices in $[2,k]$ and $[k+1,N]$, respectively.
From Lemma \ref{lem:gaussian}, the above is 
\begin{equation} \label{LB.gaussian}
\gs 
 \frac{ \E \ind( g\in E_1\cap E_2\cap E_{3}\cap \ball^{N-1}) (1-\|g\|_2^2)^{-1/2}}{ \E \ind( g\in  \ball^{N-1}) (1-\|g\|_2^2)^{-1/2}}\,.
\end{equation}
Considering first the numerator above, we note that for $g\in E_1$ we have 
\[
1-\|g\|_2^2 = 1- \|g'\|_2^2 - \|g''\|_2^2 = \overlap^2 + 
{O(LN^{-1/2})} \asymp_{\tcush } 1
\]
for $N$ sufficiently large since $\overlap\gs_{\tcush }1$ when  $\theta\ge \thetam +\tcush$. 
Thus, 
\begin{align}
&\E \ind( g\in E_1\cap E_2\cap {E_3}\cap \ball^{N-1}) (1-\|g\|_2^2)^{-1/2} \label{sphere.LB2}\\
&\gs_{\tcush } \P( g\in E_1\cap E_2 \cap E_{3}) 	\notag\\
&= \P( \|g'\|_2^2<N^{-10}) \times \P\bigg(| \|g''\|_2^2 - (1-\overlap^2)|<
LN^{-1/2} ,\notag\\
 &\qquad\qquad\qquad\qquad \|V_{[k+1,N]}g''\|_\infty \le L\sqrt{\frac{\log N}N}\,, \;
\|A g''\|_2\le L\|A\|_\HS N^{-1/2}\bigg).	\notag
\end{align}
For the first factor in the last line above, we can bound
\[
\P(\|g'\|_2^2 < N^{-10}) \ge \P( |g_2|,\dots, |g_k| < N^{-5}/\sqrt{k}) 
\ge \P(|g_2|< N^{-5}/\sqrt{k})^{k-1} = N^{-O(k)}
\]
since $g_2$ has standard deviation 
\[
\sigma_2 = (2\max(1,\theta N(\lambda_1-\lambda_2)))^{-1/2} \le 2^{-1/2}=O(1).
\]
For the second factor, first note that since $\lambda_1-\lambda_{k+1}\ge \delta_0$, we have 
\begin{align*}
\E \|g''\|_2^2 = \frac1{2\theta N} \sum_{i=k+1}^N \frac1{\lambda_1-\lambda_i} = 1-\overlap^2
\end{align*}
for all $N$ sufficiently large.  Furthermore, for $i\ge k+1$,
\[
\E (g_i^2 - \sigma_i^2)^2 \ls \E g_i^4 \ls \sigma_i^4 \ls \frac1{\theta^2\delta_0^2N^2}\,.
\]
Thus,
\[
\E(\|g''\|_2^2 - (1-\overlap^2))^2 \ls \frac1{\theta^2 \delta_0^2N}
\]
and by Markov's inequality
\begin{equation}	\label{sphere.Markov}
\P( |\|g''\|_2^2 - (1-\overlap^2)|>
LN^{-1/2}) \ls_{\delta_0,\tcush } 
L^{-2}\,.
\end{equation}
Similarly, 
\[
{
\E[\|A g''\|_2^2]=\sum_{i=k+1}^n(A^\tran A)_{ii}\E[g_{i}^{2}]\le \frac{1}{N\theta\delta_{0}} \|A\|_\HS^2 
}
\]
so that Tchebychev's inequality implies
\begin{equation}\label{newb}
\P(\|Ag''\|_2\ge 
L\|A\|_\HS N^{-1/2}) 
\ls_{\delta_0,\tcush } L^{-2}\,.
\end{equation}
 Finally, noting that for each $i\in [N]$, $(V_{[k+1,N]}g'')_i$ is a centered Gaussian of variance 
\[
\sum_{j=k+1}^N \sigma_j^2 v_j(i)^2 \ls \frac{1}{\theta \delta_0^2N} \sum_{j=k+1}^N v_j(i)^2 \ls_{\delta_0,\tcush }N^{-1}
\]
we have for any $L>0$ that
\begin{equation}	\label{Ltail3}
\P\Big( \|V_{[k+1,N]}g''\|_\infty > L\sqrt{\tfrac{\log N}N}\,\Big) \le \sum_{i=1}^N \P\Big( |(V_{[k+1,N]}g'')_i|> L\sqrt{\tfrac{\log N}N}\,\Big) 
\le N \exp( -c_{\delta_0,\tcush }L^2\log N)
\end{equation}
for some $c_{\delta_0,\tcush }>0$ depending only on $\delta_0,\tcush $. 
We can fix $L$ as a sufficiently large constant depending on $\delta_0,\tcush $ to make the probabilities in \eqref{sphere.Markov}, \eqref{newb} and \eqref{Ltail3} each smaller than $\frac1{10}$. From the union bound we have that the second factor in \eqref{sphere.LB2} is at least $\frac7{10}$, and hence the numerator in \eqref{LB.gaussian} is $\gs_{\delta_0,\tcush } N^{-O(k)}$.
Turning to the denominator in \eqref{LB.gaussian}, it suffices to show that for any fixed realization of $g''$ with $\beta:=1-\|g''\|_2^2 >0$ we have
\begin{equation}	\label{sphere.goal}
\E_{g'} \ind( \|g'\|_2^2 < \beta) \frac1{\sqrt{\beta-\|g'\|_2^2}} \le N^{O(k)}
\end{equation}
uniformly in $\beta\in (0,1)$. 
To that end, noting that the density of $g'$ on $\R^{k-1}$ is bounded by
\[
\ls \prod_{j=2}^k \sqrt{\theta N(\lambda_1-\lambda_j)} \ls  N^{O(k)}
\]
we see that the left hand side of \eqref{sphere.goal} is at most $O(N^{O(k)})$ times the Lebesgue integral
\begin{align*}
&\int_{\R^{k-1}} (\beta-x_1^2-\cdots-x_{k-1}^2)_+^{-1/2}dx_1\cdots dx_{k-1} \\
&\quad= \beta^{(k-2)/2} 
\int_{\R^{k-1}} (1-x_1^2-\cdots-x_{k-1}^2)_+^{-1/2}dx_1\cdots dx_{k-1} \\
&\quad\ls \beta^{(k-2)/2} =O(1)
\end{align*}
since $k\ge2$, where in the first equality we rescaled the variables of integration by $\beta$. 
We thus obtain \eqref{sphere.goal} and hence the claim. 
\end{proof}

\section{Ruling out fully localized eigenvectors}
\label{sec:localized}

In Section \ref{sec:monotone} we prove Lemma \ref{lem:monotone} on the monotonicity of the large deviation rate for $\lam_1$ using a Markov chain argument.
In Section \ref{sec:no-comp} we use Lemma \ref{lem:monotone} prove Proposition \ref{prop:no-comp} showing the event that $v_1$ is fully localized is negligible on the large deviation scale.

\subsection{Monotonicity of the rate function}
\label{sec:monotone}

Lemma \ref{lem:monotone} is a quick consequence of the following, lemma, in which we construct a discrete-time Markov chain $(H_n)_{n\ge0}$ on $\Sym_N$ with stationary distribution equal to the distribution of $H$, such that if  the process starts with $\lam_1(H_0) \approx   x$ for some $x>2$, then 
after time $N^{O(1)}$,  $\lam_1(H_n)$ is likely to be near the typical value $2+o(1)$. Because $H_{n}$ will be designed to take small steps, for $2<y<x$ this will ensure that $\lambda_{1}(H_{n})$ will be close to $y$ at some intermediate time, allowing us to compare the probabilities that $\lambda_{1}(H)$ is close to $x$ or $y$.

\begin{lemma}	\label{lem:goodMC}
Assume \eqref{assu:usg} holds, and that $N$ is sufficiently large depending on $\mu$.
Let $\delta\in [N^{-1/3}, 1]$.
There exists a sequence $(H_n)_{n\ge0}$ of random elements of $\Sym_N$ such that 
\begin{enumerate}
\item $H_n\eqd H$ for all $n$;
\item $\|H_{n+1}-H_n\|< 2\delta$ a.s. for all $n$;
\item 
For every $x\in[2+\delta,N]$,
\begin{equation}	\label{bd:goodMC3}
\P\Big( \lam_1(H_{N^4})\le 2+\delta \,\Big|\, |\lam_1(H_0)-x|\le\delta, \|H_0\|\le N \Big) \ge \frac12. 
\end{equation}
\end{enumerate}
\end{lemma}

\begin{proof}[Proof of Lemma \ref{lem:monotone}]
Write $\Evalx_x(\delta)=\{ M\in \Sym_N: |\lam_1(M)-x|\le \delta\}$ and recall $x>y$. 
\begin{align*}
\P( H_0\in \Evalx_x(\delta), \|H_0\|\le N ) 
&= \frac{\P( \lam_1(H_{N^4})\le 2+ \delta\,,\, H_0\in \Evalx_x(\delta), \|H_0\|\le N )}{\P(\lam_1(H_{N^4})\le 2+ \delta| H_0\in \Evalx_x(\delta), \|H_0\|\le N )}\\
& \le 2 \P( \lam_1(H_{N^4})\le 2+ \delta\,,\, H_0\in \Evalx_x(\delta), \|H_0\|\le N )\\
&\le 2  \P(\exists n\in [1, N^{4}] : \lam_1(H_n)\in \Evalx_y(\delta)\,,\, H_0\in \Evalx_x(\delta), \|H_0\|\le N )\\
&\le 2 \sum_{n=1}^{N^4} \P( \lam_1(H_n)\in \Evalx_y(\delta))\\
&= 2N^4\P( \lam_1(H_0)\in \Evalx_y(\delta))
\end{align*}
where in the 
second, third and final lines we applied properties (3), (2) and (1) from Lemma \ref{lem:goodMC}, respectively.
Taking logs and dividing by $N$ on both sides completes the proof. 
\end{proof}

For the proof of Lemma \ref{lem:goodMC} we need the following regularity estimate on the tails of $X$ under the \eqref{assu:usg} condition.

\begin{lemma}
\label{lem:stepdown} Let $\mu$ be a probability measure on the real line satisfying  \eqref{assu:usg} condition so that 
 $K_0:=\|\LLa_\mu''\|_\infty<\infty$. 
For any $r>0$ and all $y>0$ sufficiently large depending on $r,\mu$, we have
\begin{equation}	\label{stepdown.goal1}
\mu([ y + 2K_0 +r,+\infty))
 \le \frac1{10} \mu([y-2K_0,y+2K_{0}]).
\end{equation}
\end{lemma}
\begin{proof}
Fix $r>0$ arbitrary.
For $t\in\R$ write $\mu^t,\E^t$ for the tilted probability and expectation given by $\mu^t(E)= \int  e^{tX-\LLamu(t)} \ind(E)(X) d\mu(X)$. 
Then $\E^t X= \LLa_\mu'(t)$.
With $a\in[-\infty,0), b\in (0,+\infty]$ the left and right ends of the support of $\mu$, we have that $\LLa_\mu'$ is strictly increasing on $\R$ with range $(a,b)$. 
If $b<\infty$ then the left hand side of \eqref{stepdown.goal1} is zero for any $y>b$, so we may assume $b=+\infty$.
Then for any $y>0$ there exists $t(y)\in(0,\infty)$ such that $\LLa_\mu'(t(y))=y$. 
Since  $\E^t[(X-\E^t X)^2] = \LLa_\mu''(t) \le K_0$ for all $t\in \R$, from Tchebychev's inequality we get
\begin{equation}	\label{stepdown.Cheb}
\mu^{t(y)}( |X-y| \le 2K_0) \ge \frac34\qquad \forall y\in \R.
\end{equation}
On the other hand, for any $y>0$,
\[
\mu^{t(y)}(|X-y|\le 2K_0) \le e^{t(y)(y+2K_0)-\LLa_\mu(t(y))}\P(|X-y|\le 2K_0)
\]
and combining with \eqref{stepdown.Cheb} we have
\[
e^{\LLa_\mu(t(y))} \le \frac43 e^{t(y)(y+ 2K_0)} \P(|X-y|\le 2K_0).
\]
Hence, applying Markov's inequality followed by the above bound, we conclude
\begin{equation}	\label{stepdown.final}
\mu(X\ge y + 2K_0+r) \le e^{\LLa_\mu(t(y)) - (y+ 2K_0 + r)t(y)} \le \frac43 e^{-rt(y)} \P(|X-y|\le 2K_0).
\end{equation}
Since $b=+\infty$ we may take $y$ sufficiently large so that $rt(y)\ge100$, and \eqref{stepdown.goal1} follows.
\end{proof}


\begin{proof}[Proof of Lemma \ref{lem:goodMC}]
We assume throughout that $N$ is sufficiently large depending on $\mu$ without further comment.
Let $L=L_N\in (N^{1/10},\frac12\delta\sqrt{N}]$,
and partition $\R$ into intervals $I_k$ of length $L$, with
 $I_0:=(-\frac12L, \frac12L)$,  $I_k:=L [k-\frac12,k+\frac12)$ for $k\ge1$, and  $I_{-k}:=-I_k$. 
Denote $\pi_k:=\P(X\in I_k)$, and let $Q$ denote the Markov transition matrix on $\Z$ with entries
\begin{equation}
\label{def:Qtrans}
Q_{k,\ell} =\begin{cases} 
 \frac12 \min\big\{ \frac{\pi_\ell}{\pi_k} ,1\big\} & |k-\ell|=1\,\\ 
1-Q_{k,k+1}-Q_{k,k-1} & k=\ell,\\
0 & \text{otherwise.}
\end{cases}
\end{equation}
This transition matrix is reversible with stationary distribution $\pi=(\pi_k)_{k\in \Z}$. 
Moreover, 
\[
\pi_{\pm 1} \le \P( |X|\ge \tfrac12L) \le 2\exp( -c L^2)<\tfrac12<\pi_0,
\]
and for $k\ge1$, 
\[
\pi_{k+1} \le \P( |X|\ge (k+\tfrac12)L) \le \P( |X|\ge  kL + 4K_0) 
\le \frac1{10} \P( |X- kL|\le 2K_0)
\le \frac{\pi_k}{10} 
\]
where we applied Lemma \ref{lem:stepdown} with $r=2K_0$ and $y=kL\gg1$.
We similarly obtain $\pi_{k-1}\le \frac{\pi_k}{10}$ for all $k\le -1$.
Thus, the first case in \eqref{def:Qtrans} becomes
\begin{equation}
\label{chain.drift}
Q_{k,\ell} =\begin{cases} 
\frac12 & |k-\ell|=1 \text{ and } |\ell|< |k|,\\ 
\frac{\pi_\ell}{2\pi_k}\le \frac1{20}& |k-\ell|=1\text{ and } |\ell|> |k|,\\
\frac{\pi_\ell}{2\pi_k}\le \exp(-cN^{1/5}) & k=0, \ell =\pm1.
\end{cases}
\end{equation}
We easily conclude that with high probability, from any starting position $m=N^{O(1)}$, the chain reaches state 0 in time $N^{O(1)}$ and stays there for time $\ge\exp(N^{1/10})$. Indeed, 
 if $\P_m$ is a probability measure under which $(Z_n)_{n\ge0}$ is a Markov chain with transition matrix $Q$ and $Z_0=m$ a.s., and $T_0:=\inf\{ n\ge0: Z_n=0\}$, then $\E_mT_0 = O(m)$ for any $m\in\Z$ (with implicit constant depending only on $\mu$), and hence
\begin{equation}
\P_m(T_0\ge N^3m)\ls N^{-3}. 
\end{equation}
Moreover, from the strong Markov property, the third case in \eqref{chain.drift} and the union bound, we have 
\begin{equation}	\label{chain.go-stay}
\P_m( 
Z_{N^4}=0) =1-O(N^{-3}) \qquad \forall m\in[-N,N]\cap\Z.
\end{equation}

Now define a new Markov chain $(X_n)_{n\ge0}$ taking values in $\R$ that is coupled to $(Z_n)_{n\ge0}$ as follows. 
For each $k\in\Z$, $n\ge0$, on the event that $Z_n=k$ let $X_n$ be sampled from the law of $X$ conditioned on the event $\{X\in I_k\}$, independently of $X_0,\dots,X_{n-1}$. 
Since $\pi$ is stationary for $Q$ it follows that $\mu$ is a stationary distribution for $(X_n)_{n\ge0}$.

Finally, we can define the process $(H_n)_{n\ge0}$ as a Markov chain 
 as follows.
Let $(e_m)_{m\ge1}$ be iid uniform samples from $E:=\{(i,j): 1\le i\le j\le N\}$ and for each $e\in E$ let $A^e_n:=\sum_{m=1}^n \ind( e_m=e)$. 
Let $(X_m^e)_{n\ge0}, e\in E$ be iid copies of the process $(X_m)_{m\ge0}$ constructed above, with $X_0^e\sim \mu$ for all $e\in E$, and set $X_m^{(i,j)}:= X_m^{(j,i)}$ for $1\le j<i\le N$.
For each $n\ge0$ let $H_n$ have entries $(2^{1_{i=j}/2} X^{(i,j)}_{A^{(i,j)}_n})_{1\le i,j\le N}$.
Thus, at each time $n$ we sample a random entry $(i,j)$ on or above the diagonal and update the entry (with appropriate scaling by $\sqrt{2^{1_{i=j}}/N}$) according to the next step for the chain $(X_m^{(i,j)})_{m\ge0}$. 
Clearly $H_0\eqd H$, and since $\mu$ is stationary for $(X_m)_{m\ge0}$ we have $H_n\eqd H$ for all $n$, which gives the first property.

Since at each time only a single entry of $H_n$ is modified by at most $\delta/\sqrt{2}<\delta$, 
the second property follows.

From \eqref{chain.go-stay} we have that conditional on the event $\{|\lam_1(H_0)-x|\le \delta, \|H_0\|\le N\}$, the event $\event_0':=\{ H_{N^4}(i,j)\in \sqrt{2^{1_{i=j}}/N} I_0\,\forall i,j\in[N]\}$ 
holds with probability at least $1-O(N^{-1})\ge \frac9{10}$ (note that the starting interval $I_m$ varies from entry to entry, but the bound on the norm of $H_0$ ensures a uniform bound $|m|\le N/L\le N$). 
Moreover, conditional on  $\event_0'$, the entries of $H$ are still uniformly sub-Gaussian, and independent up to the symmetry condition and have the law of $N^{-1/2}X$ conditioned to lie in $\sqrt{2^{1_{i=j}}/N} I_0$. 
The conditioning only modifies the means and variances of the entries by factors $1+O(e^{-cN^{1/5}})$, so from Corollary \ref{cor:lam1-tilt.conc} we have $\P( \lam_1(H)<2+\delta| \event_0') \ge \frac{9}{10}$, and the third property follows.
\end{proof}

\subsection{Proof of Proposition \ref{prop:no-comp}}
\label{sec:no-comp}

We will assume without comment that $N$ is sufficiently large. 
By monotonicity it suffices to establish the claim with $s:=N/\log N$.
Recall the definition \eqref{def:AS} of $\Comp_N(s, \eps)$. We claim it suffices to show for arbitrary $2\le x<K<\infty$, $\delta\in(e^{-\sqrt{N}}, c_1/K)$ and $\eps\in(0,c_1/K^2)$ that
\begin{align}
&\P\Big(|\lam_1(H)-x|\le \delta \,,\; \|H\|\le K \,,\; v_1\in \Comp_N(s, \eps)\Big) \notag\\
&\qquad\qquad\qquad\qquad\qquad  
\le  e^{-2c_0N} \P\big(|\lam_1(H)-x|\le\delta\big)  + \exp(-N^{1.1}) 
\label{bd:no-comp}
\end{align}
for all $N$ sufficiently large depending on $K$ and $\mu$. 
Indeed, assuming the preceding statement holds, take $K=C_0L$ for a constant $C_0>0$ to be chosen sufficiently large depending on $\mu$. 
Let $I_k\subset I$ be a collection of disjoint intervals of length $\delta\in [\frac12N^{-1/4},N^{-1/4}]$ covering $I$. 
Applying the union bound and \eqref{bd:no-comp}, we have
\begin{align}
\P( v_1\in\Comp(s,\eps) | \lam_1\in I) 
& \le \frac{\P( \|H\|>K)}{\P(\lam_1\in I)} + \sum_k \frac{\P( \lam_1\in I_k, \|H\|\le K, v_1\in \Comp(s,\eps))}{\P(\lam_1\in I)}		\notag\\
&\le \frac{\P( \|H\|>K) + O(LN^{1/4} e^{-N^{1.1}})}{\P(\lam_1\in I)}
+ O(LN^{1/4} e^{-2c_0N})\,.		\label{no-comp1}
\end{align}
Letting $y$ be the midpoint of the right-most interval $I_k$ covering $I$, we can apply Lemma \ref{lem:monotone} followed by Lemma \ref{lem:lower1} and Lemma \ref{lem:Jprops}(\ref{Jprop.smallz}) (with $\eta=1/8$ and $R=N^{1/8}$, say) to lower bound
\begin{align*}
\P( \lam_1\in I) 
&\ge \P( |\lam_1-y|\le \delta)\\
&\ge N^{-O(1)}\P( |\lam_1-L|\le \delta, \|H\|\le N)\\
&\ge N^{-O(1)} (\P(|\lam_1-L|\le \delta) - e^{-cN^3})\\
&\ge N^{-O(1)} \Big( \exp( -N\rate^\gamma(L) + N^{1-c}) - e^{-cN^3}\Big)\\
&\ge \exp( -c L^2N).
\end{align*}
Substituting this bound in \eqref{no-comp1}, along with the upper bound $\P(\|H\|>K) \le \exp( -cK^2N)=\exp( -cC_0^2L^2N)$ from Lemma \ref{lem:tightness}, we obtain the desired bound \eqref{bd:no-comp0} by taking $C_0$ sufficiently large.

We turn to prove \eqref{bd:no-comp}.
Let $\cN_\eps$ be a $2\eps$-net for $\Comp(N/\log N, \eps)$ under the $\ell^2$ norm consisting of $N/\log N$-sparse vectors $w\in\sphereN$.
By standard volumetric considerations we  can take $\cN_\eps$ of size
\[
|\cN_\eps| \le {N\choose \lf N/\log N\rf} O(1/\eps)^{N/\log N} = O(\eps^{-1}\log N)^{N/\log N}. 
\]
We apply the union bound over $\cN_\eps$ to fix a $2\eps$-approximation $w\in\cN_\eps$ for $v_1$, thus bounding the left hand side of \eqref{bd:no-comp} by
\begin{equation}\label{union.Neps}
O(\eps^{-1}\log N)^{N/\log N} \max_{w\in\cN_\eps} \P\big( |\lam_1(H)-x|\le \delta \,,\; \|v_1-w\|_2\le 2\eps\,,\; \|H\|\le K \big) \,.
\end{equation}

Fix now an arbitrary $w\in \sphere^{N-1}$ with $|\supp(w)|\le n:= N/\log N$. For ease of notation we take $\supp(w)=[n]$.
We will often abusively treat $w$ as an element of $\R^n$.
From the bipartition $[n]\cup[n+1,N]$ of coordinates we have the block decomposition
\[
H = \begin{pmatrix} A & B^\tran \\ B & D\end{pmatrix}. 
\]
Denote events
\begin{align*}
\event_{x,w} &= \{ |\lam_1-x|\le \delta, \|v_1-w\|_2\le 2\eps\}\\
\good &= \{ \|H\|\le K\}\\
\sfA &= \{ \langle w, Aw\rangle \ge x-\frac1{100x}\}\\
\sfB &= \{ \|Bw\|_2 \le \frac1{100x}\}\\
\sfB'&= \{ \|Bw\|_2 \ge 1- \frac1{100x}\}\\
\sfD & = \{ |\langle Bw, DBw\rangle |\le C/\sqrt{N}\}.
\end{align*}
From the eigenvalue equation $Hv_1=v_1\lam_1$, on the event $\event_{x,w}\cap\good$ we have
\begin{align*}
xw &= Aw + O(\delta + K\eps)\\
\|Bw\|_2&= O(K\eps)
\end{align*}
and hence 
\begin{equation}
\event_{x,w}\cap \good \subset \sfA \cap \sfB
\end{equation}
taking  $c_1$ is sufficiently small. 
Moreover, Since 
\[
\E \|Bw\|_2^2 = (N-n)/N \ge 1-\frac1{\log N}
\]
and the entries of $\sqrt{N}Bw$ are sub-Gaussian, from the Chernoff bound for the sum of independent sub-exponential variables we have
\begin{equation}
\P ( \sfB) \le \P( \|Bw\|_2^2 < \tfrac12) \le  \exp( -c_0 N).
\end{equation}
Since $\sfA, \sfB$ are independent, 
\begin{equation}
\label{comp.1}
\frac1N\log \P( H\in \event_{x,w} \cap \good) \le  -c_0+ \frac1N\log\P( \sfA) .
\end{equation}

On the other hand, we have $\P(\sfB') \ge \frac{9}{10}$ for all $N$ sufficiently large depending on $K$. 
Writing $z:= Bw/\|Bw\|_2\in \sphere^{N-n-1}$, which is independent of $D$, we have $\E ( \langle z,Dz\rangle| B) = 0$ and $\E (\langle z,Dz\rangle^2|B) \ls N^{-1}$, 
and hence from Tchebychev's inequality,
\[
\P( \sfD| B) \ind(\sfB') \ge \frac{9}{10} \ind(\sfB')
\]
if the constant $C$ in the definition of $\sfD$ is sufficiently large. 
Thus,
\begin{equation}
\label{AtoA'}
\frac1N\log\P(\sfA) \le \frac1N\log \P(\sfA') + O(N^{-1})\,,\qquad \sfA':= \sfA \cap \sfB'\cap\sfD.
\end{equation}
On the event $\sfA'$, let 
\[
u := \begin{pmatrix} \sqrt{1-\theta^2} w \\ \theta \frac{Bw}{\|Bw\|_2} \end{pmatrix} \in \sphere^{N-1}, \qquad \theta:= \frac1{10x}\,.
\] 
We have
\begin{align*}
\langle u, Hu\rangle
&= (1-\theta^2)\langle w,Aw\rangle + 2\theta\sqrt{1-\theta^2} \|Bw\|_2^{-1} \langle w, B^\tran Bw\rangle + \theta^2 \|Bw\|_2^{-2} \langle Bw, DBw\rangle\\
&= (1-\theta^2) \langle w, Aw\rangle + 2\theta\sqrt{1-\theta^2} \|Bw\|_2 + O(N^{-1/2})\\
&\ge x+ \frac1{20x}.
\end{align*}
Hence, letting $\{x_k\}$ be a $\delta$-mesh for the interval $[x+(20x)^{-1}, N^{1/10}]$, we have
\begin{align*}
\P( \sfA') &\le \P( \lam_1(H) \ge x+ (20x)^{-1}) \\
&\le \P(  \lam_1(H) \ge x+ (20x)^{-1}, \|H\|\le N^{1/10}) + \P( \|H\|>N^{1/10})\\
&\le \sum_k \P( \lam_1(H)\in [x_k -\delta, x_k+\delta], \|H\|\le N^{1/10}) + \exp( -cN^{1.2})\\
&\le N^{1/5}\delta^{-1} N^{O(1)} \P( |\lam_1(H) - x | \le \delta) + \exp( -N^{1.1})\\
&\le N^{O(1)}e^{\sqrt{N}} \P(|\lam_1(H)-x|\le \delta) + \exp( -N^{1.1})
\end{align*}
where in the penultimate line we applied Lemma \ref{lem:monotone}. 
Combining with \eqref{comp.1} and \eqref{AtoA'}, we get
\[
\frac1N\log\P(H\in \event_{x,w}\cap\good) \le -c_0 + \max\Big\{ \frac1N\log\P( |\lam_1(H)-x|\le \delta) + O(N^{-1/2})\,,\; -N^{1/10} \Big\} + O( N^{-1})
\]
which together with \eqref{union.Neps} (replacing $c_0$ with $c_0/2$ to absorb the errors of size $o(1)$) gives \eqref{bd:no-comp} to complete the proof. 
\qed

\section{Large deviation lower bound }
\label{sec:lower}

In this section we establish the following proposition, which quickly yields the large deviation lower bounds of Lemmas \ref{lem:lower1} and \ref{lem:lower2}. The approach is by a nested tilting argument as described in Section \ref{sec:ideas}. 
Along the way we establish some key results towards the proof of Proposition \ref{prop:FE}.
The proofs are based on Propositions \ref{prop:tiltP} and \ref{prop:tiltQ} summarizing key properties of the tilted measures, which are proved in Sections \ref{sec:tiltP}--\ref{sec:annealed}. 

{
\begin{prop}
\label{prop:lower0}
Assume \eqref{assu:usg}.
For any 
$L\ge10$ and $\xcush,\zcush\in(0,\frac1{10})$
there exist $T=T(L,\zcush)\ge10$ and $\tcush_0,\eps_0\in(0,\frac1{10})$  depending only on 
$L,\xcush,\zcush$
such that the following holds. 
Let $\eta\in(\heta,\frac14-\heta)$,    $\srad\in(0,\eps_0)$,  $\log N\le R\le  N^{1/4}/\log N$, $x\in[2+\xcush,L]$ and $\chz\in(1-\zcush)\B^{n_0}$ with $n_{0}=N^{1-2\eta}$.
With $B:=[\thetam+\tcush_0,T]$, let 
\begin{equation}\label{zcurve}
B\ni\theta\mapsto \tvloc(\theta)\in\B^{n_0}
\end{equation}
be a continuous curve such that for some $\delta\in(0,1]$ and all $\theta\in B$,
\begin{enumerate}
\item $\tvloc(\theta)$ and $\chz$ have disjoint supports;
\item $\|\tvloc(\theta)\|_2^2\le (1-\zcush)^2-\|\chz\|_2^2$;
\item 
$\|\tvloc (\theta)\|_\infty\le \delta/R$.
\end{enumerate}
Then with $w_x(\theta):=\overlap_x(\theta)(\chz+\tvloc(\theta))$, 
there exists $\theta^*\in B$ such that 
\begin{align}
\label{lower:goal1}
&\frac1N\log\P(|\lambda_1-x|\le {N^{-1/20}}) \\
&\quad \notag\ge F_N(\theta^*;\Uloc_{w_x(\theta^*)}(\srad,R)) - J(x,\theta^*) - {O_{L,\xcush,\zcush}(\srad+\delta + 
N^{-c\heta} + R^2N^{-1/2})}
\end{align}
for all $N$ sufficiently large 
depending on 
$L,\xcush,\zcush,\heta$ and $\mu$ and a universal constant $c>0$.
\end{prop}
}

Note that lower bound in Theorem \ref{theomain} for $x>2$ follows immediately from Proposition \ref{prop:lower0} under the case that 
$\tvloc\equiv0$. 
We further note that if $\chz=0$ then the assumption \eqref{assu:usg} can be dropped, as it is only needed for part (b) of Proposition \ref{prop:tiltP} -- see  
Remark \ref{rmk:dropUSG.tiltP}.

\begin{proof}[Proof of Lemma \ref{lem:lower1}]
Applying Proposition \ref{prop:lower0} with $\vloc=\chz$, $\tvloc\equiv 0$, $L=\xcush^{-1}$ and $\srad=N^{-2}$, say, there exist $T\ge10,\tcush_0\in (0,\frac1{10})$ depending only on $\xcush,\zcush$ and $\theta^*\in [\thetam+\tcush_0,T]$ such that
\begin{equation*}
\frac1N\log\P(|\lam_1-x|\le N^{-1/20}) \ge F_N\big( \theta^*; \Uloc_{\overlap_x(\theta^*)\vloc}(\srad,R)\big) - J(x,\theta^*) + O_{\xcush,\zcush}( N^{-c\heta} + R^2N^{-1/2})
\end{equation*}
for all $N$ sufficiently large. Then from Proposition \ref{prop:FE}, the right hand side is equal to
\[
\free_{N,R}(\theta^*,\overlap_x(\theta^*)\vloc) - J(x,\theta^*) + O_{\xcush,\zcush}(N^{-c\heta} + R^2 N^{-1/2})
\]
for all $N$ sufficiently large. The claim follows.
\end{proof}

\begin{proof}[Proof of Lemma \ref{lem:lower2}]
We apply Proposition \ref{prop:lower0} with $L=\xcush^{-1},R=N^{1/5}, \srad=N^{-2},\eta=\frac18$ (say), and let $T_0=T(\xcush^{-1},\zcush)$, $\tcush_0=\tcush_0(\xcush^{-1},\xcush,\zcush)$ be as provided by Proposition \ref{prop:lower0}. Denote $B:=[\thetam+\tcush_0,T_0]$. 
We would like to take $\tvloc(\theta)$ in Proposition \ref{prop:lower0} to be $\overlap_x(\theta)^{-1} w^*(\theta,\overlap_x(\theta)^2\tal, t)$ with $w^{*}$ as in \eqref{def:w*}, but the latter has jump discontinuities of $\ell^2$-norm  $O(t^{1/2}(\thetam)^{-1/2}N^{-1/4}) = O_{t,\xcush}(N^{-1/4})$ at $\theta_k=\thetam+ak$ for integer $k$ and $a=\frac12tN^{-1/2}$, so we fix any continuous $\tvloc:B\to\B^{n_0}$ with
\begin{equation}	\label{lower2a}
\sup_{\theta\in B}\big\| \tvloc(\theta) - \overlap_x(\theta)^{-1} w^*(\theta,\overlap_x(\theta)^2\tal, t) \big\|_2 \le C_0N^{-1/4}
\end{equation}
for some $C_0=C_0(t,\xcush)$ sufficiently large.
Since $\chz\in \B_{\ge N^{-\eps}}$ we verify from Tchebychev's inequality that $|\supp(\chz)| \le N^{2\eps} \le N^{1-2\eta}$. Since $w^*$ is supported on $[N-n_1+1,N]\subset [N+1,N]$ for all $N$ sufficiently large, $\chz$ and $w^*$ have disjoint supports.
By invariance of $\free_{N,R}(\theta,w)$ under permutations of the coordinates of $w$ we may take $\chz$ to be supported on the first $n_0=N^{1-2\eta}$ coordinates. Since 
\[
\|\tvloc(\theta)\|_2^2= \overlap_x(\theta)^{-2}\|w^*\|_2^2 \le \tal \le 1-\zcush-\|\chz\|_2^2
\]
by \eqref{w*norm} and our hypotheses, we verify condition (2) of Proposition \ref{prop:lower0}. 
Since $\theta\overlap_x(\theta)^2 = \theta-\thetam\ge \tcush_0\gs_{\xcush,\zcush}1$ for $\theta\in[\thetam+\tcush_0,T]$, we have (from \eqref{def:w*} and \eqref{lower2a})
\[
\|\tvloc(\theta)\|_\infty = \overlap_x(\theta)^{-1} \sqrt{\frac{t}{2\theta}} N^{-1/4} +O_{\xcush,t}(N^{-1/4})\le CN^{-1/4} = CN^{-1/20}/R
\]
for all $\theta\in B$, for some $C=C(\xcush,\zcush,t)<\infty$. 
Hence condition (3) of Proposition \ref{prop:lower0} holds with $\delta=CN^{-1/20}$ (which is less than 1 for all $N$ sufficiently large). 
From Proposition \ref{prop:lower0} it thus follows that there exists $\theta^*\in B$ such that
\begin{equation}
\frac1N\log\P(|\lam_1-x|\le N^{-1/20}) \ge F_N\big( \theta^*; \Uloc_{\overlap_x(\theta^*)(\chz+\tvloc(\theta^*))}(N^{-2},N^{1/5})\big) - J(x,\theta^*) + O_{\xcush,\zcush,t}( N^{-c} )
\end{equation}
for all $N$ sufficiently large.
Then from Proposition \ref{prop:FE}, the right hand side is equal to
\[
\free_{N,N^{1/5}}\big(\theta^*,\overlap_x(\theta^*)(\chz+\tvloc(\theta^*))\big) - J(x,\theta^*) + O_{\xcush,\zcush,t}(N^{-c} )
\]
for all $N$ sufficiently large. Finally, we can replace $\overlap_x(\theta^*)(\chz+\tvloc(\theta^*))$ with $w(\theta^*)$ up to an additive error $O_{\xcush,\zcush,t}(N^{-1/4})$ using \eqref{lower2a} and \eqref{freecont.w} in Lemma \ref{lem:freeprops}. The claim follows.
\end{proof}

In the remainder of this section we prove Proposition \ref{prop:lower0} using the next two propositions, which  concern the tilted probability measures $\P^{(\theta,u)}$ from \eqref{def:tiltP} and tilted laws on the sphere naturally associated to the restricted free energies $F_N(\theta,\Uloc_w)$.
We prove the following in Section \ref{sec:tiltP}.

\begin{prop}[Behavior of $\lambda_1$ under $\P^{(\theta,u)}$]
\label{prop:tiltP}
Assume \eqref{assu:usg} holds.
\begin{enumeratea}
\item \label{tiltP:conc}(Concentration).
For any $\theta\ge0$, $u\in\sphereN$ and $s\ge0$,
\begin{equation}
\P^{(\theta,u)}( |\lam_1(H)- \E^{(\theta,u)}\lam_1(H)|\ge s)
\le 2\exp( -cs^2N/\log N)
\end{equation}
for a constant $c>0$ depending only on $\|\LLa_\mu''\|_\infty$.
The same bound holds (up to modification of $c$) with $\E^{(\theta,u)}\lam_1(H)$ replaced by any median of $\lam_1(H)$ under $\P^{(\theta,u)}$.

\item\label{tiltP:cont} (Continuity). 
For any $\theta,\phi\ge0$ and $u,v\in\sphereN$, with $\Delta(x):=\max(x,x^{1/4})$ for $x\ge0$, we have
\begin{equation}
\big| \E^{(\theta,u)} \lambda_1(H) - \E^{(\phi,v)} \lambda_1(H)\big| \ls  \Delta(|\theta-\phi|) + \Delta(\sqrt{\theta\phi} d_2(u,v)).
\end{equation}
(Recall the Wasserstein distance $d_2$ from \eqref{d2W2}.)
\item\label{tiltP:small} (Small $\theta$).
For any fixed 
$\theta\le\frac12$ and $\kappa>0$, if $1\le R\le o(N^{1/4})$, then
\begin{equation}
\sup_{u\in 
\Deloc_R
} \E^{(\theta,u)} \lambda_1(H) \le 2+\kappa
\end{equation}
for all $N$ sufficiently large. 

\item\label{tiltP:large} (Large $\theta$). 
For any $\theta\ge1$, if $u\in \sphereN$ satisfies
\begin{equation}	\label{cond:u.med}
\sum_{i=1}^N u_i^2 1_{|u_i| \le L\theta^{-1/2}N^{-1/4}} \ge \beta_0
\end{equation}
for some $\beta_0>0$ and $L\in[0,+\infty]$, then with $m_\mu(L):= \min_{t\in[-2L^2,2L^2]}\psimu(t)$, we have
\begin{equation}	\label{Etilt-lambda.bd1}
\E^{(\theta,u)} \lambda_1(H) \gs m_\mu(L) \beta_0^2\theta 
\end{equation}
for all $N$ sufficiently large depending on $\beta_0$ and $m_\mu(L)$. 
\end{enumeratea}
\end{prop}

\begin{remark}
\label{rmk:dropUSG.tiltP}
The assumption \eqref{assu:usg} is only needed for \eqref{tiltP:cont}, and one easily verifies the assumption can be dropped if $\|u\|_\infty,\|v\|_\infty=O(N^{-1/4})$.
\end{remark}

Recall notation \eqref{def:tiltQ}--\eqref{def:tiltQAB} and \eqref{def:UwrR}.
We prove the following in Section \ref{sec:annealed}.

\begin{prop}[Wasserstein localization under 
$Q^{(\theta)}(\,\cdot\,| \Uloc_w(\srad,R))$]
\label{prop:tiltQ}
Let $\eta\in (0,\frac14)$, $T\in [1,\infty)$, $\wcush\in (0,\frac12)$, {$\delta\in[0,1]$}, $\srad\in[N^{-4}, \frac\wcush{10}]$, and let $R,w$ (possibly depending on $\theta\in[0,T]$) satisfy
\begin{equation}	\label{wR.bounds0}
R\in [\log N, N^{1/4}]\,,\qquad
w= {w'+w^\le} \in(1-\wcush)\B^{n_0}\,,\qquad { \|w^\le\|_\infty\le  \delta/R } \, .
\end{equation}
{where $w',w^\le$ have disjoint supports.}
\begin{enumeratea}
\item (Concentration).
There exists $C_0(\wcush)>0$ depending only on $\wcush$ such that with $\srad_0:= C_0(\wcush)RN^{-1/2}$,
for every $\theta\in[0,T]$ 
there exists $\t v_{\theta,w,R}\in  \R^{\supp(w)^c}\cap\Deloc_R\cap \sqrt{1-\|w\|_2^2}\sphereN$ such that
\begin{align}
\label{tiltQ:conc}
&\frac1N\log Q^{(\theta)}\big( \Bset_2(w+\t v_{\theta,w,R}\,,\,\srad_0) \big| \Uloc_w(\srad,R)  \big)\\
&\qquad \ge -O_{T,\wcush}({\delta} + \srad+ R^2N^{-1/2}+ N^{-2\eta}\log N) \notag
\end{align}
for all $N$ sufficiently large depending on $T$ and $\wcush$.
{Moreover, $\t v$ depends on $w$ only through $w'$ and $\|w\|_2$.}
\item (Continuity).
{Let $\theta_0\in(0,T)$ and suppose that $R,w$ depend on $\theta\in[\theta_0,T]$ in the following way (in addition to satisfying \eqref{wR.bounds0} for all $\theta\in[\theta_0,T]$): for $\chz\in\B^{n_0}$ and $R'\ge1$ independent of $\theta$, and continuous functions $q:[\theta_0,T]\to (0,1)$, $\alpha:[\theta_0,T]\to [0,1-\wcush]$, we have
$w'=q(\theta)\chz$, $\|w^\le\|_2^2=\alpha(\theta)$}, and $R= R'/(\theta q(\theta))$.
{Then we can take the vector $\t v_{\theta,w,R}$ from part (a) to depend continuously on $\theta\in [\theta_0,T]$.} 
\end{enumeratea}
\end{prop}

We now apply Propositions \ref{prop:tiltP} and \ref{prop:tiltQ} to prove Proposition \ref{prop:lower0}.
We will assume without comment that $N$ is sufficiently large depending on $L,\xcush,\zcush,\heta$ and $\mu$. 
Let $T,\tcush_0,\eps_0$ be as in the statement of Proposition \ref{prop:lower0} to be chosen later depending on $L,\xcush,\zcush$. 

To locate $\theta^*$ we will apply a continuity argument with a one-parameter family of tilted measures $\wt Q^{(\theta)}$ on the sphere, which we now define.
We define, with $q_{x}$ as in \eqref{def:overlapx}, 
\begin{equation}
{\wt R(\theta):= R\cdot} 
\frac{T\overlap_x(T)}{\theta \overlap_x(\theta)}\,,\qquad \theta\in B.
\end{equation}
We note that $\overlap_x$, and hence $\theta\mapsto \theta \overlap_x(\theta)$, is increasing and 
continuous on $B$, with $\theta \overlap_x(\theta)=\sqrt{\theta(\theta-\thetam)} \gs_{\tcush_0} 1$ for $\theta\in B$.
It follows that $\wt R(\theta)$ is decreasing on $B$, with
\begin{equation}	\label{RT.bounds}
R=\wt R(T)\le \wt R(\theta) \ls_{\tcush_0,T} 
N^{1/4}/\log N
\,,\qquad \theta\in B.
\end{equation}
With notation as in \eqref{def:UwrR}, for $\theta\in B$ we abbreviate
\begin{align}	
\Uloc^{(\theta)} &:= \Uloc_{w(\theta)}(\srad, \wt R(\theta))\,,\qquad
\wt Q^{(\theta)} := 
 Q^{(\theta)}(\,\cdot \,| \Uloc^{(\theta)}).	\label{def:tQtheta}
\end{align}
Proposition \ref{prop:lower0} quickly follows from the next claim:

\begin{claim}\label{claim:lower2}
With hypotheses as in Proposition \ref{prop:lower0}, 
there exists $\theta^*=\theta^*_{x,\chz}\in B$ such that 
\begin{align}
\label{lower:goal2}
&\frac1N\log\P(|\lambda_1-x|\le N^{-1/20}) \\
&\qquad\ge F_N(\theta^*; \Uloc^{(\theta^*)}) - J(x,\theta^*) -O_{L,\xcush,\zcush}(\srad+{\delta}+R^2N^{-1/2}+N^{-c\heta}).	\notag
\end{align}
\end{claim}

Indeed, to deduce Proposition \ref{prop:lower0}, we note that the set $\Uloc_{w(\theta)}(\srad, R)$ is monotone increasing in $R$. Since $\wt R(\theta)$ is monotone decreasing on $B$ we have
\[
\Uloc^{(\theta)} \supset \Uloc_{w(\theta)}(\srad, \wt R(T)) = \Uloc_{w(\theta)}(\srad, R)
\]
for all $\theta\in B$. Since $F_N(\theta;\Uloc)$ is increasing in $\Uloc$, we get
\begin{align*}
F_N(\theta^*; \Uloc_{w(\theta^*)}(\srad, R)) \le 
F_N(\theta^*; \Uloc^{(\theta^*)}) 
\end{align*}
so that the lower bound in Proposition \ref{prop:lower0} follows from the lower bound in Claim \ref{claim:lower2}.

It only remains to prove Claim \ref{claim:lower2}.

Consider for now an arbitrary $\theta\in B$.
Let $K=K(L,\zcush)>x$ to be chosen sufficiently large depending only on $L$ and $\zcush$.
By Lemma \ref{lem:approxsc} (with $\xcush$ in place of $3\xcush$ and assuming $K\ge L$) we have that for any $M\in\Evalx_x(N^{-1/20})\cap\Good(K,\kappa,\eta)$, 
\begin{equation}
\frac1N\log I(M,\theta) \le  J(x,\theta) + E_1
\end{equation}
for some
\begin{equation}	\label{lower:error1}
E_1\ls _{\tcush_0,K,\xcush} N^{-1/20}+{N^{-1/2+2\eta}}\,.
\end{equation}
Thus, abbreviating $\Evalx_x:=\Evalx_x(N^{-1/20})$ and $\Good:=\Good(K,\kappa,\eta)$, we have
\begin{align*}
\P(H\in\Evalx_x)
&\ge \P( H\in\Evalx_x\cap\Good)\\
&= \E \frac{I(H,\theta)}{I(H,\theta)} \ind( H\in \Evalx_x\cap\Good)\\
&\ge e^{ - N(J(x,\theta)  + E_1)} \E I(H,\theta)\ind(H\in \Evalx_x\cap\Good)\\
&= e^{ - N(J(x,\theta)  + E_1)} \int_{\sphereN} \E e^{\theta N\langle u,Hu\rangle}\ind(\Evalx_x\cap\Good) dP(u)\\
&= e^{ - N(J(x,\theta)  + E_1)} \int_{\sphereN} \P^{(\theta,u)}(H\in\Evalx_x\cap\Good) \cdot (\E e^{\theta N\langle u,Hu\rangle})dP(u)\\
&\ge e^{ - N(J(x,\theta)  + E_1)} \int_{\Uloc^{(\theta)}}  \P^{(\theta,u)}(H\in\Evalx_x\cap\Good)\cdot (\E e^{\theta N\langle u,Hu\rangle}) dP(u)\\
&=e^{  N(F_N(\theta;\Uloc^{(\theta)})-J(x,\theta)  + E_1)} \int_{\sphereN}  \P^{(\theta,u)}(H\in\Evalx_x\cap\Good) d\wt Q^{(\theta)}(u)	
\end{align*}
where in the third bound we restricted the domain of integration to $\Uloc^{(\theta)}$, and for the final line we recall that $\wt Q^{(\theta)}$ is supported on $\Uloc^{(\theta)}$.  
To establish Claim \ref{claim:lower2} it now suffices to locate $\theta^*\in B$ such that 
\begin{equation}	\label{lower.goal2}
\frac1N\log \int_{\sphereN}  \P^{(\theta^*,u)}(H\in\Evalx_x\cap\Good) d\wt Q^{(\theta^*)}(u)	
\ge - O_{L,\xcush,\zcush}(\srad+  {\delta}+R^2N^{-1/2}+N^{-\eta}).
\end{equation}

To that end, we first note that $1-\overlap_x(\theta) \ge 1-(1-\frac{\thetam}T)^{1/2} \gs (LT)^{-1}$ for all $\theta\le T$ (since $\thetam\gs 1/x\ge 1/L$), so
\begin{equation}	\label{qw.bound}
\|w(\theta)\|_2 \le 
\overlap_x(\theta) \le 1- \frac{c}{LT}
\end{equation}
for all $\theta\in B$.
We can hence apply Proposition \ref{prop:tiltQ} with  $\wcush:=c(LT)^{-1}$  (and taking   $\eps_0$  smaller than $\wcush/10$), $\theta_0:=\thetam+\tcush_0$, $\overlap_x$ for $q$, {$w^\le:=\overlap_x(\theta)z^\le(\theta)$}, $\wt R(\theta)$ in place of $R$ and $T\overlap_x(T) R$ for $R'$, to obtain 
{a continuous curve}
\begin{equation}	\label{def:tv-theta}
B\ni \theta\mapsto \t v^{(\theta)} := \t v_{\theta,w(\theta),\t R(\theta)} \in \sqrt{1-\|w(\theta)\|_2^2}\sphereN \cap \R^{\supp(\vloc)^c} \cap \Deloc_{\wt R(\theta)}
\end{equation}
such that 
\begin{equation}	\label{tQ.Loc}
\frac1N\log \wt Q^{(\theta)}( \Bset_2( w(\theta)+\t v^{(\theta)} , \srad_0)) \ge - O_{\tcush_0, T,L}( \srad+{\delta}+ R^2N^{-1/2}+ N^{-2\eta}\log N)
\end{equation}
for some $\srad_0= 
{O_{\tcush_0,T,L}(RN^{-1/2})}$. 
Let 
\begin{equation}	\label{def:lam-xw}
\lam_{x,\vloc'}(\theta):= \E^{(\theta,\t u^{(\theta)})} \lambda_1(H)\,,\qquad 
\t u^{(\theta)}:= w(\theta)+\t v^{(\theta)}\,,\qquad \theta\in B.
\end{equation}
{Since $\theta\mapsto \t u^{(\theta)}$ is continuous on $B$, from} Proposition \ref{prop:tiltP}\eqref{tiltP:cont} 
{it follows that $\lam_{x,\vloc'}$ is continuous on $B$.}

Now we consider $\lam_{x,\vloc'}$ at the endpoints of $B$.
For the left endpoint $\theta_0=\thetam+\tcush_0$ we note that $\t u^{(\theta_0)}$ is close to a delocalized vector:
\[
\|\t u^{(\theta_0)} - \t v^{(\theta_0)}\|_2 \le \overlap_x(\theta_0) \ls \sqrt{L\tcush_0}
\]
using again that $\thetam\gs1/x\ge 1/L$. 
From \eqref{RT.bounds} we have 
\[
\t v^{(\theta_0)}\in \Deloc_{\wt R(\theta)}\subset 
\Deloc_{N^{1/4}/(\log N)^{1/2}}\,.
\]
Taking $\tcush_0$ sufficiently small depending on $\xcush$ so that $\theta_0\le 1/2$, from Proposition \ref{prop:tiltP}(b,c) it follows that
\[
\lam_{x,\vloc'}(\theta_0) = \E^{(\theta_0, \t v^{(\theta_0)})} \lam_1(H) + O_L(\tcush_0^{1/8}) \le 2+ O_L(\tcush_0^{1/8})\,.
\]
Hence, we may now fix $\tcush_0=\tcush_0(\xcush,L)>0$ smaller if necessary so that 
\begin{equation}	\label{lam-xw.small}
\lam_{x,\vloc'}(\theta_0) \le 2+\tfrac12\xcush \,.
\end{equation}
Turning to lower bound $\lam_{x,\vloc'}(T)$, 
we have from our assumptions that 
\[
\|\t v^{(T)}\|_2^2 = 1-\overlap_x(\theta)^2 \| z(\theta)\|_2^2\ge 1-\|z(\theta)\|_2^2\ge \zcush 
\]
and moreover 
\[
\|\t v^{(T)}\|_\infty\le RN^{-1/2}\le N^{-1/4-\heta/10}\le T^{-1/2}N^{-1/4}
\]
for all $N$ sufficiently large. Thus,
\[
\sum_{i=1}^N (\t u^{(T)}_i)^21_{|\t u^{(T)}_i|\le T^{-1/2}N^{-1/4}} \ge 
\|\t v^{(T)}\|_2^2 \ge \zcush
\]
and \eqref{cond:u.med} holds for $\t u^{(T)}$ with $L=1,\beta_0= \zcush$. Hence,
we get from Proposition \ref{prop:tiltP}\eqref{tiltP:large} that
$
\lam_{x,\vloc'}(T) \gs \zcush^2 T
$.
We can now fix $T=T(L,\zcush):=CL/\zcush^2$ with $C$ sufficiently large that
\begin{equation}	\label{lam-xw.large}
\lam_{x,\vloc'}(T) \ge 2L\,.
\end{equation}
Note this also fixes $\eps_0=\eps_0(L,\zcush)$.

From \eqref{lam-xw.small}, \eqref{lam-xw.large} {and the continuity of $\lam_{x,\vloc'}$ it follows from the intermediate value theorem} that there exists $\theta^*\in B$ such that 
\begin{equation}\label{findclaim}
\lam_{x,\vloc'}(\theta^*)=x.
\end{equation}
From Proposition \ref{prop:tiltP}\eqref{tiltP:cont} (and recalling that $\tcush_0,T$ have already been fixed depending on $L,\xcush,\zcush$),
\begin{equation}
\E^{(\theta^*, u)} \lam_1(H) 
= x + O_{L,\xcush,\zcush}(N^{-1/16})
\end{equation}
for all $u\in \Bset_2(\t u^{(\theta^*)}, \srad_0)$, with $\srad_0\ls_{\tcush_0,T,L}RN^{-1/2}\ls_{L,\xcush,\zcush} N^{-1/4}$ as in \eqref{tQ.Loc}. 
Applying Proposition \ref{prop:tiltP}(a) with $s=N^{-1/4}$ we have
\begin{equation}	\label{lower8}
\P^{(\theta^*, u)}( |\lam_1(H) - x| \le C_{L,\xcush,\zcush} N^{-1/16}) \ge 1- \exp( - 
cN^{1/2}/\log N) \ge \frac12\qquad \forall u \in \Bset_2(\t u^{(\theta^*)}, \srad_0)
\end{equation}
for some $C_{L,\xcush,\zcush}$ sufficiently large.

The following extends Lemma \ref{lem:good} to the tilted measures $\P^{(\theta,u)}$.

\begin{lemma}
\label{lem:good.tilt}
Let $\theta\ge0$ and $u\in\ball^N$. 
For any $\kappa_0>0$, $A\ge 100\psimax \theta^2$ and $K_0=K_0(A)$ sufficiently large depending on $A$ (and hence on $\theta$),
\begin{equation}
\frac1N\log\P^{(\theta,u)}(H\notin\Good(K_0,\kappa_0,\eta)) \le -A
\end{equation}
for all $N$ sufficiently large depending on  
$A,\kappa_0$ and $\eta$.
\end{lemma}

Before proving the lemma we complete the proof of Claim \ref{claim:lower2}.
Recall the event $\Good=\Good(K,\kappa,\eta)$ from \eqref{lower.goal2}, where $K=K(L,\zcush)$ is still to be specified. 
With $K_0(A)$ as in Lemma \ref{lem:good.tilt}, we take $A=A(L,\zcush):=\max\{1, 100\psimax T(L,\zcush)^2\}$ and fix $K(L,\zcush):=K_0(A(L,\zcush))$.
Applying Lemma \ref{lem:good.tilt} with $\theta=\theta^*${, $\kappa_0=\kappa$} and combining with \eqref{lower8}, we get
\begin{equation}	\label{lower.above2}
\P^{(\theta^*,u)}( |\lambda_1(H)- x| \le C_{L,\xcush,\zcush} N^{-1/16}, H\in \Good) \ge \frac14\qquad
\forall\;u\in \Bset_2(\t u^{(\theta^*)}, \srad_0)\,.
\end{equation}
On the other hand, the left hand side above is bounded above by $\P^{(\theta^*,u)}( H\in \Evalx_x\cap\Good)$.
Hence, for the left hand side of \eqref{lower.goal2}, by restricting the integral to $\Bset_2(\t u^{(\theta^*)}, \srad_0)$ and substituting the lower bounds \eqref{lower.above2} and \eqref{tQ.Loc}, we have
\begin{align*}
&\frac1N\log \int_{\sphereN}  \P^{(\theta,u)}(H\in\Evalx_x\cap\Good) d\wt Q^{(\theta^*)}(u)	\\
&\qquad\ge \frac1N\log \int_{\Bset_2(\t u^{(\theta^*)}, \srad_0)}  \P^{(\theta,u)}(H\in\Evalx_x\cap\Good) d\wt Q^{(\theta^*)}(u)	\\
&\qquad\ge \frac1N \log \wt Q^{(\theta^*)}( \Bset_2(\t u^{(\theta^*)}, \srad_0)) -O(N^{-1})\\
&\qquad\ge -O_{L,\xcush,\zcush}(  \srad +{\delta}+R^2N^{-1/2}+ N^{-2\eta}\log N).
\end{align*}
Thus we obtain \eqref{lower.goal2} and hence Claim \ref{claim:lower2}, which completes the proof of Proposition \ref{prop:lower0}.
\qed

\begin{proof}[Proof of Lemma \ref{lem:good.tilt}]
Fix $A\ge100\psimax\theta^2$.
Writing $f_{\theta,u}(H)
$ for the density of $\P^{(\theta,u)}$ with respect to $\P$, we have from Cauchy--Schwarz that
\begin{equation}	\label{good.tilt1}
\P^{(\theta,u)}(\Good(K_0,\kappa_0,\eta)^c)
\le \P(\Good(K_0,\kappa_0,\eta)^c)^{1/2} (\E f_{\theta,u}(H)^2)^{1/2}.
\end{equation}
From Lemma \ref{lem:good} we can take $K_0$ sufficiently large depending on $A$ so that
\begin{equation}	\label{good.tilt2}
\frac1N\log\P(H\notin \Good(K_0,\kappa_0,\eta)) \le -4A.
\end{equation}
For the second moment of $f_{\theta,u}(H)$ we have
\begin{align*}
\E f_{\theta,u}(H)^2= \frac{\E e^{2\theta N\langle u,Hu\rangle}}{(\E e^{\theta N\langle u,Hu\rangle})^2} 
&\le \E e^{2\theta N\langle u,Hu\rangle}\\
&= \exp\bigg( \sum_{i\le j} \LLa_\mu( 2^{1+\ep_{ij}}\theta\sqrt{N}u_iu_j)\bigg)
\le e^{8 \psimax \theta^2 N \|u\|_2^2} \le e^{8\psimax\theta^2N}
\end{align*}
where in the first bound we applied Jensen's inequality to bound the denominator below by 1 (recall $H$ is centered).
Combining with \eqref{good.tilt1}--\eqref{good.tilt2} and our assumption on $A$ we get
\[
\P^{(\theta,u)}(\Good(K_0,\kappa_0,\eta)^c)
\le \exp( (8\psimax \theta^2-2A)N) \le \exp( - AN)
\]
as desired.
\end{proof}

\section{The top eigenvalue of tilted Wigner matrices}
\label{sec:tiltP}

Here we prove Proposition \ref{prop:tiltP}.

\subsection{Proof of Proposition \ref{prop:tiltP}\eqref{tiltP:conc}}
This is immediate from 
a more general result, Corollary \ref{cor:lam1-tilt.conc}, that we prove in the appendix. \qed

\subsection{Proof of Proposition \ref{prop:tiltP}\eqref{tiltP:cont}}

We apply a coupling argument. 
 Let $\mu$ be a standardized distribution satisfying \eqref{assu:usg}
and let $X$ have distribution $\mu$.
For a real number $\alpha$, let 
$\mu^{\alpha}$ be the tilted measure on the real line 
\[
\mu^{\alpha}(A):=\frac{\E[ e^{\alpha X}\ind(X\in A)]}{\E[e^{\alpha X}]}.
\]
By a classical construction, we may define a family $(X_{\alpha},\alpha\in \mathbb R)$ of random variables constructed on the same probability space and such that $X_{\alpha}$ has distribution $\mu^{\alpha}$. A remarkable fact is that this family is monotone in $\alpha$. Let us recall the definition of this monotone coupling. 
With $F_\alpha(x):=
\mu^\alpha((-\infty, x])$
we let 
\begin{equation}\label{defphi}
F^{-}_{\alpha}(t):=\sup_{}\{ x: F_{\alpha}(x)\leq t\}\,,\quad t\in[0,1].
\end{equation}
For $Y$ a random variable following the   uniform law  on
$[0,1]$, we  set $ X_{\alpha}:=F_{\alpha}^{-}(Y)$. 
We have that for every real number $x$
\begin{equation}	\label{tilt-phi}
\mathbb{P}(F_{\alpha}^-(Y)\leq x)  =\mathbb{P}(Y\leq F_{\alpha}(x))=F_{\alpha}(x)
\end{equation}
and hence $X_\alpha$ follows $\mu^{\alpha}$. For a pair $\alpha,\beta\in\R$ we hence obtain a coupling 
\begin{equation}	\label{scalar-coupling}
(X_\alpha,X_\beta)=(F_{\alpha}^-(Y), F_{\beta}^-(Y)).
\end{equation}
This coupling is monotone.
\begin{lemma}\label{lem:monotony-connection}
For every $\alpha\le \beta$, $X_{\alpha}\le X_{\beta}$.
\end{lemma}
\begin{proof}
It is enough to show that $F_{\alpha}^{-}(t)\le F_{\beta}^{-}(t)$ for every $t\in [0,1)$. For $\beta\ge \alpha$, and $x\in \mathbb R$, we have 
\[
F_{\beta}(x)=\frac{1}{\mathbb{E}\left(e^{\beta X}\right)}\mathbb{E}\left(e^{(\beta-\alpha)X}e^{\alpha X}1_{X\leq x}\right)\leq\frac{\mathbb{E}\left(e^{\alpha X}\right)}{\mathbb{E}\left(e^{\beta X}\right)}e^{(\beta-\alpha)x}F_{\alpha}(x)
\]
and
\[
1-F_{\beta}(x)=\frac{1}{\mathbb{E}\left(e^{\beta X}\right)}\mathbb{E}\left(e^{(\beta-\alpha)X}e^{\alpha X}1_{X>x}\right)\geq\frac{\mathbb{E}\left(e^{\alpha X}\right)}{\mathbb{E}\left(e^{\beta X}\right)}e^{(\beta-\alpha)x}(1-F_{\alpha}(x))
\]
and therefore 
\[
\frac{F_{\beta}(x)}{1-F_{\beta}(x)}\leq\frac{F_{\alpha}(x)}{1-F_{\alpha}(x)}.
\]
We deduce that $F_{\beta}(x)\leq F_{\alpha}(x)$ for all $x\in\mathbb{R}$
and then $F_{\alpha}^-(t)\leq F_{\beta}^-(t)$ for all $t\in[0,1]$.
\end{proof}

We will need the following lemma, stating that if $|\alpha-\beta|$ is small, then the coupled pair $X_\alpha,X_\beta$ are close in $L^p$ for any $p\in[1,\infty)$ (we just need $p\in\{1,2,4\}$). We suspect such estimates have been proved before but we could not find a reference, so a proof is provided in Appendix \ref{app:coupling}.

\begin{lemma}[Continuity of coupling for tilted scalar laws]
\label{lem:coupling.scalar}
For any $k\in\mathbb{N}^{*}$,
there exists $C_{k}>0$ (depending only on $k$ and 
{$\|\LLa_\mu''\|_\infty$}) such that for all $\alpha,\beta\in\mathbb{R}$, $|\alpha-\beta|\leq1$, with $(X_\alpha,X_\beta)$ as in \eqref{scalar-coupling}, we have
\begin{equation}\label{coupling.b1}
\mathbb{E}(|X_{\beta}-X_{\alpha}|^{k})\leq C_{k}|\beta-\alpha|
\end{equation}
 and for all $\alpha,\beta\in\mathbb{R}$, 
\begin{equation}\label{b11}
\mathbb{E}(|X_{\beta}-X_{\alpha}-\mathbb{E}(X_{\beta}-X_{\alpha})|^{k})\leq C_{k} |\alpha-\beta|\wedge 1\,.
\end{equation}
\end{lemma}

Let $H_{\theta,u}$ denote a random matrix whose distribution under $\P$ is that of $H$ under $\P^{(\theta,u)}$; that is, for any Borel set $E\subseteq \Sym_N$, $\P(H_{\theta,u}\in E) = \P^{(\theta,u)}(H\in E)$. 
Using Lemma \ref{lem:coupling.scalar}, we can construct a coupling of $H_{\theta,u}, H_{\gamma,v}$ for $\theta,\gamma\ge0$ and $u,v\in \sphereN$ such that the two matrices are close in various senses when $|\theta-\gamma|$ and $\|u-v\|_2$ are small. 

\begin{lemma}
\label{lem:NiceConnection0}
For every $\theta,\gamma\in\R^+$, every integer  $N$  and  every $u,v\in \sphereN$, there exists a coupling of $H_{\theta,u}$ and $H_{\gamma,v}$ such that 
\begin{enumeratea}
\item $H_{\theta,u}-H_{\gamma, v}$ is a  symmetric matrix with independent
entries on and above the diagonal.
\item The matrix $\mathbb{E}(H_{\theta,u})$  
satisfies
\[ \|\mathbb{E}(H_{\theta,u}-H_{\gamma,v})\|_{\HS}
\ls |\theta-\gamma| + \sqrt{\theta\gamma} \|u-v\|_2. 
 \]
\item The variance matrix given for $i,j\in\{1,\ldots,N\}$ by 
\[ S_{ij}:=\mathbb{E}((H_{\theta,u}-H_{\gamma,v}-\mathbb{E}(H_{\theta,u}-H_{\gamma,v}))_{ij}^{2})\]
satisfies  for all $i\in \{1,\ldots,N\}$,
\[
\sum_{j=1}^{N}S_{ij}
\ls |\theta-\gamma| + \sqrt{\theta\gamma} \|u-v\|_2\,.
\]
\item  For all $k\ge 2$,
\[ 
\sum_{i,j=1}^{N}\mathbb{E}(H_{\theta,u}-H_{\gamma,v}-\mathbb{E}(H_{\theta,u}-H_{\gamma,v}))_{ij}^{2k}
\ls N^{3/2-k} \big( |\theta-\gamma| + \sqrt{\theta\gamma} \|u-v\|_2\big).
\]
\end{enumeratea}
\end{lemma}

\begin{proof}
Let $\boldsymbol{Y}=(Y_{ij})_{ij\leq N}\in\mathbb{R}^{N\times N}$
be a symmetric matrix with iid entries with uniform law on $[0,1]$.
From \eqref{tilt-phi}, we can realize a coupling of $H_{\theta,u}$ and $H_{\gamma,v}$ by setting:
\[
H_{\theta,u}:=\frac1{\sqrt{N}}
\left(
F_{2^{\ep_{ij}}\sqrt{N}\theta u_{i}u_{j}}^{-1}
(Y_{ij})\right)_{1\le i,j\leq N}
\,,\quad
H_{\gamma,v}:=\frac1{\sqrt{N}}
\left(
F_{2^{\ep_{ij}}\sqrt{N}\theta v_{i}v_{j}}^{-1}
(Y_{ij})\right)_{1\le i,j\leq N}
\]
(recall our notation $\ep_{ij}:=\frac12(1+1_{i\ne j})$).
We next show that it satisfies the announced properties. For (a), it is clear that $H_{\theta,u}-H_{\gamma,v}$ is symmetric with independent entries since the log density is linear in the entries. 
For (b), first notice that
\begin{align*}
\sum_{1\le i,j\leq N}(\mathbb{E}(H_{\theta,u})-\mathbb{E}(H_{\gamma,v}))_{ij}^{2} 
& =\frac{1}{N}\sum_{1\le i,j\leq N}{2^{1_{i=j}}}|\LLa_\mu'(2^{{\ep_{ij}}}\theta\sqrt{N}u_{i}u_{j})-\LLa_\mu'(2^{{\ep_{ij}}}\gamma \sqrt{N}v_{i}v_{j})|^{2}\\
 & \leq4\|\LLa_\mu''\|_{\infty}^{2}\sum_{1\le i,j\leq N}(\theta u_{i}u_{j}-\gamma v_{i}v_{j})^{2}\\
 & =4\|\LLa_\mu''\|_{\infty}^{2}\|\theta uu^{\tran}-\gamma vv^\tran\|_{\HS}^{2}\,.
\end{align*} 
We can further bound
\begin{align*}
\|\theta uu^\tran - \gamma vv^\tran\|_\HS
&\le \| \sqrt\theta u( \sqrt\theta u-\sqrt\gamma v)^\tran\|_\HS + \|(\sqrt\theta u-\sqrt\gamma v)\sqrt\gamma v^\tran\|_\HS\\
&= (\sqrt\theta + \sqrt\gamma) \|\sqrt\theta u-\sqrt\gamma v\|_2
\end{align*}
and
\begin{align}
 (\sqrt\theta + \sqrt\gamma) \|\sqrt\theta u-\sqrt\gamma v\|_2
 &\le 
  (\sqrt\theta + \sqrt\gamma) \big( \sqrt{\theta\wedge\gamma} \|u-v\|_2 + |\sqrt\theta-\sqrt\gamma|\big)	\notag\\
  &\le 2\sqrt{\theta\gamma}\|u-v\|_2 + |\theta-\gamma|.	\label{tguv.bd2}
\end{align}
Combining all of these bounds yields (b). 
Turning to (c), we have that for all $i,j$,
\[ 
S_{ij}=\frac{1}{N} \mathbb E[ (X_{ 2^{\ep_{ij}} \sqrt{N} \theta u_{i}u_{j}}-X_{ 2^{\ep_{ij}} \sqrt{N}\gamma v_{i}v_{j}}-\E[X_{2^{\ep_{ij}} \sqrt{N} \theta u_{i}u_{j}}-X_{ 2^{\ep_{ij}} \sqrt{N}\gamma v_{i}v_{j}}])^{2}]
\]
so that by Lemma \ref{lem:coupling.scalar} and more precisely \eqref{b11},
\begin{align*}
\sum_{j}S_{ij} & \leq C_{2}N^{-1/2}\sum_{j}|\theta u_{i}u_{j}-\gamma v_{i}v_{j}|\\
 & \leq C_{2}N^{-1/2}(|\sqrt{\theta}u_{i}|\sum_{j}|\sqrt{\theta}u_{j}-\sqrt{\gamma }v_{j}|+|\sqrt{\theta}u_{i}-\sqrt{\gamma }v_{i}|\sum_{j}\sqrt{\gamma}|v_{j}|)\\
 & \leq C_{2}(\sqrt{\theta}+\sqrt{\gamma })\|\sqrt{\theta}u-\sqrt{\gamma }v\|_2
\end{align*}
where we used Cauchy--Schwarz in the last step. Recalling that $C_2$ from Lemma \ref{lem:coupling.scalar} depends only on $\|\LLa_\mu''\|_\infty$, we obtain (c) upon substituting the bound \eqref{tguv.bd2}. Similarly,
\begin{align*}
&\sum_{i,j=1}^{N}\mathbb{E}[(H_{\theta,u}-H_{\gamma,v}-\mathbb{E}(H_{\theta,u}-H_{\gamma,v}))_{ij}^{2k}]\\
& \leq C_{2k}N^{-k+1/2}\sum_{i,j}|\theta u_{i}u_{j}-\gamma v_{i}v_{j}|\\
 & \leq C_{2k}N^{-k+1/2}\sum_{i}(|\sqrt{\theta}u_{i}|\sum_{j}|\sqrt{\theta}u_{j}-\sqrt{\gamma }v_{j}|+|\sqrt{\theta}u_{i}-\sqrt{\gamma }v_{i}|\sum_{j}\sqrt{\gamma}|v_{j}|)\\
 & \leq C_{2k}N^{-k+3/2}(\sqrt{\theta}+\sqrt{\gamma })\|\sqrt{\theta}u-\sqrt{\gamma }v\|_2
\end{align*}
where we finally used the Cauchy--Schwarz inequality. The claim (d) now follows from \eqref{tguv.bd2}.
\end{proof}

To finish the proof of Proposition \ref{prop:tiltP}\eqref{tiltP:cont} we need the following Lemma, obtained by applying the main result of \cite{latala} to the upper and lower triangular parts of $W$.

\begin{lemma}
\label{lem:normWignerSumVariance0}
Let $W$ be a symmetric random matrix
with centered independent entries and bounded moments such that for
all $i\leq N$, $\sum_{j=1}^{N}\mathbb{E}(W_{ij}^{2})\leq A$, $\sum_{ij=1}^{N}\mathbb{E}(W_{ij}^{4})\leq B$. 
Then there exists a universal constant $C$ such that 
\[
\E\|W\|\le C  (A^{1/2}+B^{1/4})\,.
\]
\end{lemma}


Now we complete the proof of Proposition \ref{prop:tiltP}\eqref{tiltP:cont}.
We first observe that the law $\lambda_{1}(H)$ is invariant if we replace $H$ with $H^{\varrho}=(H_{\varrho(i),\varrho(j)})_{1\le i,j\le N}$ for any fixed permutation $\varrho$, and therefore $ \lambda_{1}(H_{\theta, u})$ has the same law as 
 $\lambda_{1}(H_{\theta, u^{\varrho}})$ with $u^{\varrho}=(u_{\varrho(i)})_{1\le i\le N}$. Therefore  it is enough to show that for any $u,v\in \sphereN$,
 \begin{equation}\label{median}
 |\mathbb{E}(\lambda_{1}(H_{\theta, u}))-\mathbb{E}(\lambda_{1}(H_{\gamma, v}))|
 \ls \Delta(|\theta-\gamma|) + \Delta(\sqrt{\theta\gamma} \|u-v\|_2)  \,.
 \end{equation}
We have
\begin{align*}
|\mathbb{E}\lambda_{1}(H_{\theta,u})-\mathbb{E}\lambda_{1}(H_{\gamma,v})|
&\leq \E |\lambda_1(H_{\theta,u}) - \lambda_1(H_{\gamma,v})| 
\le  
\mathbb{E}\|H_{\theta, u}-H_{\gamma,v}\|\\
&\le \|\E H_{\theta, u} -\E H_{\gamma,v}\|
+ \mathbb{E}\|H_{\theta, u}-H_{\gamma,v}-\E H_{\theta,u}+\E H_{\gamma,v}\|\,.
\end{align*}
From Lemma \ref{lem:NiceConnection0}(b),
\[ 
\|\E H_{\theta, u}-\E H_{\gamma,v} \| \le 
\| \E  H_{\theta,u}  - \E H_{\gamma,v} \|_\HS
\ls  |\theta-\gamma| + \sqrt{\theta\gamma} \|u-v\|_2 =:a.
\]
On the other hand, from Lemma \ref{lem:normWignerSumVariance0} and  Lemma \ref{lem:NiceConnection0}(c,d) we get
\[  
\mathbb{E}\|H_{\theta, u}-H_{\gamma,v}-\E[H_{\theta,u}]+\E[H_{\gamma,v}]\| 
\ls a^{1/2} + N^{-1/8} a^{1/4}\,.
\]
Combining all of our bounds, we have
\[
|\mathbb{E}\lambda_{1}(H_{\theta,u})-\mathbb{E}\lambda_{1}(H_{\gamma,v})| \ls a + a^{1/2} + N^{-1/8}a^{1/4}\ls \max(a, a^{1/4})
\]
which gives \eqref{median} and completes the proof of Proposition \ref{prop:tiltP}\eqref{tiltP:cont}.
\qed

\subsection{Proof of Proposition \ref{prop:tiltP}\eqref{tiltP:small}}

With $R,\theta$ and $\kappa$ as in the statement of the proposition, from a slight modification of the proof of \cite[Lemma 5.2]{HuGu1} we have
\begin{equation}	\label{Lem:HuGu1}
\sup_{u\in \Deloc_R} \P^{(\theta,u)}( \lambda_1(H)>2+\tfrac12\kappa) =o(1)\,.
\end{equation}
The idea is that under the tilted measure $\P^{(\theta,u)}$, $H$ is close in operator norm to a rank-1 perturbation of a generalized Wigner matrix, with error that is small when $u$ is delocalized, and hence the top eigenvalue is close to the location predicted by the BBP transition, which in turn is close to 2 when $\theta\le\frac12$. We refer the reader to \cite{HuGu1} for the detailed argument and state here the modifications needed to obtain \eqref{Lem:HuGu1}.
Indeed, the proof 
in \cite{HuGu1} does not use the sharp sub-Gaussian assumption from their main theorem, only that $\mu$ is sub-Gaussian. 
Furthermore, while the result there is stated under the assumption that $R\le N^{1/4-\ep}$ for arbitrary fixed $\ep>\frac18$, this is only used in the proof of \cite[Lemma 5.3]{HuGu1}  to ensure the matrix $\Delta^{(e),N}$ there has spectral norm $o(1)$. However, it is shown that $\|\Delta^{(e),N}\|\ls  \sqrt{N} \|e\|_4^4$ whenever $\theta=O(1)$ and $R=O(N^{1/4})$ ($e$ there is our $u$ and lies in $\Deloc_R$), and bounding 
\[
\|e\|_4^4\le \|e\|_\infty^2\|e\|_2^2\le \|e\|_\infty^2\le R^2/N 
\]
shows that $\|\Delta^{(e),N}\|=o(1)$ as long as $R=o(N^{1/4})$. 
Hence we have that under $\P^{(\theta,u)}$, any median for $\lambda_1(H)$ is bounded by $2+\frac12\kappa$ for all $N$ sufficiently large, uniformly in $u\in \Deloc_R$. 
The claim now follows from 
part \eqref{tiltP:conc}.
\qed

\subsection{Proof of Proposition \ref{prop:tiltP}\eqref{tiltP:large}}

We have for any $K>0$ that
\begin{equation}	\label{tilt.probLB1}
\P^{(\theta,u)}(\lambda_1(H)\le K) 
\le \P^{(\theta,u)}(\langle u, Hu\rangle \le K)
\le \E^{(\theta,u)} e^{\theta N(K-\langle u,Hu\rangle)}
= \frac{e^{\theta N K}}{\E e^{\theta N\langle u,Hu\rangle}}.
\end{equation}
Now assuming we have
\begin{equation}	\label{Etilt-assuming1}
\E e^{\theta N\langle u,Hu\rangle} \ge e^{\beta \theta^2N}
\end{equation}
for some $\beta>0$, then combining with \eqref{tilt.probLB1} we would have
\begin{equation}	\label{tilt.probLB2}
\P^{(\theta,u)}(\lambda_1(H)\le \beta \theta/2) \le e^{-\frac12\beta\theta^2 N}\le \frac12
\end{equation}
for all $N$ sufficiently large depending on $\beta,\theta$. 
Moreover,
 \begin{align*}
\E^{(\theta,u)} |\lambda_1(H)|\ind( \lambda_1(H)<0)
&\le \frac{\E |\lambda_1(H)|\ind(\lambda_1(H)<0)}{\E e^{\theta N\langle u,Hu\rangle}}\\
&\le \E |\lambda_1(H)|\ind(\lambda_1(H)<0)\\
&\le (\E \|H\|^2)^{1/2}\P( \lambda_1(H)<0)^{1/2} =O(e^{-N})
\end{align*}
where in the first bound we used that $\langle u,Hu\rangle<0$ on the event that $\lambda_1(H)<0$, in the second bound we applied Jensen's inequality and the fact that $H$ is centered in the denominator, and in the third line we applied Cauchy--Schwarz.
Together with \eqref{tilt.probLB2} this implies
\begin{align*}
\E^{(\theta,u)} \lambda_1(H)
&\ge \frac12\beta\theta \cdot \P^{(\theta,u)}(\lambda_1(H)>\beta\theta/2) - \E^{(\theta,u)}|\lambda_1(H)|\ind(\lambda_1(H)<0)
\ge {\frac18\beta\theta}
\end{align*}
{for all $N$ sufficiently large depending on $\beta,\theta$.}
To prove \eqref{Etilt-lambda.bd1} it thus suffices to show that \eqref{Etilt-assuming1} holds with 
{$\beta\gs m_\mu(L)\beta_0^2$, and indeed:}
\begin{align*}
\frac1N\log\E e^{\theta N\langle u,Hu\rangle}
&=\frac1N \sum_{i\le j} \LLa_\mu(2^{\ep_{ij}} \theta \sqrt{N}u_iu_j)
=\theta^2 \sum_{i\le j} 2^{1+1_{i\ne j}}u_i^2u_j^2\psimu(2^{\ep_{ij}}\theta\sqrt{N}u_iu_j)\\
&\ge { \theta^2 m_\mu(L) \bigg( \sum_{i=1}^N u_i^2 1_{|u_i|\le L\theta^{-1/2}N^{-1/4}} \bigg)^2
=m_\mu(L) \beta_0^2 \theta^2.}
\end{align*}
\qed

\section{Constrained Gibbs variational principle}
\label{sec:gibbs}

The quantity $\VP_R(v,\al)$ (see \eqref{def:VP}) in the asymptotic expression \eqref{def:freeN} for the restricted annealed free energy provided by Proposition \ref{prop:FE} involves a constrained Gibbs variational problem.
In this section we establish some general properties of solutions for such problems, which are summarized in the following proposition. These facts will be used  to construct the vector $\t v_{\theta,w,R}$ from Proposition \ref{prop:tiltQ}, giving the optimal delocalized part of the vector $u$ as in \eqref{optimalu}.

{Recall the notation $\cP_\al(I)$ from \eqref{def:cPbeta}.}

\begin{prop}
\label{prop:gibbs} 
Let $\ham:I\to \R$ be a continuous function on a compact interval 
{$I\not\subseteq(-1,1)$ with nonempty interior,}
and for $\nu\in \cP(I)$ set
\begin{equation}	\label{def:objective}
\chi(\nu) := \int_I \ham \,d\nu - \DKL(\nu|\gamma)\,,
\end{equation}
recalling that $\gamma$ denotes the standard Gaussian measure on $\R$.
\begin{enumeratea}
\item\label{gibbs.attain}
$\chi$ achieves its maximum value on the set $\cP_1(I)$
at a unique probability measure 
$\t\nu_{I,\ham}$. 
\item\label{gibbs.density}
We have $\t\nu_{I,\ham}=\nu_{I,\ham}^{\zeta}$ with
\begin{equation}
\label{mini}
d\nu_{I,\ham}^\zeta(x):=1_{I}(x)\frac{e^{ \ham(x)-\zeta x^2} dx}{ \int_{I}e^{ \ham(y) -\zeta y^2} dy} 
\end{equation}
where $\zeta=\zeta_{I,h}$ is the unique real number such that $\nu_{I,\ham}^\zeta\in \cP_1(I)$.
As a consequence,
\begin{equation}
\sup_{\nu\in\cP_1(I)}\chi(\nu) = 
\log \int_I e^{\ham(s) - \zeta_{I,\ham}s^2}ds 
+\zeta_{I,\ham} 
- \tfrac12\log(2\pi e).
\end{equation}
\item\label{gibbs.W2cont} (Stability of optimizers).
For $a>0$ 
let $\t\nu_a:=\t\nu_{a^{-1}I,h(a\,\cdot)}\in \cP_1(a^{-1}I)$ be the measure obtained as in part (a) with $a^{-1}I$ in place of $I$ and the dilated function $\ham(a\,\cdot)$ in place of $\ham$.
Then the mapping $a\mapsto\t\nu_a\in \cP_1(\R)$ is continuous on 
{its domain $(0, \sup\{|t|:t\in I\}]$} under the $\cW_2$-distance.
{If we further assume} that $\ham$ is symmetric, with $I=[-R,R]$ for some $R>1$, then
\begin{equation}	\label{gibbs.W2}
\cW_2(\t\nu_{a},\t\nu_{b}) \le 2\bigg( 1- \frac{a}{b}\bigg)^{1/2}
\end{equation}
for $0<a\le b \le R$. In particular, $a\mapsto\t\nu_a$ is uniformly H\"older$(\frac12)$ continuous on $[a_0,R]$ for any fixed $a_0\in(0,R)$.

\item\label{gibbs.wiggle} (Stability of optima). 
For $\delta>0$ and an interval $I\subseteq \R$, let
\begin{equation}	\label{def:cP1d}
\cP_1^{\delta}(I) := 
\bigg\{ \nu\in\cP(I) : \bigg| \int x^2 d\nu(x) -1\bigg| \le \delta \bigg\} = \bigcup_{b\in[1-\delta,1+\delta]}\cP_b(I).
\end{equation}
Let $\delta,\eps>0$ and suppose 
{$\ham\circ\ssq^{-1}$}
is $K$-Lipschitz on $[-(1+\eps)R,(1+\eps)R]$ {(recall the notation \eqref{def:ssq})}. We assume as well that $h(0)=0$. Then 
\begin{equation}	\label{opt.stab1}
\sup_{\nu\in \cP_1^\delta( (1+\eps)[-R,R])}\bigg\{ \int \ham d\nu - \DKL(\nu|\gamma) \bigg\}
= \sup_{\nu\in \cP_{1}( [-R,R])}\bigg\{ \int \ham d\nu - \DKL(\nu|\gamma) \bigg\} + O(K(\delta+\eps)). 
\end{equation}

\item\label{gibbs.lim} 
If $\ham$ is defined on all of $\R$, 
then for any fixed $\al>0$,
\begin{equation}	\label{opt.stab2}
\sup_{\nu\in\cP_\al([-R,R])} \{ \nu(h) - \DKL(\nu|\gamma)\} \longrightarrow_{R\to\infty} \sup_{\nu\in\cP_\al(\R)}\{\nu(h) - \DKL(\nu|\gamma)\}.
\end{equation}
\end{enumeratea}
\end{prop}

\begin{remark}
While the explicit bound \eqref{gibbs.W2} is not needed in the present work, the fact that we have a bound that is independent of $R$ is interesting and may prove useful in subsequent work. We suspect this bound (possibly with a worse constant) extends to the asymmetric case, but we do not have a proof.
\end{remark}

We recall the following lemma from \cite{AGH}.
For a measure $\nu$ on $\R$ and $a>0$ let $\nu_{\#a}:=D_a\#\nu$ denote the pushforward of $\nu$ under the dilation map $D_a:=(s\mapsto as)$. That is,
\begin{equation}	\label{def:dilate}
\int f(at) d\nu(t) = \int f(t) d\nu_{\#a}(t)
\end{equation}
for any measurable $f$. 

\begin{lemma}	\label{lem:dilate}
Let $h:I\to \R$ be a measurable function on an interval $I\subseteq \R$ and let $a>0$.
For any $\nu\in\cP_1(a^{-1} I)$ we have $\nu_{\#a}\in \cP_{a^2}(I)$, and
\begin{equation}	\label{arg.dilate}
\int h(as)d\nu(s) - \DKL(\nu|\gamma) = \int hd\nu_{\#a} -\DKL(\nu_{\#a}|\gamma) -\frac12(1-a^2)-\log a.
\end{equation}
In particular,
\begin{equation}	\label{opt.dilate}
\sup_{\nu\in \cP_1( a^{-1}\cdot I) } \bigg\{ \int h(as)d\nu(s) - \DKL(\nu|\gamma) \bigg\}
= \sup_{\nu\in \cP_{a^2} (I) } \bigg\{ \int hd\nu - \DKL(\nu|\gamma) \bigg\} -\frac12(1-a^2) - \log a. 
\end{equation}
\end{lemma}

We include the short proof for completeness. 

\begin{proof}
Let $\nu\in\cP_1(a^{-1}I)$. Then from \eqref{def:dilate} it follows that $\nu_{\#a}\in \cP_{a^2}(I)$ and 
\begin{equation}	\label{dilate1}
\int h(as)d\nu(s) = \int h(s) d\nu_{\#a}(s). 
\end{equation}
Moreover, since relative entropy is preserved under simultaneous dilation,
\begin{align*}
\DKL(\nu|\gamma) 
&=\DKL(\nu_{\#a}|\gamma_{\#a}) 
= \int \log \frac{d\nu_{\#a}}{d\gamma}-\log\frac{d\gamma_{\#a}}{d\gamma} d\nu_{\#a}\,.
\end{align*}
Now since $\frac{d\gamma_{\#a}}{d\gamma}(t) = \frac1a\exp( \frac12t^2(1-\frac1{a^2}))$, the last expression above is
\begin{equation}	\label{dilate2}
\DKL(\nu_{\#a}|\gamma) - \int \tfrac12 t^2 (1-a^{-2}) - \log a \,d\nu_{\#a}(t)
= \DKL(\nu_{\#a}|\gamma) + \tfrac12(1-a^2) + \log a. 
\end{equation}
Substituting \eqref{dilate1} and \eqref{dilate2} in the left hand side of \eqref{opt.dilate} yields the claim. 
\end{proof}

As a consequence of Lemma \ref{lem:dilate} we obtain the following scaling property for the functional \eqref{def:VP}:
for any $R\in[1,\infty]$, $v\in \ell^2(\B)$ and $\al>0$, 
\begin{equation}	\label{VP.scaling}
\VP_R(v,\al) = \VP_{R/\sqrt\al}(\sqrt{\al}v,1) + \frac12(1-\al) + \frac12\log\al\,.
\end{equation}

\begin{remark}
\label{rmk:dilate}
Using Lemma \ref{lem:dilate}, we immediately obtain extensions of Proposition \ref{prop:gibbs}(a,b) to the case where we optimize over $\cP_{a^2}(I)$ for some $a>0$. 
Consequently,
\begin{equation}
\sup_{\nu\in\cP_{a^2}(I)}\chi(\nu) = 
\log \int_I e^{\ham(s) - \zeta_{I,\ham}^{(a)}s^2}ds 
+a^2(\zeta_{I,\ham}^{(a)} - \tfrac12) 
- \tfrac12\log(2\pi ).
\end{equation}
Furthermore, by combining Proposition \ref{prop:gibbs}(\ref{gibbs.wiggle}) with Lemma \ref{lem:dilate} we obtain that the mapping
\begin{equation}
\label{aPhi-Lipschitz}
a\mapsto \sup_{\nu\in \cP_{a^2}(I)} \bigg\{ \int hd\nu - \DKL(\nu|\gamma) \bigg\}
\end{equation}
is $O(K/\eps)$-Lipschitz on $[\eps, 1]$.
\end{remark}

\subsection{Proof of Proposition \ref{prop:gibbs}\eqref{gibbs.attain}}
We note that $\chi$ is upper-semicontinuous under the weak topology and strictly concave. 
Since $I$ is compact {and not contained in $(-1,1)$} we have that $\cP_1(I)$ is compact {and nonempty}, and hence $\chi$ attains its maximum on $\cP_1(I)$ at a unique measure $\t\nu=\t\nu_{I,\ham}\in \cP_1(I)$. 
\qed

\subsection{Proof of Proposition \ref{prop:gibbs}\eqref{gibbs.density}}
We use a perturbative argument to determine the form of $\t\nu$. 
For the remainder of the proof we write $L^p(\nu)$ for the space of functions $f$ supported on $I$ with $\int |f|^pd\nu<\infty$ (all functions are supported on $I$), and write $\|f\|_\infty$ for the $L^\infty(\gamma)$ norm of $f$.
Since $\ham$ is continuous on a compact interval we clearly have $\ham\in L^1(\nu)$ for any $\nu\in \cP(I)$. 
Since the optimizer $\t\nu$ must clearly satisfy $\DKL(\t\nu|\gamma)=\int \log\frac{d\t\nu}{d\gamma}d\t\nu<\infty$, we have $\log\frac{d\t\nu}{d\gamma}\in L^1(\t\nu)$. 

 Let $S:=\supp(\t\nu)\subseteq I$ (we will show $S= I$).
Let $f, g\in L^\infty(\gamma)$ with $f$ supported on $S$ and $g\ge0$ supported on $I\setminus S$ (taking $g=0$ if $S=I$). 
For $\eps\in (0,1/\|f\|_\infty)$ let
\begin{equation}
d\nu_\eps = d\t\nu + \eps ( f d\t\nu + g d\gamma). 
\end{equation}
We take $f, g$ to satisfy
\begin{equation}	\label{gh.cond}
\int f d\t\nu + \int gd\gamma =0\,,\qquad \int x^2 f(x)d\t\nu(x) + \int x^2 g(x)d\gamma(x)=0
\end{equation}
so that $\nu_\eps\in\cP_1(I)$. 
Since $\t\nu$ is the maximizer for $\chi$, 
\begin{equation}	\label{opt.bound.barnu}
0\le \chi(\t\nu) - \chi(\nu_\eps)
=  - \eps \int \ham ( fd\t\nu + gd\gamma) + \DKL(\nu_\eps|\gamma)-\DKL(\t\nu|\gamma).
\end{equation}
We expand the relative entropy of $\nu_\eps$ as
\begin{align}
\DKL(\nu_\eps|\gamma) 
&= \int (1+\eps f) \log\Big[(1+\eps f)\frac{d\t\nu}{d\gamma} \Big] d\t\nu
+ \eps \int  g  \log(\eps  g ) d\gamma \notag\\
&= \DKL(\t\nu|\gamma) + \int (1+\eps f)\log(1+\eps f) d\t\nu
+ \eps \int f \log\frac{d\t\nu}{d\gamma}d\t\nu
+ \eps \int  g  \log(\eps g )d\gamma.	\label{mueps.entropy}
\end{align}
Using that $0\le -x+ (1+x)\log(1+x) \le Cx^2$ on $[-1,\infty)$, 
we have 
\begin{equation}
 \int (1+\eps f)\log(1+\eps f) d\t\nu
\le 
\eps \int fd\t\nu + C \eps^2 \|f\|_\infty^2 \le C\eps^2\|f\|_\infty^2
\end{equation}
where in the final bound we used the first equation in \eqref{gh.cond} and the fact that $g\ge0$. 
Combining with \eqref{mueps.entropy} and \eqref{opt.bound.barnu} and dividing through by $\eps$, we get
\begin{align}	
0 &\le -  \int \ham ( fd\t\nu + gd\gamma) 
+ C\eps\|f\|_\infty^2 +  \int f \log\frac{d\t\nu}{d\gamma}d\t\nu
+ \int  g  \log(\eps g )d\gamma	\notag\\
& = C\eps \|f\|_\infty^2  + (\log \eps ) \int gd\gamma 
+ \int \bigg( \log \frac{d\t\nu}{d\gamma} - \ham\bigg) fd\t\nu
+ \int \big( \log g - \ham) gd\gamma.	\label{gh.all}
\end{align}
We first consider this inequality with $g$ set to 0, so that $\nu_\eps$ and $\t\nu$ have the same support. 
The condition \eqref{gh.cond} is now
\begin{equation}	\label{g.cond}
 \int fd\t\nu =  \int x^2f(x)d\t\nu(x) =0
\end{equation}
and we have 
\[
0\le C\eps \|f\|_\infty^2  
+ \int \bigg( \log \frac{d\t\nu}{d\gamma} - \ham\bigg) fd\t\nu
\]
for any $f\in L^\infty(\gamma)$ satisfying \eqref{g.cond} and any $\eps\in(0,1/\|f\|_\infty)$. 
Taking $\eps\downarrow 0$, we have $\int (\log\frac{d\t\nu}{d\gamma} - \ham) fd\t\nu \ge0$, and replacing $f$ with $-f$, we have shown
\begin{equation}	\label{g.annihilate}
\int \bigg( \log \frac{d\t\nu}{d\gamma} - \ham\bigg) fd\t\nu = 0
\end{equation}
for all $f\in L^\infty(\gamma)$ satisfying \eqref{g.cond}. 
Now let $V\subset L^1(\t\nu)$ be the 2-dimensional subspace spanned by the constant function 1 and $x\mapsto x^2$ (these lie in $L^1(\t\nu)$ since $\t\nu\in \cP_1(I)$). 
Since $\t\nu$ is finite, in particular $\sigma$-finite, we have that $L^\infty(\t\nu)$ is (isometrically isomorphic to) the dual of $L^1(\t\nu)$. Identifying $L^\infty(\t\nu)$ with $L^1(\t\nu)^*$, then the orthogonal of $V$, i.e.
\[
V^\perp = \{ f\in L^\infty(\t\nu): \eqref{g.cond} \text{ holds}\}
\]
is isometrically isomorphic to the dual of the quotient space $L^1(\t\nu)/V$. 
Recalling that $\log\frac{d\t\nu}{d\gamma}$ and $\ham$ are in $L^1(\t\nu)$, then \eqref{g.annihilate} says that 
\[
f( \log\frac{d\t\nu}{d\gamma}- \ham+V) = 0 \quad \forall f\in V^\perp
\]
(viewing $f\in L^1(\t\nu)^*$ as a linear functional)
and hence $\log\frac{d\t\nu}{d\gamma}- \ham$ lies $V$, i.e.
\begin{equation}	\label{barnu.aex}
\log\frac{d\t\nu}{d\gamma}(x) = \ham(x) - \alpha x^2 - \beta\qquad \t\nu\text{-a.e.}\; x
\end{equation}
for some $\alpha,\beta\in \R$.

Returning to \eqref{gh.all}, we now suppose $S\ne I$, and take $g=1_{I\setminus S}$. Substituting \eqref{barnu.aex} into \eqref{gh.all} and rearranging, we have
\begin{align*}
\gamma(I\setminus S) \log\frac1\eps \le C\eps \|f\|_\infty^2 -\zeta\int x^2f(x)d\t\nu(x) - \eta \int fd\t\nu - \int_{I\setminus S} \ham d\gamma\,.
\end{align*}
Since the right hand side is uniformly bounded for $\eps\in (0,1/\|f\|_\infty)$, we obtain a contradiction for $\eps$ sufficiently small. 
Hence, $S=\supp(\t\nu)=I$, and from \eqref{barnu.aex} we get that $\t\nu$ has density on $I$ with respect to Lebesgue measure proportional to $\exp( \ham(x) - (\alpha+\frac12) x^2)$, with the value of $\zeta:=\alpha+\frac12$ {uniquely} determined by the constraint that $\int x^2d\t\nu(x)=1$.
\qed

\subsection{Proof of Proposition \ref{prop:gibbs}\eqref{gibbs.W2cont}}
{To show continuity at $a=a_0$, by replacing $I$ with $a_0^{-1}I$ it suffices to consider $a_0=1$. Since $\t\nu_a$ is supported on the compact interval $2I$ for all $a\ge1/2$, we see that continuity under the $\cW_2$ metric is equivalent to continuity under the $\cW_1$ metric $\cW_1(\mu,\nu)=\sup\{ \int f d(\mu-\nu)\}$, where the supremum is taken over 1-Lipschitz functions $f:\R\to \R$. From dominated convergence we then see it suffices to show that $a\mapsto\zeta(a):= \zeta_{a^{-1}I,h(a\cdot)}$ is continuous. From  a  change of variable we see that $\zeta(a)=-a^2\beta(a)$, where $\beta(a)$ is the unique value of $\beta\in\R$ for which $F(\beta)=a^2$, where
\[
F(\beta):= \frac{\int_I x^2 e^{h(x)+\beta a^2 x^2}dx}{\int_I e^{h(x)+\beta a^2 x^2}dx}.
\] 
Since $h$ is bounded and $I$ is compact we see that $F$ is smooth, and by the implicit function theorem it suffices to show that $F'(\beta(1))>0$. But since $F'(\beta(1))$ is the variance of $X^2$ for $X$ having law $\t\nu_1$, we see that $F'(\beta(1))=0$ would imply  that  $\t\nu_1$ is a discrete measure supported on the two points $\pm \sqrt{F(\beta(1))}$. Since $\t\nu_1$ is a continuous measure we obtain a contradiction, and hence $F'(\beta(1))>0$. }

{For the quantitative bound \eqref{gibbs.W2}, consider} the probability measure $Q_0$ on $[0,\infty)$ with density proportional to
\begin{equation}
dQ_0(y)\propto 1_{[0,R^2]}(y) \frac1{\sqrt{y}} \exp\big( \ham(\sqrt{y})\big)dy
\end{equation}
and define the exponentially tilted measures 
\begin{equation}
dQ_\beta(y) := \frac{e^{\beta y}dQ_0(y)}{\int e^{\beta y} dQ_0(y)}\,,\qquad \beta\in\R\,.
\end{equation}
Let $Y_\beta\sim Q_\beta$, and for $a>0$ and $\eps$ an independent Rademacher variable set $X_{\beta,a}:= \frac1a\eps \sqrt{Y_\beta}$. 
Then $X_{\beta,a}$ has density proportional to
\[
e^{ \ham(ax) + \beta a^2x^2 }1_{[-R,R]}(ax)  dx\,.
\]
Hence, letting $\zeta(a)$ be the unique real such that $\t\nu_a=\nu^{\zeta(a)}$ (as provided by part (b)), and 
{with
$
\beta(a)= -\frac1{a^2} \zeta(a)
$
as above}
we have $X_{\beta(a),a}\sim \t\nu_a$. In particular,
\begin{equation}
1 = \E X_{\beta(a),a}^2 = \frac1{a^2} \E Y_{\beta(a)}\,.
\end{equation}
This and the fact that $\beta\mapsto \E Y_\beta$ is clearly increasing implies that $a\mapsto\beta(a)$ is increasing. 

Now fix $0<a\le b$. We construct a coupled pair $(X_a,X_b)$ with marginals $(\t\nu_a,\t\nu_b)$ as follows. Let $(Y_{\beta(a)}, Y_{\beta(b)})$ be a monotone coupling with marginals $(Q_{\beta(a)},Q_{\beta(b)})$ (see \eqref{scalar-coupling} for a construction). Thus $Y_{\beta(a)}\le Y_{\beta(b)}$ almost surely.
Now as before we let $\eps$ be a Rademacher independent of $(Y_{\beta(a)}, Y_{\beta(b)})$, and set $X_a= \frac1a\eps\sqrt{Y_{\beta(a)}}$, $X_b= \frac1b\eps \sqrt{Y_{\beta(b)}}$. 
We have
\begin{align*}
\cW_2(\t\nu_a,\t\nu_b)^2
&\le \E |X_b-X_a|^2 \\
&= \E | \frac1b \sqrt{Y_{\beta(b)}} - \frac1a\sqrt{Y_{\beta(a)}}|^2\\
&= \frac1{b^2} \E |\sqrt{Y_{\beta(b)}} - \sqrt{Y_{\beta(a)}}|^2
+ ( a^{-2}-b^{-2})\E Y_{\beta(a)} - \frac2b(a^{-1}-b^{-1}) \E \sqrt{Y_{\beta(a)}Y_{\beta(b)}}\\
&\le \frac1{b^2} \E (Y_{\beta(b)}- Y_{\beta(a)})
+ ( a^{-2}-b^{-2})\E Y_{\beta(a)} \\
&= \frac1{b^2}( b^2-a^2) + (a^{-2}-b^{-2}) a^2 = 2( 1-\frac{a^2}{b^2})
\le 4 (1-\frac{a}b)
\end{align*}
where in the fourth line we dropped the last (non-negative) term, and bounded the first term using the inequality $(\sqrt{x}-\sqrt{y})^2\le x-y$ for $0\le y\le x$. 
\qed

\subsection{Proof of Proposition \ref{prop:gibbs}\eqref{gibbs.wiggle}}

We begin with \eqref{opt.stab1}.
Obviously, we have
\[\Upsilon:=
\sup_{\nu\in \cP_1^{\delta}( [-(1+\eps)R,(1+\eps)R])}\big\{ \nu(\ham) - \DKL(\nu|\gamma) \big\}
\ge \sup_{\nu\in \cP_{1}( [-R,R])}\big\{ \nu(\ham) - \DKL(\nu|\gamma) \big\} .
\]
On the other hand
\[\Upsilon=\sup_{\nu\in \cP_1^{\delta}( [-(1+\eps)R,(1+\eps)R])}\big\{ \nu(\ham) - \DKL(\nu|\gamma) \big\}=\sup_{c\in [1-\delta,1+\delta]}\sup_{\nu\in \cP_{c}( [-(1+\eps)R,(1+\eps)R])}\big\{ \nu(\ham) - \DKL(\nu|\gamma) \big\}.
\]
For $\nu\in \cP_{c}( [-(1+\eps)R,(1+\eps)R])$, we let $\nu_{c}$ be such that $\int f(x) d\nu_{c}(x)=\int f(\frac{x}{\sqrt{c}})d\nu(x)$. We have seen {in the proof of Lemma \ref{lem:dilate}} that
\[
\DKL(\nu|\gamma)- \DKL(\nu_{c}|\gamma)=\frac{1}{2}\ln c+(c-1)
\]
and since 
{$h\circ\ssq^{-1}$} is $K$-Lipschitz and $\nu_{c}$ has 
unit second moment, we have 
$|\int hd\nu-\int hd\nu_{c}|\le K |c-1|$. We observe that $\nu_{c}$ has support in $[-R/a,R/a]$ with $a=\frac{\sqrt{c}}{(1+\eps)}$ and that $|\int h(x)d\nu_{c}(x)-\int h(ax) d\nu_{c}(x)|\le K |a^{2}-1|$. Therefore, 
\[
\Upsilon\le \sup_{|a-1|\le \delta+\epsilon } 
\sup_{\nu\in \cP_{1}( [-R/a,R/a])}\bigg\{ \int \ham  (
ax )
d\nu(x) - \DKL(\nu|\gamma) \bigg\}+O(\delta +\epsilon).
\]
To bound the right hand side, we  notice that from Proposition \ref{prop:gibbs}\eqref{gibbs.density}, the maximizer is given by  $\t\nu_a$ of the form
\[
d\t\nu_{a}(x)=1_{[-R/a,R/a]}(x)\frac{e^{ \ham(a x)-\zeta_{a} x^2} dx}{ \int_{[-R/a,R/a]}e^{ \ham(ay) -\zeta_{a}y^2} dy}
\]
so that
\begin{align}
I(a,R)&:=\sup_{\nu\in \cP_{1}( [-R/a,R/a])}\bigg\{ \int \ham  (
ax )
d\nu(x) - \DKL(\nu|\gamma) \bigg\}	\notag\\
&=\zeta_{a}-\frac{1}{2}+\log \int_{[-R/a,R/a]}e^{ \ham(ay) -\zeta_{a}y^2} dy	
-\frac12\log(2\pi)	\label{opt-zetaa}\\
&=-(\beta_{a}a^{2}+\frac{1}{2})-\log a+\log \int_{[-R,R]}e^{h(x)+\beta_{a}x^{2}}
dx	-\frac12\log(2\pi)	\label{opt-betaa}
\end{align}
where we have set $\zeta_{a}=-\beta_{a} a^{2}$ and rescaled the integral in the above right hand side. We denote by $\bar\nu_{a}$ the probability {measure} 
on $[-R,R]$ obtained by rescaling $\t\nu_{a}$ by $a$.
Then, for any $a,a'$ in a neighborhood of one,
\[
I(a,R)-I(a',R)=\ln\int_{[-R,R]} e^{{(\beta_{a}-\beta_{a'})x^{2}}} d\bar\nu_{a'}-(\beta_{a} -\beta_{a'})a^{2}-\ln a/a' 
-\beta_{a'}(a^{2}-(a')^{2})\,.
\]
We may assume without loss of generality that $I(a,R)-I(a',R)\le 0$ up to exchange $a$ and $a'$.
By Jensen's inequality and $\bar\nu_{a'}(x^{2})=(a')^{2}$, we deduce that
\begin{equation}\label{Lip}
I(a,R)-I(a',R)\ge (\beta_{a} -\beta_{a'})((a')^{2}-a^{2})-\log a/a'
-\beta_{a'}(a^{2}-(a')^{2}).
\end{equation}
Therefore we see that it is enough to show that $\beta_{a}=\beta_{a}(R)$ is bounded uniformly in $R$ and $a$ in a neighborhood of one to conclude that $I(a,R)$ is Lipschitz with a bounded  Lipschitz norm. Recall that $\beta_{a}(R)$ is defined as the point where the function
\[
F(\beta,R)= \frac{\int_{[-R,R]} x^{2}e^{h(x)+\beta x^{2}}dx}{\int_{[-R,R]} e^{h(x)+\beta x^{2}}dx} 
\]
attains the value $a^2$.
But $\partial_{\beta} F(\beta,R) $ is nonnegative (since it is equal to the covariance)  as well as 
\[
\partial_{R} F(\beta,R)=(R^{2}- \frac{\int_{[-R,R]} x^{2}e^{h(x)+\beta x^{2}}dx}{\int_{[-R,R]} e^{h(x)+\beta x^{2}}dx})\frac{e^{h(R)+\beta R^{2}}}{\int_{[-R,R]} e^{h(x)+\beta x^{2}}dx}\ge 0\,.\]
This implies that $R\rightarrow \beta_{a}(R)$ is decreasing so that for $R\ge 2$
\[\beta_{a}(\infty)\le \beta_{a}(R)\le \beta_{a}(2)\,.\]
Note here that $R$ needs to be greater or equal to $a^{2}$ to insure the existence of $\beta_{a}(R)$ and we may take $R\ge 2$ since in any case we will focus on $R$ large. When $R$ is infinite, we can use that $-K x^{2}\le h(x)\le K x^{2}$ to see that
\[
F(\beta,\infty)\le 
 \frac{(K-\beta)^{1/2}}{(-K-\beta)^{3/2}} .
\]
Since the right hand side tends to zero as $\beta\to-\infty$ it follows
that $\beta_{a}(\infty)$ is bounded below  by a finite constant, uniformly for $a$ in a neighborhood of one. At $R=2$, $\beta_{a}(2)=F^{-1}(.,2)(a^{2})$ with $F(.,2)$ strictly increasing, continuous  and $F(\beta,2)$ goes to $4$ as $\beta$ goes to infinity, hence the set of  $\beta$ such that $F(\beta,2)\le 2$ is bounded, implying that $\beta_{a}(2)$ is bounded uniformly for $a\le \sqrt{2}$. As a consequence $\beta_{a}(R)$ is uniformly bounded for $a$ in the vicinity of one, uniformly in $R\ge 2$, which allows to conclude that $I(a,R)$ is Lipschitz in $a$, uniformly in $R\ge 2$.
This concludes the proof of \eqref{opt.stab1}. \qed

\subsection{Proof of Proposition \ref{prop:gibbs}\eqref{gibbs.lim}}
From Lemma \ref{lem:dilate} we may assume $\al=1$. 
We denote the expressions on the left and right hand sides by $\phi_R$ and $\phi_\R$, respectively. 
Since $\phi_R$ is clearly monotone increasing in $R$ and bounded by $\phi_\R$, it suffices to show
\begin{equation}	\label{optd.goal}
\phi_R \ge 
\phi_\R-o_{R\to\infty}(1).
\end{equation}
{Fix an arbitrary $\eps>0$ and let $\nu^\eps\in\cP_1(\R)$ be such that 
\[
\nu^\eps(\ham)-\DKL(\nu^\eps|\gamma) \ge \phi_\R -\eps\,.
\]
For each $R>0$ denote the truncated measure $d\nu^\eps_R:= \nu^\eps([-R,R])^{-1}1_{[-R,R]}d\nu^\eps\in \cP([-R,R])$. 
For any $f:\R\to\R$ such that $f(0)=0$ and 
{$f\circ\ssq^{-1}$} is $K$-Lipschitz, we have that $f$ is dominated by the $\nu^\eps$-integrable function $s\mapsto Ks^2$, so from the dominated convergence theorem we have
\begin{equation}	\label{opt.stab2a}
\nu^\eps_R(f) = \nu^\eps([-R,R])^{-1} \int_{-R}^R fd\nu^\eps(s) \to \nu^\eps(f)
\end{equation}
as $R\to\infty$. 
In particular, with $f(s)=s^2$ we obtain that $\nu^\eps_R\in \cP_1^{\eps}([-R,R])$ 
for all $R$ sufficiently large.
Moreover,
\[
\DKL(\nu^\eps_R|\gamma) = \frac{\DKL(\nu^\eps|\gamma)}{\nu^\eps([-R,R])} - \log\nu^\eps([-R,R]) \to\DKL(\nu^\eps|\gamma)
\]
so combining with \eqref{opt.stab2a} with $f=\ham$ we get
\begin{equation}
\nu^\eps_R(\ham) - \DKL(\nu^\eps_R|\gamma) = \nu^\eps(\ham) - \DKL(\nu^\eps|\gamma) +o_{R\to\infty}(1) \ge \phi_\R -\eps+o_{R\to\infty}(1). 
\end{equation}
Now applying \eqref{opt.stab1}, we have that for all $R$ sufficiently large, 
\begin{align*}
\phi_R &\ge \sup_{\nu\in\cP_{(1-\delta,1+\delta)}([-R,R])} \{\nu(\ham)-\DKL(\nu|\gamma)\}-O(K\eps)\\
&\ge \nu^\eps_R(\ham) - \DKL(\nu^\eps_R|\gamma) - O(K\eps)\\
&\ge \phi_\R - O((K+1)\eps) +o_{R\to\infty}(1).
\end{align*}
Taking $\eps$ to zero gives \eqref{optd.goal} to complete the proof. }
\qed

\section{Annealed asymptotics for restricted spherical integrals}
\label{sec:annealed}

\subsection{
Proof of Propositions \ref{prop:FE} and \ref{prop:tiltQ}}

Our purpose in this section is to prove the following proposition, which immediately yields Propositions \ref{prop:FE} and \ref{prop:tiltQ}.

\begin{prop}
\label{prop:annealed}
Let $\eta\in (0,\frac14)$, $\zcush\in(0,\frac12)$, {$\delta\in[0,1]$} and $1\le T<\infty$, and let $R,w$ (possibly depending on $\theta\in[0,T]$) satisfy
\begin{equation}	\label{wR.bounds1}
R\in [\log N, N^{1/4}]\,,\qquad
w= {w'+w^\le} \in(1-\zcush)\B^{n_0}\,,\qquad { \|w^\le\|_\infty\le  \delta/R } \, .
\end{equation}
{where $w',w^\le$ have disjoint supports.}
\begin{enumeratea}
\item \label{annealed.UB} (Upper bound).
For any $\theta\in[0,T]$, $\srad\in [0,\frac\zcush{10}]$ and all $N$ sufficiently large depending on $T,\zcush$,
\begin{equation}
F_N(\theta;\Uloc_w(\srad,R))	\le \free_{N,R}(\theta,w)+ 
O_{T,\zcush}( \srad + R^2N^{-1/2} + N^{-2\eta}\log N)\,. \label{FE+}
\end{equation}

\item \label{annealed.LB} (Lower bound).
There exists $\t v_{\theta,w,R}\in \R^{\supp(w)^c}\cap\Deloc_R \cap\sqrt{1-\|w\|_2^2}\sphereN$ and $C_0(\zcush)>0$ depending only on $\zcush$ such that, with $\srad_0:=C_0(\zcush)RN^{-1/2}$, we have
\begin{align}		
&F_N\Big(\theta; \,\Uloc_w(N^{-4},R) \cap\Bset_2\big(w+\t v_{\theta,w,R}\,,\,\srad_0\big)\Big)	\label{FE-}\\
 &\qquad \ge \free_{N,R}(\theta,w) - O_{T,\zcush}( {\delta+}R^2N^{-1/2} + N^{-2\eta}\log N) 
\notag
\end{align}
{for all $N$ sufficiently large depending on $T,\zcush$. Moreover, $\t v$ depends on $w$ only through $w'$ and $\|w\|_2$.}
\item (Continuity). 
Let $\theta_0\in(0,T)$ and suppose that $R,w$ depend on $\theta\in[\theta_0,T]$ in the following way (in addition to satisfying \eqref{wR.bounds1} for all $\theta\in[\theta_0,T]$): for $\vloc'\in\B^{n_0}$ and $R'\ge1$ independent of $\theta$, and continuous functions $q:[\theta_0,T]\to (0,1)$, $\alpha:[\theta_0,T]\to [0,1-\zcush]$, we have
{$w'=q(\theta)\vloc'$, $\|w^\le\|_2^2=\alpha(\theta)$}, and $R= R'/(\theta q(\theta))$.
{Then we can take the vector $\t v_{\theta,w,R}$ from part (b) to depend continuously on $\theta\in [\theta_0,T]$.} 
\end{enumeratea}
\end{prop}

\begin{remark}
The proof shows that $C_0(\zcush)$ and the implicit constants in \eqref{FE+}, \eqref{FE-} are polynomial in $T$ and $1/\zcush$.
\end{remark}

{
\begin{remark}
For the application to prove Theorem \ref{theomain}  it suffices to take $w=w'$, $w^\le=0$ and $\delta=0$. 
For Theorem \ref{thm:fullLDP} it is important  that $\t v^{(\theta)}$ can be taken to depend only on the large coordinates of $w$ (of size $>\delta/R$).
\end{remark}
}

\begin{proof}[Proof of Proposition \ref{prop:FE}] 
Note that the upper bound in \eqref{FE.asymp} follows from \eqref{FE+}, while
the lower bound follows from \eqref{FE-} {with $w^\le=0$},
noting that the left hand side in \eqref{FE-} is bounded above by $F_N(\theta; \Uloc_w(\srad, R))$ (since $F_N(\theta;\Uloc)$ is monotone increasing in $\Uloc$).
\end{proof}

\begin{proof}[Proof of Proposition \ref{prop:tiltQ}] 
For part (a), we only need to note that for any nonempty measurable set $\Bset \subseteq\sphereN$, we have
\begin{align*}
\frac1N\log Q^{(\theta)}(\Bset |  \Uloc_w(\srad,R))
&= F_N(\theta; \Uloc_w(\srad,R) \cap \Bset ) - F_N(\theta;\Uloc_w(\srad,R))\\
&\ge F_N(\theta; \Uloc_w(N^{-4},R) \cap \Bset ) - F_N(\theta;\Uloc_w(\srad,R))
\end{align*}
so we get the lower bound \eqref{tiltQ:conc} by subtracting \eqref{FE+} from \eqref{FE-}. 
Proposition \ref{prop:tiltQ}(b) is 
a restatement of Proposition \ref{prop:annealed}(c). 
\end{proof}

Toward the proof of Proposition \ref{prop:annealed}, we begin by gathering some lemmas, deferring the proofs to later subsections. 
For $y,z\in {\ell^2(\N)}$ 
set 
\begin{align}
\wt f_N   (y,z) := \frac1N\sum_{1\le i, j\le N} \LLa_\mu\big( 2\sqrt{N} y_iz_j\big)	\label{def:cFNloc}
\end{align}
{and recall $\freeL(\theta,z)$ defined in \eqref{def:freeL}. B}y Fubini's theorem {(and recalling $\LLa_\mu(0)=0$)},
\begin{align*}
F_N(\theta;\Uloc) &= \frac1N\log \int_\Uloc \exp\big( N\freeL(\theta,u)\big)dP(u)\\
 &= \frac1N\log \int_\Uloc \exp\Big(N\big( \freeL(\theta,u_w) + \freeL(\theta,u_{w^c}) + \wt f_N   (\theta u_w,u_{w^c})\big)\Big)dP(u).
\end{align*}
The first step in the proof of Proposition \ref{prop:annealed} is to peel off the contributions $\freeL(\theta,u_w)$ and $\freeL(\theta,u_{w^c})$ of pairs of coordinates in the localized and delocalized parts of $u$, respectively, which we do in the following:

\begin{lemma}
\label{lem:annealed1}
Let $\zcush\in(0,\frac12)$, $w\in(1-\zcush)\ball^N$, $\srad\ge0$ and $0\le R\le N^{2/5}$.
{Suppose $w=w'+w^\le$ where $w',w^\le$ have disjoint supports and $|w^\le_j|\le \delta/R$ for all $j$ and some $\delta\le 1$.}
For an arbitrary nonempty measurable set $\Uloc\subseteq\Uloc_w(\srad,R)$,
\begin{align}
\frac1N\log\E\int_\Uloc e^{\theta N\langle u, Hu\rangle}dP(u)	
&=\frac1N\log\int_\Uloc \exp\bigg( N \wt f_N   \Big( \theta {w'}, \sqrt{1-\|w\|_2^2} \frac{u_{w^c}}{\|u_{w^c}\|_2}\Big)\bigg) dP(u)	\notag\\
&
 \quad+
\freeL(\theta,w) + 
\theta^2(1-\|w\|_2^2)^2
{+ 2\theta^2\|w^\le\|_2^2(1-\|w\|_2^2)}
\label{annealedB}\\
&\quad+O(\theta^2\srad /\sqrt{\zcush}
+\theta^3{(\delta+}R^2N^{-1/2}{)})	\,.	\notag
\end{align}
\end{lemma}

With Lemma \ref{lem:annealed1} in hand, the main step of the proof of Proposition \ref{prop:annealed} is to analyze the integral on the right hand side in \eqref{annealedB}.
The following quantitative Varadhan-type lemma 
reduces the integral to a Gibbs variational problem of the general form treated in Proposition \ref{prop:gibbs}.
Recall our notation $\Deloc_R=\Deloc_R^N:= \{ v\in\B^N: \|v\|_\infty \le RN^{-1/2}\}$ for the set of $R$-delocalized vectors,
and $\Bset_2(u,\eps)=\{v\in\R^N: d_2(u,v)<\eps\}$  for the Wasserstein $\eps$-neighborhood of a vector $u\in\R^N$ (see Section \ref{sec:notation}).

\begin{lemma}
\label{lem:Varadhan}
Let $R>1$, $I:=[-R,R]$ and $\ham:\R\to \R$ 
be such that 
{$\ham\circ\ssq^{-1}$} is $K$-Lipschitz on $2I$ for some $K\ge1$ {(recall the notation \eqref{def:ssq})}. 
We have
\begin{align}
\label{varadhan.UB}
\frac1{N}\log \int_{\Bdeloc_R} \exp\Big(  \sum_{j=1}^N \ham(\sqrt{N}u_j) \Big) dP_{N}(u)	
\le  \sup_{\nu\in \cP_1(I)} \big\{ \nu(\ham) - \DKL(\nu|\gamma) \big\} +
O\bigg(\frac{K\log N}{\sqrt{N}}\bigg)\,.
\end{align}
Moreover, 
the following holds for any $\srad'\in [C_0N^{-1/2}(R+\log N), K^{-1}]$ for a sufficiently large constant $C_0>0$.
Let $\t\nu$ be any 
element of $\cP_1((1-\srad')I)$ such that
\begin{equation}	\label{def:nearopt}
\t\nu(\ham)-\DKL(\t\nu|\gamma) \ge \sup_{\nu\in\cP_1((1-\srad')I)} \{ \nu(\ham)-\DKL(\nu|\gamma) \} -K\srad'
\end{equation}
and let the coordinates of $\t y\in \R^N$ be the non-increasing $\frac1N$-quantiles of $\t\nu$, i.e.
\begin{equation}	\label{def:quantiles}
\t y_i = \sup\{ t: \t\nu([t,\infty))\ge \tfrac iN\} \,,\qquad i\in [N].
\end{equation}
Then
\begin{align}
\frac1N\log\int_{\Bdeloc_R\cap \Bset_2(\frac1{\sqrt{N}}\t y,\srad')}  \exp\Big(  \sum_{j=1}^N \ham(\sqrt{N}u_j) \Big) dP_{N}(u)
\ge 
\sup_{\nu\in \cP_1(I)} \big\{ \nu(\ham) - \DKL(\nu|\gamma) \big\} - O(K\srad').
\label{varadhan.LB}
\end{align}
\end{lemma}

The proof is an elaboration of arguments developed in \cite[Lemma 9]{AGH} and is deferred to Appendix \ref{app:Varadhan}.

The final ingredient for the proof of Proposition \ref{prop:annealed} is the following:

\begin{lemma}
\label{lem:Ploc}
Let $w\in \R^N$ with $|\supp(w)|=n_0\le N/2$ and $\|w\|_2\le 1-\zcush$ for some $\zcush>0$, and let $\srad\in [N^{-4},\frac1{10}\zcush]$. 
We have
\[
\log P_N(u:\|u_w-w\|_2< \srad) = \frac{N}2\log(1-\|w\|_2^2) + 
O\bigg(\,\frac\srad\zcush N + n_0\log\frac N\zcush\,\bigg) \,.
\]
\end{lemma}


\begin{proof}[Proof of Proposition \ref{prop:annealed}]
For ease of notation we write 
\[
N':=N-n_0
\]
and take $\supp(w)=[N'+1,N]$. 
Consider for now a general nonempty measurable set $\cV\subseteq\ball^{N'}$ and let
\begin{align*}
\Uloc &= \Uloc_w(\srad,R)\cap (\cV\times \R^{[N'+1,N]})
= \big\{u\in \sphereN: \|u_w-w\|_2\le \srad, u|_{[N']}\in \Deloc^{N'}_R\cap \cV\big\}\,.
\end{align*}
Applying Lemma \ref{lem:annealed1} we have
\begin{align}
F_N(\theta;\Uloc)	&= 
\freeL(\theta,w) + \theta^2(1-\|w\|_2^2)^2
{+ 2\theta^2\|w^\le\|_2^2(1-\|w\|_2^2)} \label{FE.1}\\
&\quad+ \frac1N\log\int_{\Uloc} \exp\bigg( N \wt f_N   \Big( \theta {w'}, \sqrt{1-\|w\|_2^2} \frac{u_{w^c}}{\|u_{w^c}\|_2}\Big)\bigg) dP(u) \notag\\
&\quad+O_\zcush(\theta^2\srad)+ O(\theta^3({\delta}+R^2N^{-1/2}))	\,.	\notag
\end{align}
Now we can express the integral on the right hand side as an iterated integral
\begin{align}
&\int_{\Uloc} \exp\bigg( N \wt f_N   \Big( \theta w', \sqrt{1-\|w\|_2^2} \frac{u_{w^c}}{\|u_{w^c}\|_2}\Big)\bigg) dP(u)		\label{tiltQ.cross1}\\
&=\int_{\B^{[N'+1,N]}} 1_{\|z-w\|_2\le \srad} \int_{\frac1{\sqrt{1-\|z\|_2^2}}(\cV\cap \Deloc^{N'}_R)} 
e^{ 
N \wt f_N   \big(\theta w', \sqrt{1-\|w\|_2^2} v\big) 
}
dP_{N'}(v) dP_{n_0}^{(w)}(z)	\notag
\end{align}
where $P_{N'}$ is the uniform measure on $\sphere^{N'-1}$, and $P_{n_0}^{(w)}$ is the marginal law of $u_w=u|_{[N'+1,N]}$ on $\B^{[N'+1,N]}$ for $u\sim P_N$. 
Note that the inner integral depends on $z$ only through the domain of integration. 
Now letting
\begin{equation}	\label{choose.ham}
\ham:\R\to \R\,,\qquad \ham(s) =\ham(s;w):= \sum_{i=1}^N \LLa_\mu\bigg( 2\theta \sqrt{ \frac{N}{N'}} \sqrt{1-\|w\|_2^2} w_i' s \bigg)
\end{equation}
so that
\[
N \wt f_N   \Big(\theta w', \sqrt{1-\|w\|_2^2} v\Big) = \sum_{j=1}^{N'} \ham( \sqrt{N'} v_j)
\]
we have that 
$h\circ\ssq^{-1}$ is $O(\theta^2)$-Lipschitz, where we recall the notation \eqref{def:ssq}. Indeed,
 for any $s,t\in\R$ we have
\begin{align}
&|\ham({\ssq^{-1}(s)})-\ham({\ssq^{-1}(t)})|\notag \\
&\le \sum_{i=1}^N \Big|\,\LLa_\mu\big(  \ssq^{-1}( 4\theta^2(N/N')(1-\|w\|_2^2)w_i'^2\sgn(w_i')s)\,\big) \notag\\
&\qquad\qquad\qquad-
\LLa_\mu\big( \ssq^{-1}( 4\theta^2(N/N')(1-\|w\|_2^2)w_i'^2\sgn(w_i')t)\,\big)	\,\Big|\notag\\
&\le 4\|\tLL\|_{\Lip} \,\theta^2   (1-\|w\|_2^2) \frac{N}{N'} |s-t|\sum_{i=1}^N   w_i'^2 \notag\\
&\ls \theta^2 \|w\|_2^2(1-\|w\|_2^2) |s-t|	\notag\\
&\ls \theta^2|s-t|		\label{ham.lip}
\end{align}
where in the third bound we used Remark \ref{rmk:assu.reg}.

We now establish the upper bound \eqref{FE+}, for which we apply \eqref{tiltQ.cross1} with $\cV=\B^{N'}$. 
In the sequel we abbreviate
\begin{equation}
R_z:= R/ (1-\|z\|_2^2)^{1/2}
\end{equation}
for $z\in \B^N$. 
Recall from \eqref{VP.scaling} that
\begin{equation}
\VP_R(\theta w, 1-\|w\|_2^2) -\frac12\|w\|_2^2=
\VP_{R_w}\big(\theta (1-\|w\|_2^2)^{1/2}w, \,1\big) +
\frac12\log(1-\|w\|_2^2)\,.
\end{equation}

By monotonicity of $F_N(\theta;\Uloc)$ in $\Uloc$ we may assume without loss of generality that $\srad\ge N^{-1/2}$. 
For the inner integral on the right hand side of \eqref{tiltQ.cross1}, we apply Lemma \ref{lem:Varadhan} with $N'$ in place of $N$, $R_z\ls_\zcush R$ in place of $R$, and $K=O(1+\theta^2)=O_T(1)$ to get that for any fixed $z\in\B^{[N'+1,N]}$ with $\|z-w\|_2\le \srad$, 
\begin{align*}
&\frac1{N'}\log \int_{ \Deloc^{N'}_{R_z}} \exp \Big( N \wt f_N   \Big(\theta w', \sqrt{1-\|w\|_2^2} v\Big) \Big) dP_{N'}(v)\\
&\le \sup_{\nu\in\cP_1([-R_z,R_z])} \big\{ \nu(\ham)-\DKL(\nu|\gamma)\big\}
+ O_T(N^{-1/2}\log N)\\
&=\sup_{\nu\in\cP_1([-R_w,R_w])} \big\{ \nu(\ham)-\DKL(\nu|\gamma)\big\}
+ O_T(N^{-1/2}\log N) + O_{T,\zcush}(\srad)
\end{align*}
where in the final line we applied Proposition \ref{prop:gibbs}\eqref{gibbs.wiggle} to replace $z$ with $w$.
Substituting back into \eqref{tiltQ.cross1} (with $\cV=\B^{N'}$), taking logs and dividing through by $N$, we have
\begin{align*}
&\frac1N\log\int_{\Uloc_w(\srad,R)} \exp\bigg( N \wt f_N   \Big( \theta w', \sqrt{1-\|w\|_2^2} \frac{u_{w^c}}{\|u_{w^c}\|_2}\Big)\bigg) dP(u)		\\
&= \frac{N'}{N} 
\sup_{\nu\in\cP_1([-R_w,R_w])} \big\{ \nu(\ham)-\DKL(\nu|\gamma)\big\} + \frac1N\log P_N(\{u: \|u_w-w\|_2\le \srad\})\\
&\qquad\qquad\qquad\qquad\qquad\qquad\qquad\qquad\qquad
+ O_T(N^{-1/2}\log N) + O_{T,\zcush}(\srad)\\
&\le
\sup_{\nu\in\cP_1([-R_w,R_w])} \big\{ \nu(\ham)-\DKL(\nu|\gamma)\big\} + \frac12\log(1-\|w\|_2^2) \\
&\qquad\qquad\qquad\qquad\qquad\
+ O_T(N^{-1/2}\log N) + O_{T,\zcush}(\srad) + O_\zcush\Big(\frac{n_0}N\log N\Big) 
\end{align*}
where in the final line we applied Lemma \ref{lem:Ploc} and bounded $N'\le N$. 
Substituting this bound into \eqref{FE.1} (and recalling $n_0=|\supp(w)|\le N^{1-2\eta}$), we obtain the upper bound \eqref{FE+}.

Turning to the lower bound \eqref{FE-}, 
we now set
$\srad:=N^{-4}$.
With $\t v=\t v_{w, \theta,R}$ to be specified below, from the triangle inequality for the Wasserstein distance and adjusting the constant $C_0(\zcush)$,  it suffices to show
\begin{align}
\label{FE.LBgoal}
&F_N\big(\theta; \Uloc_w(\srad,R)\cap \big( \Bset_2^{N'}(\t v, \srad_0)\times \R^{[N'+1,N]}\big) \big)\\
&\qquad\ge \free_{N,R}(\theta,w) + O_{T,\zcush}( {\delta+}R^2N^{-1/2} + N^{-2\eta}\log N).\notag
\end{align}

As in the proof of the upper bound we apply \eqref{FE.1}--\eqref{tiltQ.cross1}, only now we take $\cV= \Bset_2^{N'}(\t v, \srad_0)$, to get
\begin{align}
&F_N(\theta; \Uloc_w(\srad,R)\cap (\Bset_2^{N'}(\t v, \srad_0)\times \R^{[N'+1,N]}) )	\notag\\
&\ge \frac1N \log \int_{\B^{[N'+1,N]}} 1_{\|z-w\|_2\le \srad} 
\int_{ \Deloc_{R_z}^{N'} \cap \Bset_2^{N'} ( \frac{\t v}{\sqrt{1-\|z\|_2^2}} ,\srad_z)} \exp\bigg( \sum_{j=1}^{N'} \ham(\sqrt{N'} v_j)\bigg) dP_{N'}(v) dP_{n_0}^{(w)}(z)\notag\\
&\qquad \qquad + 
\freeL(\theta,w) + + \theta^2(1-\|w\|_2^2)^2
{+ 2\theta^2\|w^\le\|_2^2(1-\|w\|_2^2)}
+ O_{T,\zcush}({\delta+}R^2N^{-1/2})	\label{FE.LB2}
\end{align}
where $\ham$ is as in \eqref{choose.ham}, and here and in the sequel we abbreviate
\begin{equation}
\srad_z:=\srad_0/(1-\|z\|_2^2)^{1/2}
\end{equation}
for $z\in \B^N$. 

Now we specify $\t v$. 
Let $\t\nu=\t\nu_{I_*,\ham}$ be the optimizing measure for $\chi$ from \eqref{def:objective} over $\cP_1(I_*)$ with 
\[
I_*:=(1-C_*(\zcush)\srad_0)[-R_w,R_w]
\]
for a constant $C_*(\zcush)\ge1$ to be chosen sufficiently large depending only on $\zcush$, 
 let $\t y\in \R^{N'}$ be as in \eqref{def:quantiles} with $N'$ in place of $N$, and set
 \begin{equation}	\label{def:tv}
\t v=\t v_{w,\theta, R}:= \sqrt{1-\|w\|_2^2} \frac{\t y}{\|\t y\|_2}\,.
\end{equation}

Now for any $z$ with $\|w-z\|_2\le \srad_0$, since
\[
(1-\tfrac12\srad_z)R_z = (1+O_\zcush(\srad_0))R_w 
\]
we have
\begin{equation}
\t\nu\in \cP_1((1-\frac12\srad_z)[-R_z,R_z])
\end{equation}
if $C_*(\zcush)$ is sufficiently large, 
and moreover
from Proposition \ref{prop:gibbs}\eqref{gibbs.wiggle} it follows that
\begin{equation}	\label{tnu.nearopt}
\chi(\t\nu) \ge \sup_{\nu\in\cP_1((1-\frac12\srad_z)[-R_z,R_z])} \chi(\nu)
+O_{T,\zcush}(\srad_0)\,.
\end{equation}
From \eqref{tnu.nearopt} and assuming $C_0(\zcush)$ is sufficiently large, we can hence apply \eqref{varadhan.LB} from Lemma \ref{lem:Varadhan} with $N'$ in place of $N$, $R_z=O_\zcush(R)$ in place of $R$, $\frac12\srad_z$ in place of $\srad'$ and $K=O_{T,\zcush}(1)$ to bound
\begin{align}
\frac1{N'}\log \int_{ \Deloc^{N'}_{R_z}
\cap \Bset_2^{N'}(\frac{\t y}{\sqrt{N'}}, \frac12\srad_z)} \exp \bigg( \sum_{j=1}^{N'} \ham(\sqrt{N'}v_j)  \bigg) dP_{N'}(v)	
&\ge \sup_{\nu\in\cP_1([-R_z,R_z])}\chi(\nu) + O_{T,\zcush}(\srad_0)	\notag\\
&\ge \sup_{\nu\in\cP_1([-R_w,R_w])}\chi(\nu) + O_{T,\zcush}(\srad_0)   \label{FE.N-n0}
\end{align}
for any $z$ with $\|w-z\|_2\le \srad_0$, where in the second line we again used Proposition \ref{prop:gibbs}\eqref{gibbs.wiggle}.
Now since we showed  in \eqref{ham.lip} that {$\ham\circ\ssq^{-1}$} is $O_T(1)$-Lipschitz and recalling also that $\ham(0)=0$ (since $\LLa_\mu(0)=0$), we have 
\begin{equation}
\chi(\nu) = \int \ham d\nu - \DKL(\nu|\gamma) \le \int \ham d\nu \ls_T \int s^2 d\nu(s) = O_T(1)
\end{equation}
for any $\nu\in \cP_1(\R)$.  From this and multiplying \eqref{FE.N-n0} through by $N'/N = 1+O(N^{-2\eta})$ we get
\begin{align}
&\frac1{N}\log \int_{ \Deloc^{N'}_{R_z}
\cap \Bset_2^{N'}(\frac{\t y}{\sqrt{N'}}, \frac12\srad_z)} \exp \bigg( \sum_{j=1}^{N'} \ham(\sqrt{N'}v_j)  \bigg) dP_{N'}(v)
\ge \sup_{\nu\in\cP_1([-R_w,R_w])}\chi(\nu) + O_{T,\zcush}(\srad_0+ N^{-2\eta})\,.	\label{FE.N.Rw}
\end{align}
We next note that 
\begin{equation}	\label{ty.norm}
\|\t y\|_2 = \sqrt{N'}(1+O(R_w^2/N)).
\end{equation}
Indeed, from Proposition \ref{prop:gibbs}\eqref{gibbs.density} we have that $\t\nu$ is continuous on its support, so the $1/N'$-quantiles $\t y_1>\dots> \t y_{N'}$ are all distinct, with $\t y_{N'}$ the left edge of the support. Denoting the right edge of the support of $\t\nu$ by $\t y_0$, we can bound
\begin{align*}
1= \int t^2 d\t\nu(t) = \sum_{i=1}^{N'} \int_{\t y_i}^{\t y_{i-1}} t^2d\t\nu(t) \ge \frac1{N'}\sum_{i=1}^{N'} \t y_i^2\,.
\end{align*}
Moreover, since $\t\nu$ is supported on $[-R_w,R_w]$ we have $\t y_0\le R_w$, so
\[
1=\sum_{i=1}^{N'} \int_{\t y_i}^{\t y_{i-1}} t^2d\t\nu(t)  \le \frac1{N'}\bigg(R_w^2 + \sum_{i=1}^{N'} \t y_i^2\bigg)
\]
and we obtain \eqref{ty.norm}.
Thus, for $\|z-w\|_2\le \srad=N^{-4}$ we have
\begin{align*}
\frac{\t y}{\sqrt{N'}} &= (1+O_\zcush(R^2/N)) \frac{\t y}{\|\t y\|_2} \\
&= (1+O_\zcush(\srad + R^2/N)) \frac{\sqrt{1-\|w\|_2^2}}{\sqrt{1-\|z\|_2^2}} \frac{\t y}{\|\t y\|_2}\\
&= (1+O_\zcush(R^2/N)) \frac{\t v}{\sqrt{1-\|z\|_2^2}}.
\end{align*}
Hence, if $C_0(\zcush)$ is sufficiently large, then 
\[
 \Bset_2^{N'} ( \frac{\t v}{\sqrt{1-\|z\|_2^2}} , \srad_z) \supset 
  \Bset_2^{N'}(\frac{\t y}{\sqrt{N'}}, \frac12\srad_z)
\]
and we can bound the inner integral in \eqref{FE.LB2} below using \eqref{FE.N.Rw}.
Noting that the lower bound in \eqref{FE.N.Rw} is independent of $z$, we can substitute it in \eqref{FE.LB2} and then apply Lemma \ref{lem:Ploc} to get
\begin{align}
&F_N(\theta; \Uloc_w(\srad,R)\cap (\Bset_2^{N'}(\t v, \srad_0)\times \R^{[N'+1,N]}) )	\notag\\
&\ge \frac1N \log \int_{\B^{[N'+1,N]}} 1_{\|z-w\|_2\le \srad} 
\exp\bigg( N \bigg( \sup_{\nu\in\cP_1([-R_w,R_w])}\chi(\nu) - O_{T,\zcush}(\srad_0+ N^{-2\eta}) \bigg) \bigg)
 dP_{n_0}^{(w)}(z)\notag\\
&\qquad \qquad + 
\freeL(\theta,w) + \theta^2(1-\|w\|_2^2)^2
{+ 2\theta^2\|w^\le\|_2^2(1-\|w\|_2^2)}+ O_{T,\zcush}({\delta+}R^2N^{-1/2})	\notag\\
&= 
\freeL(\theta,w) + \theta^2(1-\|w\|_2^2)^2
{+ 2\theta^2\|w^\le\|_2^2(1-\|w\|_2^2)}+
\frac1N\log P_{n_0}^{(w)}( z: \|z-w\|_2\le \srad) + \sup_{\nu\in\cP_1([-R_w,R_w])}\chi(\nu) 
	 \notag\\
	 &\qquad\qquad+ O_{T,\zcush}({\delta+}R^2N^{-1/2} + N^{-2\eta})	\notag\\
&= 
 \freeL(\theta,w) + \theta^2(1-\|w\|_2^2)^2
{+ 2\theta^2\|w^\le\|_2^2(1-\|w\|_2^2)}+
\frac12\log(1-\|w\|_2^2) +  \sup_{\nu\in\cP_1([-R_w,R_w])}\chi(\nu)
	\notag\\
	 &\qquad\qquad+ O_{T,\zcush}(N^{-4} + {\delta+}R^2N^{-1/2} + N^{-2\eta}\log N)	\notag\\
&= \free_{N,R}(\theta,w) + O_{T,\zcush}({\delta+} R^2N^{-1/2}  + N^{-2\eta}\log N)
\end{align}
which gives \eqref{FE.LBgoal} and hence \eqref{FE-}. 

Turning to establish part (c), let $\t y^{(\theta)}$, $\t\nu^{(\theta)}$ denote the vector and optimizing measure from the definition of $\t v^{(\theta)} := \t v_{\theta,w,R}$ in \eqref{def:tv}. {We also write $w=w^{(\theta)}$. 
It suffices to show
\begin{equation}	\label{goal:tv.cont}
d_2(\t v^{(\theta_1)}, \t v^{(\theta_2)}) = o_{\theta_1-\theta_2\to 0} (1) +  O_\zcush\bigg( \frac{R}{\sqrt{N}} \bigg)
\end{equation}
for all $\theta_0\le\theta_1,\theta_2\le T$.
Indeed, this implies  that $\t v^{(\theta)}$ is within $O_\zcush(R/\sqrt{N})$ of some $\vcont^{(\theta)}$ depending continuously on $\theta$, and from the triangle inequality and adjusting the constant $C_0(\zcush)$ we can take the Wasserstein ball from part (b) to be centered on $w+\vcont^{(\theta)}$.

Since $q$ and $\alpha$ are continuous, so is $\theta\mapsto\|w^{(\theta)}\|_2^2 =q(\theta)^2\|\vloc'\|_2^2+\alpha(\theta)$.}
With 
\begin{equation}
g( s) := \sum_{i=1}^N \LLa_\mu(2\sqrt{N/N'}\vloc'_is) \,,\qquad 
a=a(\theta):= \theta q(\theta) \sqrt{1-\|w^{(\theta)}\|_2^2}
\end{equation}
we can express $\ham$ from \eqref{choose.ham} 
as
\[
\ham(s;w) = g( a s)
\]
and $\t\nu^{(\theta)}$ is the maximizer for 
\[
\nu\mapsto \int g(a s)d\nu(s) - \DKL(\nu|\gamma)
\]
over $\cP_1( (1-C_*(\zcush)\srad_0)[ -R'/a, R'/a])$. 
Now since the entries of $\t y^{(\theta)}$ are the $1/N'$ quantiles of $\t\nu^{(\theta)}$ which is continuous and supported on $[-R'/a, R'/a]$, 
we have
\[
\cW_2( \t\nu^{(\theta)}, \hat{\mu}_{\t y^{(\theta)}} )^2 \le \frac1{N'} \bigg( (2R'/a)^2 + \sum_{i=2}^{N'} (\t y_{i-1}-\t y_i)^2 \bigg) \ls \frac{R'^2}{a^2N}\asymp_\zcush \frac{R^2}{N}
\]
(consider the coupling $(X,X')$ with $X\sim \t\nu^{(\theta)}$ and $X'\sim \hat{\mu}_{\t y^{(\theta)}}$ obtained by rounding $X$ down to the nearest quantile). 
On the other hand, from \eqref{ty.norm},
\begin{align*}
d_2( \t v^{(\theta)}, \sqrt{1-\|w^{(\theta)}\|_2^2} \frac{\t y^{(\theta)}}{\sqrt{N'}})^2
= (1-\|w^{(\theta)}\|_2^2) \bigg|1- \frac{\|\t y^{(\theta)}\|_2}{\sqrt{N'}}\bigg| ^2 
& 
\ls \frac{R^2}N.
\end{align*}
Combining these estimates with the triangle inequality, 
{we get that for all $N$ sufficiently large depending on $\zcush$ and any $\theta_1,\theta_2\in[\theta_0,T]$, 
\begin{align*}
d_2(\t v^{(\theta_1)}, \t v^{(\theta_2)}) 
&\le  d_2( \frac{\t y^{(\theta_1)}}{\sqrt{N'}}, \frac{\t y^{(\theta_2)}}{\sqrt{N'}}) 
+ \frac{\|\t y^{(\theta_1)}\|_2}{\sqrt{N'}} \Big|\, \sqrt{1-\|w^{(\theta_1)}\|_2^2} - \sqrt{1-\|w^{(\theta_2)}\|_2^2}\,\Big|
+ O_\zcush\bigg( \frac{R}{\sqrt{N}} \bigg) \\
&\le \cW_2( \t\nu^{(\theta_1)}, \t\nu^{(\theta_2)} ) 
+ 2\Big|\, \sqrt{1-\|w^{(\theta_1)}\|_2^2} - \sqrt{1-\|w^{(\theta_2)}\|_2^2}\,\Big|
 + O_\zcush\bigg( \frac{R}{\sqrt{N}} \bigg) \,.
\end{align*}
Since $\theta\mapsto\|w^{(\theta)}\|_2$ is continuous, the second term in the final bound is $o_{\theta_1-\theta_2\to 0}(1)$. From Proposition \ref{prop:gibbs}\eqref{gibbs.W2cont} we have 
that $\theta\mapsto\t\nu^{(\theta)}$ is continuous under the $\cW_2$ metric, and \eqref{goal:tv.cont} follows.}
\end{proof}

\subsection{Proof of Lemma \ref{lem:annealed1} }
\label{sec:annealed1}

We apply the following:

\begin{lemma}
\label{lem:fN.Lip}
For $v,w,z\in\ell^2(\N)$ and $t_1,t_2\ge0$, 
we have
\begin{equation}	\label{Lip1a}
\freeL(t_1,v)- \freeL(t_2,w) \ls |t_1^2-t_2^2|\|w\|_2^4 + t_1^2 \|v-w\|_2(\|v\|_2+\|w\|_2)^3
\end{equation}
and
\begin{equation}	\label{Lip2a}
\wt f_N   (t_1v,z) -\wt f_N   (t_2w,z) \ls   |t_1^2-t_2^2|\|w\|_2^2\|z\|_2^2 + t_1^2 \|v-w\|_2(\|v\|_2+\|w\|_2)\|z\|_2^2\,.
\end{equation}
\end{lemma}

\begin{proof}
See Appendix \ref{app:freeprops-cont}.
\end{proof}

To simplify notation we take $\supp(w)=[n_0]$.
By using Fubini's theorem and integrating out the entries of $H$, we find 
\begin{align}	\label{FE.Fubini}
&\frac{1}{N}\log\E \int_{\Uloc} e^{N\theta\langle u,Hu\rangle}dP(u)	\notag\\
 &\qquad= \frac{1}{N}\log\int_{\Uloc}  \exp\bigg( \sum_{i\le j} \LLa_\mu(2^{\ep_{ij}}\theta \sqrt{N} u_{i}u_{j}) \bigg)dP(u)  \notag\\
  &\qquad= \frac{1}{N}\log\int_{\Uloc}  \exp\big( N \freeL(\theta,u)\big)dP(u)
  \notag\\
 &\qquad= \frac1N \log \int_{\Uloc} \exp\big(N( \freeL(\theta,u_w) + \freeL(\theta,u_{w^c}) + \wt f_N   (\theta u_w,u_{w^c})) \big) dP(u)
 \end{align}
where in the last line we used that $\LLa_\mu(0)=0$ and that $u_w,u_{w^c}$ have disjoint supports.
From \eqref{Lip1a} we have  for $u\in \Uloc$ defined in \eqref{def:UwrR}
\begin{equation}	\label{Lip1a.1}
\left |\freeL(\theta,u_w) - \freeL(\theta,w)\right|
=O(\theta^2\srad)\,.
\end{equation}
When $n_0=N$, i.e. $\|u-w\|_2\le \srad$ for all $u\in\Uloc$ and $u_{w^c}=0$, then without loss of generality we can take $R=0$. Then the delocalized and cross terms in \eqref{FE.Fubini} vanish since $\LLa_\mu(0)=0$, so in this case we are left with
\begin{equation}
\frac{1}{N}\log\E \int_{\Uloc} e^{N\theta\langle u,Hu\rangle}dP(u)
=\freeL(\theta,w) + \frac1N\log P(\Uloc) + O(\theta^2\srad).
\end{equation}
Now for the case $n_0<N$, from \eqref{Lip2a} we get
\begin{align}
 \left| \wt f_N   (\theta u_w,u_{w^c}) - \wt f_N   (\theta w,u_{w^c})\right|
= O( \theta^{2}\srad)\,.	\label{Lip2.1}
\end{align}
Combining \eqref{FE.Fubini}, \eqref{Lip1a.1}, \eqref{Lip2.1}, we have
\begin{align*}
&\frac1N\log \E \int_\Uloc e^{\theta N\langle u, Hu\rangle}dP(u)
=\freeL(\theta,w) \\
&\qquad+ \frac1N\log\int_\Uloc \exp\Big( N \freeL(\theta,u_{w^c}) 
+ N \wt f_N   (\theta w, u_{w^c}) \Big) dP(u) 
+ O(\theta^2 \srad).
\end{align*}
By hypothesis we have $|w^\le_i(u_{w^c})_j| \le 2\theta\delta/\sqrt{N}$ for all $i,j$, and from Taylor expansion 
we get
\[
\wt f_N   (\theta w^\le,u_{w^c}) = 2\theta^2(1+O(\theta \delta))\|w^\le\|_2^2\|u_{w^c}\|_2^2.
\] 
Furthermore, denoting $\hat u_{w^c}:= u_{w^c}/\|u_{w^c}\|_2$, we have
\[
\|u_{w^c}- \sqrt{1-\|w\|_2^2}\hat u_{w^c}\|_2 = |\|u_{w^c}\|_2-\sqrt{1-\|w\|_2^2}| = O(\srad/\sqrt{\zcush})\quad \forall u\in \Uloc_w(\srad)
\]
 where we used that $\|u_{w^{c}}\|_2^{2}=1-\|u_{w}\|_2^{2}=1-\|w\|_2^{2}+O(\srad)$ and the assumption $\|w\|_2\le 1-\zcush$.
Then again from \eqref{Lip2a} (and symmetry of $\wt f_N   $ in its arguments) we get
\begin{equation}	
\label{Lip3a}
|\wt f_N   (\theta w{^>},u_{w^c}) - \wt f_N   (\theta w{^>}, \sqrt{1-\|w\|_2^2} \hat u_{w^c})| = O(\theta^2\srad/\sqrt{\zcush}).
\end{equation}

Next, observe that $ \sqrt{N} u_iu_j$  is 
   bounded  from above by  $R^2N^{-\frac{1}{2}}\le N^{-1/10}$  for $i,j>n_0$.
{Again} from Taylor expansion {of $\LLa_\mu$} we get for all $u\in \Uloc_w(\srad)$,
\begin{align}
\freeL(\theta,u_{w^c})
&=   \theta^{2}(\sum_{i>n_0 }u_{i}^{2})^{2}+ O( \theta^{3} R^2N^{-\frac{1}{2}})	
= \theta^{2}(1-\|w\|_2^{2})^{2}+O(\theta^2\srad + \theta^3R^2N^{-\frac{1}{2}})\,.	\label{devL} 
\end{align}
Substituting \eqref{Lip1a.1}--\eqref{Lip3a} 
and \eqref{devL} into \eqref{FE.Fubini}, we obtain \eqref{annealedB}.
This concludes the proof of Lemma \ref{lem:annealed1}.

\subsection{Proof of Lemma \ref{lem:Ploc}}
\label{sec:Ploc}

For notational convenience we take $\supp(w)=[n_0]$.
We write $Y:=\|u_w\|_2$ and $y:=\|w\|_2$. 
First we claim
\begin{equation}	\label{PUw.sandwich}
\log P( |Y-y|<\srad/2) - n_0\log \frac{C}{\srad} \le 
\log P( \|u_w-w\|_2<\srad) \le \log P(|Y-y|<\srad). 
\end{equation}
Indeed, the second inequality is immediate since $\|u_w-w\|_2\ge |\|u_w\|_2-\|w\|_2|$. 
For the first inequality, note that we can cover the annulus $\{v\in \R^{n_0}: y-\frac12\srad< \|v\|_2< y+ \frac12\srad\}$ with $O(1/\srad)^{n_0}$ balls of radius $\srad$ centered on points $v$ of norm $y$. 
From the union bound and rotational invariance of $P$ we get
\[
P(|Y-y|<\srad/2) \le O(1/\srad)^{n_0} P(\|u_w-w\|_2<\srad)
\]
and \eqref{PUw.sandwich} follows. It now suffices to show that for $\srad_1\in[\frac12N^{-4},\frac1{10}\zcush]$, 
\begin{equation}	\label{PUw.goal1}
\log P(|Y-y|<\srad_1) = \frac{N}2\log (1-y^2) + 
O\Big(n_0\Big(\log\frac1\zcush+ \log N\Big)\Big) + O(\zcush^{-1}\srad_1N)\,.
\end{equation}
Indeed, the claim follows by applying the above with $\srad$ and $\srad/2$ in place of $\srad_1$.

We claim that for any interval $I\subset[0,1-\frac12\zcush]$ of length $|I|\ge N^{-10}$ and any $\gamma_0\in I$, 
\begin{equation}	\label{PUw.goal2}
\log P(Y^2\in I) = \frac{N}2 \log (1-\gamma_0) + 
O\Big(n_0\Big(\log\frac1\zcush+ \log N\Big)\Big) + O(\zcush^{-1}|I|N).
\end{equation}
Indeed, {fixing such an interval $I$, first note that} $Y^2$ has the Beta$(\frac12n_0, \frac12(N-n_0))$ distribution, with density
\[
f_{Y^2}(\gamma) = \frac1{Z}\gamma^{\frac{n_0}2-1} (1-\gamma)^{\frac{N-n_0}2-1}
\]
where the normalizing constant can be estimated by Stirling's formula:
\[
Z = \frac{\Gamma( n_0/2) \Gamma( (N-n_0)/2)}{\Gamma(N/2)} = \exp( O(n_0\log N))\,.
\]
Next note that
\[
f_{Y^2}(\gamma) = \exp \bigg( \frac{N}2 \log (1-\gamma) + O\Big(n_0\Big(\log\frac1\zcush+ \log N\Big)\Big)  \bigg)
\]
for $\gamma\in [N^{-100}, 1-\frac12\zcush]$.
Now writing $I=[a,b]$, by monotonicity of $\gamma\mapsto \log(1-\gamma)$ on $(0,1]$ 
we can estimate
\[
P(Y^2\in I) = \int_a^b f_{Y^2}(\gamma) d\gamma \le \sup_{\gamma\in I} f_{Y^2}(\gamma) \le \exp\bigg( \frac{N}2\log(1-a) + O\Big(n_0\Big(\log\frac1\zcush+ \log N\Big)\Big) \bigg)
\]
and
\[
P(Y^2\in I) \ge \int_{a}^{a+N^{-10}} f_{Y^2}(\gamma) d\gamma \ge N^{-10} \exp \bigg( \frac{N}2 \log ( 1-a-N^{-10}) + O\Big(n_0\Big(\log\frac1\zcush+ \log N\Big)\Big) \bigg)\,.
\]
Now since $\gamma \mapsto \log (1-\gamma)$ has derivative of size $O(1/\zcush)$ on $[0,1-\frac12\zcush]$, we can replace  
$a$ and $a+N^{-10}$
with $\gamma_0$ in the previous two displays, incurring a multiplicative error of size $\exp( O(|I|N/\zcush))$, and \eqref{PUw.goal2} follows. 
Turning to \eqref{PUw.goal1}, if $y<\srad_1$ then applying \eqref{PUw.goal2} yields
\begin{align*}
\log P(|Y-y|<\srad_1) &=\log  P(Y^2 \in [0,(y+ \srad_1)^2]) \\
&= \frac{N}2 \log (1-y^2) + O\Big(n_0\Big(\log\frac1\zcush+ \log N\Big)\Big) + O(\zcush^{-1}\srad_1^2N)
\end{align*}
as desired.
If $y\ge \srad_1$, then
\[
P(|Y-y|<\srad_1) = P(|Y^2-y^2|<\srad_1(Y+y)).
\]
We can control the latter expression on either side by
\[
P(|Y^2-y^2|<\srad_1y) \le 
P(|Y^2-y^2|<\srad_1(Y+y))
\le P(|Y^2-y^2|<3\srad_1y).
\]
The events on the left and right are of the form \eqref{PUw.goal2} with $I$ of length at least $\srad_1y\ge \srad_1^2\ge N^{-10}$ and right endpoint at most $y(y+3\srad_1)\le (1-\zcush)(1-\zcush+\frac3{10}\zcush)\le 1-\zcush$, and \eqref{PUw.goal1} follows from \eqref{PUw.goal2}.

\section{Properties of the rate functions}
\label{sec:rateprops}

\subsection{Proof of Lemmas \ref{lem:Jprops} and \ref{lem:tJprops}}

Recalling $\thetam$ from \eqref{def:thetapm}, we will make repeated use of the following estimates:
for $x\ge2 $,
\begin{equation}	\label{thetam-small}
0\le \frac12-\thetam \asymp \sqrt{x-2}
\end{equation}
and 
\begin{equation}	\label{thetam-large}
\thetam\asymp x^{-1}.
\end{equation}

We need the following lemma gathering properties of $\free_{N,R}$ defined in \eqref{def:freeN}, which will be proved in subsequent subsections. 

\begin{lemma}[Properties of $\free_N$]
\label{lem:freeprops}
\quad
\begin{enumeratea}
\item\label{freeprop.small} (Small $w$). 
We have
\begin{equation}	\label{freeN.at0}
0\le\theta^2- \free_{N,R}(\theta,0) \ls e^{-cR^2}\qquad\forall \theta\ge0,R\in[1,\infty).
\end{equation}
Moreover,
there exists $\al_\mu>0$ depending only on $\mu$ and a universal constant $C_0>0$ such that for any $\theta\le1$, $w\in\B$ with $\|w\|_2^2\le \al_\mu$ and any $R\ge C_0$, we have
\begin{equation}	\label{bd:smallw}
\free_{N,R}(\theta,w) \le \theta^2 -\frac14\|w\|_2^2 \,.
\end{equation}
Similarly
\begin{equation}	\label{bd:smallw.t}
\tfree_{N,R}(\theta,w,\tal) \le \theta^2 -\frac14(\|w\|_2^2+\tal) 
\end{equation}
when $\|w\|_2^2+\tal\le \al_\mu$.

\item\label{freeprop.cont} (Continuity).
Let $R\ge1$.
For any $w\in\B$,
\begin{equation}	\label{freecont.theta}
\free_{N,R}(\theta_1,w) - \free_{N,R}(\theta_2,w) \ls |\theta_1^2 - \theta_2^2| \qquad\forall \theta_1,\theta_2\ge0\,.
\end{equation}
For any $\theta\ge0$, $z_1,z_2\in\ell^2(\N)$ and $q_1,q_2\ge0$ such that $q_iz_i\in (1-\rho)\B$ for $i=1,2$ and some $\rho>0$, we have
\begin{equation}	\label{freecont.w}
\free_{N,R}(\theta,q_1z_1) - \free_{N,R}(\theta,q_2z_2) 
\ls \Big( 1+ \frac{\theta^2}{\sqrt{\rho}}\Big) \big( |q_1^2 - q_2^2| + \|z_1-z_2\|_2\big). 
\end{equation}
Furthermore, for any $\tal\in[0,1]$ and $\chw\in \sqrt{1-\tal}\B$, 
\begin{equation}	\label{tfreecont.theta}
\tfree_{N,R}(\theta_1,\chw,\tal) - \tfree_{N,R}(\theta_2,\chw,\tal) \ls |\theta_1^2 - \theta_2^2|\,,\qquad \forall \theta_1,\theta_2\ge0\,.
\end{equation}
For $\theta\ge0$, $\chz_1,\chz_2\in\ell^2(\N)$, $q_1,q_2\ge0$ and $\tal_1,\tal_2\in[0,1]$ such that $q_i^2\|\chz_i\|_2^2 + \tal_i \le 1-\rho$ for $i=1,2$ and some $\rho>0$, 
\begin{equation}	\label{tfreecont.w}
\tfree_{N,R}(\theta,q_1\chz_1,\tal_1) - \tfree_{N,R}(\theta,q_2\chz_2,\tal_2)
\ls \Big( 1+ \frac{\theta^2}{\sqrt{\rho}}\Big) \big( |q_1^2 - q_2^2| + \|\chz_1-\chz_2\|_2 + |\tal_1-\tal_2|\big).
\end{equation}

\item\label{freeprop.large} (Large $\theta$). 
For any $R\ge1,\zcush>0$ and $w\in(1-\zcush)\B$, 
\begin{equation}
\free_{N,R}(\theta,w) \gs \zcush^2 \theta^2
\end{equation}
for all $\theta\ge C\zcush^{-1} \sqrt{\log(1/\zcush)}$. 
The conclusion also holds for $\tfree_{N,R}(\theta,w,\al)$ if $\al+\|w\|_2^2\le (1-\zcush)^2$.
\end{enumeratea}
\end{lemma}

\begin{proof}[Proof of Lemma \ref{lem:Jprops}\eqref{Jprop.smallz}]
This is immediate from \eqref{freeN.at0}.
\end{proof}

\begin{proof}[Proof of Lemmas \ref{lem:Jprops}\eqref{Jprop.smallx} and \ref{lem:tJprops}\eqref{tJprop.smallx}]
Fix $x\in[2,x_\mu)$, $\vloc\in\B$ and $R\ge C_0$.
For Lemma \ref{lem:Jprops}(\ref{Jprop.smallx}) it suffices to show  that there exists $\theta_*\ge0$ such that
\[
J(x,\theta_*) - \free_{N,R}(\theta_*,\overlap_x(\theta_*)\vloc) \ge \rate^\gamma(x) + c_0\sqrt{x-2}\|\vloc\|_2^2
\]
assuming $x_\mu$ is sufficiently small.
We show this holds with $\theta_*=\thetap$ (see \eqref{def:thetapm}). 
Indeed, with $\al_\mu$ as in 
Lemma \ref{lem:freeprops}(\ref{freeprop.small}) 
we can take $x_\mu$ sufficiently small that $\thetap\le1$ and $\overlap_x(\thetap)\le \al_\mu^{1/2}$ (recall that $\overlap_x(\theta)^2=1-\frac{\thetam}{\theta}$ for $\theta\ge\thetam$, and that $\thetap/\thetam\downarrow1$ as $x\downarrow2$). 
Then since $\|\overlap_x(\thetap)\vloc\|_2^2\le \overlap_x(\thetap)^2\le \alpha_\mu$, it follows from 
Lemma \ref{lem:freeprops}(\ref{freeprop.small}) 
that for all $N$ sufficiently large (so that $R\ge C_0$), 
\[
\free_{N,R}(\thetap,q_x(\thetap)\vloc)\le (\thetap)^2 -\frac14\overlap_x(\thetap)^2\|\vloc\|_2^2
\]
and hence (recalling \eqref{igamma.var}),
\begin{align*}
J(x,\thetap) -\free_{N,R}(\thetap,\overlap_x(\thetap)\vloc)
&\ge J(x,\thetap)-(\thetap)^2 + \frac14\overlap_x(\thetap)^2\|\vloc\|_2^2\\
&= \rate^\gamma(x)+ \frac14\overlap_x(\thetap)^2\|\vloc\|_2^2\,.
\end{align*}
It only remains to note 
\[
\overlap_x(\thetap)^2=
1-\frac{\thetam}{\thetap} = 1-\frac{x-\sqrt{x^2-4}}{x+\sqrt{x^2-4}} \asymp \sqrt{1-\frac4{x^2}}\asymp\sqrt{x-2}
\]
for $x=O(1)$. 
For Lemma \ref{lem:tJprops}(\ref{tJprop.smallx}), fixing $x\in[2,x_\mu)$, $\chz\in\B$ and $\tal\in[0,1-\|\chz\|_2^2]$, it suffices to show there exists $\theta_*\ge0$ such that
\[
J(x,\theta_*) - \tfree_{N,R}(\theta_*,\overlap_x(\theta_*)\chz, \overlap_x(\theta_*)^2\tal) \ge \rate^\gamma(x) + c_0\sqrt{x-2}(\|\chz\|_2^2+\tal)
\]
if $x_\mu$ is sufficiently small. The proof of this follows similar lines as above, using \eqref{bd:smallw.t} in place of \eqref{bd:smallw}.
\end{proof}

\begin{proof}[Proof of Lemma \ref{lem:Jprops}\eqref{Jprop.T} and Lemma \ref{lem:tJprops}\eqref{tJprop.T}]
We only prove the claim about $\cJ_{N,R}$; the claim for $\wt\cJ_{N,R}$ follows from the same argument, using the analogous properties of $\tfree_{N,R}$ stated in Lemma \ref{lem:freeprops}. 

Writing
\[
\cK_{x,\vloc}(\theta) := J(x,\theta) - \free_{N,R}(\theta,\overlap_x(\theta)\vloc)
\]
we have from \eqref{freeN.at0} that 
\begin{equation}	\label{Exz-small1}
\cK_{x,z}(\theta)=O(e^{-cR^2}) = O(1) \qquad\forall \theta\le\thetam\,.
\end{equation}
On the other hand, from \eqref{def:Jsc} we see that 
$J(x,\theta)\le \theta x+ O(1)$ for all $x\ge2,\theta\ge1$,
which together with Lemma \ref{lem:freeprops}(\ref{freeprop.large}) shows
\begin{equation}
\cK_{x,\vloc}(\theta) \le -c\zcush^2\theta^2\qquad \forall \theta\ge CK/\zcush^2
\end{equation}
for a sufficiently large constant $C>0$. The first claim follows with $T=O(K/\zcush^2)$.

For the second claim, 
denote
\[
\cK_x(\theta):= J(x,\theta) - \theta^2,\qquad \Delta(s):= \cK_{x,z}(\thetam s) - \cK_x(\thetam s)\,,\quad s\ge0. 
\]
From Lemma \ref{lem:freeprops}(\ref{freeprop.small}) we have
\begin{equation}	\label{Delta1}
\Delta(1) = (\thetam)^2-\free_{N,R}(\thetam,0) \in[0,\eps_R]
\end{equation}
for some $\eps_R=O(e^{-cR^2})$. 
From \eqref{freecont.w} we get that for any $\theta>\thetam$,
\[
\free_{N,R}(\theta,\overlap_x(\theta)\vloc) - \free_{N,R}(\thetam,0) \ls  O(1+\frac{\theta^2}{\sqrt{\zcush}}) \overlap_x(\theta)^2
\ls (1+\frac{\theta^2}{\sqrt{\zcush}}) \frac{\theta-\thetam}{\thetam}
\]
so
\begin{equation}
|\Delta(s)| \le \eps_R + (\thetam)^2(s^2-1) + \frac{C}{\sqrt{\zcush}}s^2(s-1)\,,\qquad s\ge1.
\end{equation}
Since $\cK_x(\thetam)=\cK_x'(\thetam)=0$ and
\begin{equation}	\label{E''UB}
\cK_x''(\theta) = \frac1{2\theta^2}-2\le \frac1{2(\thetam)^2} \ls K^2\,,\qquad \theta\ge\thetam
\end{equation}
we can bound
\begin{equation}
\cK_x(\theta) \ls K^2 (\theta-\thetam)^2 \qquad\theta\ge\thetam
\end{equation}
and hence
\begin{equation}	\label{vloc.small0}
\cK_{x,\vloc}(\thetam s) = \cK_x(\thetam s) + \Delta(s) \le \eps_R + O(\frac{K^2}{\sqrt{\zcush}}(s-1))\,,\qquad s\in [1,2],
\end{equation}
showing the objective function for $\cJ_{N,R}(x,\vloc):=\sup_{\theta\ge0} \cK_{x,z}(\theta)$ is small in a neighborhood of $\thetam$. 

It remains to show that $\cK_{x,\vloc}$ attains a larger value for larger $\theta$. 
We first consider the case that $\vloc$ is small: suppose
\begin{equation}	\label{vloc.small1}
\al:=\|\vloc\|_2^2\le c_0\kappa^3/K^4
\end{equation}
for a sufficiently small constant $c_0>0$. 
From Lemma \ref{lem:freeprops}(\ref{freeprop.cont},\ref{freeprop.small}) we have
\begin{equation}
\free_{N,R}(\theta,\overlap_x(\theta)\vloc) = \free_{N,R}(\theta,0) + O(T^2\sqrt{\al}) \le \theta^2+O(T^2\sqrt{\al})
\end{equation}
for $\theta\le T$. Thus, taking $T=\thetap\ls K$ and recalling \eqref{igamma.var}, we have
\[
\sup_{\theta\ge0} \cK_{x,\vloc}(\theta) \ge \sup_{\theta\in[0,T]} \cK_x(\theta) - O(K^2\sqrt{\al})
=\rate^\gamma(x) - O(K^2\sqrt{\al}).
\]
Since $\rate^\gamma(x) \gs (x-2)^{3/2}\ge\xcush^{3/2}$, we see 
\begin{equation}
\sup_{\theta\ge0} \cK_{x,\vloc}(\theta) \ge \rate^\gamma(x)/2
\end{equation}
 if the constant $c_0$ in \eqref{vloc.small1} is sufficiently small. 
On the other hand, from \eqref{vloc.small0} (with $\zcush=1/2$, say) and taking $c_0$ smaller if necessary and $R$ sufficiently large, we see 
\begin{equation}
\sup_{\theta\le (1+\delta)\thetam} \cK_{x,\vloc}(\theta) \le \rate^\gamma(x)/4
\end{equation}
if $\delta=(\xcush/K)^C$ for a suitable constant $C>0$. Since $\thetam\gs 1/K$, the claim follows with $\tcush=(\xcush/K)^{C+1}$ for the case that \eqref{vloc.small1} holds. 

Now we consider the complementary case 
\begin{equation}
\label{vloc.small2}
\al_0:=c_0\xcush^3/K^4<\al\le 1-\zcush.
\end{equation}
From Lemma \ref{lem:freeprops}(\ref{freeprop.small}) we have
\begin{equation}
\Delta(s)\ge \Delta^-(s) := \frac14(1-\frac1s)\al\,,\qquad 1\le s\le s_0:= \frac1{1-\frac12\al_\mu}. 
\end{equation}
From the equality in \eqref{E''UB} we see that $\cK_x(\theta)\ge0$ for $\theta\le\frac12$, and since $\thetam\le \frac12-c\sqrt{\xcush}$ (see \eqref{thetam-small}), we have $\cK_x(\thetam s) \ge0$ for all $s\in[1,1+c\sqrt{\xcush}]$. 
Hence,
\begin{equation}
\cK_{x,\vloc}(\thetam s) =\cK_x(\thetam s) + \Delta(s) \ge \Delta^-(s) \,,\qquad 1\le s\le s_1:=1+c\sqrt{\xcush}
\end{equation}
taking $c>0$ small enough that $s_1\le s_0$. In particular, combing with  \eqref{vloc.small2},
\begin{equation}
\cK_{x,\vloc}(\thetam s_1) \gs \Delta^-(s_1) \gs \al\sqrt{\xcush} \gs (\xcush/K)^4.
\end{equation}
On the other hand, from \eqref{vloc.small1} and the assumption $R\ge C\sqrt{\log(K/\xcush)}$ we can make
\[
\sup_{\theta\le (1+\delta')\thetam} \le c'(\xcush/K)^5
\]
for any fixed small constant $c'>0$ by taking $\delta'=\sqrt{\zcush}(\xcush/K)^{C'}$ for suitable $C'>0$ depending on $c'$. 
This yields the claim with $\tcush=\sqrt{\zcush}(\xcush/K)^{C'+1}$ for the case that \eqref{vloc.small2} holds, which concludes the proof. 
\end{proof}

For the proof of the continuity properties we need the following. 

\begin{lemma}
\label{lem:q2Lip}
For any $\theta>0$, the function $x\mapsto \overlap_x(\theta)^2= (1-\frac{\thetam}\theta)_+$ is $O(L\xcush^{-1/2})$-Lipschitz on $[2+\xcush,L]$. 
\end{lemma}

\begin{proof}
Fixing $\theta>0$, if $\theta\ge\frac12$ then since $\thetam = \frac14(x-\sqrt{x^2-4})\in (0,\frac12)$ for $x>2$ we have that $\overlap_x(\theta)^2 = 1-\frac{\thetam}\theta=:g_\theta(x)$ for all $x>2$, and $g_\theta$ is smooth on $(2,\infty)$ with derivative bounded by $O(\theta^{-1}\xcush^{-1/2}) = O(\xcush^{-1/2})$. 

If $\theta\in(0,\frac12)$, then $\overlap_x(\theta)^2$ is the maximum two Lipchitz functions, namely $g_\theta$ and the function that is identically zero, so it is Lipchitz with Lipschitz constant equal to that of $g_\theta$. Moreover, when $\overlap_x(\theta)^2=g_\theta(x)$ we have $\theta\ge\thetam$, so the derivative of $g_\theta$ is $\ls \theta^{-1}\xcush^{-1/2} \le (\thetam)^{-1}\xcush^{-1/2} \ls L\xcush^{-1/2}$. 
\end{proof}

\begin{proof}[Proof of Lemmas \ref{lem:Jprops}\eqref{Jprop.cont} and \ref{lem:tJprops}\eqref{tJprop.cont}]
From \eqref{thetam-large} we have $\|\overlap_x(\theta)\vloc\|_2 \le \overlap_x(\theta) \le (1-\frac{\thetam}T)^{1/2}\le 1-\frac{c}{LT}$ for $x\le L$ and $\theta\le T$. The continuity in $z$ is now immediate from Lemma \ref{lem:freeprops}(\ref{freeprop.cont}). 

For the continuity in $x$, it suffices to show that $x\mapsto J(x,\theta)$ and $x\mapsto\free_{N,R}(\theta,\overlap_x(\theta)\vloc)$ are $O(T^3L^2\xcush^{-1/2})$-Lipschitz on $[2+\xcush,L]$, uniformly for $\theta\le T$ and $\vloc\in\B$. 
For $J(x,\theta)$ we  note
\begin{equation}
\partial_xJ(x,\theta) =
1_{\theta\ge\thetam} \left( \theta - \tfrac12G_\sigma(x) \right) = (\theta-\thetam)_+\ls T.
\end{equation}
For $\free_{N,R}(\theta,\overlap_x(\theta)\vloc)$, note that if $x\le L$ and $\theta\le T$ then 
\[
\overlap_x(\theta) \le \sqrt{1-\frac{\thetam}T} \le 1-\frac{c}{LT}
\]
for a constant $c>0$, by \eqref{thetam-large}. On the other hand, we have from \eqref{freecont.w} that $q^2\mapsto \free_{N,R}(\theta,q\vloc)$ is $O(T^{5/2}L^{1/2})$-Lipchitz on $[0,\frac{c}{LT}]$ for fixed $\theta\le T$ and $\vloc\in\B$. 
The claim follows by combining this with Lemma \ref{lem:q2Lip}.

The claim for $\wt\cJ_{N,R}$ follows from the same argument, using \eqref{tfreecont.w} in place of \eqref{freecont.w}. 
\end{proof}

This completes the proof of Lemma \ref{lem:Jprops}.
In the remaining subsections of this appendix we establish Lemma \ref{lem:freeprops}.

\subsection{Proof of Lemma \ref{lem:freeprops}(\ref{freeprop.small})}

We begin with  \eqref{freeN.at0}. From \eqref{def:freeN} and that fact that $\LLa_\mu(0)=0$ we see that
\[
\free_{N,R}(\theta,0) = \theta^2 + \freeL(\theta,0)+\VP_R(0, 1) = \theta^2 - \inf_{\nu\in\cP_1([-R,R])} \DKL(\nu|\gamma) 
\]

Since the relative entropy $\DKL(\nu|\gamma)$ is non-negative (in fact one can see from the argument below that it is uniformly positive for fixed finite $R$, but we do not need this) the first inequality in \eqref{freeN.at0} follows. For the second inequality,
by considering $d\nu=\gamma([-R,R])^{-1}1_{[-R,R]}d\gamma$ the Gaussian measure conditioned to the interval $[-R,R]$, we verify
\begin{equation}
\label{minDKL}
0\le \inf_{\nu\in\cP_1([-R,R])} \DKL(\nu|\gamma) \le \log\frac1{\gamma([-R,R])} \ls \gamma([-R,R]^c)\ls e^{-cR^2}
\end{equation}
(note that $\gamma([-R,R])\ge \gamma([-1,1])\gs 1$). 
We hence obtain \eqref{freeN.at0}.

We turn to the proof of \eqref{bd:smallw}; the proof of \eqref{bd:smallw.t} follow similar lines and is omitted. 
We abbreviate $\alpha:=\|w\|_2^2$ and set $I:=[-R/(1-\al)^{1/2},R/(1-\al)^{1/2}]$. Assume $\al\le \al_\mu$ for $\al_\mu\in(0,1)$ to be taken sufficiently small depending on $\mu$.
From \eqref{VP.scaling} we have
\begin{equation}	\label{smallw-cross}
\VP_R(\theta w, 1-\al) - \frac12\al
= \sup_{\nu\in\cP_1(I) } \{ \nu(h)-\DKL(\nu|\gamma)\} + \frac12\log(1-\alpha)
\end{equation}
where
\[
h(s):= \sum_{i=1}^N \LLa_\mu(2\theta w_i(1-\alpha)^{1/2}s)\,.
\]
Moreover, from Proposition \ref{prop:gibbs}\eqref{gibbs.density} we know that the supremum in \eqref{smallw-cross} is attained at the measure $\nu^\zeta\in\cP_1(I) $ with density proportional to $1_I(s) \exp( h(s) -\zeta s^2)$. We claim that if $\al_\mu$ is sufficiently small depending on $\mu$ then
\begin{equation}	\label{smallw-goal1}
\nu^\zeta(h) = 2\theta^2\alpha(1-\alpha) + O(\alpha^{{3/2}})\,.
\end{equation}
Indeed, we have
\[
\nu^\zeta(h) - 2\theta^2\alpha(1-\alpha) 
= 4\theta^2 (1-\alpha) \sum_{i=1}^N w_i^2 \int s^2(\psimu(2\theta w_i(1-\alpha)^{1/2}s) - \tfrac12)d\nu^\zeta(s)
\]
so we can bound 
\begin{equation}
\big| \nu^\zeta(h) - 2\theta^2\alpha(1-\alpha) \big| 
\le 4\sum_{i=1}^N w_i^2 \int s^2 \big|\psimu(2\theta w_i(1-\alpha)^{1/2}s) - \tfrac12\big|d\nu^\zeta(s).
\end{equation}
Since $|2\theta w_i(1-\alpha)^{1/2}| \le 2\alpha^{1/2}$, then \eqref{smallw-goal1} will follow once we show that for any fixed $b\in[-2\alpha^{1/2},2\alpha^{1/2}]$,
\begin{equation}
\label{smallw-goal2}
\int s^2|\psimu(bs)-\tfrac12|d\nu^\zeta(s) 
= O({\al^{1/2}})
\end{equation}
if $\al_\mu$ is sufficiently small.
To that end, note that since $\psimu(t) =t^{-2}\LLa_\mu(t) = 
{\frac12+O(|t|)}$ 
for $|t|\le 1$ by Taylor expansion, we can bound the left hand side by
\begin{align*}
&\int_{|s|\le \frac12\alpha^{-1/2}} s^2 O(|bs|)d\nu^\zeta(s)
+ (\tfrac12+\psimax) \int_{|s|>\frac12\alpha^{-1/2}}s^2d\nu^\zeta(s)
{ \ls \al^{1/2} \int |s|^3d\nu^\zeta(s) + 
\alpha \int s^4d\nu^\zeta(s)}
\end{align*}
so it suffices to show $\nu^\zeta$ has bounded fourth moment if $\al_\mu$ is sufficiently small and $C_0$ is sufficiently large.
To see this, we first establish bounds on $\zeta$, which we recall is determined by the second-moment constraint
\begin{equation}	\label{zeta.constraint}
1 = \frac{\int_I s^2 e^{h(s)-\zeta s^2}ds}{\int_I e^{h(s)-\zeta s^2}ds} =: F(\zeta).
\end{equation}
We claim
\begin{equation}	\label{zeta.bounds10}
\frac1{16}\le \zeta \le C_1
\end{equation}
for a suitable absolute constant $C_1<\infty$ when $C_0$ and $\al_\mu$ are suitably large and small, respectively. 
For this we note 
the pointwise bounds
\[
0\le  h(s) \le 4\theta^2\alpha(1-\alpha)\psimax s^2\le 4\alpha\psimax s^2=: C\alpha s^2.
\]
For the lower bound on $\zeta$,
since $F$ is monotone decreasing  it suffices to show $F(\frac1{16})>1$. Assuming $\al\le (32C)^{-1}$, we can use the above bounds on $h$ to lower bound
\[
F(\tfrac1{16})\ge 
\frac{\int_{-R}^R s^2e^{-s^2/16}ds}{\int_\R \exp( -(\tfrac1{16}-C\alpha)s^2) ds} \ge  \sqrt{\frac1{32\pi}} \int_{-R}^R s^2e^{-s^2/16}ds\,.
\]
Since $\int_{-R}^R s^2e^{-s^2/16}ds\to32\sqrt{\pi}$ as $R\to\infty$, it follows that $F(\frac1{16})\ge 2$ when $R$ is a sufficiently large constant. 
For the upper bound on $\zeta$ it suffices to show $F(C_1) <1$. 
Arguing similarly as above we have
\[
F(C_1) \le \frac{\int_\R s^2 e^{-(C_1 -C\alpha) s^2}ds}{ \int_{-R}^R e^{- C_1s^2}ds}
\ls \frac1{C_1}
\]
when $C\alpha\le C_1/2$ and $R\ge 2C_1$, say.  Taking $C_1$ sufficiently large yields the upper bound in \eqref{zeta.bounds10}. 

Now to bound the fourth moment, with $\zeta\asymp 1$ and $C\alpha\le \zeta/2$ we have
\begin{align*}
\int s^4 d\nu^\zeta(s) 
&= \frac{\int_I s^4 e^{h(s) - \zeta s^2}ds}{\int_I e^{h(s) - \zeta s^2}ds}
\le \frac{\int_\R s^4 \exp( - (\zeta - C\alpha)s^2)ds }{ \int_{-R}^R e^{-\zeta s^2}ds}
\ls 1
\end{align*}
whenever $R\gs1$. 
We thus obtain \eqref{smallw-goal2} and hence \eqref{smallw-goal1} since  we multiply the previous estimate by $\|w\|_{2}^{2}=\alpha$.

Returning to \eqref{smallw-cross}, using \eqref{smallw-goal1} and the bounds $\DKL(\nu|\gamma)\ge0$ (for any probability measure $\nu$) and $\log(1-\alpha)\le -\alpha$, we have \eqref{smallw-cross} is
\[
 \nu^\zeta(h) -\DKL(\nu^\zeta|\gamma) + \frac12\log(1-\alpha) 
\le 2\theta^2\alpha(1-\alpha) -\frac12\alpha+O(\alpha^{3/2}).
\]
Since 
\[
\freeL(\theta,w) \le \frac1N\sum_{i\le j} 2^{2\ep_{ij}} \theta^2Nw_i^2w_j^2\psimax = 2\theta^2\psimax \alpha^2
\]
we have altogether that
\begin{align*}
\free_{N,R}(\theta,w) 
&\le \theta^2\big[ (1-\alpha)^2+ 2\alpha(1-\alpha)+ 2\psimax \alpha^2 \big] -\frac12\alpha+O(\alpha^{3/2})\\
&=\theta^2 (1+ (2\psimax-1)\alpha^2)-\frac12\alpha+O(\alpha^{3/2})\\
&= \theta^2 -\frac12\alpha + O(\alpha^{3/2}).
\end{align*}
Taking $\al_\mu$ smaller, if necessary, so that the error term is bounded by $\frac14\alpha$,  \eqref{bd:smallw} follows.

\subsection{Proof of Lemma \ref{lem:freeprops}(\ref{freeprop.cont})}
\label{app:freeprops-cont}

Recall the notation
\begin{equation}
\freeL(\theta,w) := \frac1N\sum_{1\le i\le j\le N} \LLa_\mu(2^{\ep_{ij}}\theta\sqrt{N}w_iw_j)\,,\qquad
\wt f_N   (v,w):= \frac1N\sum_{i,j=1}^N \LLa_\mu(2\sqrt{N}v_iw_j)
\end{equation}
We note the scaling and symmetry properties
\begin{equation}	\label{freeLfN.scaling}
\freeL(\theta,\al w) = \freeL(\al^2\theta,w)\,,\qquad \wt f_N   (v,\al w) = \wt f_N   (\al v,w)\,.
\end{equation}
for any $\al\in \R$. 
Estimates \eqref{Lip1} and \eqref{Lip2} in the following were also used in the proof of Lemma \ref{lem:annealed1}.

\begin{lemma}
\label{lem:fN.Lip2}
For $v,w,z\in\ell^2(\N)$ and $t_1,t_2\ge0$, 
we have
\begin{equation}	\label{Lip1}
\freeL(t_1,v)- \freeL(t_2,w) \ls |t_1^2-t_2^2|\|w\|_2^4 + t_1^2 \|v-w\|_2(\|v\|_2+\|w\|_2)^3
\end{equation}
and
\begin{equation}	\label{Lip2}
\wt f_N   (t_1v,z) -\wt f_N   (t_2w,z) \ls   |t_1^2-t_2^2|\|w\|_2^2\|z\|_2^2 + t_1^2 \|v-w\|_2(\|v\|_2+\|w\|_2)\|z\|_2^2\,.
\end{equation}
Moreover, for any $\beta>0$ and $R\ge \sqrt{\beta}$, 
\begin{equation}	\label{Lip3}
\VP_R(t_1v,\beta) - \VP_R(t_2w,\beta) \ls  |t_1^2-t_2^2|\|w\|_2^2\beta + t_1^2 \|v-w\|_2(\|v\|_2+\|w\|_2)\beta\,,
\end{equation}
and for $0<\beta_1\le \beta_2\le R^2$,
\begin{equation}
\VP_R(w,\beta_1) - \VP_R(w,\beta_2) \ls \|w\|_2^2 (\sqrt{\beta_2}-\sqrt{\beta_1}) \ls \frac{\|w\|_2^2}{\sqrt{\beta_1}} (\beta_2-\beta_1)\,.	\label{Lip4}
\end{equation}

\end{lemma}

\begin{proof}
From \eqref{freeLfN.scaling} and rescaling $t_1,t_2$ we may assume $\|v\|_2,\|w\|_2\in\{0,1\}$. 
Recall from Remark \ref{rmk:assu.reg} that $t\mapsto \tLL(t)=\LLa_\mu(\sgn(t)\sqrt{|t|})$ is $O(1)$-Lipschitz on $\R$. 
Recalling also the notation $\ssq(t):= \sgn(t)t^2$, 
we have for any $t_1,t_2\ge0$ and $v,w\in \ell^2(\N)$,
\begin{align}
|\freeL(t_1,v)-\freeL(t_2,w)| 
&\le \frac1N\sum_{1\le i\le j} |\LLa_\mu(2^{\ep_{ij}}t_1\sqrt{N}v_iv_j) - \LLa_\mu(2^{\ep_{ij}}t_2\sqrt{N}w_iw_j))| \notag\\
&=  \frac1N\sum_{1\le i\le j} |\tLL(2^{1+1_{i\ne j}}t_1^2 N\ssq(v_iv_j)) - \tLL(2^{1+1_{i\ne j}}t_2^2N \ssq(w_iw_j))|\notag\\
&\ls \sum_{i,j\ge1} |t_1^2\ssq(v_iv_j) - t_2^2\ssq(w_iw_j)|\\
&\le \sum_{i,j\ge1} |t_1^2-t_2^2|w_i^2w_j^2 + t_1^2|\ssq(v_iv_j) -\ssq(v_iw_j)| + t_1^2 |\ssq(v_iw_j) -\ssq(w_iw_j)|\notag\\
& =  |t_1^2-t_2^2| \|w\|_2^4 + t_1^2 \sum_{i,j=1}^N v_i^2 |\ssq(v_j) - \ssq(w_j)| + w_j^2 |\ssq(v_i)-\ssq(w_i)|\notag\\
& =   |t_1^2-t_2^2| \|w\|_2^4 + t_1^2 (\|v\|_2^2+\|w\|_2^2)  \sum_{i=1}^N |\ssq(v_j) - \ssq(w_j)| \,.	\label{Lip1-bd1}
\end{align}
Now for any $a,b\in \R$,
\begin{align}
|\ssq(a)-\ssq(b)| &= |a^2-b^2|1_{ab\ge0} + (a^2+b^2)1_{ab<0} 	\notag\\
&\le |a^2-b^2|1_{ab\ge0} + |a-b|^21_{ab<0} \notag\\
&= |a-b| \big( |a+b|1_{ab\ge0} + |a-b|1_{ab<0} \big) 	\notag\\
&\le   |a-b| (|a|+|b|).	\label{bd-ssq}
\end{align}
Combining this bound with Cauchy--Schwarz we have
\begin{align*}
 \sum_{i\ge1} |\ssq(v_j) - \ssq(w_j)| \le  \sum_{i\ge1} |v_i-w_i| (|v_i|+|w_i|) \le \|v-w\|_2 ( \|v\|_2+\|w\|_2).
\end{align*}
Together with \eqref{Lip1-bd1} this implies  \eqref{Lip1}. 

Following similar lines as above we get that for $s\in\R$,
\begin{equation}	\label{Lip.Lam}
\sum_{i\ge1} \LLa_\mu(2t_1w_is) - \LLa_\mu(2t_2v_is) \ls |t_1^2-t_2^2|\|w\|_2^2s^2+ t_1^2\|v-w\|_2(\|v\|_2+\|w\|_2)s^2.
\end{equation}
Substituting $z_j$ for $s$ and summing over $j$ yields \eqref{Lip2}. 
For \eqref{Lip3}, we bound the left hand side by
\begin{align*}
\sup_{\nu\in\cP_\beta([-R,R])}\bigg\{ \int \sum_{i\ge1} \LLa_\mu(2t_1v_is) - \LLa_\mu(2t_2w_is) d\nu(s) \bigg\}
\end{align*}
and the claim follows upon substituting the bound \eqref{Lip.Lam} and integrating in $s$. 
Finally, \eqref{Lip4} is a direct consequence of \eqref{aPhi-Lipschitz}. 
\end{proof}

Recall
\[
\free_{N,R}(\theta,w) = \theta^2(1-\|w\|_2^2)^2 + \freeL(\theta,w) + \VP_R(\theta w, 1-\|w\|_2^2) - \frac12\|w\|_2^2. 
\]

\begin{proof}[Proof of Lemma \ref{lem:freeprops}\eqref{freeprop.cont}]
We only prove \eqref{freecont.theta} and \eqref{freecont.w}, the arguments for \eqref{tfreecont.theta} and \eqref{tfreecont.w} being similar (and slightly simpler). 
For \eqref{freecont.theta}, writing $\al:=\|w\|_2^2$, we bound the left hand side by
\begin{align*}
|\theta_1^2-\theta_2^2| 
+ |\freeL(\theta_1,w) - \freeL(\theta_2,w)| + |\VP_R(\theta_1w,1-\al) - \VP_R(\theta_2w,1-\al)|\,.
\end{align*}
\eqref{freecont.theta} now follows from \eqref{Lip1} and \eqref{Lip3}. 

For \eqref{freecont.w}, writing $\al_i:=\|z_i\|_2^2$, we bound the left hand side by
\begin{align*}
&
\theta^2 | (1-q_1^2\al_1)^2 - (1-q_2^2\al_2)^2| + |q_1^2\al_1 - q_2^2\al_1|  \\
& + |\freeL(\theta q_1^2,z_1) - \freeL(\theta q_2^2 z_2) |\\
&+ |\VP_R(\theta q_1 z_1, 1- q_1^2\al_1) - \VP_R(\theta q_2z_2, 1-q_1^2\al_1)| \\
&+ |\VP_R(\theta q_2 z_2, 1- q_1^2\al_1) - \VP_R(\theta q_2z_2, 1-q_2^2\al_2)|\,.
\end{align*}
The first line is bounded by 
\[
O( (1+\theta^2) |q_1^2\al_1-q_2^2\al_2|) \ls (1+\theta^2)(  |q_1^2-q_2^2| + |\al_1-\al_2|)
\le (1+\theta^2)(  |q_1^2-q_2^2| + \|z_1-z_2\|_2)
\]
as desired. The second and third lines satisfy the same bound by \eqref{Lip1} and \eqref{Lip3}. 
From \eqref{Lip4} the fourth line is bounded by
\[
O\bigg(  \frac{\theta^2}{\sqrt{\rho}} |q_1^2\al_1-q_2^2\al_2| \bigg)
\]
which is bounded by the right hand side of \eqref{freecont.w}. 
This completes the proof of \eqref{freecont.w}. 
\end{proof}

\subsection{Proof of Lemma \ref{lem:freeprops}(\ref{freeprop.large})}

We may assume without loss of generality that $\zcush\in(0,\frac12)$.
To lighten notation we write $\alpha=\|w\|_2^2$. From \eqref{VP.scaling},
\begin{equation}	\label{free.large1}
\free_{N,R}(\theta,w) 
=\freeL(\theta,w) + \theta^2(1-\al)^2 + \VP_{R/\sqrt{1-\al}}(\theta \sqrt{1-\al}w, 1)+ \frac12\log(1-\al).
\end{equation}
Since $\LLa_\mu$ is non-negative,  $\freeL(\theta,w)$ is non-negative, while our hypothesis on $w$ implies the second term is bounded below by $\zcush^2\theta^2$. Turning to the third term, 
 bounding $\LLa_\mu\ge0$, we have
\begin{align*}
\VP_{R/\sqrt{1-\al}}(\theta \sqrt{1-\al}w, 1)
&= \sup_{\nu\in\cP_1([-\frac{R}{\sqrt{1-\alpha}}, \frac{R}{\sqrt{1-\alpha}}])}\bigg\{ \LLa_\mu(2\theta w_i \sqrt{1-\alpha} s)d\nu(s) -\DKL(\nu|\gamma) \bigg\} \\
&\ge \sup_{\nu\in\cP_1([-\frac{R}{\sqrt{1-\alpha}}, \frac{R}{\sqrt{1-\alpha}}])} \{ - \DKL(\nu|\gamma)\} .
\end{align*}
Applying the estimate in \eqref{minDKL} to the first term we obtain
\begin{align}
\VP_{R/\sqrt{1-\al}}(\theta \sqrt{1-\al}w, 1)
+ \frac12\log(1-\al)&\ge \frac12\log(1-\alpha) + O(e^{-cR^2})
\ge -C\log(1/\zcush)\,.
 \label{LB.tfree}
\end{align}
Inserting this along with our other estimates in \eqref{free.large1} gives
\[
\free_{N,R}(\theta,w) 
\gs \zcush^2  \theta^2 -C\log(1/\zcush)
\gs \zcush^2 \theta^2 
\]
if $\theta\ge C\zcush^{-1}\sqrt{\log \frac1\zcush}$, as desired. 
\qed

\section{Proof of Corollary \ref{cor:nonuniv}}
\label{sec:nonuniv}

We begin with \eqref{bd1:nonuniv}.
Since $\rate^\mu(x)\le \rate^\gamma(x)$ by Theorem \ref{thm:smallx}, we may assume 
$x>2\sqrt{2}(\Delta^{1/2}+\Delta^{-1/2})$. 
First note that for any $R\ge10$,
\begin{align*}
\wt\free_{N,R}(\theta,0,\al) 
&= \theta^2 [ (1-\al)^2 + 2\al(1-\al) + 2\psimax\al^2] + \VP_{R/\sqrt{1-\al}}(0,1) + \tfrac12\log(1-\al)\\
&= \theta^2 [ 1 + 2\Delta\al^2] + \frac12\log(1-\al) + O(e^{-cR^2})
\end{align*}
where we applied \eqref{VP.scaling} and \eqref{minDKL}. 
With $\zcush_x=c_\mu x^{-4}\le \frac1{16}$ and $\eps\in(0,\frac1{10})$, since $\wt\rate_{N,N^{-\eps}}(x,\zcush)$ is monotonically decreasing in $N$ for any fixed $\zcush\in(0,\zcush_x)$, we have
\begin{align*}
\rate^\mu(x)
&\le \wt\rate_{N,N^{-\eps}}(x,\zcush)\\
& \le \inf_{\al\in[0,\frac{15}{16}]} \wt\cJ_{N,N^{1/5}}(x,0,\al)\\
&= \inf_{\al\in[0,\frac{15}{16}]}  \sup_{\theta\ge0} \Big\{ J(x,\theta) - \theta^2 ( 1+ 2\Delta\al^2 ) \Big\}- \tfrac12\log(1-\al)+O(e^{-cN^{2/5}})\\
&\le \sup_{\theta\ge0} \Big\{ J(x,\theta) - \theta^2(1+ \Delta ) \Big\}+\tfrac12+O(e^{-cN^{2/5}})
\end{align*}
where in the final line we took $\al=\frac1{\sqrt{2}}$ (noting $\log(1/(1-2^{-1/2}))<1$). 
Taking $N\to\infty$ we get
\begin{equation}	\label{nonuniv-1}
\rate^\mu(x)\le \frac12+
\sup_{\theta\ge0} \Big\{ J(x,\theta) - \theta^2 (1+\Delta)\Big\}\,.
\end{equation}
The supremum  is attained at
\begin{equation*}
\theta_x= \frac{x+\sqrt{x^2-4(1+\Delta)}}{4(1+\Delta)} < \frac14(x+\sqrt{x^2-4})= \thetap
\end{equation*}
if $x>2(1+\Delta)^{1/2}$.
Substituting $\theta_x$ for $\theta$ in \eqref{nonuniv-1} we have
\begin{align*}
\rate^\mu(x)
&\le \frac12+
J(x,\theta_x) - \theta_x^2(1+\Delta) \\
&\le\frac12+ \rate^\gamma(x) - \Delta\theta_x^2\\
&\le \rate^\gamma(x) + \frac12- \frac{\Delta x^2}{16(1+\Delta)^2}
\end{align*}
where for the second line we recall from \eqref{igamma.var} that $\rate^\gamma(x) = \sup_{\theta\ge0}\{J(x,\theta)-\theta^2\}$. We hence obtain \eqref{bd1:nonuniv}.

Turning to \eqref{bd2:nonuniv}, let $I:=[x,x')$, $\eta:=\frac14-\delta$, and let $K:=C_0x$ for a constant $C_0\ge10$ to be taken sufficiently large depending only on $\mu$.
We split $I=I_0\cup I_1$ with $I_0=[x,\min(K, x')]$ and $I_1=(\min(K, x'), x')$ (where $I_1$ may be empty). 
For arbitrary fixed $\al<\frac1{10}$, from Proposition \ref{prop:upper-joint} we have
\begin{align*}
\frac1N\log\P\big( \|v_1^{(\eta)}\|_2^2\le \al\,,\, \lam_1\in I_0\big)
&\le -\inf_{y\in I_0, \vloc\in \sqrt{\al}\B^{n_0}} \cJ_{N, N^{2\eta}}(y,\vloc) + N^{-c\delta}.
\end{align*}
From Lemma \ref{lem:Jprops}(\ref{Jprop.T},\ref{Jprop.cont},\ref{Jprop.smallz}), for any $y\in I_0, \vloc\in \sqrt{\al}\B^{n_0}$ we can estimate 
\begin{align*}
\cJ_{N, N^{2\eta}}(y,\vloc)
&= \cJ_{N, N^{2\eta}}(y,0) + 
O_{K}(\sqrt{\al})
= \rate^\gamma(y) + O_{K}(\sqrt{\al}) + O(e^{-cN^{2\eta}}). 
\end{align*}
Thus,
\begin{align}
\frac1N\log\P\big( \|v_1^{(\eta)}\|_2^2\le \al\,,\, \lam_1\in I_0\big)
&\le- \inf_{y\in I_0}\rate^\gamma(y) + O_{K}(\sqrt{\al}) +o(1) \notag\\
&= -\rate^\gamma(x) +  O_{K}(\sqrt{\al}) +o(1) .   \label{nonuniv-2}
\end{align}
From Lemma \ref{lem:tightness} we get 
\[
\frac1N\log\P( \lam_1\in I_1) \le -cK^2 = -cC_0^2x^2 .
\]
We can take $C_0$ sufficiently large that the right hand side above is bounded by $-\rate^\gamma(x) - 10$ (note from \eqref{def:igamma} that $\rate^\gamma(x) = O(x^2)$). 
Then taking $\al=a^2\delta_0^2$ with $a=a_\mu(x)>0$ sufficiently small depending on $x$ so that the right hand side in \eqref{nonuniv-2} lies in $[-\rate^\gamma(x) -\frac{\delta_0}{10}, -\rate^\gamma(x)+\frac{\delta_0}{10}]$ for all $N$ sufficiently large, it follows that 
\begin{equation}
\frac1N\log\P\big( \|v_1^{(\eta)}\|_2^2\le \al\,,\, \lam_1\in I\big)
\le  -\rate^\gamma(x) +  \frac{\delta_0}{10} +o(1) .   \label{nonuniv-3}
\end{equation}
On the other hand, from Theorem \ref{thm:fullLDP} we have 
\[
\frac1N\log\P(\lam_1\in I)
\ge \frac1N\log\P( \lam_1\in [x,x+\delta])
=-\inf_{y\in [x,x+\delta]} \rate^\mu(y)+o(1)
=-\rate^\mu(x) + o(1)
\] 
where we used the continuity and monotonicity of $\rate^\mu$.
Combining with \eqref{bd1:nonuniv} and \eqref{nonuniv-2} gives
\begin{align*}
\frac1N\log\P(  \|v_1^{(\eta)}\|_2^2\le \al\,\big|\, \lam_1\in I\big)
&= \frac1N\log\P\big( \|v_1^{(\eta)}\|_2^2\le \al\,,\, \lam_1\in I\big) - \frac1N\log\P(\lam_1\in I)\\
&\le - \rate^\gamma(x) + \rate^\mu(x) +\frac{\delta_0}{10} +o(1)
\le -\frac{\delta_0}2
\end{align*}
for all $N$ sufficiently large, as desired. \qed

\section{Proof of Theorem \ref{thm:increasing}}
\label{sec:increasing}

\subsection{Proof of Theorem \ref{thm:increasing}(a)}
\label{sec:increasing-a}

Recalling $\hfree,\hcJ$ from \eqref{def:hfree}--\eqref{def:hcJ}, for $R\ge1$ we denote
\begin{align}
\hfree_R(\theta,\al) 
&:= \theta^2\big[ (1-\al)^2 + 2\psiinfty\al^2 \big] 
+\VP_R(\theta\al^{1/2}e_1, 1-\al) - \frac\al2		\label{def:hfreeR}\\
\hcJ_R(x,\al) 
&:= \sup_{\theta\ge0} \big\{ J(x,\theta) - \hfree_R\big(\theta,\overlap_x(\theta)^2\al\big)\big\} \label{def:hcJR}
\end{align}
where $e_{1}$ denotes the first vector of the canonical basis.
The following provides analogues of Lemma \ref{lem:Jprops}(\ref{Jprop.T},\ref{Jprop.cont}) for $\hcJ_R(x,\al)$. The proof follows similar lines and is omitted.

\begin{lemma}[Properties of $\hcJ$]
\label{lem:hJprops}
\quad
\begin{enumeratea}
\item \label{hJprop.T}  
For any  $\xcush,\zcush\in(0,\frac1{10})$ there exists $T<\infty$ depending only on $\xcush,\zcush$ such that for any $R\ge1$, $x\in[2+\xcush,\xcush^{-1}]$, $\al\in[0,1-\zcush]$, the supremum in \eqref{def:hcJR} is attained in $[\thetam+T^{-1},T]$. 

\item \label{hJprop.cont} 
For any $R\ge1$, $\hcJ_R$ and $\hcJ$ are locally Lipschitz on $(2,\infty)\times[0,1)$. 
\end{enumeratea}

\end{lemma}

\begin{lemma}
\label{lem:FN.incr}
Assume $\mu$ is symmetric and $\psimu$ is non-decreasing on $\R^+$.
Then for any $\theta\ge0$, $w\in\B$ and $R\ge1$, 
\begin{align}
\free_{N,R}(\theta,w) 
&\le 
\free_{N,R}(\theta,\|w\|_2e_1)\le
\hfree_R(\theta,\|w\|_2^2)\,.	\label{FN.incr1}
\end{align}
If we further assume $R=R(N)$ is non-decreasing in $N$, then for any fixed $\theta\in(0,T]$ and $\al\in[0,1]$, 
$\free_{N,R}(\theta,\al^{1/2}e_1)$ is non-decreasing in $N$, and
\begin{equation}	
\free_{N,R}(\theta, \al^{1/2} e_1) = \hfree_R(\theta,\al) + o_T(1)	\label{FN.incr2}
\end{equation}
where the rate of convergence in $o_T(1)$ depends only on $T$ and $\mu$. 
\end{lemma}

\begin{proof}
We follow similar arguments as in \cite[Prop.\ 8]{AGH}. 
Since $\psimu$ is symmetric and non-decreasing on $\R^+$, for any probability measure $\nu$ and $v\in \ell^2(\N)$,
\begin{align}
\sum_{i}\int \LLa_\mu(v_{i}
 s)d\nu(s)
= \sum_{i}v_{i}^{2}\int s^2 \psimu(v_{i}s) d\nu(s)
&\le \sum_{i}v_{i}^{2} \int  s^2\psimu(\|v\|_2 s) d\nu(s)\nonumber\\
&=\int \LLa_\mu( \|v\|_2 s)d\nu(s)\,.	\label{gree}
\end{align}
Hence,
\begin{equation}
\VP_R(\theta w, \beta) \le \VP_R(\theta\|w\|_2e_1,\beta)
\end{equation}
for any $\theta\ge0,R\ge1$ and $\beta\in[0,1]$. 
Moreover, because $\psimu\le\psiinfty$, for any $\theta\ge0$ and $w\in\ell^2(\N)$,
\begin{align}
\freeL(\theta,w)&=
\frac1N\sum_{i\le j}\LLa_\mu(2^{{\ep_{ij}}}\theta\sqrt{N}w_{i} w_{j})	\notag\\
&=2\theta^2 \sum_{i, j} w_{i}^{2} w_{j}^{2}\psi_{\mu}(2^{\ep_{ij}} \theta\sqrt{N} w_{i} w_{j})\notag\\
&\le 2\theta^2 \|w\|_2^4 \psimu(\theta\sqrt{2N} \|w\|_2^2) = \frac1N\LLa_\mu( \theta\sqrt{2N}\|w\|_2^2) = \freeL(\theta,\|w\|_2e_1)\label{incr.loc0}\\
&\le 2 \theta^2 \psiinfty \|w\|_2^{4}	\label{incr.loc}
\end{align}
where in the first bound we used that $|w_iw_j|\le \frac12(w_i^2+w_j^2)\le \frac12\|w\|_2^2$ when $i\ne j$, and otherwise $w_i^2\le \|w\|_2^2$. 
We hence obtain \eqref{FN.incr1}. For \eqref{FN.incr2}, 
we need to show the difference between \eqref{incr.loc} and \eqref{incr.loc0} is $o_T(1)$. 
With $w=\al^{1/2}e_1$, and $M$ a large constant,
if $\theta\al\le M N^{-1/4}$ then 
\[
 \frac1N \LLa_\mu(\theta\sqrt{2N}\al) = 2\theta^2\al^2\psi_\mu( \theta\sqrt{2N}\al) \le 2N^{-1/2} \psiinfty = o(1).
\]
Otherwise, since $\psi_\mu(t) = \psiinfty + o_{|t|\to\infty}(1)$,   for $M$ large enough,
\[
2\theta^2\al^2\psi_\mu( \theta\sqrt{2N}\al) = 2\theta^2\al^2(\psiinfty +o(1))
\]
so in either case we have
\begin{equation}
\frac1N\sum_{i\le j}\LLa_\mu(2^{\ep_{ij}} \theta\sqrt{N}w_{i} w_{j})
= 2\theta^2\al\psiinfty + o_T(1)
\end{equation}
for $\theta\in[0,T]$ and $w=\al^{1/2}e_1\in\B$.
Comparing with \eqref{incr.loc}, we see from this and \eqref{gree} (with $v=2\theta w$), that \eqref{FN.incr2} holds as claimed.
\end{proof}

\begin{proof}[Proof of Theorem \ref{thm:increasing}(a)]
The claim that the infimum is attained on a closed nonempty set $A_x^*$ is immediate from the continuity of $\hcJ$ given by Lemma \ref{lem:hJprops}(\ref{hJprop.cont}). 

Denote the expression on the right hand side of  \eqref{def:rate.increasing} by $\wh\rate^\mu(x)$.
As in the proof of Theorem \ref{thm:fullLDP}, it suffices to show that $\wh\rate^\mu$ is lower-semicontinuous and that the weak LDP holds, i.e. for every fixed $x\in \R$,
\begin{equation}
\label{goal:weakLDP.i}
\lim_{\delta\downarrow0} \limsup_{N\to\infty} \frac1N\log\P(|\lam_1-x|\le \delta)
= \lim_{\delta\downarrow0} \liminf_{N\to\infty} \frac1N\log\P(|\lam_1-x|\le \delta) = -\wh\rate^\mu(x)\,.
\end{equation}
The lower-semicontinuity will follow from Theorem \ref{thm:fullLDP} once we show \eqref{goal:weakLDP.i} and hence that $\rate^\mu=\wh\rate^\mu$ (it can also be verified directly). 
For \eqref{goal:weakLDP.i}, the case $x<2$ follows from Lemma \ref{lem:tightness}. For the case $x=2$, from \eqref{FuKo} we only need to verify that $\wh\rate^\mu(2)=0$, and indeed for all $x\ge2$ we have 
\[
0\le \inf_{\al\in[0,1]}\hcJ(x,\al) \le\wh\rate^\mu(x)\le \hcJ(x,0) = \sup_{\theta\ge0} \{ J(x,\theta) - \theta^2\} = \rate^\gamma(x)
\]
and $\rate^\gamma(2)=0$. It only remains to show \eqref{goal:weakLDP.i} for fixed $x>2$. 

Fix $x>2$.
From Theorem \ref{thm:rateN} it suffices to show
\begin{equation}
\label{incr.goal1}
\rate^\mu_N(x) \to \wh\rate^\mu(x)\,.
\end{equation}
Note that a non-asymptotic one-sided bound is immediate from \eqref{FN.incr1}:
\begin{align*}
\rate^\mu_N(x) 
&=\inf_{\vloc\in(1-\zcush_x)\B^n}\cJ_N(x,\vloc) \\
&=\inf_{\vloc\in(1-\zcush_x)^{1/2}\B^n}\sup_{\theta\ge0} \big\{ J(x,\theta) - \free_{N,N^{1/5}}(\theta,\overlap_x(\theta)\vloc)\big\}\\
&\ge \inf_{\vloc\in(1-\zcush_x)^{1/2}\B^n}\sup_{\theta\ge0} \big\{ J(x,\theta) - \hfree_{N^{1/5}}(\theta,\overlap_x(\theta)^2\|\vloc\|_2^2)\big\}\\
&= \inf_{0\le \al \le 1-\zcush_x}\sup_{\theta\ge0} \big\{ J(x,\theta) - \hfree_{N^{1/5}}(\theta,\overlap_x(\theta)^2\al)\big\}\\
&\ge \inf_{0\le \al \le 1-\zcush_x}\sup_{\theta\ge0} \big\{ J(x,\theta) - \hfree(\theta,\overlap_x(\theta)^2\al)\big\}\\
&=\wh\rate(x)
\end{align*}
(in the second line we used the third point in Remark \ref{rmk:dropUSG} to replace $1-\zcush_x$ with $(1-\zcush_x)^{1/2}$). 
For an asymptotically matching upper bound, fix an arbitrary $\al\in[0,1-\zcush_x]$. 
From Lemma \ref{lem:Jprops}(\ref{Jprop.T}) and the fact that $\zcush_x=c_\mu/x^4$ depends only on $x$ and $\mu$, there exists $T\ls_x1$ such that 
\begin{align*}
\rate^\mu_N(x) &\le \cJ_N(x,\al^{1/2}e_1)
=\sup_{\theta\in[\thetam+T^{-1},T]}\big\{ J(x,\theta) - \free_{N,N^{1/5}}(\theta,\overlap_x(\theta)\al^{1/2}e_1)\big\}\,.
\end{align*}
It only remains to show
\begin{equation}
\label{incr.goal2}
\sup_{\al\in[0,  1-\zcush_x]} \inf_{\theta\in[\thetam+T^{-1},T]} \phi_N(\al,\theta)
\to \sup_{\al\in[0,1-\zcush_x]} \inf_{\theta\in[\thetam+T^{-1},T]}\phi(\al,\theta)
\end{equation}
as $N\to\infty$, where
\[
\phi_N(\al,\theta):= 
\free_{N,N^{1/5}}(\theta,\overlap_x(\theta)\al^{1/2}e_1)-J(x,\theta)\,,
\qquad
\phi(\al,\theta):= \hfree(\theta,\overlap_x(\theta)^2\al) -J(x,\theta)\,.
\]
First, we claim $\phi_N$ increases pointwise to $\phi$ on $[0,1)\times\R^+$. 
Indeed, from Lemma \ref{lem:FN.incr} we have that for fixed $\al,\theta$, the sequence $\phi_N(\al,\theta)$ is monotone in $N$, and by \eqref{FN.incr2} we only need to show
\begin{equation}	\label{incr.goal3}
\VP_R(\theta(1-\beta)^{1/2}e_1,\beta)\uparrow \VP_\infty(\theta(1-\beta)^{1/2}e_1,\beta)
\end{equation}
as $R\to\infty$ for fixed $\theta\ge0$ and $\beta\in(0,1]$. 
But \eqref{incr.goal3} follows directly from Proposition \ref{prop:gibbs}(\ref{gibbs.lim}). 

Next, we claim that $\phi_N$ is continuous in $\theta$ for each fixed $\al\in[0,1)$ (in fact it is jointly continuous on $[0,1)\times \R^+$).
Indeed, this is a consequence of Lemma \ref{lem:freeprops}(\ref{freeprop.cont}) and the continuity of $\theta\mapsto\overlap_x(\theta)$. 
\eqref{incr.goal3} now follows from Lemma \ref{lem:maxmin} below. 
\end{proof}

In the proof above we used the following elementary fact.

\begin{lemma}
\label{lem:maxmin}
If $g_n:X\to\R$ is a monotone non-decreasing sequence of functions on a set $X$ converging pointwise to a function $g$, then
\begin{equation}	\label{maxmin1}
\sup_{x\in X} g_n(x) \uparrow \sup_{x\in X}g(x).
\end{equation}
If we further assume $X$ is a compact topological space and $g_n$ is continuous for each $n$, then
\begin{equation}
\label{maxmin2}
\inf_{x\in X} g_n(x) \uparrow \inf_{x\in X} g(x).
\end{equation}
As a consequence, for an arbitrary set $X$ and a compact topological space $Y$, if $f_n:X\times Y\to \R$ is a monotone non-decreasing sequence converging pointwise to some $f:X\times Y\to\R$, and $f_n(x,\cdot)$ is continuous for each fixed $n$ and $x$, then
\begin{equation}\label{maxmin3}
\sup_{x\in X}\inf_{y\in Y} f_n(x,y) \uparrow \sup_{x\in X}\inf_{y\in Y} f(x,y).
\end{equation}
\end{lemma}

\begin{proof}
To deduce \eqref{maxmin3} from the first two claims, from \eqref{maxmin1} it suffices to show that for fixed $x\in X$, $\inf_{y\in Y} f_n(x,y)\uparrow\inf_{y\in Y} f(x,y)$. But this follows from \eqref{maxmin2} with $Y$ in place of $X$ and $f_n(x,\cdot),f(x,\cdot)$ in place of $g_n,g$.

Turning to establish the first two claims,
\eqref{maxmin1} is obvious. For \eqref{maxmin2}, the sequence on the left hand side is clearly monotone and bounded by the right hand side for every $n$.
Letting $a:=\inf_{x\in X}g(x)$, assume toward a contradiction that $a_n:=\inf_{x\in X} g_n(x) \to b\le a-\eps$ for some $\eps>0$. 
Since $X$ is compact there is a sequence $(x_n)$ with $g_n(x_n)=a_n$ for all $n$, and (passing to a subsequence) with $x_n$ converging to some $x_*$. Since $g_n\uparrow g$, we can pass to a further subsequence to assume $g_1(x_*)\ge b+\frac\eps2$. Since $g_1$ is continuous, there exists an open neighborhood $U$ of $x_*$ such that $g_1\ge b+\frac\eps4$ on $U$. 
But since $x_n\to x_*$ there exists $m$ such that $x_m\in U$, 
and hence $b\ge a_m =g_m(x_m) \ge g_1(x_m) \ge b+ \frac\eps4$, a contradiction. 
\end{proof}

\subsection{Proof of Theorem \ref{thm:increasing}(b)}
\label{sec:increasing-b}

The claim is a consequence of the following:

\begin{prop}
\label{prop:increasing}
Let $\xcush,\eta,\eps\in(0,\frac1{10})$. For any interval $I\subset[2+\xcush,\xcush^{-1}]$ of length at least $2N^{-1/20}$,
\begin{equation}	\label{incr1}
\P\Big( \|v_1^{(\eta)}\|_\infty \ge \|v_1^{(\eta)}\|_2-\eps  \,\Big|\, \lam_1\in I\Big)
\ge1- \exp ( - c_0\eps^{3}N + N^{1-c\eta})
\end{equation}
for all $N$ sufficiently large depending on $\xcush,\eta$ and $\mu$, where $c_0>0$ depends only on $\xcush$ and $\mu$.

Furthermore, for any $x>2$ and $\eta,\eps\in(0,\frac1{10})$ there exist $\delta_0,\delta_1>0$ depending only on $x,\eps$ such that for any $\delta\in(0,\delta_0)$, 
\begin{equation}	\label{incr2}
\P\Big( \sqrt{\al_x^*} - \eps\le  \|v_1^{(\eta)}\|_2 \le \|v_1^{(\eta)}\|_\infty+\eps \,\Big| \, |\lam_1-x|\le \delta\Big) 
\ge 1- e^{-\delta_1N}
\end{equation}
for all $N$ sufficiently large depending on $x,\eta,\eps,\delta$ and $\mu$.
\end{prop}

From \eqref{FN.incr1} in Lemma \ref{lem:FN.incr} it follows that for $\psimu$ symmetric and non-decreasing on $\R^+$,
\begin{equation}	\label{Jbounds.incr}
\cJ_{N,R}(x,\vloc) \ge \cJ_{N,R}(x,\|\vloc\|_2e_1) \ge \hcJ(x,\|\vloc\|_2^2)\qquad \forall \theta\ge0,R\ge1,\vloc\in\B.
\end{equation}
The following gives a stronger stability form of this bound when $\psimu$ is strictly increasing on $\R^+$.

\begin{lemma}
\label{lem:JNRe1}
Let $\xcush,\rho\in(0,\frac1{10})$. There exists $c_0>0$ depending only on $\xcush,\rho$ and $\mu$ such that for any  $R\ge2$, $x\in[2+\xcush,\xcush^{-1}]$ and $\vloc\in (1-\rho)\B$,
\begin{equation}
\cJ_{N,R}(x,\vloc) 
\ge
\cJ_{N,R}(x,\|\vloc\|_2e_1)  + c_0 \|\vloc\|_2^2(\|\vloc\|_2-\|\vloc\|_\infty) .
\end{equation}
\end{lemma}

\begin{proof}
Recall from \eqref{VP.scaling} that for any $v\in\ell^2(\N)$ and $\al\in (0,1]$,
\begin{align*}
\VP_R(v,\al) 
& = \VP_{R/\sqrt{\al}}(\al^{1/2}v, 1) + \frac12(1-\al)+\frac12\log\al\\
& = \sup_{\nu\in\cP_1(I)} \Big\{ \int h d\nu - \DKL(\nu|\gamma)\Big\} +  \frac12(1-\al)+\frac12\log\al
\end{align*}
with $I:= [-R/\sqrt\al, R/\sqrt\al]$.
We claim the supremum is attained on 
\begin{equation*}
\cP'_{1,\eps_0}(I) = \Big\{ \nu\in \cP_1(I): \nu\ll \gamma, \frac{d\nu}{\d\gamma} \ge \eps_0 \; \gamma\text{-a.e. on } [-2,2]\Big\}
\end{equation*}
with 
\begin{equation}	\label{incr.eps0}
\eps_0= \exp( - C(1+\psiinfty\al\|v\|_2^2)).
\end{equation}
Indeed, writing  $h(s) = \sum_i \LLa_\mu(2\al^{1/2}v_i s)$, we know from Proposition \ref{prop:gibbs} that the supremum is attained at a measure $\hat\nu$ with density
\[
\frac{d\hat\nu}{ds} = 1_I(s)\exp( h(s)-\hat\zeta s^2 - \log Z(\hat\zeta))  \,,\qquad Z(\zeta):= \int_I e^{h(s) - \zeta s^2}ds
\]
with $\hat\zeta$ the unique positive real such that $\int s^2\hat\nu(s) = 1$. 
Then, from Proposition \ref{prop:gibbs}(\ref{gibbs.density}),
\begin{align*}
\log Z(\hat\zeta) + \hat\zeta - \frac12\log(2\pi e)
&=\sup_{\nu\in \cP_1(I)}\{ \int hd\nu-\DKL(\nu|\gamma) \} \\
&= \int hd\hat\nu - \DKL(\hat\nu|\gamma) \\
&\le \int hd\hat\nu \\
&= 4\al\sum_i v_i^2\int s^2\psimu( 2\al^{1/2}v_i s) d\hat\nu(s)\\
&\le 4\psiinfty\al\|v\|_2^2 \,.
\end{align*}
Thus,
\begin{align*}
\frac{d\hat\nu}{ds} \ge 1_I(s)\exp( -\hat\zeta s^2 - \log Z(\hat\zeta)) 
\ge 1_I(s)\exp( - (1+s^2)( \frac12\log(2\pi e) + 4 \psiinfty\al\|v\|_2^2 ) ) 
\end{align*}
as desired. 

Now for $\theta\ge0$ and $w\in\B$ denote
\begin{equation}
h_{\theta,w}(s):= \sum_i \LLa_\mu(2\theta \sqrt{1-\|w\|_2^2} w_is)
= 4\theta^2(1-\|w\|_2^2) s^2 \sum_i w_i^2\psimu(2\theta\sqrt{1-\|w\|_2^2}w_is)\,.
\end{equation}
For any $R\ge2$, $\nu\in \cP'_{1,\eps_0}([-R,R])$, $0\le \theta\le T$ and $w\in \B$ with $\|w\|_2^2=\beta$, we have
\begin{align*}
\int h_{\theta,\beta^{1/2}e_1} -h_{\theta,w}d\nu
&= 4\theta^2 (1-\beta) \sum_i w_i^2 \int s^2\big[ \psimu(2\theta\sqrt{\beta(1-\beta)}s) - \psimu(2\theta\sqrt{1-\beta}w_is) \big] d\nu(s)\\
&\gs \theta^3(1-\beta)^{3/2} (\inf_{t\in [0,4T]} \psi'_\mu(t)) \sum_i w_i^2(\beta^{1/2}-|w_i|) \int_0^2 s^3d\nu(s)\\
&\gs \theta^3(1-\beta)^{3/2} \beta (\beta^{1/2}-\|w\|_\infty) (\inf_{t\in [0,4T]} \psi'_\mu(t)) \nu([1,2])\\
&\gs \eps_0\theta^3(1-\beta)^{3/2}\beta (\beta^{1/2}-\|w\|_\infty) (\inf_{t\in [0,4T]} \psi'_\mu(t))
\end{align*}
and hence
\begin{align*}
 \free_{N,R}(\theta,\beta^{1/2}e_1)- \free_{N,R}(\theta,w)
&\ge\VP_R(\theta\beta^{1/2}e_1,1-\beta) - \VP_R(\theta w, 1-\beta) \\
&= \VP_{R/\sqrt{1-\beta}}(\theta\sqrt{\beta(1-\beta)} e_1,1) - \VP_{R/\sqrt{1-\beta}}(\theta\sqrt{1-\beta} w,1) \\
&\gs \exp( - C(1+T^2)) \theta^3(1-\beta)^{3/2}\beta (\beta^{1/2}-\|w\|_\infty) (\inf_{t\in [0,4T]} \psi'_\mu(t))
\end{align*}
where we applied \eqref{incr.eps0} with $v=\theta w$ and $\al=1-\beta$.

For any $\xcush,\zcush\in(0,\frac1{10})$ we can take $T$ sufficiently large depending only on $\kappa,\rho$ such that for any $x\in[2+\xcush,\xcush^{-1}]$ and any $\vloc\in \B$ with $\|\vloc\|_2^2=\al\le (1-\zcush)^2$,
abbreviating $B:=[\thetam+T^{-1},T]$ and $\overlap=\overlap_x(\theta)$, we have
\begin{align*}
&\cJ_{N,R}(x,\vloc) 
= \sup_{\theta\in B} \{J(x,\theta)-\free_{N,R}(\theta,\overlap\vloc) \}\\
&\ge \sup_{\theta\in B} \{ J(x,\theta) - \free_{N,R}(\theta,q\al^{1/2}e_1) \} 
+e^{-C(1+T^2)} (\al^{1/2}-\|\vloc\|_\infty)  (\inf_{t\in [0,4T]} \psi'_\mu(t))\inf_{\theta\in B}\theta^3 (1-q^2\al)^{3/2}q^3\al\\
&\ge \sup_{\theta\in B} \{ J(x,\theta) - \free_{N,R}(\theta,q\al^{1/2}e_1) \}  +c_0 \al (\al^{1/2}-\|\vloc\|_\infty) \\
&= \cJ_{N,R}(x,\al^{1/2}e_1) + c_0 \al (\al^{1/2}-\|\vloc\|_\infty)
\end{align*}
as desired.
\end{proof}

\begin{proof}[Proof of Proposition \ref{prop:increasing}]
We begin with \eqref{incr1}.
From Proposition \ref{prop:no-comp} there exists $\zcush=\zcush(\xcush)>0$ such that it suffices to show
\begin{equation}
\label{incr-goal1}
\frac1N\log\P\big(v_1^{(\eta)}\in A_\eps \,\big|\, \lam_1\in I\big) \le  -c_0\eps^3  + N^{-c\eta})
\end{equation}
where
\begin{equation}
A_\eps = \{ \vloc\in (1-\zcush)\B^N: \|\vloc\|_\infty\le \|\vloc\|_2-\eps \}\,.
\end{equation}
Letting $R=N^{1/8}$, from Proposition \ref{prop:upper-joint} we have
\begin{align*}
\frac1N\log \P\big(v_1^{(\eta)}\in A_\eps \,,\, \lam_1\in I\big)
&\le -\inf_{\substack{y\in I\\ \vloc\in A_\eps\cap\B^{n_0}}} \cJ_{N,R}(y,\vloc) + O_\xcush(N^{-c\eta})\\
&\le -\inf_{\substack{y\in I\\ \vloc\in A_\eps\cap \B^{n_0}}} \Big\{ \cJ_{N,R}(y,\|\vloc\|_2e_1) + c_0\|\vloc\|_2^2(\|\vloc\|_2-\|\vloc\|_\infty)\Big\}+  O_\xcush(N^{-c\eta})\\
&\le -\inf_{\substack{y\in I\\ 0\le\al\le (1-\zcush)^2 }} \Big\{ \cJ_{N,R}(y,\sqrt\al e_1)\Big\} - c_0\eps^3+  O_\xcush(N^{-c\eta})
\end{align*}
where we applied Lemma \ref{lem:JNRe1} in the second line.
On the other hand, from Lemma \ref{lem:lower1},
 \begin{align*}
  \frac1N\log \P\big( \lam_1\in I\big)
&\ge - \inf_{\substack{y: [y-N^{-1/20},y+N^{-1/20}]\subseteq I \\ \vloc\in (1-\zcush)\B^{n_0}}} \cJ_{N,R}(y,\vloc) + O_\xcush(N^{-c\eta})\\
&= -\inf_{\substack{y\in I \\ \vloc\in (1-\zcush)\B^{n_0}}} \cJ_{N,R}(y,\vloc) + O_\xcush(N^{-c\eta})\\
&\ge -\inf_{\substack{y\in I \\ 0\le \al\le (1-\zcush)^2}} \cJ_{N,R}(y,\sqrt{\al}e_1) + O_\xcush(N^{-c\eta})\,.
\end{align*}
Combining these bounds we get
\begin{align*}
\frac1N\log\P\big(v_1^{(\eta)}\in A_\eps \,\big|\, \lam_1\in I\big)
&= 
\frac1N\log \P\big(v_1^{(\eta)}\in A_\eps \,,\, \lam_1\in I\big)
- \frac1N\log \P\big( \lam_1\in I\big)\\
&\le -c_0\eps^3 + O_\xcush(N^{-c\eta})
\end{align*}
which gives \eqref{incr-goal1} and hence \eqref{incr1}.

Turning to \eqref{incr2}, with $\delta_0,\delta_1$ to be chosen sufficiently small and $\delta\in(0,\delta_0)$, from \eqref{incr1} we only need to show 
\begin{equation}
\label{incr-goal2}
\frac1N\log\P\big(v_1^{(\eta)}\in (\sqrt{\al_x^*}-\eps)\B^N \,\big|\, |\lam_1-x|\le\delta\big) \le  -\delta_1 \,.
\end{equation}
Again with $R=N^{1/8}$, from Proposition \ref{prop:upper-joint}, 
\begin{align*}
\frac1N\log \P\big(v_1^{(\eta)}\in (\sqrt{\al_x^*}-\eps)\B^N \,,\, |\lam_1-x|\le \delta\big)
\le -\inf_{\substack{|y-x|\le \delta\\ \vloc\in (\sqrt{\al_x^*}-\eps)\B^{n_0}}} \cJ_{N,R}(y,\vloc) + O_\xcush(N^{-c\eta})
\end{align*}
From Lemma \ref{lem:Jprops}(\ref{Jprop.T},\ref{Jprop.cont}) we have that $\cJ_{N,R}(y,\vloc)$ is $O_\xcush(1)$-Lipschitz on $[2+\xcush,\xcush^{-1}]$, so 
\begin{align}
\frac1N\log \P\big(v_1^{(\eta)}\in (\sqrt{\al_x^*}-\eps)\B^N \,,\, |\lam_1-x|\le \delta\big)
&\le -\inf_{ \vloc\in (\sqrt{\al_x^*}-\eps)\B^{n_0}} \cJ_{N,R}(x,\vloc) + O_\xcush(\delta_0+N^{-c\eta}) \notag\\
&\le -\inf_{0\le \al\le\al_x^*-\eps^2} \hcJ(x,\al) +  O_\xcush(\delta_0+N^{-c\eta})	\label{incr2.2}
\end{align}
where we used \eqref{Jbounds.incr} in the second line. 
Since $\hcJ(x,\cdot)$ is continuous on $[0,1)$ by Lemma \ref{lem:hJprops}(\ref{hJprop.cont}) and $\al_x^*=\inf\{ \al\in[0,1-\zcush_x]: \hcJ(x,\al) = \rate^\mu(x)\}$, it follows that 
\[
\inf_{0\le \al\le\al_x^*-\eps^2} \hcJ(x,\al) \ge \rate^\mu(x) + 2\delta_1
\]
if $\delta_1$ is sufficiently small depending on $x$ and $\eps$. 
Substituting this into \eqref{incr2.2} and taking $\delta_0$ sufficiently small, the desired bound \eqref{incr-goal2} then follows from Theorem \ref{thm:increasing}(a). 
\end{proof}

\begin{appendix}

\section{Concentration properties for sub-Gaussian Wigner matrices}
\label{app:conc}

In this appendix we gather some concentration of measure tools and use them to prove Lemma \ref{lem:tightness}, Lemma \ref{lem:good}, and Proposition \ref{prop:tiltP}(a).

\subsection{Extension of an inequality of Talagrand to sub-Gaussian variables}
\label{app:talagrand}

A well-known result of Talagrand (see \cite[Theorem 6.6]{talagrand}) states that for a random vector $X\in [-1,1]^d$ with independent components and a convex 1-Lipschitz function $f:[-1,1]^d\to \R$, we have
\begin{equation}
\label{talagrand}
\P( |f(X) - m|\ge s) \le 4 \exp( - s^2/16)\,,\qquad s\ge0
\end{equation}
for any median $m\in \R$ of the random variable $f(X)$.
A similar tail bound 
holds for vectors with entries enjoying a log-Sobolev inequality, without the requirement that $f$ be convex. 
Both results have been widely applied in random matrix theory.
In particular, it was shown in \cite{GuZe} that for Wigner matrices $H$ with entries either having bounded support or enjoying a log-Sobolev inequality, the ESD $\hat\mu_H$ concentrates around the semicircular measure $\sigma$ with speed $N^2$.

Talagrand's inequality is often combined with truncation  arguments to treat Wigner matrices with unbounded entries. However, a straightforward truncation argument with a union bound is generally insufficient for applications to large deviations, as the exceptional event must be small compared to the rare event of interest. 

The following result provides an extension of \eqref{talagrand} to sub-Gaussian variables with only a logarithmic loss in the exponent, which is more than sufficient for our purposes. 
In particular, in Corollary \ref{cor:conc} we deduce concentration for convex linear statistics with a tail speed of $N^2/\log N$. 
The key is to control the impact of the truncation error in $\ell^2$ rather than $\ell^\infty$. 

\begin{prop}
\label{prop:conc-subG}
Let $X=(X_1,\dots, X_d)$ be a vector of independent random variables that are $K$-sub-Gaussian
for some 
$K>0$. Let $f:\R^d\to \R$ be a convex 1-Lipschitz function. 
Then 
\begin{equation}	\label{CL.tail}
\P( |f(X) - \E f(X)|\ge s) \le 2\exp( -\frac{cs^2}{K^2\log d})\,, \qquad s\ge0
\end{equation}
for a universal constant $c>0$. Moreover, the same bound holds (with a possibly modified value of $c$) with $\E f(X)$ replaced by any median of $f(X)$. 
\end{prop}

\begin{remark}
Since the initial posting of this article to arXiv we learned 
of stronger versions of Proposition \ref{prop:conc-subG} obtained in \cite{KlZh,HuTi} by more elaborate arguments. We include the proof below as it is simpler than the arguments in those works, while the tail bound \eqref{CL.tail} is already more than sufficient for applications in random matrix theory (in our applications it yields bounds of shape $\exp(-cN^2/\log N)$ where we only need $\exp(-\omega(N))$).

Specifically, \cite[Lemma 1.6]{KlZh} shows we can take $2\exp( -cs^2/M^2)$ on the right hand side in \eqref{CL.tail}, where $M$ is any constant such that $\max_{i\le d} |X_i|$ is $M$-sub-Gaussian. 
Since it is elementary that $\max_{i\le d} |X_i|$ is $O(K\sqrt{\log d})$-sub-Gaussian when $X_i$ are $K$-sub-Gaussian, \cite[Lemma 1.6]{KlZh} implies Proposition \ref{prop:conc-subG}.
\cite[Theorem 1.1]{HuTi} generalizes Proposition \ref{prop:conc-subG} to the case that the $X_i$ have uniformly bounded $\psi_p$-norm (the sub-Gaussian case being $p=2$); such a generalization also follows from an easy modification of the argument below. 
Moreover, \cite[Theorem 1.3]{HuTi} provides optimal tail bound for the sub-Gaussian case, showing that $\log d$ on the right hand side of \eqref{CL.tail} can be replaced by $\log(2+ \frac{K^2d}{s^2})$.
\end{remark}

\begin{proof}
By replacing $f$ with $K^{-1}f(K\cdot)$ and rescaling $X,s$ we may assume $K=1$.
By adjusting the constant $c$ in \eqref{CL.tail} we may assume
\begin{equation}	\label{CL.assume-eps}
s\ge C_0\sqrt{\log d}
\end{equation}
for any fixed constant $C_0>0$. 
Similarly, by adjusting $c$ it suffices to establish \eqref{CL.tail} with the prefactor $2$ replaced by any constant $C>0$. 
For the claim about deviation from a median $m$, we simply note that \eqref{CL.tail} implies that $m=\E f(X)+ O(\sqrt{\log d})$, so the claim follows by applying the triangle inequality and \eqref{CL.tail} with $s/2$ in place of $s$, and adjusting constants. 

Letting $Y_i:= X_i-\E X_i$ 
we have that $Y_i$ is $O(1)$-sub-Gaussian (see \cite[Lemma 2.6.8]{Vershynin:book}) for each $1\le i\le d$.
Letting $g(z):= f(\E X + z)$, which is convex and 1-Lipschitz, it suffices to show
\begin{equation}	\label{CL.goal1}
\P( | g(Y) - \E g(Y)|\ge s) \ls \exp( - \frac{cs^2}{\log d})\,,\qquad s\ge0
\end{equation}
for a possibly modified constant $c>0$.
We split
\[
Y= Y^\le + Y^>\,, \quad Y^\le_i := Y_i \ind(|Y_i|\le B)\,,\qquad B:= C_1\sqrt{\log d}\,.
\]
where 
$C_1>0$ is large enough that $Y_i$ is $\frac12C_1$-sub-Gaussian for all $i$.
We first claim 
\begin{equation}	\label{CL.claim1}
\P( \|Y^>\|_2 > t) \le \exp(1-\frac{t^2}{C_1^2})\qquad\forall \,t\ge 0.
\end{equation}
Indeed, the left hand side above is bounded by 
\begin{equation}	\label{CL.1}
\P( \sum_{i\le d} |Y^>_i|^2 >  t^2) 
\le \exp( - \frac{2t^2}{C_1^2}) \prod_{i\le d} \E \exp( \frac{2|Y^>_i|^2}{C_1^2} ).
\end{equation}
For the exponential moments on the right hand side we have
\[
\E\exp(\frac{2 |Y_i^>|^2}{C_1^2}) 
\le 1+ \E \ind(|Y_i|>B) e^{2Y_i^2/C_1^2}
\le 1+ \E\exp\Big(\frac{2Y_i^2-2B^2}{C_1^2}\Big) e^{2Y_i^2/C_1^2}
\le 1+2e^{-2B^2/C_1^2}
\]
where we used that the $Y_i$ are $\frac12C_1$-sub-Gaussian. 
Further bounding the right hand side by $\exp( 2\exp( - \frac{2B^2}{C_1^2}))$ and substituting into \eqref{CL.1} we get
\begin{equation*}
\P( \|Y^>\|_2 > t) 
\le \exp( - \frac{2t^2}{C_1^2} +2d \exp( - \frac{2B^2}{C_1^2})) 
=  \exp( - \frac{2t^2}{C_1^2} +\frac2d) 
\end{equation*}
and \eqref{CL.claim1} follows (clearly we may assume $t\ge C_1 $). 

By the Lipschitz assumption we have
$
| g(Y) - g(Y^\le) |\le \|Y^>\|_2
$,
so from \eqref{CL.claim1},
\begin{equation}	\label{CL.2}
\P( |g(Y) - g(Y^\le)|> t) \le \P ( \|Y^>\|_2 > t) \le \exp( 1- \frac{t^2}{C_1^2})\quad\forall t\ge0.
\end{equation}
Integrating this bound gives
\begin{equation}	\label{CL.3}
| \E g(Y) - \E g(Y^\le) | \le \E | g(Y) -  g(Y^\le) | =O(1). 
\end{equation}

From \eqref{talagrand} we have that for any median $m$ of $g(Y^\le)$,
\begin{equation}	\label{CL.4}
\P( |g(Y^\le) - m|> t) \ls  \exp( -  c't^2 / B^2)
\end{equation}
for some universal constant $c'>0$. 
In particular,
\begin{equation}	\label{CL.5}
|\E  g(Y^\le )-m|\le \E| g(Y^\le) - m| \ls B \ls  \sqrt{\log d}.
\end{equation}
Now taking $C_0$ in \eqref{CL.assume-eps} sufficiently large to bound the right hand sides of \eqref{CL.3} and \eqref{CL.5} by $s$, and 
combining with \eqref{CL.2} and \eqref{CL.4} with $t=s$, we get
\[
\P( | g(Y) - \E  g(Y)| > 4s) \ls \exp( -\frac{ c''s^2}{\log d})
\]
for a universal constant $c''>0$. 
Replacing $s$ with $s/4$ and adjusting constants, we obtain \eqref{CL.goal1} and hence the claim. 
\end{proof}

\subsection{Concentration of linear statistics}
\label{app:conc.linear}

From an argument in \cite{GuZe} we have the following corollary of Proposition \ref{prop:conc-subG}.

\begin{cor}[Concentration of linear statistics]
\label{cor:conc}
Let $\sqrt{N}H\in \Sym_N$ have independent $K$-sub-Gaussian entries on and above the diagonal. 
There exists a universal constant $c>0$ such that
for any convex and 1-Lipschitz function $f:\R\to \R$ and any $t>0$, we have
\begin{equation}	\label{conc.bound}
\P( |\Tr f(H) - \E \Tr f(H)|> t) \le 2 \exp( -\frac{ct^2}{K^2\log N}).
\end{equation}
The same holds with $\E \Tr f(H)$ replaced with any median of $\Tr f(H)$ (up to modification of the constant $c$). 
\end{cor}

\begin{remark}
An extension to non-convex linear statistics could be obtained by the same lines as in \cite{GuZe}, but the statistics of interest in the present work happen to all be convex.
We note that extensions of results in \cite{GuZe} in a similar spirit were recently established in \cite{HuMc} under the stronger sharp sub-Gaussian assumption. The arguments in \cite{HuMc} use the interlacing property of eigenvalues, similarly to the proof of Lemma \ref{lem:ktail} below, rather than proceeding through a general concentration estimate for functions on product spaces like Proposition \ref{prop:conc-subG}.
\end{remark}

\begin{proof}
From \cite[Lemma 1.2]{GuZe} we have that $\Tr f(\frac1{\sqrt{N}}\,\cdot)$ is a convex and $O(1)$-Lipschitz function on $\Sym_N$ with the Euclidean Hilbert--Schmidt metric.  We may hence apply Proposition \ref{prop:conc-subG} with $d={N+1\choose 2}$ (identifying $\Sym_N$ with $\R^d$) and $\Tr f(H)$ in place of $f$ (rescaled to have Lipschitz constant 1). 
\end{proof}

By combining Corollary \ref{cor:conc} with the following, we can show concentration of linear statistics around their value at the semicircle measure instead of the mean or median. 

\begin{lemma}[Weak concentration of linear statistics]
\label{lem:median}
Let $H$ be a Wigner matrix as in \eqref{def:H} with $X_{ij}$ standardized and uniformly sub-Gaussian.
For any $L$-Lipschitz function  $f:\R\to \R$ and any $\eps>0$, we have
\begin{equation}	\label{weak.bound}
\hat{\mu}_H(f) =\sigma(f) + O_\eps(LN^{-1/2+\eps}) 
\end{equation}
with probability $1-o(1)$. 
In particular, for any median $m$ of $\hat{\mu}_H(f)$ and any $\eps>0$ we have
\begin{equation}
\label{weak.median}
|m-\sigma(f)| \le LN^{-1/2 + \eps}
\end{equation}
for all $N$ sufficiently large depending on $\eps$ and $\mu$. 
\end{lemma}

\begin{proof}
By replacing $f$ with $f/L$ we may assume $f$ is 1-Lipschitz. 
By subtracting $f(0)$ from both sides of \eqref{weak.bound} we may assume $f(0)=0$, and in particular that $\|f\|_\infty=O(1)$. 

Let $\cI$ be a collection of $O(N^{1/2-\eps})$ intervals of length $\asymp N^{-1/2+\eps}$ with disjoint interiors covering $[-3,3]$. 
From the local semicircle law (see for instance \cite[Theorem 2.8]{BeKn}) and the union bound we have
\begin{equation}	\label{median.locallaw}
\hat{\mu}_H(I) = \sigma(I) + O_\eps(N^{-1+2\eps}) \qquad \forall I\in \cI
\end{equation}
with probability $1-o(1)$.
On the other hand, from \eqref{FuKo} applied with $H$ and $-H$ and the union bound, we have 
\begin{equation}	\label{median.support}
\hat{\mu}_H((-3,3)^c)=0
\end{equation}
with probability $1-o(1)$.
Let $h$ be a function supported on $[-3,3]$ that is constant on intervals in $\cI$ with $\|f-h\|_{L^\infty([-3,3])} =O(N^{-1/2+\eps})$. 
On the event that \eqref{median.support} and \eqref{median.locallaw} hold we thus have
\begin{align*}
|\hat{\mu}_H(f) - \sigma(f)|
&\le |\hat{\mu}_H(f-h)| + |(\hat{\mu}_H-\sigma)(h)|\\
&\le \|f-h\|_{L^\infty([-3,3])} + \|h\|_\infty \sum_{I\in\cI} |(\hat{\mu}_H-\sigma)(I)|\\
&\ls N^{-1/2+\eps} + |\cI| O_\eps(N^{-1+2\eps}) \ls_\eps N^{-1/2+\eps}.
\end{align*}
as desired. 
\end{proof}

From Corollary \ref{cor:conc} and Lemma \ref{lem:median} we deduce the following, which a generalization of \eqref{tight-Lside} in Lemma \ref{lem:tightness}.

\begin{lemma}
\label{lem:Lside.app}
With $H$ as in Corollary \ref{cor:conc},
for any $\eps\ge N^{-1/10}$ and all $N$ sufficiently large,
\[
\P( \lam_1(H)\le 2-\eps) \le 2\exp( - \frac{c\min(\eps^4,1) N^2}{K^2\log N})\,.
\]
\end{lemma}

\begin{proof}
We may assume $\eps\in(0,1)$ (the claim for larger $\eps$ follows from the case $\eps=1$). Let $f:\R\to\R$ be the convex 1-Lipschitz function given by $f(x):= \max( 0, x-2+\eps)$.
Then $\sigma(f) \ge c'\eps^2$ for some absolute constant $c'>0$, so letting $m$ be a median for $\hat\mu_H(f)=\frac1N\Tr f(H)$,  from Lemma \ref{lem:median} we have that $m\ge c'\eps^2-O(N^{-1/4})\ge \frac12c'\eps^2$ for all $N$ sufficiently large. Thus,
\begin{align*}
\P( \lam_1(H)\le 2-\eps) 
&\le \P( \hat\mu_H(f) =0)\le \P( \Tr f(H) - mN <- \tfrac12c'\eps^2N)
\end{align*}
and the claim now follows from Corollary \ref{cor:conc}.
\end{proof}

\subsection{Concentration of the largest eigenvalue of tilted matrices}

Recall the tilted probability measures $\P^{(\theta,u)}$ defined in \eqref{def:tiltP},
under which $H$ is a symmetric random matrix with  independent entries. 
We have the following corollary of Proposition \ref{prop:conc-subG}, which generalizes Proposition \ref{prop:tiltP}(a).

\begin{cor}
\label{cor:lam1-tilt.conc}
There is a universal constant $c>0$ such that the following holds for any $N\in \N$, $\theta\ge0$ and $u\in\sphereN$.
Let $\sqrt{N}H\in\Sym_N$ have independent centered entries on and above the diagonal with laws $\mu_{ij},1\le i\le j\le N$.
Assume $\max_{1\le i\le j\le N}\sup_{t\in\R}\LLa_{\mu_{ij}}''(t)\le K^2<\infty$.
Then 
\begin{equation}
\P^{(\theta,u)}( |\lam_1(H)- \E^{(\theta,u)}\lam_1(H)|\ge s)
\le 2\exp( -\frac{cs^2N}{K^2\log N})\,,\qquad s\ge0.
\end{equation}
Moreover, the same bound holds (with a possibly modified value of $c$) with $\E^{(\theta,u)}\lam_1(H)$ replaced by any median of $\lam_1(H)$ under $\P^{(\theta,u)}$.
\end{cor}

\begin{remark}
A bound of the form $\exp( -cs\sqrt{N})$ is easy to establish by a standard truncation argument combined with \eqref{talagrand}, avoiding the use of Proposition \ref{prop:conc-subG}, and in fact any bound of the form $o(1)$ suffices for our purposes as we are already under the tilted measure. 
In our proofs the strong bound of Proposition \ref{prop:conc-subG} is only crucial for controlling linear statistics, via Corollary \ref{cor:conc}.
\end{remark}

\begin{lemma}	\label{lem:subG-tilts}
The condition \eqref{assu:usg} is equivalent to assuming that the tilted measures  $d \mu^{t}(x) =e^{tx-\LLa_\mu(t)} d\mu(x)$, after recentering, are uniformly sub-Gaussian for $t\in\R$. 
\end{lemma}

\begin{proof}
Since $\mu^{t}$ has mean $\LLa_\mu'(t)$, the centered measure $d\mu_0^t(x) = d\mu^{t}(x+\LLa_\mu'(t))$ satisfies
\begin{align}
\LLa_{\mu_0^t}(s)&=
\log\int e^{s(x-\LLa_\mu'(t))}d\mu^t(x)	
=\LLa_\mu(t+s) - \LLa_\mu(t) - s\LLa_\mu'(t).	\label{LLa-3pt}
\end{align}
Since the last expression is bounded by $ \frac12\|\LLa_\mu''\|_\infty s^2$ we see that $\psi_{\mu_0^t}^{\max}$ is uniformly bounded in $t$ if \eqref{assu:usg} holds. On the other hand, taking $s\to0$ in \eqref{LLa-3pt} gives the identity $ \frac12\LLa_\mu''(t)= \psi_{\mu_0^t}(0)$, so that \eqref{assu:usg} holds if $\sup_t\psi_{\mu_0^t}^{\max}<\infty$.
\end{proof}

\begin{proof}[Proof of Corollary \ref{cor:lam1-tilt.conc}]
From the assumption on $\LLa_{\mu_{ij}}''$ and Lemma \ref{lem:subG-tilts} we get that the entries of $\sqrt{N}H$ are uniformly sub-Gaussian under $\P^{(\theta,u)}$. We may hence apply Proposition \ref{prop:conc-subG} as we did in the proof of Corollary \ref{cor:conc}, this time to the convex and $N^{-1/2}$-Lipschitz function $\lam_1(\frac1{\sqrt{N}}\,\cdot)$ on the ${N+1\choose 2}$-dimensional Euclidean space $\Sym_N$. 
\end{proof}

\subsection{Proof of Lemmas \ref{lem:tightness} and \ref{lem:good}}
\label{app:good}

From the union bound it suffices to establish the claimed bound for each of $\Good_0,\Good_1,\Good_2$ in place of $\Good$. 
The desired bound on $\P(H\notin \Good_0(K))$ is an immediate consequence of the following standard fact for sub-Gaussian matrices {(for instance one can apply results from \cite[$\mathsection$4.4.2]{Vershynin:book} to the upper and lower triangular parts of $H$).
Together with Lemma \ref{lem:Lside.app} this also yields Lemma \ref{lem:tightness}. 

\begin{lemma}\label{lem:tight.upper}
There are universal constants $C,c>0$ such that
\[
\P( \|H\|> t) \le 2\exp( - c t^2N)
\]
for all $t\ge C\sqrt{\psimax}$.
\end{lemma}
}

The desired bound on $\P(H\notin \Good_1(\kappa,\eta))$ is immediate from the following:

\begin{lemma}
\label{lem:ktail}
For any $k\ge\frac12 N^{1/2+\eta}$ and $\eps\in(0,1)$ we have
\begin{equation}
\P(\lambda_k(H) >2+\eps) \le 2\exp(-c\eps^2N^{1+\eta})
\end{equation}
for a constant $c>0$ depending only on $\mu$.
\end{lemma}

\begin{proof}
By modifying $c$ we may assume $N$ is sufficiently large depending on $\mu$.
We split $H=H^\le + H^>$, where
\[
H^\le_{ij} = H_{ij} \ind( |H_{ij}|\le N^{-1/4}). 
\]
Let $k\ge \frac12N^{1/2+\eta}$. By monotonicity we may assume $k\asymp N^{1/2+\eta}$.  
Let $\sigma_1(M)\ge\cdots\ge \sigma_N(M)\ge0$ be the singular values of $M$; as $M$ is symmetric these are simply the moduli of the eigenvalues of $M$ labeled in non-increasing order. 
Consider the function $F_k: \Sym_N\to \R$ on symmetric matrices given by 
\begin{equation}
F_k(M) = (\sum_{i\le k}\sigma_i(M)^2)^{1/2} = \sup_{\dim E= k} \|P_E M P_E\|_{\HS}
\end{equation}
where the supremum runs over subspaces of $\R^N$ and $P_E$ denotes the orthogonal projection operator to $E$. 
Clearly, $F_k$ is the supremum of convex functions $M\mapsto \|P_EMP_E\|_\HS$ that are 1-Lipschitz under the Euclidean  Hilbert--Schmidt norm on $\Sym_N$. 
From \eqref{FuKo} (applied with $H$ and $-H$) it follows that $\sigma_1(H)=\max\{\lam_1(H), |\lam_N(H)|\}=2+o(1)$ with probability $1-o(1)$, and hence any median for $F_k(X)$ is at most $(2+o(1))\sqrt{k}$. 
From Talagrand's inequality \eqref{talagrand} we thus have
\begin{equation}
\P(\lambda_k(H^\le)> 2+\eps) \le \P( F_k(H^\le) >(2+\eps)\sqrt{k}) \ls \exp( -c \eps^2 kN^{1/2}) 
\end{equation}
for all $N$ sufficiently large.
On the other hand, applying the union bound to fix the large entries on and above the diagonal, by noticing that for all $t\ge 0$
\begin{equation}\label{ineq1}
\P(\rank(H^>)\ge k) \le \P( |\{(i,j): |H_{ij}|>N^{-1/4}\}| \ge k) \le e^{-tk}\prod_{i,j} \E\exp(t 1_{|H_{ij}|>N^{-1/4}}).
\end{equation}
Now for all $s\ge0$, with $u:=s/2\psimax$ we have
\[ 
\mathbb P(|H_{ij}|\ge sN^{-1/2})\le e^{-us}\E\exp(u\sqrt{N}|H_{ij}|) 
\le 2\exp(-us+\psimax u^{2})
\le  2\exp(-\frac{s^{2}}{4\psimax})\,.
\]
Hence,
\[ 
\E\exp(t 1_{|H_{ij}|>N^{-1/4}})
\le 1+2\exp(t-\frac{\sqrt{N}}{4\psimax})\,.
\]
We finally take $t=\frac{\sqrt{N}}{8\psimax }$ so that the second term goes to zero  fast enough to deduce  from \eqref{ineq1} that
\begin{align*}
\P(\rank(H^>)\ge k)\le 2 \exp(-\frac{\sqrt{N}}{8\psimax } k)
\end{align*}
for $N$ sufficiently large. 
Thus, with probability at least $1-O(\exp( -c\eps^2 k N^{1/2}))$ (for a modified constant $c>0$) we have that $\lambda_k(H^\le)\le 2+\eps$ and $\rank(H^>)\le k$, and hence by the Cauchy interlacing law, that $\lambda_{2k}(H)\le 2+\eps $. Replacing $2k$ with $k$, the claim follows. 
\end{proof}

To establish Lemma \ref{lem:good} it now suffices to show
\begin{equation}	\label{G2.goal1}
\P(H\in \Good_1(\kappa,\eta)\setminus \Good_2(K,\kappa,\eta)) = \exp(-\omega_{ \eta,K,\kappa}(N)). 
\end{equation}
By a straightforward continuity argument it suffices to prove concentration for all $y$ in an $N^{-100}$-mesh of $[2+2\kappa,K]$, and from the union bound it suffices to consider an arbitrary fixed $y$ in this range.
Thus, writing
\[
H^{(\kappa)}:= \sum_{j: \lambda_j(H) \le 2+\kappa} \lambda_j v_jv_j^\tran
\]
it suffices to show that for such $y$,
\begin{equation}	\label{G2.goal2}
\P\big(H\in \Good_1(\kappa,\eta), |G_{H^{(\kappa)}}(y)-G_\sigma(y)| > N^{-1/2+2\eta}\big) 
= \exp( - \omega_{\eta,\kappa}(N))
\end{equation}
and similarly with $\Lpot_{H^{(\kappa)}}$ and $\Lpot_\sigma$ in place of $G_{H^{(\kappa)}}$ and $G_\sigma$, respectively. 
In both cases this follows from the following:

\begin{claim}
\label{claim:conc}
Let $f$ be convex and $L$-Lipschitz on $(-\infty,2+\kappa]$, with $f(0)=O(1)$.
Then
\[
\P(H\in \Good_1(\kappa,\eta), |\hat{\mu}_{H^{(\kappa)}}(f) - \sigma(f)| >CL N^{-1/2+\eta}) 
= \exp( -\omega_{\eta}(N))
\]
for some $C=C(\eta)>0$ sufficiently large. 
\end{claim}

Indeed, for \eqref{G2.goal2} we apply the claim with $f(\lambda)= 1/(y-\lambda)$ and $L=O(\kappa^{-2})$,
and for the concentration of the log-potential $\Lpot_{H^{(\kappa)}}(y)$ we take $f(\lambda) = -\log(y-\lambda)$ and $L=O(\kappa^{-1})$. 

\begin{proof}[Proof of Claim \ref{claim:conc}]
Let $f_1$ be the continuous and convex extension of $f$ to all of $\R$ which is linear on $[2+\kappa,\infty)$ with slope $L$, and let $f_2$ be the continuous extension of $f$ to $\R$ which is identically equal to $f(2+\kappa)$ on $[2+\kappa,\infty)$. 
Thus, 
\[
h(\lambda):= f_1(\lambda)-f_2(\lambda) = \max\{0, L(\lambda-2-\kappa) 
\}
\]
is convex and $L$-Lipschitz on $\R$. 
For $H\in \Good_1(\kappa,\eta)$ we have
\begin{equation}	\label{claim.conc1}
\hat{\mu}_{H^{(\kappa)}}(f) = \hat{\mu}_H(f_2) - f(2+\kappa)|\{ j: \lambda_j>2+\kappa\}| = \hat{\mu}_H(f_2) + O((1+\kappa)L N^{1/2+\eta}). 
\end{equation}
On the other hand, if $m(f_1), m(h)$ are medians for $\hat{\mu}_H(f_1), \hat{\mu}_H(h)$, respectively, from Lemma \ref{lem:median} we have
\begin{align*}
m(f_1) &= \sigma(f_1) + O_\eta(L N^{-1/2+\eta}) = \sigma(f) + O_\eta(LN^{-1/2+\eta})\,,\\
m(h) &= \sigma(h) + O_\eta(LN^{-1/2+\eta}) = O_\eta(LN^{-1/2+\eta}). 
\end{align*}
Combining with Corollary \ref{cor:conc}, we have that except with probability $1-O( \exp( - cN^{1+2\eta} / (\log N))) = 1-\exp( -\omega_{\eta}(N))$, 
\begin{equation*}
\hat{\mu}_H(f_1) = \sigma(f) + O_\eta(L N^{-1/2+\eta}) \,,\qquad
\hat{\mu}_H(h) = O_\eta(L N^{-1/2+\eta})
\end{equation*}
and hence
\[
\hat{\mu}_H(f_2) = \hat{\mu}_H( f_1) - \hat{\mu}_H(h) = \sigma(f) + O_\eta(LN^{-1/2+\eta}).
\]
Combining with \eqref{claim.conc1} yields the claim.
\end{proof}

\section{Coupling of tilted laws}
\label{app:coupling}

In this appendix we prove Lemma \ref{lem:coupling.scalar} on the $L^p$-continuity for couplings of scalar random variables. Recall the coupling $(X_\alpha,X_\beta)$ from \eqref{scalar-coupling}. 
Without loss of generality we can assume $\alpha\leq\beta$ and, because
of Lemma \ref{lem:monotony-connection}, $X_{\alpha}\leq X_{\beta}$
a.s. We first consider the case $\beta-\alpha\le 1$.
 In the case $k=1$, we can simply use that \[ \mathbb{E}(|X_{\beta}-X_{\alpha}|)=\mathbb{E}(X_{\beta})-\mathbb{E}(X_{\alpha})=\LLa_\mu'(\beta)-\LLa_\mu'(\alpha)\le \|\LLa_\mu''\|_{\infty}|\beta-\alpha|\,.\]
Let $k\in2\mathbb{N}^{*}$. Let $x\in\mathbb{R}$ and write
\begin{align}
|X_{\beta}-X_{\alpha}|^{k} & =|X_{\beta}-X_{\alpha}|^{k}1_{X_{\alpha}\geq x}+|X_{\beta}-X_{\alpha}|^{k}1_{X_{\beta}\geq x>X_{\alpha}}+|X_{\beta}-X_{\alpha}|^{k}1_{x>X_{\beta}}\nonumber\\
 & \leq2^{k-1}\Big((X_{\beta}-x)^{k}1_{X_{\alpha}\geq x}-(X_{\alpha}-x)^{k}1_{X_{\alpha}\geq x}+(X_{\beta}-x)^{k}1_{X_{\beta}\geq x>X_{\alpha}}\nonumber\\
 & +(X_{\alpha}-x)^{k}1_{X_{\beta}\geq x>X_{\alpha}}+(X_{\alpha}-x)^{k}1_{x>X_{\beta}}-(X_{\beta}-x)^{k}1_{x>X_{\beta}}\Big)\nonumber\\
 & =2^{k-1}\Big((X_{\beta}-x)^{k}1_{X_{\beta}\geq x}-(X_{\alpha}-x)^{k}1_{X_{\alpha}\geq x}\Big)\nonumber\\
 & +2^{k-1}\Big((X_{\alpha}-x)^{k}1_{x>X_{\alpha}}-(X_{\beta}-x)^{k}1_{x>X_{\beta}}\Big)\label{boundmom}
\end{align}
where we use that $b^{k}-a^{k}\geq(b-a)^{k}$,  and $2^{k-1}(b^{k}+a^{k})\geq (b+a)^{k}$ for all $b\geq a\geq0$
in the first inequality and $1_{X_{\beta}\geq x>X_{\alpha}}+1_{X_{\alpha}\geq x}=1_{X_{\beta}\geq x}$ in the
second equality.  We choose $x=\mathbb{E}(X_{\alpha})$ hereafter. We now estimate the  expectation of the first term in the sum \eqref{boundmom}. 
\begin{align*}
C_{1}(\alpha,\beta) & :=\mathbb{E}\Big((X_{\beta}-\mathbb{E}(X_{\alpha}))^{k}1_{X_{\beta}\geq \mathbb{E}(X_{\alpha})}-(X_{\alpha}-\mathbb{E}(X_{\alpha}))^{k}1_{X_{\alpha}\geq x}\Big)\\
 & =\mathbb{E}\Big((X_{\alpha}-\mathbb{E}(X_{\alpha}))^{k}1_{X_{\alpha}\geq \mathbb{E}(X_{\alpha})}\frac{e^{(\beta-\alpha)(X_{\alpha}-\mathbb{E}(X_{\alpha}))}}{\mathbb{E}(e^{(\beta-\alpha)(X_{\alpha}-\mathbb{E}(X_{\alpha}))})}\Big)\\
 &\quad -\mathbb{E}\left((X_{\alpha}-\mathbb{E}(X_{\alpha}))^{k}1_{X_{\alpha}\geq \mathbb{E}(X_{\alpha})}\right)\,.
\end{align*}
We  write 
\[ g(t)=\mathbb{E}\Big((X_{\alpha}-\mathbb{E}(X_{\alpha}))^{k}1_{X_{\alpha}\geq\mathbb{E}(X_{\alpha})}e^{t(X_{\alpha}-\mathbb{E}(X_{\alpha}))}\Big),\quad h(t)=\frac{1}{\mathbb{E}(e^{t (X_{\alpha}-\mathbb E[X_{\alpha}])})}\]
so that
\[ C_{1}(\alpha,\beta)=g(\beta-\alpha)h(\beta-\alpha)- g(0)h(0)=\int_{0}^{\beta-\alpha} (g'h+h'g)(t) dt\,.\]
We obtain that for all $0\leq t\leq1$
\begin{align*}
0\leq g'(t) & =\mathbb{E}\Big((X_{\alpha}-\mathbb{E}(X_{\alpha})){}^{k+1}1_{X_{\alpha}\geq\mathbb{E}(X_{\alpha})}e^{t(X_{\alpha}-\mathbb{E}(X_{\alpha}))}\Big)\\
 & \leq(k+1)!\times\mathbb{E}\Big(e^{(1+t)(X_{\alpha}-\mathbb{E}(X_{\alpha}))}\Big)\\
 & =(k+1)!\times e^{\LLa_\mu(\alpha+1+t)-\LLa_\mu(\alpha)-(1+t)\LLa_\mu'(\alpha)}\\
 & \leq(k+1)!\times e^{2\|\LLa_\mu''\|_{\infty}}
\end{align*}
where in the second line we used that for all non-negative real number $y$, $\frac{y^{k+1}}{(k+1 )!} \le e^{y}$.
We also have  for all $t\le 1$, 
\begin{align*}
 h'(t) & =\frac{d}{dt}e^{-\LLa_\mu(\alpha+t)+\LLa_\mu(\alpha)+t\LLa_\mu'(\alpha)}\\
 & =(\LLa_\mu'(\alpha)-\LLa_\mu'(\alpha+t))e^{-\LLa_\mu(\alpha+t)+\LLa_\mu(\alpha)+t\LLa_\mu'(\alpha)}\\
 & \geq-\|\LLa_\mu''\|_{\infty}e^{\frac{1}{2}\|\LLa_\mu''\|_{\infty}}
\end{align*}
and $h'$ is non-positive because $\LLa_\mu$ is convex.
Noting that we can similarly bound 
\[
g(0) = \E (X_\alpha- \E X_{\alpha})^k1_{X_\alpha\ge\E X_\alpha}\le k! \E e^{X_\alpha-\E X_\alpha} = k!e^{2\|\LLa_\mu''\|_\infty}
\]
we see that the above differential inequalities imply
that $h$ and $g$ are uniformly bounded and with uniformly bounded derivatives on $[0,1]$ 
and 
 there exists a constant $C_{k}$ that depends only on $\|\LLa_\mu''\|_{\infty}$ and $k$
such that for $|\beta-\alpha|\le 1$,
\[
|C_{1}(\alpha,\beta)|\leq C_{k}|\beta-\alpha|\,.
\]
The second term \[ C_{2}(\alpha,\beta):=\mathbb{E}\Big((X_{\alpha}-\E[X_{\alpha}])^{k}1_{\E[X_{\alpha}]>X_{\alpha}}-(X_{\beta}-\E[X_{\alpha}])^{k}1_{\E[X_{\alpha}]>X_{\beta}}\Big)=\int_{0}^{\beta-\alpha}(h\tilde g)' dt
\]
with \[ \tilde g(t)=-\mathbb{E}\Big((X_{\alpha}-\mathbb{E}(X_{\alpha})){}^{k+1}1_{X_{\alpha}<\mathbb{E}(X_{\alpha})}e^{t(X_{\alpha}-\mathbb{E}(X_{\alpha}))}\Big)\]
so that $\tilde g'$ is non positive and bounded below
\[ \tilde g'(t)=-\mathbb{E}\Big((X_{\alpha}-\mathbb{E}(X_{\alpha})){}^{k+2}1_{X_{\alpha}<\mathbb{E}(X_{\alpha})}e^{t(X_{\alpha}-\mathbb{E}(X_{\alpha}))}\Big)\ge -(k+2)! \mathbb E\Big( e^{(t+1)(X_{\alpha}-\mathbb{E}(X_{\alpha}))}\Big)
\]
 is controlled similarly. Because $|a|^{2k+1}\leq a^{2k}+a^{2k+2}$ for every real number $a$,
the above estimate is also valid for $k\in2\mathbb{N}^{*}+1$ and
this finishes the proof of \eqref{coupling.b1}.

For \eqref{b11} in the case $|\beta-\alpha|\le 1$, with $(X_\beta',X_\alpha')$ and independent copy of $(X_\beta,X_\alpha)$ we have
\begin{align*}
\E |X_{\beta}-X_{\alpha}-\mathbb{E}(X_{\beta}-X_{\alpha})|^{k}
& =\E |X_{\beta}-X_{\alpha}-\mathbb{E}(X_{\beta}'-X_{\alpha}')|^{k}\\
&\le \E | (X_\beta - X_\alpha)- (X_\beta'-X_\alpha')|^k\\
&\le \E ( |X_\beta-X_\alpha| + |X_\beta'-X_\alpha'|)^k\\
&\le 2^{k+1} \E |X_\beta-X_\alpha|^k
\end{align*}
where in the second line we applied Jensen's inequality. Hence \eqref{b11} now follows from \eqref{coupling.b1} in this case.
For the case $\beta-\alpha\geq1$,  we notice that
\begin{align*}
\mathbb{E}(|X_{\beta}-X_{\alpha}-\mathbb{E}(X_{\beta}-X_{\alpha})|^{k}) & \leq2^{k-1}\left(\mathbb{E}(|X_{\beta}-\mathbb{E}(X_{\beta})|^{k})+\mathbb{E}(|X_{\alpha}-\mathbb{E}(X_{\alpha})|^{k})\right)
\end{align*}
and each term is bounded independently of $\alpha$ and $\beta$ by
the same argument as above since if $k$ is even, for any non negative real number $\alpha$,
\begin{align*}
\mathbb{E}(|X_{\alpha}-\mathbb{E}(X_{\alpha})|^{k})
&\le \mathbb{E}(|X_{\alpha}-\mathbb{E}(X_{\alpha})|^{k} 1_{X_{\alpha}\ge \E[X_{\alpha}]})+\mathbb{E}(|X_{\alpha}-\mathbb{E}(X_{\alpha})|^{k} 1_{X_{\alpha}< \E[X_{\alpha}]})\\
&\le g(0)+k!\E[ e^{\mathbb E[X_{\alpha}]-X_{\alpha}}]\\
&=g(0)+ k! e^{\LLa_\mu(\alpha-1)-\LLa_\mu(\alpha)+\LLa_\mu'(\alpha)}\le 2 k! e^{\|\LLa_\mu''\|_{\infty}}\,.
\end{align*}
This completes the proof of Lemma \ref{lem:coupling.scalar}.

\section{Quantitative Varadhan lemma}
\label{app:Varadhan}

In this appendix we prove Lemma \ref{lem:Varadhan}.
We note that by subtracting $\ham(0)$ from all sides we may assume $\ham(0)=0$. 
We may also assume $R\le \sqrt{N}$ since any unit vector lies in $\Deloc_{\sqrt{N}}$. In particular we have $\log R\le \log N$.
We also recall the notation $\ssq(x):=x^2(1_{x\ge0}-1_{x<0})$.

\subsection{Preliminary lemmas}

The first step to prove both upper and lower bounds will be to replace the unit vector $u\sim P_N$ with a Gaussian vector. 

\begin{lemma}
\label{lem:vara.G}
Let $\delta\in (C_1N^{-1/2}, 1)$  for a sufficiently large constant $C_1>0$ and set
\begin{align*}
A(\delta)&:= \{y\in \R^N: |\frac1{\sqrt{N}}\|y\|_2 - 1|\le \delta\}\,,\\
A_\pm(R,\delta) &:= A(\delta) \cap [-(1\pm \delta)R,(1\pm\delta)R]^N.
\end{align*}
Let $g\in\R^N$ be a standard Gaussian vector and let $v\in \R^N$ be fixed.
We have
\begin{align}
\int_{\Deloc_R} \exp\Big( \sum_{j=1}^N\ham(\sqrt{N}u_j)\Big) dP_N(u)
&\le 2e^{3K\delta N} \E 
\exp \Big( \sum_{j=1}^N\ham(g_j)\Big)  \ind( {g\in A_+(R,\delta)} ) 
\label{varG.UB}
\end{align}
and for any $\eps_0\in (0,\frac12)$, 
\begin{align}
&\E 
\exp \Big( \sum_{j=1}^N\ham(g_j)\Big) \ind({g\in A_-(R,\delta)\cap\sqrt{N} \Bset_2(v,\eps_0-\delta)})
\notag\\
&\qquad\qquad\qquad\le
2e^{3K\delta N}\int_{\Deloc_R\cap \Bset_2(v,\eps_0)} \exp\Big( \sum_{j=1}^N\ham(\sqrt{N}u_j)\Big) dP_N(u) \,.	\label{varG.LB}
\end{align}
\end{lemma}

\begin{proof}
For $y\in A(\delta)$ we have for each $i\in [N]$ that
\[
|y_i|\le (1-\delta) R \; \Rightarrow\;|y_i|/\|y\|_2\le R/\sqrt{N} \;\Rightarrow \; |y_i| \le (1+\delta)R
\]
so
\begin{equation}	\label{Apm.contain}
A_-(R,\delta) \subseteq A(\delta)\cap \{ y: y/\|y\|_2\in \Deloc_R\}\subseteq A_+(R,\delta).
\end{equation}
Moreover, for $y\in A(\delta)$ we have 
\[
d_2(\frac{y}{\sqrt{N}},\frac{y}{\|y\|_2}) \le \|\frac{y}{\sqrt{N}}-\frac{y}{\|y\|_2}\|_2 = |\frac{\|y\|_2}{\sqrt{N}}-1| \le \delta
\]
so by the triangle inequality,
\begin{equation}	\label{ABW.contain}
A(\delta)\cap\sqrt{N} \Bset_2(v,\eps_0-\delta) \subseteq \{ y: y/\|y\|_2\in \Bset_2(v,\eps_0)\}.
\end{equation}
Now since the distribution of $g$ and the set $A(\delta)$ are rotationally invariant, we have that  $g/\|g\|_2$ conditioned on the event $\{g\in A(\delta)\}$ has law $P$. From this and \eqref{Apm.contain}, \eqref{ABW.contain} we deduce
\begin{align*}
&\E \bigg(\exp \Big( \sum_{j=1}^N\ham(\sqrt{N}g_j/\|g\|_2)\Big)  \ind\big(g\in {A_-(R,\delta)\cap \sqrt{N}\Bset_2(v,\eps_0-\delta)}\big)\,\bigg| \, g\in A(\delta) \bigg)  \\
&\le\int_{\Deloc_R\cap \Bset_2(v,\eps_0)} \exp\Big( \sum_{j=1}^N\ham(\sqrt{N}u_j)\Big) dP_N(u)
\le \int_{\Deloc_R} \exp\Big( \sum_{j=1}^N\ham(\sqrt{N}u_j)\Big) dP_N(u)\\
&\le \E \bigg(\exp \Big( \sum_{j=1}^N\ham(\sqrt{N}g_j/\|g\|_2)\Big)  \ind( g\in {A_+(R,\delta)} ) \,\bigg| \, g\in A(\delta) \bigg)\,.
\end{align*}
Next we note that since $\ham\circ\ssq^{-1}$ is $K$-Lipschitz {on $[-2R,2R]$}, on the event that $g\in A(\delta)$ we have
\begin{align*}
|\sum_{j=1}^N\ham( \sqrt{N} g_j/\|g\|_2) - \sum_{j=1}^N \ham(g_j) |
&\le K |\frac{N}{\|g\|_2^2}-1| \sum_{j=1}^Ng_j^2\le 3K\delta N\,.
\end{align*}
Finally, it only remains to note that since $\P(g\in A(\delta))\ge 1/2$ for $\delta\ge C_1/\sqrt{N}$ if $C_1$ is sufficiently large, we can give up a factor of 2 in the bounds to remove the conditioning on the event that $g\in A(\delta)$.
\end{proof}

Recalling our notation $\hat{\mu}_y$ for the empirical measure of the components of a vector $y\in\R^N$ (see Section \ref{sec:notation}), 
we have $\sum_{j=1}^N\ham(g_j)=N\int \ham d\hat{\mu}_g$.
The next step for the proof of both bounds will be to coarse-grain the range $\R$ for the $g_j$ by a partition $\{E_\sigma\}_{\sigma\in \Sigma}$ of small size (compared with $N$), and approximate $\int \ham d\hat{\mu}_g$ with an average over $\Sigma$. 

Let $\eps\in (0,\frac1{10})$ be a small parameter (we will choose $\eps$ differently in the proofs of the upper and lower bound) and set
$
\ell_0:=\min\{ \ell\in\N: (1+\eps)^{-\ell} <\eps\}
$.
We note that
\begin{equation}	\label{bd:ell0}
\ell_0\asymp \frac1\eps\log\frac1\eps.
\end{equation}
Let
\begin{align*}
E_{-\ell_0}^+&:= [0, (1+ \eps)^{-\ell_0}]\,,\qquad E_k^+ :=((1+\eps)^{k-1}, (1+\eps)^k] \,,\quad k>-\ell_0,
\end{align*}
and set $E_k^-:= -E_k^+$ for each $k\ge-\ell_0$.
Setting $k_0:= \min\{k: (1+\eps)^k > (1+C_2\eps)R\}$ for a sufficiently large constant $C_2\ge1$ to be chosen later, we have
\begin{equation}	\label{Ek0.contain}
(1+\tfrac12C_2\eps)[-R,R] \subseteq [-(1+\eps)^{k_0-1}, (1+\eps)^{k_0-1}] \subseteq (1+C_2\eps)[-R,R]\,.
\end{equation}
In particular,
\begin{equation}	\label{bd:k0}
k_0 \asymp\frac1\eps\log R
\end{equation}
and
\begin{equation}	\label{E_k++}
\nu(E_k^\pm) = 0
\qquad\forall \,k>k_0\
\end{equation}
for any measure $\nu$ supported on $[-(1+C_2\eps)R,(1+C_2\eps)R]$. 
The set
\begin{equation}
\Sigma=\Sigma_\eps := (\{-\ell_0,\dots, k_0\}\times\{-,+\})\cup \{*\}\,.
\end{equation}
labels the sets $E_k^\pm$ for $-\ell_0\le k\le k_0$ and the set $E_*:= \R\setminus[-(1+\eps)^{k_0},(1+\eps)^{k_0}]$, and the collection $\{E_\sigma\}_{\sigma\in\Sigma}$ partitions $\R$.
(For $\sigma=(k,\pm)$ we write $E_\sigma:=E_k^\pm$.)
From \eqref{bd:ell0}, \eqref{bd:k0} we have
\begin{equation}	\label{bd.Alphabet}
|\Sigma| = 1+ 2(k_0+\ell_0+1) \ls \frac1\eps\log\frac{R}{\eps}.
\end{equation}
Any measure $\nu$ on $\R$ induces a measure $[\nu]$ on $\Sigma$ defined
\[
[\nu](\sigma) := \nu(E_\sigma)
\]
(abusively identifying $[\nu]$ with its mass function $[\nu]:\Sigma\to \R^+$). 

Fixing an arbitrary point $a_*\in E_*$ (such as $(1+C_2\eps)^{10k_0}$),
we write
\[
\iota_\eps:\Sigma\to \R\,,\quad \iota_\eps(k,\pm) :=\pm(1+\eps)^k\,,\quad \iota_\eps(*):=a_*. 
\]
With $\cP(\Sigma)=\{\pi:\Sigma\to [0,1]: \sum_{\sigma\in \Sigma}\pi(\sigma)=1\}$ the set of probability measures on $\Sigma$, we denote by
\begin{equation}
\cP_N(\Sigma) := \cP(\Sigma)\cap (\tfrac1N\cdot\Z)^\Sigma
\end{equation}
the subset of measures taking values in the integral multiples of $1/N$. Note that $[\hat{\mu}_y]\in \cP_N(\Sigma)$ for any $y\in \R^N$.
Any $\pi \in \cP(\Sigma)$ pushes forward to a discrete measure $\iota_\eps\#\pi$ on $\R$, and 
\begin{equation}	\label{pi.push-pull}
[\iota_\eps\#\pi] = \pi
\end{equation}
Indeed, for any $\sigma\in\Sigma$,
\[
[\iota_\eps\#\pi](\sigma) = \iota_\eps\#\pi(E_\sigma) = \pi( \iota_\eps^{-1}(E_\sigma)) = \pi(\sigma).
\]

\begin{lemma}[Coarse-graining]
\label{lem:coarse}
With the above setup, let
$f:[-(1+C_2\eps)R,(1+C_2\eps)R]\to \R$ be such that $g:=f\circ\ssq^{-1}$ is $K$-Lipschitz, and let $\nu$ be a probability measure supported on $[-(1+C_2\eps)R,(1+C_2\eps)R]$ with second moment $\int x^2d\nu(x)\le M$. We have
\begin{equation}	\label{Lip.f-mu}
\bigg| \int f d\nu - \int f\circ \iota_\eps d[\nu]\bigg| \ls (M+\eps)K\eps.
\end{equation}
In particular, for any $y\in A_+(R,C_2\eps)$, 
\begin{equation}	\label{Lip.f-muy}
\bigg| \frac1N\sum_{j=1}^N f(y_j) - \sum_{\sigma\in \Sigma} f(\iota_\eps(\sigma))[\hat{\mu}_y](\sigma) \bigg|
\ls K\eps.
\end{equation}
\end{lemma}

\begin{proof}
Since $\hat{\mu}_y$ is supported on $[-(1+C_2\eps)R,(1+C_2\eps)R]$ for any $y\in A_+(R,C_2\eps)$, and $\int x^2d\hat{\mu}_y(x) = \frac1N\|y\|_2^2 \ls 1$, we have that 
\eqref{Lip.f-muy} follows immediately from \eqref{Lip.f-mu}.

Turning to \eqref{Lip.f-mu},
in view of \eqref{E_k++} we can express
\begin{align*}
\int f d\nu - \int f\circ \iota_\eps d[\nu]
&= \sum_{k=-\ell_0}^{k_0} \sum_{s=\pm} \int_{E_k^s} f d\nu - f(s(1+\eps)^k) [\nu](k,s) \\
&=  \sum_{k=-\ell_0}^{k_0} \sum_{s=\pm} \int_{E_k^s} f - f(s(1+\eps)^k) d\nu 
\end{align*}
so
\begin{align*}
\bigg| \int f d\nu  - \int f\circ \iota_\eps d[\nu]  \bigg|
&\le \sum_{k=-\ell_0}^{k_0} \sum_{s=\pm}\int_{E_k^s}| f - f(s(1+\eps)^k) |d\nu  \\
&= \sum_{k=-\ell_0}^{k_0} \sum_{s=\pm} \int_{E_k^s} |g(x^2) - g((1+\eps)^{2k}) |d\nu \\
&\ls K \eps^2\nu (E_{-\ell_0}^+\cup E_{-\ell_0}^-) + K\sum_{k=-\ell_0+1}^{k_0} (1+\eps)^{2k}( 1- (1+\eps)^{-2}) \nu ( E_k^+\cup E_k^-)\\
&\ls K\eps^2 +K \eps  \sum_{k=-\ell_0+1}^{k_0} (1+\eps)^{2k} \nu (E_k^+\cup E_k^-)\,.
\end{align*}
For the last term, 
\begin{align*}
\sum_{k=-\ell_0+1}^{k_0} (1+\eps)^{2k} \nu (E_k^+\cup E_k^-)
&\le (1+\eps)^{2} \int x^2d\nu (x) \ls M
\end{align*}
which, combined with the previous bound, yields the claim.
\end{proof}


\begin{lemma}[Cf.\ {\cite[Lemma 2.1.9]{DZ}}]
\label{lem:DeZe}
For any $\pi\in \cP_N(\Sigma)$, we have
\[
-\DKL(\pi| [\gamma])- \frac1N|\Sigma|\log(N+1) \le \frac1N\log\P([\hat{\mu}_g]= \pi)
\le -\DKL(\pi|[\gamma]). 
\]
\end{lemma}

\subsection{Proof of Lemma \ref{lem:Varadhan} (upper bound)}

We now establish the upper bound \eqref{varadhan.UB}. 
Here we take the coarse-graining parameter $\eps$ from Lemma \ref{lem:coarse} to be
\begin{equation}	\label{eps.upper}
\eps = C N^{-1/2}\log N
\end{equation}
for an absolute constant $C>0$ to be taken sufficiently large. 

Letting $C_3>0$ be a sufficiently large absolute constant and
\begin{equation}	\label{def:cA+}
\cA_+(\eps) := \bigg\{ \pi\in\cP_N(\Sigma): \pi(*)=\pi(k_0,+)=\pi(k_0,-)=0, \; \bigg| \int_\Sigma \iota_\eps^2d\pi - 1\bigg| < C_3\eps \,\bigg\} 
\end{equation}
we claim
\begin{equation}	\label{A+A+}
y\in A_+(R, \tfrac12C_2\eps) \quad\Longrightarrow\quad [\hat{\mu}_y] \in \cA_+(\eps)\,.
\end{equation}
Indeed, the first condition follows from the first containment in \eqref{Ek0.contain}; for the second, by applying Lemma \ref{lem:coarse} with the function $f(t)=t^2$ (for which we can take $K=1$) we find
\begin{align*}
\int_\Sigma \iota_\eps^2 d[\hat{\mu}_y] = \int_\R s^2 d\hat{\mu}_y(s) + O(\eps)
\end{align*}
and \eqref{A+A+} follows by taking $C_3$ sufficiently large. 

Now assuming the constant $C$ in \eqref{eps.upper} is at least $2C_2^{-1}C_1$ we can apply Lemma \ref{lem:vara.G} with $\delta=\frac12C_2\eps$, followed by Lemma \ref{lem:coarse} and \eqref{A+A+}, to bound
\begin{align*}
 \int_{\Bdeloc_R} \exp\Big(  \sum_{j=1}^N \ham(\sqrt{N}u_j) \Big) dP_{N}(u)	
&\le e^{O(K\eps N)} \E \exp\bigg(N \int_\R \ham d\hat{\mu}_g\bigg) \ind(g\in {A_+(R, \tfrac12C_2\eps)} )	\\
&\le e^{O(K\eps N)}  \E \exp\bigg( N \int_\Sigma \ham\circ \iota_\eps d[\hat{\mu}_g]\bigg)
\ind( [\hat{\mu}_g]\in \cA_+(\eps))\,.
\end{align*}
Applying Lemma \ref{lem:DeZe}, the last expression is 
\begin{align}
&  e^{O(K\eps N)} \sum_{\pi\in \cA_+(\eps)} 
\exp\bigg( N \int_\Sigma \ham\circ \iota_\eps d\pi\bigg)
\P([\hat{\mu}_g]=\pi)	\notag\\
&\le e^{O(K\eps N)} \sum_{\pi\in \cA_+(\eps)} \exp\bigg( N \bigg(\int_\Sigma \ham\circ \iota_\eps d\pi - \DKL(\pi|[\gamma])\bigg) \bigg)	\notag\\
&\le e^{O(K\eps N)} | \cA_+(\eps)| \exp\bigg( N \max_{\pi\in\cA_+(\eps)} \bigg\{ \int_\Sigma \ham\circ \iota_\eps d\pi - \DKL(\pi|[\gamma]) \bigg\} \bigg)	\notag\\
&\le e^{O(K\eps N)}\exp\bigg( N \max_{\pi\in\cA_+(\eps)} \bigg\{ \int_\Sigma \ham\circ \iota_\eps d\pi - \DKL(\pi|[\gamma]) \bigg\} \bigg)
\label{VUB-2}
\end{align}
where in the last line we used \eqref{bd.Alphabet}, \eqref{eps.upper} and our assumptions $R\le N^{1/2}$, $K\ge1$ to bound
\[
|\cA_+(\eps)|\le|\cP_N(\Sigma)|
\le (N+1)^{|\Sigma|} 
= N^{O( \frac1\eps \log\frac{R}\eps)} = e^{O(N^{1/2}\log N)} = e^{O(K\eps N)}.
\]

Our next step is to replace the discrete measures $\pi$ on $\Sigma$ with continuous measures on $\R$. 
For a Borel measure $\nu$ on $\R$, define a measure $\ol\nu\ll\gamma$ with
\begin{equation}	\label{flatten}
\frac{d\ol\nu}{d\gamma} = \sum_{\sigma\in\Sigma} 1_{E_\sigma} \frac{\nu(E_\sigma)}{\gamma(E_\sigma)} .
\end{equation}
Since $\ol\nu$ and $\nu$ assign the same measure to the sets $E_\sigma$, we have
\begin{equation}	\label{flat-coarse}
[\ol\nu] = [\nu].
\end{equation}
Note that
\begin{equation}	\label{H.flat-coarse}
\DKL(\ol\nu|\gamma)  = \DKL([\nu]|[\gamma])\,.
\end{equation}
For $\pi\in \cP(\Sigma)$ 
we denote by
\[
d\nu_\pi:=\ol{\iota_\eps\#\pi} = \sum_{\sigma\in\Sigma} 1_{E_\sigma} \frac{\pi(\sigma)}{\gamma(E_\sigma)} d\gamma
\]
the continuous measure on $\R$ with mass $\pi(\sigma)$ distributed with constant density relative to $\gamma$ within each interval $E_\sigma$.
From \eqref{H.flat-coarse} and \eqref{pi.push-pull} we get
\begin{equation}	\label{H.push-flat}
\DKL(\nu_\pi | \gamma) = \DKL( \pi|[\gamma]). 
\end{equation}
From \eqref{flat-coarse} we have $[\nu_\pi]=\pi$, and so from \eqref{Lip.f-mu} we get
\begin{equation}	\label{VUB-2.1}
\int_\Sigma \ham\circ\iota_\eps d\pi 
= \int_\R \ham d \nu_\pi  + O(K\eps).
\end{equation}
Combining \eqref{H.push-flat} and \eqref{VUB-2.1}, for the argument of the maximum in \eqref{VUB-2} we thus have
\begin{equation}	\label{VUB-2.2}
\int_\Sigma \ham\circ \iota_\eps d\pi - \DKL(\pi|[\gamma]) 
= \int_\R \ham d(\iota_\eps\#\pi) - \DKL(\nu_\pi|\gamma) 
= \int_\R\ham d\nu_\pi - \DKL(\nu_\pi|\gamma) + O(K\eps).
\end{equation}
Substituting \eqref{VUB-2.2} into \eqref{VUB-2}, taking logs and dividing through by $N$, we have shown
\begin{align}
\frac1N\log  \int_{\Bdeloc_R} \exp\Big(  \sum_{j=1}^N \ham(\sqrt{N}u_j) \Big) dP_{N}(u)		
&\le   \max_{\pi\in\cA_+(\eps)} \bigg\{ 
\int_\R \ham d\nu_\pi - \DKL(\nu_\pi|\gamma) 
 \bigg\} + O(K\eps)	\,. \label{varadhan.up1}
\end{align}

We now claim that for any $\pi \in \cA_+(\eps) $, 
\begin{equation}	\label{VUB.fudge}
\nu_\pi \in \bigcup_{|b-1|=O(\eps)} \cP_{b}((1+O(\eps))[-R,R])\,.
\end{equation}
Indeed, 
fixing an arbitrary $\pi \in \cA_+(\eps)$, from the condition $\pi(*)=\pi(k_0,\pm)=0$ and \eqref{Ek0.contain} we have 
$\supp(\nu_\pi)\subseteq (1+C_2\eps)[-R,R]$.
Moreover, applying \eqref{Lip.f-mu} with $f(s)=s^2$ (for which we can take $K=1$), we have
\begin{align*}
\int_\R x^2 d\nu_\pi(x) = \int_\Sigma \iota_\eps^2 d\pi + O(\eps) = 1+ O(\eps)
\end{align*}
and \eqref{VUB.fudge} follows. 
Thus the right hand side of \eqref{varadhan.up1} is bounded above by
\begin{align*}
\sup_{|b-1|=O(\eps)} \sup_{\nu\in \cP_{b}((1+O(\eps))[-R,R])} \bigg\{ 
\int_\R \ham d\nu - \DKL(\nu|\gamma) 
 \bigg\}  + O(K\eps)\,.
\end{align*}
The upper bound \eqref{varadhan.UB} now follows from Proposition \ref{prop:gibbs}\eqref{gibbs.wiggle} and our choice \eqref{eps.upper} for $\eps$.

\subsection{Proof of Lemma \ref{lem:Varadhan} (lower bound)}

Turning to establish the lower bound \eqref{varadhan.LB},
here we take
\begin{equation}	\label{eps.lower}
\eps = c_0\srad'
\end{equation}
for a sufficiently small absolute constant $c_0>0$. We also denote
\begin{equation}
\srad_0:= C_0N^{-1/2}(R+\log N)
\end{equation}
(with the constant $C_0$ as in the statement of Lemma \ref{lem:Varadhan}) so that $\srad'\in [\srad_0,\frac1K]$.

We begin by gathering some comparisons between the near-optimizing measure $\t \nu$ (see \eqref{def:nearopt}) and the empirical measure $\hat{\mu}_{\t  y}$ of its quantiles.
First, from the definition \eqref{def:quantiles} of $\t  y$ it follows that $|\hat{\mu}_{\t  y}(E)-\t \nu(E)|\le \frac1N$ for any interval $E\subset\R$. In particular
\begin{equation}	\label{hathat.infty}
\|[\hat{\mu}_{\t  y}] - [\t \nu]\|_{\ell^\infty(\Sigma)} \le \frac1N. 
\end{equation}
To compare the entropy of these measures relative to the discretized Gaussian $[\gamma]$ we will combine the above bound with the following:

\begin{lemma}
\label{lem:divergence}
For any $\pi_1,\pi_2\in \cP(\Sigma)$ and any $K_0\ge1$,
\[
\big| \DKL(\pi_1|[\gamma]) - \DKL(\pi_2|[\gamma]) \big| 
\ls  O(\eps)^{K_0} +K_0\|\pi_1-\pi_2\|_{\ell^\infty(\Sigma)} \bigg( |\Sigma|\log\frac1\eps + \frac{R^2}\eps\bigg).
\]
\end{lemma}

\begin{proof}
First we claim that for all $\sigma\in\Sigma$ such that $\max\{\pi_1(\sigma),\pi_2(\sigma)\} \ge \gamma(E_\sigma)^{K_0}$, 
\begin{equation}	\label{divergence1}
\bigg| \pi_1(\sigma) \log\frac{\pi_1(\sigma)}{\gamma(E_\sigma)} - 
\pi_2(\sigma) \log\frac{\pi_2(\sigma)}{\gamma(E_\sigma)} \bigg|
\le \|\pi_1-\pi_2\|_{\ell^\infty(\Sigma)} \bigg( 1+ (K_0+1)\log\frac1{\gamma(E_\sigma)} \bigg) \,.
\end{equation}
Indeed, fixing $\sigma\in \Sigma$, without loss of generality suppose $\pi_1(\sigma)\le \pi_2(\sigma)$. 
Then writing $r_i(\sigma):= \pi_i(\sigma)/\gamma(E_\sigma)$, we have that the left hand side above is 
\begin{align*}
| \pi_1(\sigma) \log r_1(\sigma) - \pi_2(\sigma)\log r_2(\sigma)|
&\le |\pi_2(\sigma) - \pi_1(\sigma)| |\log r_2(\sigma)| 
+ \pi_1(\sigma)\log\frac{r_2(\sigma)}{r_1(\sigma)}\\
&\le \|\pi_1-\pi_2\|_{\ell^\infty(\Sigma)} |\log r_2(\sigma)| 
+ \pi_1(\sigma) \bigg(\frac{r_2(\sigma)}{r_1(\sigma)} -1\bigg)\\
&= \|\pi_1-\pi_2\|_{\ell^\infty(\Sigma)} |\log r_2(\sigma)| + \pi_2(\sigma)-\pi_1(\sigma) \\
&\le  \|\pi_1-\pi_2\|_{\ell^\infty(\Sigma)} (1+ |\log r_2(\sigma)|)
\end{align*}
where in the second line we used the concavity of $\log$ to bound it by its linearization at $1$. Now since
\[
|\log r_2(\sigma)| \le  |\log \pi_2(\sigma)| + |\log \gamma(E_\sigma)| \le (K_0+1)\log \frac1{\gamma(E_\sigma)}
\]
we get \eqref{divergence1}.

For the case $\pi_1(\sigma)\le \pi_2(\sigma) < \gamma(E_\sigma)^{K_0}$, we have $r_1(\sigma)\le r_2(\sigma)<\gamma(E_\sigma)^{K_0-1}\le 1/10$ (using that $K_0\ge2$ and recalling $\eps<1/10$). 
Noting that $f(s) :=- s\log s$ is non-negative and increasing on $[0,1/e]$ (with $f(0):=0$), we have
\begin{align}
| \pi_1(\sigma) \log r_1(\sigma) - \pi_2(\sigma)\log r_2(\sigma)|
& = \gamma(E_\sigma) |f(r_1(\sigma)) - f(r_2(\sigma))|	\notag\\
&\le \gamma(E_\sigma)f(r_2(\sigma)) = \gamma(E_\sigma) r_2(\sigma) \log\frac1{r_2(\sigma)}	\notag \\
&\le \gamma(E_\sigma) r_2(\sigma)^{3/2}\le \gamma(E_\sigma)^{\frac32K_0-\frac12}\le \gamma(E_\sigma)^{K_0+1}.	\label{divergence2}
\end{align}

Combining \eqref{divergence1} and \eqref{divergence2}, we have
\begin{align}
&\sum_{\sigma\in \Sigma} |\pi_1(\sigma) \log r_1(\sigma) - \pi_2(\sigma)\log r_2(\sigma)|		\notag\\
&\qquad\le \|\pi_1-\pi_2\|_{\ell^\infty(\Sigma)} \bigg( |\Sigma| + (K_0+1) \sum_{\sigma\in \Sigma} \log \frac1{\gamma(E_\sigma)} \bigg) 	
+\sum_{\sigma\in \Sigma}
\gamma(E_\sigma)^{K_0+1}\,.	\label{divergence3}
\end{align}
Now we can bound
\begin{align*}
\sum_{\sigma\in\Sigma} \log\frac1{\gamma(E_\sigma)} \ls \log \frac1\eps + \sum_{k=-\ell_0+1}^{k_0} \log\frac1\eps + (1+\eps)^{2k}
&\le |\Sigma|\log\frac1\eps + \frac{R^2}\eps
\end{align*}
and
\begin{align*}
\sum_{\sigma\in \Sigma}
\gamma(E_\sigma)^{K_0+1}
&\le \max_{\sigma\in \Sigma} \{\gamma(E_\sigma)\}^{K_0} \sum_{\sigma\in\Sigma}\gamma(E_\sigma) = O(\eps)^{K_0}.
\end{align*}
Substituting these bounds in \eqref{divergence3} yields the claim.
\end{proof}

Applying Lemma \ref{lem:divergence} with $K_0=2$ along with \eqref{hathat.infty}, we get
\[
|\DKL([\hat{\mu}_{\t  y}]| [\gamma])  - \DKL([\t \nu]| [\gamma]) | \ls \eps^2 +\frac1{N} (  |\Sigma|\log\frac1\eps + \frac{R^2}\eps) \,.
\]
Note that since $R\le \sqrt{N}$ and $\eps \gs \srad_0\gs N^{-1/2}\log N$, from \eqref{bd.Alphabet} we have
\begin{equation}	\label{bd:Alphabet2}
|\Sigma|  
\ls \frac1\eps \log\frac R\eps\ls\sqrt{N}\ls N\srad_0/\log N
\end{equation}
and since $\srad_0\ge C_0R/\sqrt{N}$, taking $C_0\ge 1/c_0^2$ gives
\begin{equation}	\label{eps.non-dominant}
\frac{R^2}{\eps N} \le \frac{R^2}{c_0\srad_0N} \le \frac{R}{c_0C_0\sqrt{N}} \le c_0\srad_0\le\eps\,.
\end{equation}
Thus, 
\[
|\DKL([\hat{\mu}_{\t  y}]| [\gamma])  - \DKL([\t \nu]| [\gamma]) | \ls \eps^2 +\frac1{N} (  |\Sigma|\log\frac1\eps + \frac{R^2}\eps) 
=O(\srad')\,.
\]
Since $\DKL(\t \nu, \gamma)\ge\DKL([\t \nu], [\gamma])$ by Jensen's inequality, we obtain the lower bound
\begin{equation}	\label{hathat.DKL}
\DKL(\t \nu, \gamma)\ge
\DKL([\t \nu], [\gamma])\ge
\DKL([\hat{\mu}_{\t  y}]| [\gamma])  - 
O(\srad').
\end{equation}
Furthemore, we claim that for any function $f:[-R,R]\to \R$ such that $f(0)=0$ and $f(\sqrt{|\cdot|})$ is $K$-Lipschitz, we have
\begin{equation}	\label{hathat.f}
\int_\Sigma f\circ \iota_\eps d[\hat{\mu}_{\t  y}]
= \t \nu(f) + O(K\srad').
\end{equation}
Indeed, since $\t \nu$ and $\hat{\mu}_{\t  y}$ are both supported on $[-R,R]$, we have
\begin{align*}
\int_\Sigma f\circ \iota_\eps d[\hat{\mu}_{\t  y}]
&= \int_\Sigma f\circ\iota_\eps d[\t \nu] + O( \frac KN\sum_{k=-\ell_0}^{k_0} (1+\eps)^{2k})\\
&= \int_\Sigma f\circ \iota_\eps d[\t \nu]+ O(\frac{KR^2}{\eps N})\\
&= \t \nu(f) + O( K\eps + \frac{KR^2}{\eps N}) = \t \nu(f) + O(K\srad')
\end{align*}
where in the first line we applied \eqref{hathat.infty} and the assumption on $f$, and in the last line we applied Lemma \ref{lem:coarse} and \eqref{eps.non-dominant}.

Next we claim that for any $y\in \R^N$, 
\begin{equation}	\label{A-A-}
[\hat{\mu}_y]=[\hat{\mu}_{\t  y}] \quad\Longrightarrow\quad y\in A_-(R,\tfrac12\srad') \quad\text{ and } \quad \cW_2(\hat{\mu}_y,\hat{\mu}_{\t  y}) \le \tfrac12\srad'
\end{equation}
if $c_0$ in \eqref{eps.lower} is sufficiently small.
Indeed, assuming $[\hat{\mu}_y]=[\hat{\mu}_{\t  y}]$, since $\t  \nu$ is supported on $(1-\srad')[-R,R]$ it follows that $\hat{\mu}_{\t  y}$ is supported on $(1-\srad')[-R,R]$, and hence $\hat{\mu}_y$ is supported on $(1+\eps)(1-\srad')[-R,R]\subset (1-\frac12\srad')[-R,R]$ (taking $c_0$ sufficiently small). 
Moreover, applying Lemma \ref{lem:coarse} and \eqref{hathat.f} with $f(s)=s^2$, we have
\[
\frac1N\|y\|_2^2 = \int_\R s^2d\hat{\mu}_y(s) = \int_\Sigma \iota_\eps^2d[\hat{\mu}_{\t  y}] + O(\eps) = \int_\R s^2d\t \nu(s) + O(\eps + \frac{R^2}{\eps N}) = 1+ O(\eps) 
\]
(using \eqref{eps.non-dominant} in the final bound).
Taking $c_0$ sufficiently small we hence have
 $y\in A_-(R,\frac12\srad')$.
Furthermore,
\begin{align*}
\cW_2(\hat{\mu}_y,\hat{\mu}_{\t  y})^2 = \frac1N \min_\varrho \sum_{i=1}^N |y_{\varrho(i)} - \t  y_i|^2 \le \sum_{\sigma\in\Sigma} [\hat{\mu}_{\t  y}](\sigma) \diam(E_\sigma)^2 \ls  \eps^2 \sum_{k=-\ell_0}^{k_0} [\hat{\mu}_{\t  y}](\sigma) (1+\eps)^{2(k-1)}
\ls \eps^2
\end{align*}
where in the first bound we took $\varrho$ to be any permutation such that $y_{\varrho(i)}\in E_\sigma$ whenever $\t  y_i\in E_\sigma$ for all $\sigma\in\Sigma$ (which we can do since $\hat{\mu}_y(E_\sigma)=\hat{\mu}_{\t  y}(E_\sigma)$ for all $\sigma\in\Sigma$), and in the last bound we applied \eqref{hathat.f} to bound the sum by $O(1)$ (recalling $\srad'\le\frac1K$). Taking $c_0$ smaller if necessary we obtain $\cW_2(\hat{\mu}_y,\hat{\mu}_{\t  y})\le\frac12\srad'$ and hence \eqref{A-A-}. 

Now we assemble all of our bounds to conclude the proof of \eqref{varadhan.LB}. 
We have
\begin{align*}
&\int_{\Bdeloc_R\cap \Bset_2(\frac1{\sqrt{N}}\t  y,\srad')}  \exp\Big(  \sum_{j=1}^N \ham(\sqrt{N}u_j) \Big) dP_{N}(u)\\
(\text{Lemma \ref{lem:vara.G}}) &\qquad \ge
e^{-O(K\srad' N)}\E e^{N\hat{\mu}_g(\ham)} \ind(g\in A_-(R,\tfrac12\srad'), \cW_2(\hat{\mu}_g, \hat{\mu}_{\t  y}) \le \tfrac12\srad')\\
\eqref{A-A-}&\qquad\ge e^{-O(K\srad' N)}\E e^{N\hat{\mu}_g(\ham)}\ind( [\hat{\mu}_g]=[\hat{\mu}_{\t  y}])\\
(\text{Lemma \ref{lem:coarse}})&\qquad = 
e^{-O(K\srad' N)}  \exp\Big( N\int_\Sigma \ham\circ\iota_\eps d[\hat{\mu}_{\t  y}] \Big) \P([\hat{\mu}_g]=[\hat{\mu}_{\t  y}])\\
(\text{Lemma \ref{lem:DeZe}})&\qquad\ge 
\exp\bigg( N\bigg( \int h\circ \iota_\eps d[\hat{\mu}_{\t  y}] - \DKL( [\hat{\mu}_{\t  y}] |[\gamma]) \bigg) - |\Sigma|\log (N+1) - O(K\srad'N)\bigg)\,.
\end{align*}
The claim \eqref{varadhan.LB} now follows upon
substituting  the bounds \eqref{hathat.f} with $f=\ham$ along with \eqref{bd:Alphabet2} and \eqref{hathat.DKL}, taking logs and dividing through by $N$.

\end{appendix}

\begin{acks}[Acknowledgments]
We thank Ofer Zeitouni for giving us the idea 
to look closer at the convergence of the rate function in the case where $\psi_\mu$ goes to zero at infinity, which led us to the proof of the full LDP of Theorem \ref{thm:fullLDP}. 
This work was initiated in the fall of 2021 while N.A.C.\ and A.G.\ were participants in the MSRI (now SLMath) program ``Universality and Integrability in Random Matrix Theory and Interacting Particle Systems''. We thank the institute and the program organizers for providing a stimulating work environment.
Finally, we thank the referees for their careful reading of the paper and many helpful suggestions  to improve the exposition.
\end{acks}
\begin{funding}
This project has received funding from the European Research Council (ERC) under the European Union
Horizon 2020 research and innovation program (grant agreement No. 884584).
NAC was supported in part by NSF grant DMS-2154029.
\end{funding}



\bibliographystyle{imsart-number} 
\bibliography{bibLDP.bib}       

\begin{thebibliography}{88}

\bibitem{AEK:QVE}
\begin{barticle}[author]
\bauthor{\bsnm{Ajanki},~\bfnm{Oskari~Heikki}\binits{O.~H.}},
  \bauthor{\bsnm{Erd\H{o}s},~\bfnm{L\'{a}szl\'{o}}\binits{L.}} \AND
  \bauthor{\bsnm{Kr\"{u}ger},~\bfnm{Torben}\binits{T.}}
(\byear{2019}).
\btitle{Quadratic vector equations on complex upper half-plane}.
\bjournal{Mem. Amer. Math. Soc.}
\bvolume{261}
\bpages{v+133}.
\bdoi{10.1090/memo/1261}
\bmrnumber{4031100}
\end{barticle}
\endbibitem

\bibitem{ABB}
\begin{barticle}[author]
\bauthor{\bsnm{Arguin},~\bfnm{Louis-Pierre}\binits{L.-P.}},
  \bauthor{\bsnm{Belius},~\bfnm{David}\binits{D.}},
  \bauthor{\bsnm{Bourgade},~\bfnm{Paul}\binits{P.}},
  \bauthor{\bsnm{Radziwi\l\l},~\bfnm{Maksym}\binits{M.}} \AND
  \bauthor{\bsnm{Soundararajan},~\bfnm{Kannan}\binits{K.}}
(\byear{2019}).
\btitle{Maximum of the {R}iemann zeta function on a short interval of the
  critical line}.
\bjournal{Comm. Pure Appl. Math.}
\bvolume{72}
\bpages{500--535}.
\bdoi{10.1002/cpa.21791}
\bmrnumber{3911893}
\end{barticle}
\endbibitem

\bibitem{ABA}
\begin{barticle}[author]
\bauthor{\bsnm{Auffinger},~\bfnm{Antonio}\binits{A.}} \AND
  \bauthor{\bsnm{Ben~Arous},~\bfnm{Gerard}\binits{G.}}
(\byear{2013}).
\btitle{Complexity of random smooth functions on the high-dimensional sphere}.
\bjournal{Ann. Probab.}
\bvolume{41}
\bpages{4214--4247}.
\bdoi{10.1214/13-AOP862}
\bmrnumber{3161473}
\end{barticle}
\endbibitem

\bibitem{ABAC}
\begin{barticle}[author]
\bauthor{\bsnm{Auffinger},~\bfnm{Antonio}\binits{A.}},
  \bauthor{\bsnm{Ben~Arous},~\bfnm{G{\'e}rard}\binits{G.}} \AND
  \bauthor{\bsnm{{\v{C}}ern{\'y}},~\bfnm{Ji{\v{r}}{\'{\i}}}\binits{J.}}
(\byear{2013}).
\btitle{Random matrices and complexity of spin glasses}.
\bjournal{Commun. Pure Appl. Math.}
\bvolume{66}
\bpages{165--201}.
\bdoi{10.1002/cpa.21422}
\end{barticle}
\endbibitem

\bibitem{Augeri:stretched}
\begin{barticle}[author]
\bauthor{\bsnm{Augeri},~\bfnm{Fanny}\binits{F.}}
(\byear{2016}).
\btitle{Large deviations principle for the largest eigenvalue of {W}igner
  matrices without {G}aussian tails}.
\bjournal{Electron. J. Probab.}
\bvolume{21}
\bpages{Paper No. 32, 49}.
\bdoi{10.1214/16-EJP4146}
\bmrnumber{3492936}
\end{barticle}
\endbibitem

\bibitem{Augeri:moments}
\begin{barticle}[author]
\bauthor{\bsnm{{Augeri}},~\bfnm{Fanny}\binits{F.}}
(\byear{2018}).
\btitle{{On the large deviations of traces of random matrices}}.
\bjournal{{Ann. Inst. H. Poincar\'e, Probab. Stat.}}
\bvolume{54}
\bpages{2239--2285}.
\bdoi{10.1214/17-AIHP870}
\end{barticle}
\endbibitem

\bibitem{Augeri:ER}
\begin{barticle}[author]
\bauthor{\bsnm{Augeri},~\bfnm{Fanny}\binits{F.}}
(\byear{2020}).
\btitle{Nonlinear large deviation bounds with applications to {W}igner matrices
  and sparse {E}rd\"{o}s-{R}\'{e}nyi graphs}.
\bjournal{Ann. Probab.}
\bvolume{48}
\bpages{2404--2448}.
\bdoi{10.1214/20-AOP1427}
\bmrnumber{4152647}
\end{barticle}
\endbibitem

\bibitem{Augeri:esd}
\begin{barticle}[author]
\bauthor{\bsnm{Augeri},~\bfnm{Fanny}\binits{F.}}
(\byear{2025}).
\btitle{Large deviations of the empirical spectral measure of supercritical
  sparse {W}igner matrices}.
\bjournal{Adv. Math.}
\bvolume{466}
\bpages{Paper No. 110156, 53}.
\bdoi{10.1016/j.aim.2025.110156}
\bmrnumber{4869057}
\end{barticle}
\endbibitem

\bibitem{AuBa}
\begin{bunpublished}[author]
\bauthor{\bsnm{Augeri},~\bfnm{Fanny}\binits{F.}} \AND
  \bauthor{\bsnm{Basak},~\bfnm{Anirban}\binits{A.}}
\btitle{Large deviations of the largest eigenvalue of supercritical sparse
  {W}igner matrices}.
\bnote{Preprint, arXiv:2304.13364}.
\end{bunpublished}
\endbibitem

\bibitem{AGH}
\begin{barticle}[author]
\bauthor{\bsnm{Augeri},~\bfnm{Fanny}\binits{F.}},
  \bauthor{\bsnm{Guionnet},~\bfnm{Alice}\binits{A.}} \AND
  \bauthor{\bsnm{Husson},~\bfnm{Jonathan}\binits{J.}}
(\byear{2021}).
\btitle{Large deviations for the largest eigenvalue of sub-{G}aussian
  matrices}.
\bjournal{Comm. Math. Phys.}
\bvolume{383}
\bpages{997--1050}.
\bdoi{10.1007/s00220-021-04027-9}
\bmrnumber{4239836}
\end{barticle}
\endbibitem

\bibitem{BBP}
\begin{barticle}[author]
\bauthor{\bsnm{Baik},~\bfnm{Jinho}\binits{J.}},
  \bauthor{\bsnm{Ben~Arous},~\bfnm{G{\'e}rard}\binits{G.}} \AND
  \bauthor{\bsnm{P{\'e}ch{\'e}},~\bfnm{Sandrine}\binits{S.}}
(\byear{2005}).
\btitle{Phase transition of the largest eigenvalue for nonnull complex sample
  covariance matrices}.
\bjournal{Ann. Probab.}
\bvolume{33}
\bpages{1643--1697}.
\bdoi{10.1214/009117905000000233}
\bmrnumber{2165575 (2006g:15046)}
\end{barticle}
\endbibitem

\bibitem{Basak}
\begin{barticle}[author]
\bauthor{\bsnm{Basak},~\bfnm{Anirban}\binits{A.}}
(\byear{2023}).
\btitle{Upper tail of the spectral radius of sparse {E}rd\"{o}s-{R}\'{e}nyi
  graphs}.
\bjournal{Probab. Theory Related Fields}
\bvolume{187}
\bpages{885--947}.
\bdoi{10.1007/s00440-023-01232-6}
\bmrnumber{4664587}
\end{barticle}
\endbibitem

\bibitem{BaBa}
\begin{barticle}[author]
\bauthor{\bsnm{Basak},~\bfnm{Anirban}\binits{A.}} \AND
  \bauthor{\bsnm{Basu},~\bfnm{Riddhipratim}\binits{R.}}
(\byear{2023}).
\btitle{Upper tail large deviations of regular subgraph counts in
  {E}rd{\H{o}}s-{R}\'{e}nyi graphs in the full localized regime}.
\bjournal{Comm. Pure Appl. Math.}
\bvolume{76}
\bpages{3--72}.
\bmrnumber{4544794}
\end{barticle}
\endbibitem

\bibitem{BHG}
\begin{barticle}[author]
\bauthor{\bsnm{Belinschi},~\bfnm{Serban}\binits{S.}},
  \bauthor{\bsnm{Guionnet},~\bfnm{Alice}\binits{A.}} \AND
  \bauthor{\bsnm{Huang},~\bfnm{Jiaoyang}\binits{J.}}
(\byear{2022}).
\btitle{Large deviation principles via spherical integrals}.
\bjournal{Probab. Math. Phys.}
\bvolume{3}
\bpages{543--625}.
\bdoi{10.2140/pmp.2022.3.543}
\bmrnumber{4520314}
\end{barticle}
\endbibitem

\bibitem{BBMcK}
\begin{barticle}[author]
\bauthor{\bsnm{Ben~Arous},~\bfnm{G\'{e}rard}\binits{G.}},
  \bauthor{\bsnm{Bourgade},~\bfnm{Paul}\binits{P.}} \AND
  \bauthor{\bsnm{McKenna},~\bfnm{Benjamin}\binits{B.}}
(\byear{2022}).
\btitle{Exponential growth of random determinants beyond invariance}.
\bjournal{Probab. Math. Phys.}
\bvolume{3}
\bpages{731--789}.
\bdoi{10.2140/pmp.2022.3.731}
\bmrnumber{4552227}
\end{barticle}
\endbibitem

\bibitem{BDG}
\begin{barticle}[author]
\bauthor{\bsnm{Ben~Arous},~\bfnm{G.}\binits{G.}},
  \bauthor{\bsnm{Dembo},~\bfnm{A.}\binits{A.}} \AND
  \bauthor{\bsnm{Guionnet},~\bfnm{A.}\binits{A.}}
(\byear{2001}).
\btitle{Aging of spherical spin glasses}.
\bjournal{Probab. Theory Related Fields}
\bvolume{120}
\bpages{1--67}.
\bdoi{10.1007/PL00008774}
\bmrnumber{1856194}
\end{barticle}
\endbibitem

\bibitem{BAG97}
\begin{barticle}[author]
\bauthor{\bsnm{Ben~Arous},~\bfnm{G.}\binits{G.}} \AND
  \bauthor{\bsnm{Guionnet},~\bfnm{A.}\binits{A.}}
(\byear{1997}).
\btitle{Large deviations for {W}igner's law and {V}oiculescu's non-commutative
  entropy}.
\bjournal{Probab. Theory Related Fields}
\bvolume{108}
\bpages{517--542}.
\bdoi{10.1007/s004400050119}
\bmrnumber{1465640}
\end{barticle}
\endbibitem

\bibitem{BAMMN}
\begin{barticle}[author]
\bauthor{\bsnm{Ben~Arous},~\bfnm{G\'{e}rard}\binits{G.}},
  \bauthor{\bsnm{Mei},~\bfnm{Song}\binits{S.}},
  \bauthor{\bsnm{Montanari},~\bfnm{Andrea}\binits{A.}} \AND
  \bauthor{\bsnm{Nica},~\bfnm{Mihai}\binits{M.}}
(\byear{2019}).
\btitle{The landscape of the spiked tensor model}.
\bjournal{Commun. Pure Appl. Math}
\bvolume{72}
\bpages{2282--2330}.
\bdoi{10.1002/cpa.21861}
\bmrnumber{4011861}
\end{barticle}
\endbibitem

\bibitem{BGGM}
\begin{barticle}[author]
\bauthor{\bsnm{Benaych-Georges},~\bfnm{F.}\binits{F.}},
  \bauthor{\bsnm{Guionnet},~\bfnm{A.}\binits{A.}} \AND
  \bauthor{\bsnm{Ma\"{\i}da},~\bfnm{M.}\binits{M.}}
(\byear{2012}).
\btitle{Large deviations of the extreme eigenvalues of random deformations of
  matrices}.
\bjournal{Probab. Theory Related Fields}
\bvolume{154}
\bpages{703--751}.
\bdoi{10.1007/s00440-011-0382-3}
\bmrnumber{3000560}
\end{barticle}
\endbibitem

\bibitem{BeKn}
\begin{bincollection}[author]
\bauthor{\bsnm{Benaych-Georges},~\bfnm{Florent}\binits{F.}} \AND
  \bauthor{\bsnm{Knowles},~\bfnm{Antti}\binits{A.}}
(\byear{2017}).
\btitle{Local semicircle law for {W}igner matrices}.
In \bbooktitle{Advanced topics in random matrices}.
\bseries{Panor. Synth\`eses}
\bvolume{53}
\bpages{1--90}.
\bpublisher{Soc. Math. France, Paris}.
\bmrnumber{3792624}
\end{bincollection}
\endbibitem

\bibitem{BGR}
\begin{barticle}[author]
\bauthor{\bsnm{Bercu},~\bfnm{B.}\binits{B.}},
  \bauthor{\bsnm{Gamboa},~\bfnm{F.}\binits{F.}} \AND
  \bauthor{\bsnm{Rouault},~\bfnm{A.}\binits{A.}}
(\byear{1997}).
\btitle{Large deviations for quadratic forms of stationary {G}aussian
  processes}.
\bjournal{Stochastic Process. Appl.}
\bvolume{71}
\bpages{75--90}.
\bdoi{10.1016/S0304-4149(97)00071-9}
\bmrnumber{1480640}
\end{barticle}
\endbibitem

\bibitem{BBG}
\begin{barticle}[author]
\bauthor{\bsnm{Bhattacharya},~\bfnm{Bhaswar~B.}\binits{B.~B.}},
  \bauthor{\bsnm{Bhattacharya},~\bfnm{Sohom}\binits{S.}} \AND
  \bauthor{\bsnm{Ganguly},~\bfnm{Shirshendu}\binits{S.}}
(\byear{2021}).
\btitle{Spectral edge in sparse random graphs: {U}pper and lower tail large
  deviations}.
\bjournal{Ann. Probab.}
\bvolume{49}
\bpages{1847--1885}.
\bdoi{10.1214/20-aop1495}
\bmrnumber{4260469}
\end{barticle}
\endbibitem

\bibitem{BhGa}
\begin{barticle}[author]
\bauthor{\bsnm{Bhattacharya},~\bfnm{Bhaswar~B.}\binits{B.~B.}} \AND
  \bauthor{\bsnm{Ganguly},~\bfnm{Shirshendu}\binits{S.}}
(\byear{2020}).
\btitle{Upper tails for edge eigenvalues of random graphs}.
\bjournal{SIAM J. Discrete Math.}
\bvolume{34}
\bpages{1069--1083}.
\bdoi{10.1137/18M1230852}
\bmrnumber{4083586}
\end{barticle}
\endbibitem

\bibitem{BDMN11}
\begin{barticle}[author]
\bauthor{\bsnm{Bianchi},~\bfnm{P.}\binits{P.}},
  \bauthor{\bsnm{Debbah},~\bfnm{M.}\binits{M.}},
  \bauthor{\bsnm{Ma\"{\i}da},~\bfnm{M.}\binits{M.}} \AND
  \bauthor{\bsnm{Najim},~\bfnm{J.}\binits{J.}}
(\byear{2011}).
\btitle{Performance of statistical tests for single-source detection using
  random matrix theory}.
\bjournal{IEEE Trans. Inform. Theory}
\bvolume{57}
\bpages{2400--2419}.
\bdoi{10.1109/TIT.2011.2111710}
\bmrnumber{2809098}
\end{barticle}
\endbibitem

\bibitem{Giulio}
\begin{barticle}[author]
\bauthor{\bsnm{Biroli},~\bfnm{Giulio}\binits{G.}} \AND
  \bauthor{\bsnm{Guionnet},~\bfnm{Alice}\binits{A.}}
(\byear{2020}).
\btitle{Large deviations for the largest eigenvalues and eigenvectors of spiked
  {G}aussian random matrices}.
\bjournal{Electron. Commun. Probab.}
\bvolume{25}
\bpages{Paper No. 70, 13}.
\bdoi{10.3390/mca25010013}
\bmrnumber{4158230}
\end{barticle}
\endbibitem

\bibitem{BCG:subgaussian}
\begin{barticle}[author]
\bauthor{\bsnm{Bobkov},~\bfnm{S.~G.}\binits{S.~G.}},
  \bauthor{\bsnm{Chistyakov},~\bfnm{G.~P.}\binits{G.~P.}} \AND
  \bauthor{\bsnm{G\"{o}tze},~\bfnm{F.}\binits{F.}}
(\byear{2024}).
\btitle{Strictly subgaussian probability distributions}.
\bjournal{Electron. J. Probab.}
\bvolume{29}
\bpages{--}.
\bdoi{10.1214/24-ejp1120}
\bmrnumber{4736269}
\end{barticle}
\endbibitem

\bibitem{BordCap}
\begin{barticle}[author]
\bauthor{\bsnm{Bordenave},~\bfnm{C.}\binits{C.}} \AND
  \bauthor{\bsnm{Caputo},~\bfnm{P.}\binits{P.}}
(\byear{2014}).
\btitle{A large deviation principle for {W}igner matrices without {G}aussian
  tails}.
\bjournal{Ann. Probab.}
\bvolume{42}
\bpages{2454--2496}.
\bdoi{10.1214/13-AOP866}
\bmrnumber{3265172}
\end{barticle}
\endbibitem

\bibitem{BoGu24}
\begin{bunpublished}[author]
\bauthor{\bsnm{Boursier},~\bfnm{Jeanne}\binits{J.}} \AND
  \bauthor{\bsnm{Guionnet},~\bfnm{Alice}\binits{A.}}
\btitle{Large deviations for the smallest eigenvalue of a deformed GOE with an
  outlier}.
\bnote{Preprint, arXiv:2408.09256}.
\end{bunpublished}
\endbibitem

\bibitem{cartan}
\begin{barticle}[author]
\bauthor{\bsnm{{Cartan}},~\bfnm{E.}\binits{E.}}
(\byear{1929}).
\btitle{{Sur la d\'etermination d'un syst\`eme orthogonal complet dans un
  espace de Riemann sym\'etrique clos.}}
\bjournal{{Rend. Circ. Mat. Palermo}}
\bvolume{53}
\bpages{217--252}.
\bdoi{10.1007/BF03024106}
\end{barticle}
\endbibitem

\bibitem{ChaDe}
\begin{barticle}[author]
\bauthor{\bsnm{Chatterjee},~\bfnm{Sourav}\binits{S.}} \AND
  \bauthor{\bsnm{Dembo},~\bfnm{Amir}\binits{A.}}
(\byear{2016}).
\btitle{Nonlinear large deviations}.
\bjournal{Adv. Math.}
\bvolume{299}
\bpages{396--450}.
\bdoi{10.1016/j.aim.2016.05.017}
\bmrnumber{3519474}
\end{barticle}
\endbibitem

\bibitem{ChVa2}
\begin{barticle}[author]
\bauthor{\bsnm{Chatterjee},~\bfnm{Sourav}\binits{S.}} \AND
  \bauthor{\bsnm{R.~S.~Varadhan},~\bfnm{S}\binits{S.}}
(\byear{2012}).
\btitle{Large deviations for random matrices}.
\bjournal{Comm Stoch Anal}
\bvolume{6}
\bpages{1-13}.
\bdoi{10.31390/cosa.6.1.02}
\end{barticle}
\endbibitem

\bibitem{ChaVa}
\begin{barticle}[author]
\bauthor{\bsnm{Chatterjee},~\bfnm{Sourav}\binits{S.}} \AND
  \bauthor{\bsnm{Varadhan},~\bfnm{S.~R.~S.}\binits{S.~R.~S.}}
(\byear{2011}).
\btitle{The large deviation principle for the {E}rd{\H{o}}s-{R}\'{e}nyi random
  graph}.
\bjournal{European J. Combin.}
\bvolume{32}
\bpages{1000--1017}.
\bdoi{10.1016/j.ejc.2011.03.014}
\bmrnumber{2825532}
\end{barticle}
\endbibitem

\bibitem{Chung:book}
\begin{bbook}[author]
\bauthor{\bsnm{Chung},~\bfnm{Fan R.~K.}\binits{F.~R.~K.}}
(\byear{1997}).
\btitle{Spectral graph theory}.
\bseries{CBMS Regional Conference Series in Mathematics}
\bvolume{92}.
\bpublisher{Conference Board of the Mathematical Sciences, Washington, DC; by
  the American Mathematical Society, Providence, RI}.
\bmrnumber{1421568}
\end{bbook}
\endbibitem

\bibitem{CoDe}
\begin{barticle}[author]
\bauthor{\bsnm{Cook},~\bfnm{Nicholas}\binits{N.}} \AND
  \bauthor{\bsnm{Dembo},~\bfnm{Amir}\binits{A.}}
(\byear{2020}).
\btitle{Large deviations of subgraph counts for sparse
  {E}rd{\H{o}}s-{R}\'{e}nyi graphs}.
\bjournal{Adv. Math.}
\bvolume{373}
\bpages{107289, 53}.
\bdoi{10.1016/j.aim.2020.107289}
\bmrnumber{4130460}
\end{barticle}
\endbibitem

\bibitem{Majum2}
\begin{barticle}[author]
\bauthor{\bsnm{Dean},~\bfnm{D.}\binits{D.}} \AND
  \bauthor{\bsnm{Majumdar},~\bfnm{S.}\binits{S.}}
(\byear{2006}).
\btitle{Large deviations of extreme eigenvalues of random matrices}.
\bjournal{Phys. Rev. Lett.}
\bvolume{97}
\bpages{160201, 4}.
\bdoi{10.1103/PhysRevLett.97.160201}
\bmrnumber{2274338}
\end{barticle}
\endbibitem

\bibitem{DZ}
\begin{bbook}[author]
\bauthor{\bsnm{Dembo},~\bfnm{A.}\binits{A.}} \AND
  \bauthor{\bsnm{Zeitouni},~\bfnm{O.}\binits{O.}}
(\byear{1998}).
\btitle{Large deviations techniques and applications},
\bedition{second} ed.
\bseries{Applications of Mathematics (New York)}
\bvolume{38}.
\bpublisher{Springer-Verlag, New York}.
\bdoi{10.1007/978-1-4612-5320-4}
\bmrnumber{1619036}
\end{bbook}
\endbibitem

\bibitem{DZ15}
\begin{barticle}[author]
\bauthor{\bsnm{Dembo},~\bfnm{Amir}\binits{A.}} \AND
  \bauthor{\bsnm{Zeitouni},~\bfnm{Ofer}\binits{O.}}
(\byear{2015}).
\btitle{Matrix optimization under random external fields}.
\bjournal{J. Stat. Phys.}
\bvolume{159}
\bpages{1306--1326}.
\bdoi{10.1007/s10955-015-1228-7}
\end{barticle}
\endbibitem

\bibitem{DGH24}
\begin{bmisc}[author]
\bauthor{\bsnm{Ducatez},~\bfnm{Raphal}\binits{R.}},
  \bauthor{\bsnm{Guionnet},~\bfnm{Alice}\binits{A.}} \AND
  \bauthor{\bsnm{Husson},~\bfnm{Jonathan}\binits{J.}}
(\byear{2024}).
\btitle{Large deviation principle for the largest eigenvalue of random matrices
  with a variance profile}.
\bnote{arXiv:2403.05413}.
\end{bmisc}
\endbibitem

\bibitem{dyson}
\begin{barticle}[author]
\bauthor{\bsnm{Dyson},~\bfnm{F.~J.}\binits{F.~J.}}
(\byear{1962}).
\btitle{A {B}rownian-motion model for the eigenvalues of a random matrix}.
\bjournal{J. Math. Phys.}
\bvolume{3}
\bpages{1191--1198}.
\end{barticle}
\endbibitem

\bibitem{EPRY}
\begin{barticle}[author]
\bauthor{\bsnm{Erd{\H{o}}s},~\bfnm{L{\'a}szl{\'o}}\binits{L.}},
  \bauthor{\bsnm{P{\'e}ch{\'e}},~\bfnm{Sandrine}\binits{S.}},
  \bauthor{\bsnm{Ram{\'{\i}}rez},~\bfnm{Jos{\'e}~A.}\binits{J.~A.}},
  \bauthor{\bsnm{Schlein},~\bfnm{Benjamin}\binits{B.}} \AND
  \bauthor{\bsnm{Yau},~\bfnm{Horng-Tzer}\binits{H.-T.}}
(\byear{2010}).
\btitle{Bulk universality for {W}igner matrices}.
\bjournal{Comm. Pure Appl. Math.}
\bvolume{63}
\bpages{895--925}.
\bdoi{10.1002/cpa.20317}
\bmrnumber{2662426}
\end{barticle}
\endbibitem

\bibitem{ESY09a}
\begin{barticle}[author]
\bauthor{\bsnm{Erd{\H{o}}s},~\bfnm{L{\'a}szl{\'o}}\binits{L.}},
  \bauthor{\bsnm{Schlein},~\bfnm{Benjamin}\binits{B.}} \AND
  \bauthor{\bsnm{Yau},~\bfnm{Horng-Tzer}\binits{H.-T.}}
(\byear{2009}).
\btitle{Local semicircle law and complete delocalization for {W}igner random
  matrices}.
\bjournal{Comm. Math. Phys.}
\bvolume{287}
\bpages{641--655}.
\end{barticle}
\endbibitem

\bibitem{ESY}
\begin{barticle}[author]
\bauthor{\bsnm{Erd\H{o}s},~\bfnm{L\'aszl\'o}\binits{L.}},
  \bauthor{\bsnm{Schlein},~\bfnm{Benjamin}\binits{B.}} \AND
  \bauthor{\bsnm{Yau},~\bfnm{Horng-Tzer}\binits{H.-T.}}
(\byear{2010}).
\btitle{Wegner estimate and level repulsion for {W}igner random matrices}.
\bjournal{Int. Math. Res. Not. IMRN}
\bvolume{3}
\bpages{436--479}.
\bmrnumber{2587574}
\end{barticle}
\endbibitem

\bibitem{ESY11}
\begin{barticle}[author]
\bauthor{\bsnm{Erd{\H{o}}s},~\bfnm{L{\'a}szl{\'o}}\binits{L.}},
  \bauthor{\bsnm{Schlein},~\bfnm{Benjamin}\binits{B.}} \AND
  \bauthor{\bsnm{Yau},~\bfnm{Horng-Tzer}\binits{H.-T.}}
(\byear{2011}).
\btitle{Universality of random matrices and local relaxation flow}.
\bjournal{Invent. Math.}
\bvolume{185}
\bpages{75--119}.
\bdoi{10.1007/s00222-010-0302-7}
\bmrnumber{2810797}
\end{barticle}
\endbibitem

\bibitem{For93}
\begin{barticle}[author]
\bauthor{\bsnm{Forrester},~\bfnm{P.~J.}\binits{P.~J.}}
(\byear{1993}).
\btitle{The spectral edge of random matrix ensembles}.
\bjournal{Nuclear Phys. B}
\bvolume{402}
\bpages{709--728}.
\end{barticle}
\endbibitem

\bibitem{furedi}
\begin{barticle}[author]
\bauthor{\bsnm{F\"{u}redi},~\bfnm{Z.}\binits{Z.}} \AND
  \bauthor{\bsnm{Koml\'{o}s},~\bfnm{J.}\binits{J.}}
(\byear{1981}).
\btitle{The eigenvalues of random symmetric matrices}.
\bjournal{Combinatorica}
\bvolume{1}
\bpages{233--241}.
\end{barticle}
\endbibitem

\bibitem{FD14}
\begin{barticle}[author]
\bauthor{\bsnm{Fyodorov},~\bfnm{Yan~V.}\binits{Y.~V.}} \AND
  \bauthor{\bsnm{Le~Doussal},~\bfnm{Pierre}\binits{P.}}
(\byear{2014}).
\btitle{Topology trivialization and large deviations for the minimum in the
  simplest random optimization}.
\bjournal{J. Stat. Phys.}
\bvolume{154}
\bpages{466--490}.
\bdoi{10.1007/s10955-013-0838-1}
\end{barticle}
\endbibitem

\bibitem{GHN}
\begin{barticle}[author]
\bauthor{\bsnm{Ganguly},~\bfnm{Shirshendu}\binits{S.}},
  \bauthor{\bsnm{Hiesmayr},~\bfnm{Ella}\binits{E.}} \AND
  \bauthor{\bsnm{Nam},~\bfnm{Kyeongsik}\binits{K.}}
(\byear{2024}).
\btitle{Spectral large deviations of sparse random matrices}.
\bjournal{J. Lond. Math. Soc. (2)}
\bvolume{110}
\bpages{Paper No. e12954, 64}.
\bdoi{10.1112/jlms.12954}
\bmrnumber{4767705}
\end{barticle}
\endbibitem

\bibitem{GaNa}
\begin{barticle}[author]
\bauthor{\bsnm{Ganguly},~\bfnm{Shirshendu}\binits{S.}} \AND
  \bauthor{\bsnm{Nam},~\bfnm{Kyeongsik}\binits{K.}}
(\byear{2022}).
\btitle{Large deviations for the largest eigenvalue of {G}aussian networks with
  constant average degree}.
\bjournal{Probab. Theory Related Fields}
\bvolume{184}
\bpages{613--679}.
\bdoi{10.1007/s00440-022-01164-7}
\bmrnumber{4507932}
\end{barticle}
\endbibitem

\bibitem{HuGu1}
\begin{barticle}[author]
\bauthor{\bsnm{Guionnet},~\bfnm{Alice}\binits{A.}} \AND
  \bauthor{\bsnm{Husson},~\bfnm{Jonathan}\binits{J.}}
(\byear{2020}).
\btitle{Large deviations for the largest eigenvalue of {R}ademacher matrices}.
\bjournal{Ann. Probab.}
\bvolume{48}
\bpages{1436--1465}.
\bdoi{10.1214/19-AOP1398}
\bmrnumber{4112720}
\end{barticle}
\endbibitem

\bibitem{GuMa05}
\begin{barticle}[author]
\bauthor{\bsnm{Guionnet},~\bfnm{A.}\binits{A.}} \AND
  \bauthor{\bsnm{Ma\"{\i}da},~\bfnm{M.}\binits{M.}}
(\byear{2005}).
\btitle{A {F}ourier view on the {$R$}-transform and related asymptotics of
  spherical integrals}.
\bjournal{J. Funct. Anal.}
\bvolume{222}
\bpages{435--490}.
\bmrnumber{2132396}
\end{barticle}
\endbibitem

\bibitem{GMa20}
\begin{barticle}[author]
\bauthor{\bsnm{{Guionnet}},~\bfnm{Alice}\binits{A.}} \AND
  \bauthor{\bsnm{{Ma\"{\i}da}},~\bfnm{Myl\`ene}\binits{M.}}
(\byear{2020}).
\btitle{{Large deviations for the largest eigenvalue of the sum of two random
  matrices}}.
\bjournal{{Electron. J. Probab.}}
\bvolume{25}
\bpages{24}.
\bnote{Id/No 14}.
\bdoi{10.1214/19-EJP405}
\end{barticle}
\endbibitem

\bibitem{GuZe}
\begin{barticle}[author]
\bauthor{\bsnm{Guionnet},~\bfnm{A.}\binits{A.}} \AND
  \bauthor{\bsnm{Zeitouni},~\bfnm{O.}\binits{O.}}
(\byear{2000}).
\btitle{Concentration of the spectral measure for large matrices}.
\bjournal{Electron. Commun. Prob.}
\bvolume{5}
\bpages{119--136 (electronic)}.
\end{barticle}
\endbibitem

\bibitem{GZ3}
\begin{barticle}[author]
\bauthor{\bsnm{Guionnet},~\bfnm{A.}\binits{A.}} \AND
  \bauthor{\bsnm{Zeitouni},~\bfnm{O.}\binits{O.}}
(\byear{2002}).
\btitle{Large deviations asymptotics for spherical integrals}.
\bjournal{J. Funct. Anal.}
\bvolume{188}
\bpages{461--515}.
\end{barticle}
\endbibitem

\bibitem{HMS}
\begin{barticle}[author]
\bauthor{\bsnm{Harel},~\bfnm{Matan}\binits{M.}},
  \bauthor{\bsnm{Mousset},~\bfnm{Frank}\binits{F.}} \AND
  \bauthor{\bsnm{Samotij},~\bfnm{Wojciech}\binits{W.}}
(\byear{2022}).
\btitle{Upper tails via high moments and entropic stability}.
\bjournal{Duke Math. J.}
\bvolume{171}
\bpages{2089--2192}.
\bdoi{10.1215/00127094-2021-0067}
\bmrnumber{4484206}
\end{barticle}
\endbibitem

\bibitem{HuTi}
\begin{barticle}[author]
\bauthor{\bsnm{Huang},~\bfnm{Han}\binits{H.}} \AND
  \bauthor{\bsnm{Tikhomirov},~\bfnm{Konstantin}\binits{K.}}
(\byear{2023}).
\btitle{On dimension-dependent concentration for convex {L}ipschitz functions
  in product spaces}.
\bjournal{Electron. J. Probab.}
\bvolume{28}
\bpages{Paper No. 63, 23}.
\bdoi{10.1214/23-ejp944}
\bmrnumber{4583676}
\end{barticle}
\endbibitem

\bibitem{Hu}
\begin{barticle}[author]
\bauthor{\bsnm{Husson},~\bfnm{Jonathan}\binits{J.}}
(\byear{2022}).
\btitle{Large deviations for the largest eigenvalue of matrices with variance
  profiles}.
\bjournal{Electron. J. Probab.}
\bvolume{27}
\bpages{44}.
\bnote{Id/No 74}.
\bdoi{10.1214/22-EJP793}
\end{barticle}
\endbibitem

\bibitem{HuMc}
\begin{barticle}[author]
\bauthor{\bsnm{Husson},~\bfnm{Jonathan}\binits{J.}} \AND
  \bauthor{\bsnm{McKenna},~\bfnm{Benjamin}\binits{B.}}
(\byear{2024}).
\btitle{Large deviations for the largest eigenvalue of generalized sample
  covariance matrices}.
\bjournal{Electron. J. Probab.}
\bvolume{29}
\bpages{Paper No. 187, 48}.
\bdoi{10.1214/24-ejp1228}
\bmrnumber{4841063}
\end{barticle}
\endbibitem

\bibitem{johansson}
\begin{barticle}[author]
\bauthor{\bsnm{Johansson},~\bfnm{K.}\binits{K.}}
(\byear{1998}).
\btitle{On fluctuations of eigenvalues of random {H}ermitian matrices}.
\bjournal{Duke Math. J.}
\bvolume{91}
\bpages{151--204}.
\end{barticle}
\endbibitem

\bibitem{johanssonHS}
\begin{barticle}[author]
\bauthor{\bsnm{Johansson},~\bfnm{K.}\binits{K.}}
(\byear{2001}).
\btitle{Universality of the local spacing distribution in certain ensembles of
  {H}ermitian {W}igner matrices}.
\bjournal{Comm. Math. Phys.}
\bvolume{215}
\bpages{683--705}.
\end{barticle}
\endbibitem

\bibitem{KeatingSnaith}
\begin{barticle}[author]
\bauthor{\bsnm{Keating},~\bfnm{J.~P.}\binits{J.~P.}} \AND
  \bauthor{\bsnm{Snaith},~\bfnm{N.~C.}\binits{N.~C.}}
(\byear{2000}).
\btitle{Random matrix theory and {$\zeta(1/2+it)$}}.
\bjournal{Comm. Math. Phys.}
\bvolume{214}
\bpages{57--89}.
\bdoi{10.1007/s002200000261}
\bmrnumber{1794265}
\end{barticle}
\endbibitem

\bibitem{KlZh}
\begin{barticle}[author]
\bauthor{\bsnm{Klochkov},~\bfnm{Yegor}\binits{Y.}} \AND
  \bauthor{\bsnm{Zhivotovskiy},~\bfnm{Nikita}\binits{N.}}
(\byear{2020}).
\btitle{Uniform {H}anson-{W}right type concentration inequalities for unbounded
  entries via the entropy method}.
\bjournal{Electron. J. Probab.}
\bvolume{25}
\bpages{Paper No. 22, 30}.
\bdoi{10.1214/20-ejp422}
\bmrnumber{4073683}
\end{barticle}
\endbibitem

\bibitem{LFD}
\begin{barticle}[author]
\bauthor{\bsnm{Lacroix-A-Chez-Toine},~\bfnm{Bertrand}\binits{B.}},
  \bauthor{\bsnm{Fyodorov},~\bfnm{Yan~V.}\binits{Y.~V.}} \AND
  \bauthor{\bsnm{Le~Doussal},~\bfnm{Pierre}\binits{P.}}
(\byear{2024}).
\btitle{Replica-symmetry breaking transitions in the large deviations of the
  ground-state of a spherical spin-glass}.
\bjournal{J. Stat. Phys.}
\bvolume{191}
\bpages{Paper No. 11, 76}.
\bdoi{10.1007/s10955-024-03232-9}
\bmrnumber{4695832}
\end{barticle}
\endbibitem

\bibitem{LaSo}
\begin{bunpublished}[author]
\bauthor{\bsnm{Landon},~\bfnm{Benjamin}\binits{B.}} \AND
  \bauthor{\bsnm{Sosoe},~\bfnm{Philippe}\binits{P.}}
\btitle{Almost optimal bulk regularity conditions in the {CLT} for {W}igner
  matrices}.
\bnote{Preprint, arXiv:2204.03419}.
\end{bunpublished}
\endbibitem

\bibitem{latala}
\begin{barticle}[author]
\bauthor{\bsnm{Lata\l{a}},~\bfnm{Rafa\l}\binits{R.}}
(\byear{2005}).
\btitle{Some estimates of norms of random matrices}.
\bjournal{Proc. Amer. Math. Soc.}
\bvolume{133}
\bpages{1273--1282}.
\bdoi{10.1090/S0002-9939-04-07800-1}
\bmrnumber{2111932}
\end{barticle}
\endbibitem

\bibitem{LeNa23}
\begin{bunpublished}[author]
\bauthor{\bsnm{Lee},~\bfnm{Jaehun}\binits{J.}} \AND
  \bauthor{\bsnm{Nam},~\bfnm{Kyeongsik}\binits{K.}}
\btitle{Extremal spectral behavior of weighted random $d$-regular graphs}.
\bnote{Preprint, arXiv:2306.03479}.
\end{bunpublished}
\endbibitem

\bibitem{LuZh12}
\begin{barticle}[author]
\bauthor{\bsnm{Lubetzky},~\bfnm{Eyal}\binits{E.}} \AND
  \bauthor{\bsnm{Zhao},~\bfnm{Yufei}\binits{Y.}}
(\byear{2015}).
\btitle{On replica symmetry of large deviations in random graphs}.
\bjournal{Random Structures Algorithms}
\bvolume{47}
\bpages{109--146}.
\bdoi{10.1002/rsa.20536}
\bmrnumber{3366814}
\end{barticle}
\endbibitem

\bibitem{Maida:deformed}
\begin{barticle}[author]
\bauthor{\bsnm{Ma\"{\i}da},~\bfnm{Myl\`ene}\binits{M.}}
(\byear{2007}).
\btitle{Large deviations for the largest eigenvalue of rank one deformations of
  {G}aussian ensembles}.
\bjournal{Electron. J. Probab.}
\bvolume{12}
\bpages{1131--1150}.
\bdoi{10.1214/EJP.v12-438}
\bmrnumber{2336602}
\end{barticle}
\endbibitem

\bibitem{Ma07}
\begin{barticle}[author]
\bauthor{\bsnm{Ma\"{\i}da},~\bfnm{Myl\`ene}\binits{M.}}
(\byear{2007}).
\btitle{Large deviations for the largest eigenvalue of rank one deformations of
  {G}aussian ensembles}.
\bjournal{Electron. J. Probab.}
\bvolume{12}
\bpages{1131--1150}.
\bmrnumber{2336602}
\end{barticle}
\endbibitem

\bibitem{MS14}
\begin{barticle}[author]
\bauthor{\bsnm{Majumdar},~\bfnm{Satya~N.}\binits{S.~N.}} \AND
  \bauthor{\bsnm{Schehr},~\bfnm{Gr{\'e}gory}\binits{G.}}
(\byear{2014}).
\btitle{Top eigenvalue of a random matrix: large deviations and third order
  phase transition}.
\bjournal{J. Stat. Mech. Theory Exp.}
\bvolume{2014}
\bpages{31}.
\bnote{Id/No p01012}.
\bdoi{10.1088/1742-5468/2014/01/P01012}
\end{barticle}
\endbibitem

\bibitem{May72}
\begin{barticle}[author]
\bauthor{\bsnm{May},~\bfnm{Robert~M}\binits{R.~M.}}
(\byear{1972}).
\btitle{Will a large complex system be stable?}
\bjournal{Nature}
\bvolume{238}
\bpages{413--414}.
\end{barticle}
\endbibitem

\bibitem{McKenna:deformed}
\begin{barticle}[author]
\bauthor{\bsnm{McKenna},~\bfnm{Benjamin}\binits{B.}}
(\byear{2021}).
\btitle{Large deviations for extreme eigenvalues of deformed {W}igner random
  matrices}.
\bjournal{Electron. J. Probab.}
\bvolume{26}
\bpages{Paper No. 34, 37}.
\bdoi{10.1214/20-EJP571}
\bmrnumber{4235485}
\end{barticle}
\endbibitem

\bibitem{ME}
\begin{bbook}[author]
\bauthor{\bsnm{Mehta},~\bfnm{M.~L.}\binits{M.~L.}}
(\byear{2004}).
\btitle{Random matrices},
\bedition{third} ed.
\bseries{Pure and Applied Mathematics (Amsterdam)}
\bvolume{142}.
\bpublisher{Elsevier/Academic Press, Amsterdam}.
\bmrnumber{2129906 (2006b:82001)}
\end{bbook}
\endbibitem

\bibitem{Montgo}
\begin{binproceedings}[author]
\bauthor{\bsnm{Montgomery},~\bfnm{H.~L.}\binits{H.~L.}}
(\byear{1973}).
\btitle{The pair correlation of zeros of the zeta function}.
In \bbooktitle{Analytic number theory ({P}roc. {S}ympos. {P}ure {M}ath., {V}ol.
  {XXIV}, {S}t. {L}ouis {U}niv., {S}t. {L}ouis, {M}o., 1972)}
\bpages{181--193}.
\bmrnumber{0337821}
\end{binproceedings}
\endbibitem

\bibitem{PR08}
\begin{barticle}[author]
\bauthor{\bsnm{Parisi},~\bfnm{Giorgio}\binits{G.}} \AND
  \bauthor{\bsnm{Rizzo},~\bfnm{Tommaso}\binits{T.}}
(\byear{2010}).
\btitle{Large deviations of the free energy in diluted mean-field spin-glass}.
\bjournal{J. Phys. A}
\bvolume{43}
\bpages{045001, 18}.
\bdoi{10.1088/1751-8113/43/4/045001}
\bmrnumber{2578720}
\end{barticle}
\endbibitem

\bibitem{RaAb:neural}
\begin{barticle}[author]
\bauthor{\bsnm{Rajan},~\bfnm{Kanaka}\binits{K.}} \AND
  \bauthor{\bsnm{Abbott},~\bfnm{LF}\binits{L.}}
(\byear{2006}).
\btitle{Eigenvalue spectra of random matrices for neural networks}.
\bjournal{Physical review letters}
\bvolume{97}
\bpages{188104}.
\end{barticle}
\endbibitem

\bibitem{sosh}
\begin{barticle}[author]
\bauthor{\bsnm{Soshnikov},~\bfnm{Alexander}\binits{A.}}
(\byear{1999}).
\btitle{Universality at the edge of the spectrum in {W}igner random matrices}.
\bjournal{Comm. Math. Phys.}
\bvolume{207}
\bpages{697--733}.
\bdoi{10.1007/s002200050743}
\bmrnumber{1727234}
\end{barticle}
\endbibitem

\bibitem{ST02}
\begin{bincollection}[author]
\bauthor{\bsnm{Spielman},~\bfnm{D.~A.}\binits{D.~A.}} \AND
  \bauthor{\bsnm{Teng},~\bfnm{S.~H.}\binits{S.~H.}}
(\byear{2002}).
\btitle{Smooth analysis of algorithms}.
In \bbooktitle{Proceedings of the international congress of Mathematicians
  (Beijing 2002)},
\bvolume{I}
\bpages{597--606}.
\bpublisher{Higher Ed. Press}, \baddress{Beijing}.
\end{bincollection}
\endbibitem

\bibitem{talagrand}
\begin{barticle}[author]
\bauthor{\bsnm{Talagrand},~\bfnm{M.}\binits{M.}}
(\byear{1996}).
\btitle{A new look at independence}.
\bjournal{Annals Probab.}
\bvolume{24}
\bpages{1--34}.
\end{barticle}
\endbibitem

\bibitem{TVun}
\begin{barticle}[author]
\bauthor{\bsnm{Tao},~\bfnm{Terence}\binits{T.}} \AND
  \bauthor{\bsnm{Vu},~\bfnm{Van}\binits{V.}}
(\byear{2010}).
\btitle{Random matrices: universality of local eigenvalue statistics up to the
  edge}.
\bjournal{Comm. Math. Phys.}
\bvolume{298}
\bpages{549--572}.
\bmrnumber{2669449}
\end{barticle}
\endbibitem

\bibitem{TW}
\begin{barticle}[author]
\bauthor{\bsnm{Tracy},~\bfnm{C.~A.}\binits{C.~A.}} \AND
  \bauthor{\bsnm{Widom},~\bfnm{H.}\binits{H.}}
(\byear{1994}).
\btitle{Level spacing distributions and the {A}iry kernel}.
\bjournal{Commun. Math. Phys.}
\bvolume{159}
\bpages{151--174}.
\end{barticle}
\endbibitem

\bibitem{Vershynin:book}
\begin{bbook}[author]
\bauthor{\bsnm{Vershynin},~\bfnm{Roman}\binits{R.}}
(\byear{2018}).
\btitle{High-dimensional probability}.
\bseries{Cambridge Series in Statistical and Probabilistic Mathematics}
\bvolume{47}.
\bpublisher{Cambridge University Press, Cambridge}
\bnote{An introduction with applications in data science, With a foreword by
  Sara van de Geer}.
\bdoi{10.1017/9781108231596}
\bmrnumber{3837109}
\end{bbook}
\endbibitem

\bibitem{Majum}
\begin{barticle}[author]
\bauthor{\bsnm{Vivo},~\bfnm{P.}\binits{P.}},
  \bauthor{\bsnm{Majumdar},~\bfnm{S.}\binits{S.}} \AND
  \bauthor{\bsnm{Bohigas},~\bfnm{O.}\binits{O.}}
(\byear{2007}).
\btitle{Large deviations of the maximum eigenvalue in {W}ishart random
  matrices}.
\bjournal{J. Phys. A}
\bvolume{40}
\bpages{4317--4337}.
\bdoi{10.1088/1751-8113/40/16/005}
\bmrnumber{2316708}
\end{barticle}
\endbibitem

\bibitem{voi91}
\begin{barticle}[author]
\bauthor{\bsnm{Voiculescu},~\bfnm{D.}\binits{D.}}
(\byear{1991}).
\btitle{Limit laws for random matrices and free products}.
\bjournal{Invent. Math.}
\bvolume{104}
\bpages{201--220}.
\end{barticle}
\endbibitem

\bibitem{voicstflour}
\begin{binbook}[author]
\bauthor{\bsnm{Voiculescu},~\bfnm{D.}\binits{D.}}
(\byear{2000}).
\btitle{Lectures on Probability Theory and Statistics: {E}cole
  {D'\'{E}}t{\'{e}} de {P}robabilit{\'{e}}s de {S}aint-{F}lour {XXVIII} -
  1998}.
\bseries{Lecture Notes in Mathematics}
\bvolume{1738}
\bpages{283--349}.
\bpublisher{Springer}, \baddress{New York, NY}.
\end{binbook}
\endbibitem

\bibitem{GoVN}
\begin{barticle}[author]
\bauthor{\bparticle{von} \bsnm{Neumann},~\bfnm{John}\binits{J.}} \AND
  \bauthor{\bsnm{Goldstine},~\bfnm{H.~H.}\binits{H.~H.}}
(\byear{1947}).
\btitle{Numerical inverting of matrices of high order}.
\bjournal{Bull. Amer. Math. Soc.}
\bvolume{53}
\bpages{1021--1099}.
\bdoi{10.1090/S0002-9904-1947-08909-6}
\bmrnumber{24235}
\end{barticle}
\endbibitem

\bibitem{weyl}
\begin{bbook}[author]
\bauthor{\bsnm{Weyl},~\bfnm{H.}\binits{H.}}
(\byear{1939}).
\btitle{The classical groups: their invariants and representations}.
\bpublisher{Princeton University Press}, \baddress{Princeton, NJ}.
\end{bbook}
\endbibitem

\bibitem{wigner}
\begin{barticle}[author]
\bauthor{\bsnm{Wigner},~\bfnm{E.~P.}\binits{E.~P.}}
(\byear{1955}).
\btitle{Characteristic vectors of bordered matrices with infinite dimensions}.
\bjournal{Annals Math.}
\bvolume{62}
\bpages{548--564}.
\end{barticle}
\endbibitem

\bibitem{wishart}
\begin{barticle}[author]
\bauthor{\bsnm{Wishart},~\bfnm{J.}\binits{J.}}
(\byear{1928}).
\btitle{The generalized product moment distribution in samples from a Normal
  multivariate population}.
\bjournal{Biometrika}
\bvolume{20A}
\bpages{32--52}.
\end{barticle}
\endbibitem

\end{thebibliography}


\end{document}